\documentclass[twoside,11pt]{article}
\usepackage[abbrvbib, preprint]{jmlr2e}

\usepackage{jmlr2e}

\usepackage[english]{babel}
\usepackage[T1]{fontenc}
\usepackage[utf8]{inputenc}
\usepackage{amsmath}
\usepackage{bbding}
\usepackage{pifont}
\usepackage{wasysym}
\usepackage{amssymb}
\usepackage[export]{adjustbox}
\usepackage{algpseudocode}
\usepackage{accents}
\usepackage{mathrsfs}
\usepackage{graphicx}
\usepackage{tikz-cd}
\usetikzlibrary{arrows.meta, positioning}
\usepackage{mathtools}
\usepackage{caption}
\usepackage{subcaption}
\usepackage{algorithm}
\usepackage{algpseudocode}
\usepackage{multirow}
\usepackage{rotating}
\usepackage{tabularx}
\usepackage{booktabs}
\usepackage{makecell}

\usepackage{hyperref}

\hypersetup{
  colorlinks=true,
  linkcolor=blue,
  citecolor=black,
  urlcolor=blue}

\usepackage{mathtools}

\hypersetup{ hidelinks }

\tikzset{%
  symbol/.style={
    draw=none,
    every to/.append style={
      edge node={node [sloped, allow upside down, auto=false]{$#1$}}
    },
  },
}

\usepackage{caption}
\usetikzlibrary{positioning}

\newcommand*{\dd}{\ensuremath{\mathrm{d}}}
\newcommand*{\mathbbL}{\ensuremath{\mathbb{L}^2}}
\newcommand*{\din}{\ensuremath{d_{{in}}}}
\newcommand*{\dout}{\ensuremath{d_{{out}}}}
\newcommand*{\dhidden}{\ensuremath{d_{{h}}}}
\newcommand*{\numlayer}{\ensuremath{n_{{l}}}}

\newcommand*{\bigla}{\ensuremath{\Big\langle}}
\newcommand*{\bigra}{\ensuremath{\Big\rangle}}
\newcommand*{\ut}{\ensuremath{\textbf{U}_{\theta_t}}}
\newcommand*{\varphit}{\ensuremath{\boldsymbol{\Phi}_{\eta_t, \xi_t }}}
\newcommand*{\err}{\ensuremath{\textrm{err}}}
\newcommand*{\Utheta}{\ensuremath{\mathbf U_\theta}}
\newcommand*{\Phietaxi}{\ensuremath{\boldsymbol{\Phi}_{\eta, \xi}}}

\newcommand*{\Uthetat}{\ensuremath{\mathbf U_{\theta_t}}}
\newcommand*{\Phietaxit}{\ensuremath{\boldsymbol{\Phi}_{\eta_t, \xi_t}}}
\newcommand*{\varphietat}{\ensuremath{\boldsymbol{\varphi}_{\eta_t}}}
\newcommand*{\psixit}{\ensuremath{\boldsymbol{\psi}_{\xi_t}}}

\newcommand*{\parUt}{\ensuremath{\partial \mathbf U_{\theta_t}}}
\newcommand*{\parPhit}{\ensuremath{\partial \boldsymbol{\Phi}_{\eta_t, \xi_t}}}
\newcommand*{\parvarphit}{\ensuremath{\partial \boldsymbol{\varphi}_{\eta_t}}}
\newcommand*{\parpsit}{\ensuremath{\partial \boldsymbol{\psi}_{\xi_t}}}
\newcommand*{\dotut}{\ensuremath{\dot{\textbf{U}}_{\theta_t}}}

\newcommand*{\primalspc}{\ensuremath{\widetilde{\mathbb{H}}} }

\mathtoolsset{showonlyrefs}

\newcommand{\blue}[1]{\textcolor{black}{#1}}

\usepackage{lastpage}
\jmlrheading{27}{2026}{1-\pageref{LastPage}}{11/24; Revised
1/26}{5/26}{24-1929}{Shu Liu, Stanley Osher, Wuchen Li}
\ShortHeadings{A Natural Primal-Dual Gradient Method for Solving PDEs}{Liu, Osher, and Li}

\begin{document}

\title{A Natural Primal-Dual Hybrid Gradient Method for Adversarial Neural Network Training on Solving \\ Partial Differential Equations}

\author{\name Shu Liu \email sliu11@fsu.edu \\
       \addr Department of Mathematics\\
       Florida State University, \\
       Tallahassee, FL 32306, USA
       \AND
       \name Stanley Osher \email sjo@math.ucla.edu \\
       \addr Department of Mathematics\\
       University of California, Los Angeles\\
       Los Angeles, CA 90095, USA
       \AND
       \name Wuchen Li \email wuchen@mailbox.sc.edu \\
       \addr Department of Mathematics\\
       University of South Carolina\\
       Columbia, SC 29208, USA
       }

\editor{Weijie Su}

\maketitle
\begin{abstract}
We propose a scalable preconditioned primal-dual hybrid gradient algorithm for solving partial differential equations (PDEs). We multiply the PDE with a dual test function to obtain an inf-sup problem whose loss functional involves lower-order differential operators. The Primal-Dual Hybrid Gradient (PDHG) algorithm is then leveraged for this saddle point problem. By introducing suitable precondition operators to the proximal steps in the PDHG algorithm, we obtain an alternative natural gradient ascent-descent optimization scheme for updating the neural network parameters. We apply the Krylov subspace method (MINRES) to evaluate the natural gradients efficiently. Such treatment readily handles the inversion of precondition matrices via matrix-vector multiplication. An \textit{a posteriori} convergence analysis is established for the time-continuous version of the proposed algorithm for general linear PDEs. By incorporating appropriate boundary loss terms, we further obtain a refined \textit{a priori} convergence result for elliptic equations in divergence form. The algorithm is tested on various types of PDEs with dimensions ranging from $1$ to $50$, including linear and nonlinear elliptic equations, reaction-diffusion equations, and Monge-Amp\`ere equations stemming from the $L^2$ optimal transport problems. We compare the performance of the proposed method with several commonly used deep learning algorithms such as physics-informed neural networks (PINNs), the DeepRitz method and weak adversarial networks (WANs) using either the Adam or the L-BFGS optimizer. The numerical results suggest that the proposed method performs efficiently and robustly and converges more stably with higher accuracy.
\end{abstract}

\begin{keywords}
    Deep learning for solving PDEs; Neural Networks; Inf-sup problem; Primal-Dual Hybrid Gradient (PDHG) algorithm; Natural Gradient; Convergence analysis; Monge-Amp\`ere equation.
\end{keywords}

\section{Introduction}\label{sec: introduction}

Machine learning, particularly deep learning, is a fast-developing direction with modern computational technologies \citep{amari1998natural} and applications \citep{goodfellow2020generative,arjovsky2017wasserstein}. Typical examples of applications often come from computer science, including creating new images, videos, and voices and generating languages. During the development of modern applications, machine learning has introduced a variety of nonlinear methodologies, including computational nonlinear models such as neural networks, as well as variational frameworks such as generative adversarial networks \citep{goodfellow2020generative}. The impact of machine learning on scientific computing has therefore been profound and cannot be overstated.

In recent years, deep learning algorithms have been developed to solve partial differential equations (PDEs). The physics-informed neural networks (PINN) method \citep{raissi2019physics} employs neural networks to approximate PDE solutions by minimizing the discrepancy between observed data and the equation's residual. The DeepRitz method \citep{yu2018deep} computes neural network surrogate solutions for PDEs using a variational approach, minimizing the associated energy functional. The Forward-Backward Stochastic Differential Equation (FBSDE) method \citep{han2017deep} makes use of the nonlinear Feynman-Kac formula for semi-linear parabolic equations to derive numerical solutions at specific time-space points. \blue{Additionally, the Weak Adversarial Network (WAN) approach \citep{zang2020weak, cai2024weak} leverages the weak formulation of PDEs by multiplying the original equation with a test function, resulting in an inf-sup saddle point problem for computing the equation.} This approach is applicable to various types of equations and is scalable to PDEs in high dimensions.

While these methods demonstrate the potential of applying machine learning techniques in solving PDEs, challenges such as hyperparameter tuning, loss function designing, and convergence guarantees remain unresolved. More critically, due to the nonlinearity of neural networks, conventional optimizers such as Adam \citep{kingma2014adam} or RMSProp \citep{tieleman2012lecture} suffer from strong fluctuations and do not achieve stable convergence, which complicates the implementation of current algorithms.

In this research, we aim to tackle these challenges by adopting the adversarial training strategy and propose a crucial preconditioned optimizer that takes advantage of the primal-dual hybrid gradient (PDHG) algorithm \citep{zhu2008efficient, chambolle2011first}. We utilize suitable preconditioned gradients known as the natural gradients \citep{muller2023achieving} to update the parameters of the neural networks. The proposed algorithm, named the Natural Primal-Dual Hybrid Gradient (NPDG) method, performs efficiently and converges more stably than classical machine learning-based PDE solvers. In addition, we also provide a theoretical convergence guarantee for the proposed algorithm.

To illustrate the main idea, we consider the following linear equation posed with suitable boundary condition,
\begin{equation}
  \mathcal L u = f ~ \textrm{on } \Omega,  \quad  \mathcal B u = g ~ \textrm{on } \partial \Omega. \label{intro linear PDE}
\end{equation}
Here $\Omega\subset\mathbb{R}^d$ is a bounded open region, $\partial \Omega$ denotes the boundary of $\Omega$, $f\colon \Omega\rightarrow\mathbb{R}$, $g\colon\partial\Omega\rightarrow\mathbb{R}$ are $L^2$ functions, 
and $u\colon \Omega \rightarrow \mathbb{R}$ belongs to $H^2(\Omega)$. We assume $\mathcal L$ being a second-order elliptic operator and $\mathcal B$ as a linear boundary operator, which indicates the Dirichlet or Neumann or more general boundary conditions. Here we assume that $\mathcal L$ can be split as $\mathcal L = \mathcal M_d^* \widetilde{\mathcal L} \mathcal M_p$ with $\mathcal M_p, \mathcal M_d$ being first-order differential operators, $\widetilde{ \mathcal{L}}$ is a well-conditioned bounded linear operator, and $\mathcal M_d^*$ denotes the $L^2$ adjoint of $\mathcal M_d$. Suppose this equation admits a unique classical solution $u_*\in C^2(\Omega)\cap C(\bar{\Omega})$. The goal is to efficiently compute $u_*$. 

By introducing the test functions (dual variables) $\varphi \in H_0^1(\Omega)$ and $\psi \in L^2(\partial \Omega)$ into the equation \eqref{intro linear PDE}, we consider the following inf-sup problem with quadratic regularization terms. Here, $\blue{\nu} > 0$ denotes the regularization coefficient,
\begin{equation}
\begin{split}\label{intro inf sup problem for root-finding} 
  \inf_{u} ~ \sup_{\varphi, \psi} ~ \mathscr{E}(u, \varphi, \psi) := &  \langle \widetilde{\mathcal L} \mathcal M_p u, \mathcal M_d \varphi \rangle_{L^2(\Omega)} - \langle f , \varphi \rangle_{L^2(\Omega)} - \frac{\blue{\nu}}{2} \|\mathcal M_d\varphi\|_{L^2(\Omega)}^2   \\
  & + \langle \mathcal B u - g, \psi \rangle_{L^2(\partial \Omega)} - \frac{\blue{\nu}}{2}\|\psi\|^2_{L^2(\partial \Omega)}.  
\end{split}
\end{equation}
Note that $u=u_*, \varphi=0, \psi=0$ form the saddle point of $\mathscr{E}$. In order to approach this saddle point and hence solve for $u_*$, we apply a preconditioned version of the PDHG algorithm to the inf-sup problem \eqref{intro inf sup problem for root-finding}.  The algorithm utilizes alternative proximal point steps and an intermediate extrapolation to solve the inf-sup problem with selected preconditioning operators $\mathcal M_p, \widetilde{\mathcal L}, \mathcal M_d$. More specifically, the algorithm repeats the following three-line iteration
\begin{equation}
\begin{split}\label{intro precond PDHG functional space}
  & (\varphi_{n+1}, \psi_{n+1}) = \underset{ \varphi, \psi}{\mathrm{argmin}} ~ \left\{ \frac{1}{2\tau_\varphi}(\|\mathcal M_d \varphi - \mathcal M_d \varphi_n\|^2_{L^2(\Omega)} + \|\psi - \psi_n\|^2_{L^2(\partial \Omega)}) - 
  \mathscr{E}(u_n, (\blue{\varphi}, \blue{\psi}))
   \right\},  \\ 
  & \widetilde{\varphi}_{n+1} = \varphi_{n+1} + \omega (\varphi_{n+1} - \varphi_n),  \quad  \widetilde{\psi}_{n+1} = \psi_{n+1} + \omega (\psi_{n+1} - \psi_n) \\
  & u_{n+1} =  \underset{u}{\mathrm{argmin}} ~ \left\{ \frac{1}{ 2\tau_{u} }(\|\mathcal M_p u - \mathcal M_p u_n\|_{L^2(\Omega)}^2 + \|\mathcal B u - \mathcal B u_n\|_{L^2(\partial \Omega)}^2) + \mathscr{E} (\blue{u}, (\widetilde{\varphi}_{n+1}, \widetilde{\psi}_{n+1}))  \right\}.
\end{split}
\end{equation}
Here $\tau_u, \tau_\varphi > 0$ are the step sizes of the algorithm and $\omega > 0$ denotes the extrapolation coefficient. We briefly illustrate the motivation of preconditioning steps in \eqref{intro precond PDHG functional space}. In general, the differential operator $\mathcal L$ is usually ill-conditioned.  As shown in \citep{liuJCP2023}, the convergence rate of the un-preconditioned dynamic equals $1 - \mathcal O(\frac{1}{\kappa^2})$ with $\kappa$ denoting the condition number of the spatial discretization of $\mathcal L$. The convergence speed decreases fast as $\kappa$ gets larger. This pronounced slowdown motivates us to introduce appropriate preconditioning in the proximal steps (i.e., the first and third lines) of \eqref{intro precond PDHG functional space} in order to mitigate the resulting inefficiency. %

So far, the algorithm we have developed remains at the functional level, which is generally intractable for practical implementation. To realize the proposed PDHG algorithm, we parameterize $u(\cdot)$, $\varphi(\cdot)$ and $\psi(\cdot)$ as $u_\theta(\cdot), \varphi_\eta(\cdot)$ and $\psi_\xi(\cdot)$ with the tunable parameters $\theta\in \Theta_\theta \subseteq \mathbb{R}^{m_\theta}, \eta\in \Theta_\eta \subseteq \mathbb{R}^{m_\eta}$ and $\xi \in \Theta_\xi \subseteq \mathbb{R}^{m_\xi}$.  A straightforward parameterization approach involves expressing these functions as linear combinations of predefined basis functions---a method traditionally employed in finite element methods. However, as the problem's dimensionality increases, such parameterization becomes computationally prohibitive due to the curse of dimensionality, since it requires a significant number of basis functions to maintain accuracy \citep{hu2024tackling}. Recent advances in deep learning have highlighted the potential of neural networks as computational tools to solve PDEs. Given their flexibility and expressive power, we adopt three neural network functions, such as Multilayer Perceptrons (MLPs, see Appendix \ref{append: MLP} for detailed definition), to represent $u_\theta, \varphi_\eta,$ and $ \psi_\xi$. Therefore, we reduce the original algorithm in functional spaces to a time-discrete dynamic in which the parameters $\theta^n, \eta^n, \xi^n$ evolve together.

We replace the implicit proximal step for updating \( \eta, \xi, \theta \) with an explicit scheme known as the linearized PDHG algorithm. We come up with the following algorithm:
\begin{equation}
\begin{split}\label{def: intro alg NPDG}
  \left[\begin{array}{c}
          \eta^{n+1}\\
          \xi^{n+1}
        \end{array}
  \right] = & \left[\begin{array}{c} 
                        \eta^n    \\
                        \xi^n
                      \end{array}
                     \right]
                +  \tau_\varphi 
                \left[ \begin{array}{c} 
                M_d(\eta^n)^\dagger \nabla_\eta {\mathscr{E}}(u_{\theta^n}, \varphi_{\eta^n}, \psi_{\xi^n})  \\
                M_{bdd}(\xi^n)^\dagger  \nabla_\xi {\mathscr{E}}(u_{\theta^n}, \varphi_{\eta^n}, \psi_{\xi^n})
                \end{array} \right], \\
  \left[ 
  \begin{array}{c}
    \widetilde{\varphi}_{n+1} \\
    \widetilde{\psi}_{n+1}
  \end{array}
  \right] = & \left[ \begin{array}{c}
    \varphi_{\eta^{n+1}} \\
    \psi_{\xi^{n+1}}
  \end{array}
  \right] 
  + \omega \left(
  \left[ \begin{array}{c}
    \varphi_{\eta^{n+1}} \\
    \psi_{\xi^{n+1}}
  \end{array}
  \right] - \left[ \begin{array}{c}
    \varphi_{\eta^{n}} \\
    \psi_{\xi^{n}}
  \end{array}
  \right]
  \right),  \\
  \theta^{n+1} = & ~  \theta^n - \tau_u M_p(\theta^n)^\dagger \nabla_\theta {\mathscr{E}}(u_{\theta^n}, \widetilde{\varphi}_{n+1}, \widetilde{\psi}_{n+1}  ).
\end{split}
\end{equation}
Here $M_d(\eta)\in \mathbb{R}^{m_\eta\times m_\eta}$, $M_{bdd}(\xi)\in\mathbb{R}^{m_\xi\times m_\xi}$, $M_{p}(\theta)\in\mathbb{R}^{m_\theta\times m_\theta}$ are Gram type matrices. They are derived from the bilinear form approximation of the proximal steps in the PDHG algorithm \eqref{intro precond PDHG functional space}. Here, we denote ``$^\dagger$'' as the Moore–Penrose inverse of a matrix. The precondition matrix $M_d(\eta^n)$ incorporates the information of the precondition operator $\mathcal M_p$, which is built in the original operator $\mathcal L$; we call $M_d(\eta^n)^\dagger \nabla_\eta {\mathscr{E}}(u_{\theta^n}, \varphi_{\eta^n}, \psi_{\xi^n})$ the \textit{natural gradient} of ${\mathscr{E}}(u_\theta, \varphi_\eta, \psi_\xi)$ with respect to $\eta$. Similarly, we can define the natural gradient ascent and descent directions for variables $\xi$ and $\theta$. The algorithm alternatively updates the primal and dual parameters along the natural gradient directions. An additional extrapolation step in the functional space is introduced to enhance the convergence of the method. We denote the above updates as the \textbf{Natural Primal-Dual Hybrid Gradient} algorithm. For simplicity, we refer to this method as the \textbf{NPDG} algorithm in the following discussion. We refer the readers to Section \ref{sec: derivation} for a detailed derivation of the algorithm. 

While the NPDG algorithm is designed around linear PDEs, it effectively accommodates equations with nonlinear terms. Additionally, it can be extended to address certain fully nonlinear equations, such as the Monge-Ampère equation, which emerges in the context of the $L^2$ optimal transport (OT) problem \citep{villani2021topics, de2014monge}. Since the OT problem can be formulated as a constrained optimization problem, introducing the Lagrange multiplier method leads to a saddle point scheme. This scheme involves adversarial training with the pushforward map and the dual potential function to solve the Monge-Amp\`ere equation, substituting both the map and potential function with neural network approximations and applying the NPDG algorithm with precondition matrices. The $L^2$ Gram type matrices lead to stable and efficient numerical results. Further analysis around the saddle point of the loss function suggests a more canonical preconditioning approach, where the mapping still uses the $L^2$ Gram type matrix while the potential uses the $H^1$ Gram type matrix. For a detailed discussion, readers are referred to Section \ref{subsec: nonlinear eq}.

In this research, we provide an \textit{a posteriori} convergence analysis for the time-continuous version of the NPDG algorithm when applied to the general linear PDE \eqref{intro linear PDE}. Let $(\theta_t, \eta_t, \xi_t)$ be the solution obtained from the time-continuous algorithm for $0\leq t \leq T$. Under specific conditions regarding the approximation capabilities of the tangent spaces spanned by $\{\partial_{\theta_k}u_{\theta_t}\}, \{\partial_{\eta_k}\varphi_{\eta_t}\},$ and $\{\partial_{\xi_k}\psi_{\xi_t}\}$, we establish the linear convergence of the numerical solution $u_{\theta_t}$
\begin{equation*}
  \|\mathcal M_p (u_{\theta_t} - u_* )\|_{L^2( \Omega )}^2 + \lambda \|\mathcal B (u_{\theta_t} - u_*) \|_{L^2(\partial \Omega) }^2 \leq \blue{C_0} \cdot \exp(-rt)   \quad \textrm{for } 0 \leq t \leq T.
\end{equation*}
Here \blue{$C_0 > 0$} \blue{denotes the initial error}, $r>0$ is the convergence rate depending on the preconditioned operator $\widetilde{\mathcal L}$, the hyperparameters of the NPDG algorithm, and the neural network parameters. \blue{An explicit lower bound for \( r \) in the case where \( u_\theta \), \( \varphi_\eta \), and \( \psi_\xi \) are linear combinations of basis functions is provided in \eqref{explicit convergence rate origin analysis}. A fast convergence rate of the time-continuous algorithm can be expected when the operator $\widetilde{\mathcal L}$ is well-conditioned and the hyperparameters are appropriately chosen.} 

{\color{black} We further refine the above convergence analysis and remove the restriction on finite time horizon \( [0,T] \) for an important class of elliptic equations in divergence form. By incorporating a boundary loss term stronger than the standard \( L^2(\partial \Omega) \) norm, we establish an \emph{a priori} convergence bound in the setting where \( u_\theta \), \( \varphi_\eta \), and \( \psi_\xi \) are linear combinations of prescribed basis functions \( \{u_k\} \), \( \{\varphi_k\} \), and \( \{\psi_k\} \), respectively. More precisely, we have
\[
\Big(
    \|\nabla u_{\theta_t}-\nabla u_*\|_{L^2}^2
    + \lambda \|u_{\theta_t}-u_*\|_{\mathcal X}^2
\Big)^{\frac12}
\;\le\;
C_0 e^{-rt}
+ \frac{1-e^{-rt}}{r}
\Big(
    C_1 \sqrt{\mathcal E_u}
    + C_2 \sqrt{\mathcal E_{\nabla \varphi}}
    + C_3 \sqrt{\mathcal E_\psi}
\Big).
\]
Here, \( \|\cdot\|_{\mathcal X} \) denotes a boundary norm stronger than \( L^2(\partial \Omega) \), such as \( H^{1/2}(\partial \Omega) \) or \( H^1(\partial \Omega) \). The constant \( C_0 \) depends on the initial error, while \( C_1, C_2, \) and \( C_3 \) are coefficients determined by the {elliptic operator} and the choice of hyperparameters. The quantities \( \mathcal E_u \), \( \mathcal E_{\nabla \varphi} \), and \( \mathcal E_\psi \) represent the approximation errors of the chosen basis functions in approximating the exact solution \( u_* \) and the boundary data \( g \). Moreover, the convergence rate \( r \) admits a uniform lower bound \( r \ge \frac{2}{3\sqrt{3}} \) under a suitable selection of hyperparameters. This result reveals a clear two-phase behavior of the dynamics: at early stages, the numerical error is dominated by the decay of the initial error, while in the long-time regime, the error is governed by the intrinsic approximation of the chosen basis functions.} The readers are referred to Section \ref{section: Numer Anal} for detailed discussions on this series of results.

In implementation, we apply the Monte-Carlo algorithm to approximate $\mathscr{E}(u_\theta, \varphi_\eta, \psi_\xi)$; we use automatic differentiation to compute the derivatives of $\mathscr{E}(u_\theta, \varphi_\eta, \psi_\xi)$ with respect to the parameters $\theta, \eta, \xi$. It is usually prohibitively expensive to explicitly form the precondition matrices $M_p(\theta), M_d(\eta), M_{bdd}(\xi)$ given that $m_\theta, m_\eta, m_\xi$ might be very large. To cope with this, we evaluate the pseudo-inverse in \eqref{def: intro alg NPDG} via the iterative solver such as the Minimal residual method (MINRES) \citep{paige1975solution}, which, instead of forming entire matrices, only requires matrix-vector multiplication. Further details of the implementation can be found in Section \ref{sec: alg}.

Numerical examples of linear PDEs \eqref{intro linear PDE}, nonlinear PDEs \eqref{def: L+N equation }, and Monge-Ampère equations \eqref{insection Monge Ampere} in Section \ref{sec: numerical examples } illustrate the accuracy, efficiency, and robustness of the NPDG method compared to classical methods, including the Physics-Informed Neural Network (PINN), the DeepRitz method, and the Weak Adversarial Network (WAN). Based on these numerical results, the algorithm demonstrates linear convergence for the high-dimensional PDEs tested in this section. Additionally, the proposed method converges more efficiently and achieves higher accuracy in both $L^2$ and $H^1$ norms compared to the other tested methods.

\subsection{Related references}\label{subsec: literature}

In recent years, machine learning algorithms have attracted increasing attention from the scientific computing community due to their flexibility and scalability. A considerable amount of these investigations are based on the Physics-Informed Neural Network (PINN) algorithm \citep{raissi2019physics, lu2021deepxde}; further approaches that address the pathologies during PINN training include calibration of interior-boundary loss coefficients \citep{wang2022and}, and variable splitting techniques \citep{basir2022investigating, park2024beyond}. The adaptive sampling methods \citep{tang2022adaptive, tang2023pinns} are introduced to gain better accuracy of the neural network approximation. In addition to PINNs, a range of deep learning-based algorithms is introduced for solving various types of PDEs, demonstrating scalability to high-dimensional problems. These include the Deep Galerkin Method \citep{sirignano2018dgm}, Deep Ritz method \citep{yu2018deep, deepritzCOLT, liu2023deep}, Forward-Backward Stochastic Differential Equation (FBSDE) approaches \citep{han2017deep, han2018solving, hutzenthaler2021multilevel}, Extreme Learning Machines \citep{dong2021local, ni2023numerical, wang2024extreme}, Tensor Neural Networks \citep{wang2022tensor, wang2024solving}, etc.

Recent research trends leverage adversarial training strategies \citep{goodfellow2020generative, arjovsky2017wasserstein} to improve algorithm performance. In the Weak Adversarial Network (WAN) algorithm, discriminator neural networks are used to enhance training efficiency by employing the weak formulation of PDEs \citep{zang2020weak, bao2020numerical}. \blue{The weak formulation is further employed in \citep{cai2024weak} to train a generative model for generating samples from the invariant measure of stochastic dynamics.} Additionally, a residual-attention-based approach has been introduced in \citep{mcclenny2020self, mcclenny2023self, anagnostopoulos2023residual, zeng2022competitive} for seeking numerical solutions with higher precision.

The Primal-Dual Hybrid Gradient (PDHG) method, which is widely used in image processing problems \citep{zhu2008efficient, chambolle2011first}, has been introduced to handle nonlinear PDEs on classical numerical schemes \citep{liu2022primal, liuJCP2023, meng2023primal}. Suitable preconditioning is introduced to improve the convergence of the algorithm significantly. The method is shown to converge linearly in \citep{liu2024numerical}.

Large-scale optimization algorithms play a crucial role in machine learning research. Stochastic gradient descent (SGD) is a widely used first-order optimization method \citep{robbins1951stochastic, saad1998online, bottou-bousquet-2008}. One can improve the SGD's performance by incorporating momentum terms \citep{rumelhart1986learning, nag, su2016differential}. Various modified versions of SGD with per-parameter learning rates—such as AdaGrad \citep{duchi2011adaptive}, Adadelta \citep{zeiler2012adadelta}, RMSProp \citep{tieleman2012lecture}, and Adam \citep{kingma2014adam}—are popular optimizers in deep learning \citep{Pytorch}. Additionally, second-order algorithms like the BFGS method \citep{fletcher2000practical}, LBFGS method \citep{liu1989limited}, and inexact-Newton methods \citep{dembo1982inexact, brown1990hybrid, brown1994convergence, eisenstat1994globally, martens2010deep, roosta2022newton, rathore2024challenges} are also widely explored in machine learning research.

The natural gradient method is another critical category of second-order optimizers, initially introduced in \citep{amari1998natural} with further developments in \citep{amari2016information, thomas2016energetic, song2018accelerating}. An efficient, scalable variant known as the K-FAC (Kronecker-factored Approximate Curvature) method was proposed in \citep{martens2015optimizing}. The natural gradient method finds its application under different scenarios, including optimization involving combined loss functionals \citep{ying2021natural}, PDE-constrained optimization \citep{nurbekyan2023efficient}, simulation and acceleration of Wasserstein gradient flows \citep{LM1, CL, wang2020information, shen2020sinkhorn, liuPFPE}. A series of research that utilizes the concept of the natural gradient to solve general time-dependent PDEs have been conducted, as detailed in \citep{du2021evolutional, bruna2024neural, gaby2023neural, chenteng} and the references therein. Recently, a natural gradient primal-dual algorithm for decentralized learning problems is proposed in \citep{NGPDdecentralized}.

The natural gradient algorithm has recently been applied to training PINNs, achieving highly accurate solutions \citep{muller2023achieving}. The K-FAC method is exploited in the follow-up work \citep{dangel2024kronecker} to enable scalability in high-dimensional settings. Beyond natural gradients, the Gauss-Newton method has been introduced in \citep{hao2024gauss} for computing variational PDEs. Additional preconditioning techniques for solving PDEs include the multigrid-augmented method \citep{azulay2022multigrid}, domain decomposition strategies \citep{kopanivcakova2024enhancing}, and incomplete LU preconditioning \citep{liu2024preconditioning}. However, these methods typically need to scale more effectively to compute high-dimensional problems.

Compared to these methods, we summarize the advantages of the proposed approach in two key aspects: the primal-dual hybrid gradient algorithmic framework and the application of natural gradients in neural network functions.
\begin{list}{$\bullet$}{}
\item On the primal-dual framework:
\begin{itemize}
    \item By applying integration by parts, we reduce the order of the differential operator $\mathcal L$ in the primal-dual formulation, lowering computational complexity when performing automatic differentiation on the neural networks.
    \item The proposed primal--dual training scheme is versatile and adaptable, rendering the algorithm applicable to a broad class of partial differential equations, including \blue{linear elliptic problems, semi-linear equations with dominant viscosity terms, fully nonlinear PDEs such as the Monge--Amp\`ere equation, etc.}
\end{itemize}
\item On the primal-dual hybrid natural gradients:
\begin{itemize}
    \item Unlike other second-order optimization algorithms, such as L-BFGS, which are unable to handle the training involving random batches, the proposed algorithm is well-suited to data stochasticity, performing robustly under stochastic approximation.
    \item To address the computation of large-scale linear systems (specifically, the pseudo-inverse of preconditioning matrices), we introduce the iterative method (MINRES). Consequently, our approach readily accommodates high-dimensional PDEs requiring neural networks with a large number of parameters. In experiments, we handle neural networks with parameter counts ranging from 20,000 to 300,000.
\end{itemize}
\end{list}
Generally, the proposed algorithm converges smoothly, avoiding the intense fluctuations and spikes commonly observed in the loss decay curves of classical momentum-based optimizers such as Adam and RMSProp. With appropriate preconditioning, theoretical analysis (Theorem \ref{thm: convergence analysis of NPDG flow}, Theorem \ref{thm: convergence analysis on elliptic PDE of divergence type}) indicates linear convergence of the method. In practice, the approach performs more efficiently than classical machine learning methods and achieves higher precision in the norms $L^2$ and $H^1$. Furthermore, as reflected in later Table \ref{tab: GPU time to accuracy }, the method demonstrates robustness with respect to its hyperparameters, including the regularization coefficient $\blue{\nu}$, step sizes $\tau_\varphi$, $\tau_u$, and the extrapolation coefficient $\omega$. Typically, a standard configuration of $\blue{\nu} = 0.1, \tau_\varphi = 0.095$, $\tau_u = 0.05$, and $\omega = 1$ yields satisfactory performance.

This paper is organized as follows. In Section \ref{sec: derivation}, we provide a detailed derivation of the algorithm. Supplementary discussions on treating the semi-linear PDEs and the Monge-Amp\`ere equations are provided in Section \ref{subsec: nonlinear eq}. Then, in Section \ref{section: Numer Anal}, we establish a series of convergence analysis results for the time-continuous version of the algorithm. Implementation details are demonstrated in Section \ref{sec: alg}. We demonstrate the numerical examples in Section \ref{sec: numerical examples }. We provide further materials related to the algorithm, proof, and numerical examples in the Appendix.

\section{Derivation of Natural Primal-Dual Hybrid Gradient (NPDG) method}\label{sec: derivation}
In this section, we provide a detailed derivation of the proposed method by first introducing the Primal-Dual Hybrid Gradient algorithm for root-finding problems. We then apply this algorithm to solving PDEs in the functional space. We improve the algorithm's performance by introducing suitable preconditioning. Finally, we discuss how we realize the algorithm by substituting the functions with neural networks and introduce the Natural Primal-Dual Hybrid Gradient (NPDG) algorithm for adversarial training of the neural networks for solving PDEs.

\subsection{Primal-Dual algorithm for root-finding problem}
We first consider a root-finding problem defined on Hilbert space $\mathbb{X}$,
\begin{equation}
  \mathcal F(x) = 0.  \label{root-finding}
\end{equation}
Here, we assume that $\mathcal F:\mathbb{X}\rightarrow \mathbb{Y}$ is a function from $\mathbb{X}$ to another Hilbert space $\mathbb{Y}$. The goal is to find a solution $x\in \mathbb{X}$. For a certain convex functional $\iota: \mathbb{Y}\rightarrow \mathbb{R}$ that satisfies $\iota(y)>0$ iff $y\neq 0$ and $\iota(y)=0$ whenever $y=0$, the root-finding problem is equivalent to the following minimization problem
\begin{equation}
  \inf_{x\in\mathbb{X}} ~ \iota(\mathcal F(x)).  \label{minimization problem equiv to root finding problem}
\end{equation}
We denote the Legendre dual of $\iota(\cdot)$ as $\iota^*(\cdot)$ which is defined as $\iota^*(y) = \sup_{w\in\mathbb{Y}} ~ \langle y, w   \rangle_{\mathbb{Y}} - \iota(w)$. Here, we denote $\langle \cdot, \cdot \rangle_{\mathbb{Y}}$ as the inner product defined in the space $\mathbb{Y}$. Then 
\begin{equation}
  \iota(z) = \iota^{**}(z) = \sup_{y\in\mathbb{Y}}~ \langle z, y    \rangle_{\mathbb{Y}} - \iota^*(y).  \label{legendre dual of legendre dual}
\end{equation}
Substituting \eqref{legendre dual of legendre dual} into \eqref{minimization problem equiv to root finding problem} yields the following saddle point problem
\begin{equation}
  \inf_{x\in \mathbb{X} } ~ \sup_{y \in \mathbb{Y} } ~ \mathscr{E}(x, y) := \langle \mathcal F(x), y \rangle_{\mathbb{Y}} - \iota^*(y).  \label{inf sup problem for root-finding Hilbert space} 
\end{equation}
We now apply the PDHG algorithm to deal with the inf-sup problem \eqref{inf sup problem for root-finding Hilbert space}, yielding
\begin{align}
  y_{n+1} = & \underset{y\in\mathbb{Y}}{\mathrm{argmin}} ~ \frac{\|y - y_n \|_{\mathbb{Y}}^2}{2\tau_y} - \mathscr{E}(x_n, y) = (\mathrm{Id} - \blue{\tau_y D_y}\mathscr{E}(x_n, \cdot))^{-1}~y_n,  \label{PDHG  prox dual}\\ 
  \widetilde{y}_{n+1} = & y_{n+1} + \omega (y_{n+1} - y_n),  
  \nonumber\\  
  x_{n+1} = & \underset{x\in\mathbb{X}}{\mathrm{argmin}} ~ \frac{\|x - x_n\|_{\mathbb{X}}^2}{2\tau_x} + \mathscr{E}(x, \widetilde{y}_{n+1}) = (\mathrm{Id} + \blue{\tau_x D_x}\mathscr{E}(\cdot, \widetilde{y}_{n+1}))^{-1} ~ x_n.  \label{PDHG prox primal}
\end{align}
Here $\tau_x, \tau_y > 0$ are the step sizes of the PDHG algorithm, $D_x\mathscr{E}\in\mathbb{X}, D_y\mathscr{E}\in\mathbb{Y}$ are the Fr\'echet derivatives and $\omega > 0$ denotes the extrapolation coefficient. The proximal steps \eqref{PDHG  prox dual}, \eqref{PDHG prox primal} can be interpreted as the implicit update of the gradient ascent/descent algorithm of functional $\mathscr{E}$ as $\tau_x, \tau_y$ are small enough. In practice, one can choose $\iota(\cdot) = \chi(\cdot)$, where $\chi$ is the indicator function defined as $\chi(y)=+\infty$ for $y\neq 0$ and $\chi(0)=0$. In this case, the Legendre dual satisfies $\iota^*(\cdot)\equiv 0$. Another popular choice is $\iota(\cdot)=\frac{1}{2\blue{\nu}}\|\cdot\|_{\mathbb{Y}}^2$ with $\iota^*(\cdot)=\frac{\blue{\nu}}{2}\|\cdot\|_{\mathbb{Y}}^2$. Here, $\blue{\nu} > 0$ is a tunable hyperparameter. We will mainly focus on the latter throughout the subsequent discussion of the paper.

\subsection{Primal-Dual Hybrid Gradient algorithm for solving PDEs} \label{subsec: PDHG for PDE}
\blue{From now on, we assume that $\Omega \subset \mathbb{R}^d$ is a bounded open set. }Let us start by considering a linear equation defined on a Hilbert space $\mathbb H$,
\begin{equation}
  \mathcal L u = f ~ \textrm{on } \Omega, \quad \textrm{with boundary condition } \mathcal B u = g ~ \textrm{on } \partial \Omega. \label{linear PDE}
\end{equation}
\blue{We denote $\mathbb{K}\subseteq L^2(\Omega)$, $\mathbb{K}_{\partial \Omega}\subseteq L^2(\partial \Omega)$ as two Hilbert spaces.} Then,
$\mathcal L:\mathbb{H}\rightarrow \mathbb{K}\subseteq L^2(\Omega)$ is a linear differential operator, $\mathcal B: \mathbb{H} \rightarrow \mathbb{K}_{\partial \Omega} \subseteq L^2(\partial \Omega)$ is a linear boundary operator. We assume $u_*\in C^2(\Omega)\cap C(\bar{\Omega})\subset\mathbb{H}$ to be the classical solution to \eqref{linear PDE}.

We now set $\mathcal F: \mathbb{H} \rightarrow \mathbb{K} \times \mathbb{K}_{\partial \Omega},  u \mapsto (\mathcal L u - f, \mathcal B u - g)$. By introducing the \blue{test} variables $\varphi \in \blue{\mathbb{K}^{test}} \subseteq L^2(\Omega)$ and 
$\psi \in \blue{\mathbb{K}^{test}_{\partial \Omega}} \subseteq L^2(\partial \Omega)$, 
and by defining $\mathbb{L}^2 := L^2(\Omega)\times L^2(\partial \Omega)$, 
we arrive at the following saddle-point problem:
\begin{align}
  \inf_{u \in \mathbb{H} } ~ \sup_{
    \substack{ 
      \varphi  \in  \blue{\mathbb{K}^{test}}  \\
      \psi  \in  \blue{\mathbb{K}^{test}_{\partial \Omega}}}
      } ~ \blue{\mathscr{E}_0}(u, \varphi, \psi  ) := &  \langle \mathcal F(u), (\varphi, \psi) \rangle_{ \mathbb{L}^2 } - \frac{\blue{\nu}}{2}\|(\varphi, \psi)\|^2_{\mathbb{L}^2}.  \label{inf sup problem for root-finding} \\
      = & \langle \mathcal L u - f, \varphi \rangle_{L^2(\Omega)} - \frac{\blue{\nu}}{2}\|\varphi\|_{L^2(\Omega)}^2 + \langle \mathcal B u - g, \psi \rangle_{L^2(\partial \Omega)} - \frac{\blue{\nu}}{2}\|\psi\|_{L^2(\partial \Omega)}^2. \nonumber 
\end{align}
It is not hard to verify that $u=u_*, \varphi=0, \psi=0$ form the saddle point of the inf-sup problem \eqref{inf sup problem for root-finding}. We refer to \citep{huo2024inf} for further discussion of the saddle point structure of related inf-sup formulations. \blue{In practice, it is usually convenient to introduce a boundary loss coefficient $\lambda>0$ and consider\footnote{\footnotesize \blue{The new functional is obtained by considering root-finding problem $\mathcal F_\lambda(u)=0$ with $\mathcal F_\lambda: u\mapsto (\mathcal L u - f, \sqrt{\lambda}(\mathcal B u-g))$, and setting $\mathscr{E}_0(u, \varphi, \psi):=\langle \mathcal F_\lambda(u), (\varphi, \sqrt{\lambda}\psi) \rangle_{ \mathbb{L}^2 } - \frac{\blue{\nu}}{2}\|(\varphi, \sqrt{\lambda}\psi)\|^2_{\mathbb{L}^2}.$}},
\[ \mathscr{E}_0(u, \varphi, \psi)=\langle \mathcal L u - f, \varphi \rangle_{L^2(\Omega)} - \frac{\blue{\nu}}{2}\|\varphi\|_{L^2(\Omega)}^2 + \lambda(\langle \mathcal B u - g, \psi \rangle_{L^2(\partial \Omega)} - \frac{\blue{\nu}}{2}\|\psi\|_{L^2(\partial \Omega)}^2). \]}
In this work, we propose the following PDHG algorithm to deal with inf-sup problem \eqref{inf sup problem for root-finding},
{
\begin{equation}
\begin{split}\label{PDHG functional space}
  & \left[\begin{array}{c}
        {\varphi}_{n+1} \\
        {\psi}_{n+1}
      \end{array}\right] = \underset{  (\varphi, \psi)\in \blue{\mathbb{K}^{test}} \times \blue{\mathbb{K}^{test}_{\partial \Omega}} }{\mathrm{argmin}} ~ \left\{ \frac{1}{2\tau_\varphi}(\|\varphi - \varphi_n\|^2_{L^2(\Omega)} + \|\psi - \psi_n\|^2_{L^2(\partial \Omega)}) - 
  \mathscr{E}_0(u_n, \blue{\varphi}, \blue{\psi})
   \right\},  \\ 
  & \left[\begin{array}{c}
        \widetilde{\varphi}_{n+1} \\
        \widetilde{\psi}_{n+1}
      \end{array}\right] = \left[\begin{array}{c}
        \varphi_{n+1} \\
        \psi_{n+1}
      \end{array}\right] + \omega \left(\left[\begin{array}{c}
        \varphi_{n+1} \\
        \psi_{n+1}
      \end{array}\right] - \left[\begin{array}{c}
        \varphi_{n } \\
        \psi_{n }
      \end{array}\right]\right), \\
  & u_{n+1} =  \underset{u\in \mathbb{H}}{\mathrm{argmin}} ~ \left\{ \frac{1}{ 2\tau_{u} }(\|u - u_n\|_{L^2(\Omega)}^2 + \|\mathcal B u - \mathcal B u_n\|_{L^2(\partial \Omega)}^2) + \mathscr{E}_0(\blue{u}, \widetilde{\varphi}_{n+1}, \widetilde{\psi}_{n+1})  \right\}.
\end{split}
\end{equation}}
To develop an intuitive understanding of why the algorithm \eqref{PDHG functional space} has difficulties in approaching the PDE solution, we consider the square region $\Omega$ and discretize it into $N_x^d$ lattices. We apply the finite difference scheme to discretize \eqref{linear PDE} into grids. Solving the PDE yields a linear equation $Ax-b=0$. Here, $A$ is the matrix obtained upon discretizing $\mathcal L$. Roughly speaking, $A\in\mathbb{R}^{N_x^d\times N_x^d}$ is self-adjoint and non-singular, $x\in\mathbb{R}^{N_x^d}$ denotes the numerical solution of the PDE on the grid points, $b\in\mathbb{R}^{N_x^d}$ is the vector encoding $f$ and its boundary condition. The proposed PDHG algorithm yields 
\begin{equation*}
\begin{split}
  y_{n+1} = & \underset{y}{\mathrm{argmin}} ~ \frac{\|y - y_n \|^2}{2\tau_y}  -  (Ax_n-b)^\top y ,  \\ 
  \widetilde{y}_{n+1} = & 2y_{n+1} - y_n,   \\
  x_{n+1} = & \underset{x}{\mathrm{argmin}} ~ \frac{\|x - x_n\|^2}{2\tau_x} + (Ax-b)^\top  \widetilde{y}_{n+1}.
\end{split}
\end{equation*}
Here, we set $\blue{\nu}=0$ and $\omega=1$ to simplify the discussion. And $\|\cdot\|$ denotes the $\ell_2$ norm of $\mathbb{R}^N$. It is not hard to verify that the above algorithm is equivalent to the following update: 
\begin{equation*}
\begin{split}
\left[\begin{array}{c}
    \widehat{x}_{n+1}  \\
    y_{n+1}  
\end{array}\right] & =   \underbrace{\left[\begin{array}{cc}
     I - 2\tau_x\tau_y A^\top A & -\tau_x A^\top \\
      \tau_y A &    I
\end{array}\right] }_{\textrm{denote as }\Gamma }  \left[\begin{array}{c}
    \widehat{x}_n  \\
    y_n  
\end{array}\right]. \\ 
\end{split}
\end{equation*}
Here, we denote $x_*$ as the solution to $Ax-b=0$ and $\widehat{x}_n = x_n - x_*.$ The convergence rate of the PDHG algorithm depends on the spectrum radius $\rho(\Gamma)$ of $\Gamma$. The value of $\rho(\Gamma)$ equals $\sqrt{1 - \frac{c}{\kappa^2}}$, where $c\in [1, \frac{4}{3})$ and $\kappa$ denotes the condition number of $A$ \blue{(we refer readers to Theorem 1 in \citep{liuJCP2023} for a detailed discussion)}. For Laplace operator $\mathcal L = \Delta$, the matrix $A$ obtained via central difference scheme takes condition number $\kappa=\mathcal O(N_x^2)$ \citep{kulkarni1999eigenvalues}. This indicates that the convergence rate of the PDHG method is $\sqrt{1-\frac{c}{\kappa^2}}=1-\mathcal O(\frac{1}{N_x^4})$, which is very inefficient as $N_x$ increases.  

\subsection{Preconditioning of the primal-dual algorithm}\label{subsec: precond PDHG method for PDE}
The discussion in Section \ref{sec: introduction} suggests that we should introduce preconditioning to the original algorithm \eqref{PDHG functional space}. As mentioned previously, we assume that $\mathcal L$ admits the splitting $\mathcal L = \mathcal M_d^*\, \widetilde{\mathcal L}\, \mathcal M_p,$ where $\mathcal M_d^*$, $\widetilde{\mathcal L}$, and $\mathcal M_p$ are linear differential operators \blue{acting between the functional spaces specified below.}
\[
{\color{black}
\begin{tikzcd}[
  column sep=3.2em,
  row sep=2.2em,
  ampersand replacement=\&,
  every label/.append style={font=\small},
  nodes={font=\normalsize},
]
 \mathbb H
  \arrow[r, "\mathcal M_p"]
\& \widetilde{\mathbb H}
  \arrow[r, "\widetilde{\mathcal L}"]
\& \widetilde{\mathbb K}
  \arrow[r, "\mathcal M_d^{*}"]
\& \mathbb K \subseteq L^2(\Omega)
\\
\& \&[-0.6em]
L^2(\Omega;\mathbb R^r)\supseteq {\widetilde{\mathbb K}^{test}}
\&[-0.2em]
{\mathbb K^{test}}
  \arrow[l, "\mathcal M_d"']\subseteq L^2(\Omega)
\end{tikzcd}}
\]
Here we assume $\widetilde{\mathbb{H}}, \widetilde{\mathbb{K}}\subseteq L^2(\Omega; \mathbb{R}^r)$ are Hilbert spaces. Moreover, $\mathcal M_d:\blue{\mathbb{K}^{test}} \rightarrow \blue{\widetilde{\mathbb{K}}^{test}}$ is a linear operator. $\mathcal M_d^*$ is treated as the ``adjoint'' of $\mathcal M_d$ in the sense of
\begin{equation*}
  \langle \mathcal M_d^* \textbf{u}  , \varphi \rangle_{L^2(\Omega)} =  \langle \textbf{u} , \mathcal M_d \varphi \rangle_{L^2(\Omega; \mathbb{R}^r)} , \quad \forall ~ \textbf{u}  \in  \widetilde{\mathbb{K}},  ~\varphi \in \blue{\mathbb{K}^{test}}.
\end{equation*}
Now recall that $u_*\in\mathbb{H}$ is the solution to \eqref{linear PDE}. For any $u\in\mathbb{H}, \varphi\in\blue{\mathbb{K}^{test}},$ we have
\begin{align} 
  \langle \mathcal L u - f, \varphi \rangle_{L^2(\Omega)} = \langle \mathcal L (u - u_*), \varphi \rangle_{L^2(\Omega)}  &  =  \langle \mathcal M_d^* \widetilde{\mathcal L} \mathcal M_p (u - u_*), \varphi    \rangle_{L^2(\Omega)}  \nonumber\\
  & = \langle \widetilde{\mathcal L} ~ \mathcal M_p (u - u_*),  \mathcal M_d \varphi    \rangle_{L^2(\Omega; \mathbb{R}^r)}.  \label{apply adjoint motivate precond }
\end{align}
\begin{example}\label{ex: 1}
Taking the negative Laplace operator $\mathcal L = -\Delta$ as an example, by setting $\mathbb{H} = H^2(\Omega), \widetilde{\mathbb{H}} = \widetilde{\mathbb{K}} = H^1(\Omega, \mathbb{R}^d), \mathbb{K} = L^2(\Omega)$, and $\blue{\mathbb{K}^{test}} = H^1_0(\Omega), \blue{\widetilde{\mathbb{K}}^{test}} = L^2(\Omega; \mathbb{R}^d)$, and choosing $\mathcal M_d = \mathcal M_p =\nabla$ and $\widetilde{\mathcal{L}} = \mathrm{Id}$, we obtain
\begin{align*} 
  \langle -\Delta u - f, \varphi \rangle_{L^2(\Omega)} = &  \langle - \Delta (u - u_*), \varphi \rangle_{L^2(\Omega)} =  \langle \nabla (u - u_*),  \nabla \varphi  \rangle_{L^2(\Omega; \mathbb{R}^d)},
\end{align*}
for any $u\in \mathbb{H}=H^2(\Omega), \varphi \in \blue{\mathbb{K}^{test}}=H^1_0(\Omega)$.
\end{example}

\begin{example}\label{ex: 2}
 Consider the elliptic operator $\mathcal L = \mathrm{Id} - \Delta$, where $\mathrm{Id}$ is an identity operator. By setting $\mathbb{H}=H^2(\Omega)$, $\widetilde{\mathbb{H}} = \widetilde{\mathbb{K}} = H^2(\Omega)\times H^1(\Omega; \mathbb{R}^d)$, $\mathbb{K}=L^2(\Omega)$, and $\blue{\mathbb{K}^{test}} = H_0^1(\Omega)$, $\blue{\widetilde{\mathbb{K}}^{test}} = H^1_0(\Omega)\times L^2(\Omega; \mathbb{R}^d)$, we can split the elliptic operator as
\[  
    \mathrm{Id} - \Delta = [\begin{array}{cc}
    \mathrm{Id}  & -\nabla \cdot  \\
\end{array}] \left[\begin{array}{cc}
    \mathrm{Id}    &  \\
     &  \mathrm{Id}    
\end{array}\right] \left[\begin{array}{c}
     \mathrm{Id} \\
     \nabla
\end{array}\right] = \left[\begin{array}{c}
     \mathrm{Id} \\
     \nabla
\end{array}\right]^* \left[\begin{array}{cc}
    \mathrm{Id}    &  \\
     &  \mathrm{Id}    
\end{array}\right] \left[\begin{array}{c}
     \mathrm{Id} \\
     \nabla
\end{array}\right] =  \mathcal M_d^* \widetilde{\mathcal L} \mathcal M_p.   \] 
Denoting (by abuse of notation) ``$\mathrm{Id}$'' as the identity map on its corresponding space, we have
\begin{align*}
  \langle u - \Delta u - f, \varphi    \rangle_{L^2(\Omega)} & = \langle (\mathrm{Id} - \Delta)(  u - u_* ), \varphi    \rangle_{L^2(\Omega)}
  = \Big\langle  \left[\begin{array}{c}
       u - u_* \\
       \nabla(u - u_*) 
  \end{array}\right], \left[ \begin{array}{c}
       \varphi \\
       \nabla\varphi
  \end{array} \right]  \Big\rangle_{L^2(\Omega; \mathbb{R}^{1+d})}
\end{align*}
for any $u\in H^2(\Omega), \varphi \in H^1_0(\Omega)$.
\end{example}
Further examples of linear equations with elliptic operators $\mathcal L$ belonging to divergence form $\mathcal L = -\nabla\cdot(\kappa(x) \nabla \, )$ with $\kappa\in C^1(\Omega)$ are given in Section~\ref{sec: numerical examples }.

{\color{black}
\begin{remark}[$\mathcal L$ of non-divergence form]\label{rk: L non-divergence form}
  Consider a general second-order differential operator $\mathcal L = \sum_{i,j=1}^d a_{ij}(x)\frac{\partial^2}{\partial x_i \partial x_j}$ with $a_{ij}(\cdot)\in L^2(\Omega)$. The operator shows up as the generator of diffusion processes and acts as a fundamental role in a wide range of applications. It is not always tractable to split $\mathcal L$ into differential operators with lower orders as illustrated in Example \ref{ex: 1} and \ref{ex: 2}. In such a case, one possible treatment is to set $\mathcal M_p = \mathcal L, \widetilde{\mathcal L}=\mathrm{Id}, \mathcal M_d=\mathrm{Id}$ with $\mathbb{H}=H^2(\Omega), \widetilde{\mathbb{H}}=\widetilde{\mathbb{K}}= \mathbb{K} = L^2(\Omega)$ and apply the algorithm. However, in the current work, we will mainly focus on elliptic equations of the divergence form and leave the non-divergence cases for future investigation.
\end{remark}}

Similar to \eqref{apply adjoint motivate precond }, recall that $\mathcal B$ is a linear boundary operator, for any $u\in\mathbb{H}, \psi \in \blue{\mathbb{K}^{test}_{\partial \Omega}},$ we have
\[ \langle \mathcal B u - g, \psi \rangle_{L^2(\partial \Omega)} = \langle \mathcal B (u - u_*), \psi \rangle_{L^2(\partial \Omega)}.  \]
As mentioned in Section \ref{sec: introduction}, we shall substitute $u, \varphi, \psi$ in the proximal steps of \eqref{PDHG functional space} with $(\mathcal M_p(u - u_*), \blue{\sqrt{\lambda}}\mathcal B(u-u_*))$ and $\mathcal M_d \varphi$, $\blue{\sqrt{\lambda}\psi}$. Correspondingly, we use the following modified functional $\mathscr{E}:\mathbb{H}\times \blue{\mathbb{K}^{test}}\times\blue{\mathbb{K}^{test}_{\partial \Omega}}\rightarrow \mathbb{R}$ (\blue{we denote $L^2:=L^2(\Omega)$, $L_{\partial\Omega}^2:=L^2(\partial \Omega)$ for simplicity}),
\begin{equation}
\begin{split}\label{preconditioned J loss funcitional }
    {\mathscr{E}}(u, \varphi, \psi)  = & \langle  \widetilde{\mathcal L} \mathcal M_p (u - u_*), \mathcal M_d    \varphi    \rangle_{L^2} + \blue{\lambda} \langle \mathcal B (u - u_*), \psi \rangle_{L^2_{\partial \Omega}} - \frac{\blue{\nu}}{2}( \|\mathcal M_d \varphi\|^2_{L^2} + \blue{\lambda} \|\psi\|^2_{L^2}  )\\
    =& \langle  \widetilde{\mathcal L} \mathcal M_p u, \mathcal M_d \varphi    \rangle_{L^2} - \langle \mathcal L u_*, \varphi \rangle_{L^2} + \blue{\lambda} \langle \mathcal B (u - u_*), \psi \rangle_{L^2_{\partial \Omega}} - \frac{\blue{\nu}}{2}( \|\mathcal M_d \varphi\|^2_{L^2} + \blue{\lambda} \|\psi\|^2_{L^2} ) \\
    =& \langle  \widetilde{\mathcal L} \mathcal M_p  u , \mathcal M_d \varphi    \rangle_{L^2}  - \langle f,   \varphi \rangle_{L^2} + \blue{\lambda} \langle \mathcal B u - g, \psi \rangle_{L^2_{\partial \Omega}} - \frac{\blue{\nu}}{2}( \|\mathcal M_d \varphi\|^2_{L^2} + \blue{\lambda} \|\psi\|^2_{L^2}  ).
\end{split}
\end{equation}
{\color{black} 
We then consider the inf-sup problem
\begin{equation}
  \inf_{u \in \mathbb{H} } ~ \sup_{\varphi  \in  \blue{\mathbb{K}^{test}}, \, \psi \in \blue{\mathbb{K}^{test}_{\partial \Omega}}} ~ \mathscr{E}(u, \varphi, \psi).  \label{def: inf-sup new E}
\end{equation}}
{\color{black}
The following theorem proves the consistency between the solution to this inf-sup problem and the solution to the PDE \eqref{linear PDE}.
\begin{theorem}[Consistency]\label{thm: consistency inf-sup and sol to PDE}
Assume the test spaces $\blue{\mathbb{K}^{test}}, \blue{\mathbb{K}^{test}_{\partial \Omega }  }$ are dense in the spaces $L^2(\Omega, \mu)$, $L^2(\partial \Omega, \mu_{\partial \Omega})$, respectively. Suppose that $(\widehat{u}, \widehat{\varphi}, \widehat{\psi}) \in \mathbb{H}\times \blue{\mathbb{K}^{test}} \times \blue{\mathbb{K}^{test}_{\partial \Omega}} $ is a solution to the inf-sup problem \eqref{def: inf-sup new E}. Then $\widehat{u}$ is a strong solution to \eqref{linear PDE} in the sense that $\mathcal L u - f=0$, almost everywhere (a.e.) on $\Omega$ and $\mathcal Bu = g$, a.e. on $\partial \Omega.$
\end{theorem}
The proof of the theorem is provided in Appendix \ref{append: proof of consistency thm}. As long as \eqref{linear PDE} admits a unique strong solution, the function $\widehat{u}$ must coincide with the classical solution $u_*$.}

To seek for the solution of the inf-sup problem \eqref{def: inf-sup new E}, we treat $(\mathcal M_p (u - u_*), \blue{\sqrt{\lambda}}\mathcal B (u-u_*))$, together with $(\mathcal M_d \varphi, \blue{\sqrt{\lambda}}\psi)$, as the \textit{new} primal and dual variables of the algorithm. By doing so, we substitute $(u, \mathcal B u)$, $(\varphi, \psi)$ in the proximal steps (the 1st and the 3rd line) of \eqref{PDHG functional space} with $(\mathcal M_p(u - u_*), \blue{\sqrt{\lambda}}\mathcal B(u-u_*)), (\mathcal M_d \varphi, \blue{\sqrt{\lambda}}\psi)$. Therefore, we come up with the following preconditioned version of the PDHG algorithm. This treatment is also known as the G-prox PDHG algorithm introduced in \citep{pdhgjacobs2019solving}.
\begin{equation}
\begin{split}\label{precond PDHG functional space}
      & \left[\begin{array}{c}
        {\varphi}_{n+1} \\
        {\psi}_{n+1}
      \end{array}\right] =     \underset{\substack{\varphi  \in  \blue{\mathbb{K}^{test}} \\ \psi \in \blue{\mathbb{K}^{test}_{\partial \Omega}}}}{\mathrm{argmin}} ~ \left\{\frac{1}{2\tau_\varphi}{(\|\mathcal M_d \varphi - \mathcal M_d  \varphi_n\|^2_{L^2(\Omega)} + \blue{\lambda}\|\psi - \psi_n \|^2_{L^2(\partial \Omega)})} - {\mathscr{E}}(u_n, \varphi, \psi)\right\},  \\
      & \left[\begin{array}{c}
        \widetilde{\varphi}_{n+1} \\
        \widetilde{\psi}_{n+1}
      \end{array}\right] = \left[\begin{array}{c}
        \varphi_{n+1} \\
        \psi_{n+1}
      \end{array}\right] + \omega \left(\left[\begin{array}{c}
        \varphi_{n+1} \\
        \psi_{n+1}
      \end{array}\right] - \left[\begin{array}{c}
        \varphi_{n } \\
        \psi_{n }
      \end{array}\right]\right),  \\
      & u_{n+1} =   \underset{u \in \mathbb{H} }{\mathrm{argmin}} ~ \left\{ \frac{1}{ 2\tau_{u}}{(\|\mathcal M_p u - \mathcal M_p u_n\|^2_{L^2(\Omega)} + \blue{ \lambda} \|\mathcal B u - \mathcal B u_n\|^2_{L^2(\partial \Omega)} ) } + {\mathscr{E}}(u, \widetilde{\varphi}_{n+1}, \widetilde{\psi}_{n+1})  
      \right\},  
\end{split}
\end{equation}

\subsection{Natural Primal-Dual Hybrid Gradient (NPDG) algorithm for neural networks}\label{subsec: NPDG for NN training}
At the beginning of this section, we briefly introduce the idea of the Natural Gradient method \citep{amari1998natural, amari2016information, martens2020new}.
\subsubsection{Natural Gradient method}
For a wide range of machine learning problems, we assume that the loss function $J(\theta)=\mathscr{J}(u_\theta)$ where $\mathscr{J}:\mathbb{U}\rightarrow\mathbb{R}$ denotes the loss functional and $u_\theta$ is the parametrized function on the metric space $\mathbb{U}$ with the parameter $\theta \in \Theta \subset \mathbb{R}^m$ to be determined. The essential idea of the natural gradient algorithm is to conduct gradient descent on $u_\theta$ as an entity in the functional space rather than on the parameter $\theta$. This can be realized by considering the proximal algorithm
\begin{equation} 
  \inf_{u_\theta\in\mathbb{U}} ~ \frac{d^2(u_\theta, u_{\theta^n})}{2\tau} + J(\theta).  \label{def: proximal for NG}
\end{equation}
The preconditioning matrix $G(\theta)$ can thus be obtained by investigating the \blue{infinitesimal} distance $d^2(u_\theta, u_{\theta^n}) \approx (\theta-\theta^n)^\top G(\theta) (\theta - \theta^n)$, where $d(\cdot, \cdot)$ is a distance function enriches the Hessian information of the loss functional $\mathscr{J}$. By sending $\tau\rightarrow 0,$ the implicit scheme \eqref{def: proximal for NG} reduces to the natural gradient flow
\begin{equation*}
  \dot\theta_t = - G(\theta)^{-1}\nabla_\theta J(\theta).
\end{equation*}
As a result, viewing from the parameter space, the natural gradient algorithm can be realized by applying $G(\theta)$-preconditioned gradient descent steps to loss function $J(\theta)$. We refer the interested readers to \citep{amari2016information} for a comprehensive illustration of the Natural Gradient methods.

Let us continue the discussion on the derivation of NPDG algorithm. We substitute $u(\cdot), \varphi(\cdot), \psi(\cdot)$ with neural networks $u_\theta(\cdot), \varphi_\eta(\cdot)$ and $\psi_\xi(\cdot)$ with tunable parameters $\theta\in\Theta_\theta \subseteq \mathbb{R}^{m_\theta}, \eta\in\Theta_\eta \subseteq \mathbb{R}^{m_\eta},  \xi \in \Theta_\xi \subseteq  \mathbb{R}^{m_\xi}$. Here, we assume that the parameter spaces $\Theta_\theta, \Theta_\eta, \Theta_\xi$ are open sets of the Euclidean space. Then, algorithm \eqref{precond PDHG functional space} becomes
\begin{equation}
\begin{split}\label{precond PDHG nn space}
      & \left[\!\begin{array}{c}
        \eta^{n+1} \\
        \xi^{n+1}
      \end{array}\!\right] \! =  \underset{\substack{\blue{\eta}\in \mathbb{R}^{m_\eta}\\ \blue{\xi} \in \mathbb{R}^{m_\xi}}}{\mathrm{argmin}}  \left\{ { \frac{1}{2\tau_\varphi}{(\|\mathcal M_d \varphi_\eta - \mathcal M_d  \varphi_{\eta^n}\|^2_{L^2(\Omega)} + \blue{\lambda} \|\psi_{\xi} - \psi_{\xi^n} \|^2_{L^2(\partial \Omega)})} - {\mathscr{E}}(u_{\theta^n}, \varphi_\eta, \psi_\xi) } \right\},  \\
      & \left[ \! \begin{array}{c}
        \widetilde{\varphi}_{n+1} \\
        \widetilde{\psi}_{n+1}
      \end{array}\! \right] \! = \! \left[\!\begin{array}{c}
        \varphi_{\eta^{n+1}} \\
        \psi_{\xi^{n+1}}
      \end{array}\!\right] + \omega \left(\left[\!\begin{array}{c}
        \varphi_{\eta^{n+1}} \\
        \psi_{\xi^{n+1}}
      \end{array}\!\right] - \left[\!\begin{array}{c}
        \varphi_{\eta^n } \\
        \psi_{\xi^n }
      \end{array}\!\right]\right),  \\
      & \theta^{n+1} \! = \! \underset{  \theta \in \mathbb{R}^{m_\theta}  }{\mathrm{argmin}}  \left\{ \frac{1}{ 2\tau_{ u } } {(\|\mathcal M_p u_\theta - \mathcal M_p u_{\theta^n}\|^2_{L^2(\Omega)} + \blue{\lambda} \|\mathcal B u_\theta - \mathcal B u_{\theta^n}\|^2_{L^2(\partial\Omega)})}  + {\mathscr{E}}(u_\theta, \widetilde{\varphi}_{n+1}, \widetilde{\psi}_{n+1})  
      \right\}.
\end{split}
\end{equation}
Let us take a closer look at the first line of \eqref{precond PDHG nn space}. Since $\varphi$ and $\psi$ are separable in ${\mathscr{E}}$, it is not hard to verify that the updating rule of $\eta^{n+1}$ can be formulated as
\begin{equation}
  \eta^{n+1} = \underset{\eta \in \mathbb{R}^{m_\eta} }{\mathrm{argmin}} ~ \left\{ \frac{1}{2\tau_\varphi}{\|\mathcal M_d \varphi_\eta -   \mathcal M_d \varphi_{\eta^n} \|^2_{L^2(\Omega)}} - {\mathscr{E}}(u_{\theta^n}, \varphi_{\eta}, \psi_{\xi}) 
   \right\}.  \label{G-prox PDHG PDE param MLP dual varphi}
\end{equation}
As suggested in Section \ref{sec: introduction}, we approximate $\mathcal M_d \varphi_\eta(x) - \mathcal M_d \varphi_{\eta^n}(x)$ with $\frac{\partial }{\partial \eta} \mathcal M_d \varphi_{\eta^n}(x) (\eta - \eta^n)$, where we denote $\frac{\partial}{\partial \eta} \mathcal M_d \varphi_{\eta_n}(x)\in\mathbb{R}^{ r \times m_\eta }$ as the Jacobian matrix of $\mathcal M_d \varphi_\eta(\cdot)$ with respect to the parameter $\eta$ at $\eta^n$. Therefore,
\begin{align}
  {\|\mathcal M_d \varphi_\eta - \mathcal M_d \varphi_{\eta^n}\|^2_{L^2(\Omega; \mathbb{R}^r)}} \approx &  \int_{\Omega} \|\frac{\partial }{\partial \eta} \mathcal M_d \varphi_{\eta^n}(x) (\eta - \eta^n)\|^2 \,\dd\mu(x)\nonumber  \\
  = & \sum_{i=1}^{m_\eta} \sum_{j=1}^{m_\eta} \Big\langle \frac{\partial}{\partial \eta_i} \mathcal M_d \varphi_{\eta^n}, \frac{\partial}{\partial \eta_j} \mathcal M_d \varphi_{\eta^n} \Big\rangle_{L^2(\Omega; \mathbb{R}^r)}  ( \eta_i - \eta^n_i) (\eta_j - \eta^n_j)  \nonumber  \\
   = &  (\eta - \eta^n)^\top M_d(\eta^n) (\eta - \eta^n),  \label{quad proximal}
\end{align}
where we denote  
\begin{equation}
    M_d(\eta^n) = \int_{\Omega}  \frac{\partial}{\partial \eta} \mathcal M_d \varphi_{\eta^n}(x)^\top  \frac{\partial}{\partial \eta} \mathcal M_d \varphi_{\eta^n}(x) ~ \dd\mu(x),  \label{def: M_d}
\end{equation}
as an $m_\eta \times m_\eta$ symmetric, positive semidefinite Gram matrix that encodes the information of $\mathcal M_d$. 

Replacing the proximal term in \eqref{G-prox PDHG PDE param MLP dual varphi} with the quadratic term \eqref{quad proximal} yields 
\begin{equation}
  \frac{1}{2\tau_\varphi}\Delta \eta^\top M_d(\eta^n) \Delta \eta - \widehat{E}(\theta^n, \eta^n + \Delta\eta, \xi^n).  \label{quad term of proximal step}
\end{equation}
Here we denote $\Delta \eta = \eta - \eta^n$ and $\widehat{E}(\theta, \eta, \xi) = {\mathscr{E}}(u_\theta, \varphi_\eta, \psi_\xi)$ for shorthand. By linearizing $\widehat{E}(\theta^n, \eta, \xi^n)$ at $\eta=\eta^n$, the quantity \eqref{quad term of proximal step} yields 
\begin{align*}
  & \frac{1}{2}\Delta \eta^\top M_d(\eta^n) \Delta \eta - \tau_\varphi (\widehat{E}(\theta^n, \eta^n + \Delta\eta, \xi^n) - \widehat{E}(\theta^n, \eta^n, \xi^n))  \\
  \approx & \frac{1}{2}\Delta\eta^\top M(\eta^n) \Delta \eta -\tau_\varphi \nabla_\eta \widehat{E}(\theta^n, \eta^n, \xi^n)^\top \Delta \eta + \mathcal O(\tau_\varphi \blue{\|\Delta \eta\|}^2).
\end{align*}
We further omit the term $\mathcal O(\tau_\varphi \blue{\|\Delta \eta\|}^2)$ to obtain the linearized version of \eqref{G-prox PDHG PDE param MLP dual varphi}
\begin{equation}
  \min_{\Delta \eta \in \mathbb{R}^{ m_\eta }} \left\{ \frac{1}{2} \Delta \eta^\top M_d(\eta^n) \Delta \eta - \tau_\varphi \nabla_\eta \widehat{E}(\theta^n, \eta^n, \xi^n)^\top \Delta \eta \right\}.  \label{derive least square problem}
\end{equation}
\blue{According to Lemma \ref{lemma: orthogonal projection } from the next Section \ref{section: Numer Anal}, we have $\nabla_\eta \widehat{E}(\theta^n, \eta^n, \xi^n) \in \mathrm{Ran}(M_d(\eta^n))$. Therefore, the minimum value of \eqref{derive least square problem} is finite. Recall that $M_d(\eta^n)^\dagger$ denotes the Moore-Penrose inverse of $M_d(\eta^n)$, an optimal solution to this least squares problem \eqref{derive least square problem} can be denoted as}
\begin{equation*}
      \Delta\eta = \tau_\varphi \cdot M_d(\eta^n)^\dagger \nabla_\eta \widehat{E}(\theta^n, \eta^n, \xi^n). 
\end{equation*}
The resulting formula suggests that we explicitly update $\eta$ along the gradient ascent direction preconditioned by the Gram matrix $M_d(\eta^n)$,
\begin{equation*}
    \eta^{n+1} = \eta^n + \tau_\varphi \cdot M_d(\eta^n)^\dagger \nabla_\eta \widehat{E}(\theta^n, \eta^n, \xi^n).
\end{equation*}
By doing so, we exchange some of the numerical stability enjoyed by the proximal step for computational feasibility and efficacy.

\begin{remark}\label{rk: pseudo inverse}
    It is worth mentioning that the Moore-Penrose inverse used here is generally unnecessary. In order to determine a solution $\mathbf{v}$ to \eqref{derive least square problem}, one only needs to guarantee that $M_d(\eta^n)\mathbf{v} = \nabla_\eta \widehat{E}(\theta^n, \eta^n, \xi^n)$. Consider any pseudo-inverse matrix $M^+_d(\eta^n)$ such that $M_d(\eta^n)M^+_d(\eta^n)$ preserves the column vectors of $M_d(\eta^n)$, i.e., \(M_d(\eta^n)M_d^+(\eta^n)M_d(\eta^n) = M_d(\eta^n)\). Then, $\mathbf{v} = M_d^+(\eta^n)\nabla_\eta \widehat{E}(\theta^n, \eta^n, \xi^n)$ will be a solution to \eqref{derive least square problem}. Thus, we can set 
    \[ \Delta \eta = M_d^+(\eta^n)\nabla_\eta \widehat{E}(\theta^n, \eta^n, \xi^n). \]
    In this research, we pick $M^+_d(\eta^n)$ as the Moore-Penrose inverse for convenience in discussion.
\end{remark}

Moreover, we utilize the same idea to update the parameters $\xi$ and $\theta$, such that 
\begin{align*}
  \xi^{n+1} = &  \xi^n + \tau_\varphi \cdot M_{bdd}(\xi^n)^\dagger \nabla_\xi \widehat{E}(\theta^n, \eta^n, \xi^n),\\
  \theta^{n+1} = &  \theta^n - \tau_u \cdot M_p(\theta^n)^\dagger \nabla_\theta {\mathscr{E}}(u_{\theta^n}, \widetilde{\varphi}_{n+1},  \widetilde{\psi}_{n+1}),   \end{align*}
where $\widetilde{\varphi}_{n+1} = \varphi_{\eta^{n+1}} + \omega (\varphi_{\eta^{n+1}} - \varphi_{\eta^n}), \widetilde{\psi}_{n+1} = \psi_{\xi^{n+1}} + \omega (\psi_{\xi^{n+1}} - \psi_{\xi^n})$ are obtained via the extrapolation. The Gram type matrices $M_p(\theta), M_{bdd}(\xi)$ are computed as
\begin{align}
  M_p(\theta) = &  \int_{\Omega} \frac{\partial}{\partial \theta} \mathcal M_p u_\theta(x)^\top \frac{\partial}{\partial \theta} \mathcal M_p u_\theta(x) ~ \dd\mu + {\color{black} \lambda}\int_{\partial \Omega} \frac{\partial}{\partial \theta} \mathcal B u_\theta(y)^\top \frac{\partial}{\partial \theta} \mathcal B u_\theta(y)\,\dd\mu_{\partial \Omega}.  \label{def: M_p}\\
  M_{bdd}(\xi) = &  \blue{\lambda} \int_{\partial \Omega} \frac{\partial  \psi_{\xi}(y)  }{\partial \xi}^\top \frac{\partial \psi_\xi(y)}{\partial \xi} ~ \dd\mu_{\partial \Omega},  \label{def: M_bdd}  
\end{align}
This yields our NPDG algorithm \eqref{def: intro alg NPDG}.

In practice, choosing the stepsize $\tau_\varphi$ ranging from $10^{-2}$ to $10^{-1}$ usually yields stable and efficient performance of this explicit scheme. The study on the optimal choice of the stepsizes, as well as the application of more meticulous line search strategies will serve as the future research directions.

{\color{black}
\begin{remark}[Adoption of stronger boundary norm]
For Dirichlet problem, $\mathcal B$ is treated as the trace operator. Then $\mathcal B : H^1(\Omega) \to H^{1/2}(\partial \Omega) \subsetneq L^2(\partial \Omega)$ is a continuous and surjective mapping \citep{grisvard2011elliptic}, it is thus more natural to employ the $H^{1/2}(\partial \Omega)$ norm—or even stronger boundary norms—rather than $L^2(\partial \Omega)$ for the boundary loss in \eqref{G-prox PDHG PDE param MLP dual varphi}. A detailed treatment is deferred to Section~\ref{subsec: numerical analysis with stronger bdd norm}. 
\end{remark}}

{\color{black}
We conclude this subsection by briefly discussing the advantages and limitations of different ways to split the operator $\mathcal L$ associated with the linear PDE~\eqref{linear PDE}. In general, $\mathcal L$ may admit multiple splittings, each leading to a distinct preconditioning strategy. For example, the operator $-\Delta$ can be decomposed as in Example~\ref{ex: 1}, or alternatively as described in Remark~\ref{rk: L non-divergence form}. 

Comparing these two choices, the former is typically more computationally efficient, as it avoids the evaluation of second-order derivatives. Moreover, the theoretical analysis presented in Theorem~\ref{thm: convergence analysis of NPDG flow} in the following section shows that, under the former splitting, the error term $u_\theta-u_*$ can be bounded using a more effective $H^1$ seminorm, rather than the $(-\Delta)$-weighted $L^2$ norm that arises in the latter case.

As a trade-off, as discussed in Remark~\ref{rk: adaptive sampling} of Section~\ref{sec: alg}, adaptive sampling strategies cannot be readily incorporated into the former framework, whereas such techniques can be naturally integrated under the latter choice.

In practice, we focus on the first splitting for the elliptic equations considered in this work. A comprehensive investigation and systematic comparison of different splitting and preconditioning strategies will be pursued in future investigations.}

\vspace{0.3cm}

\subsection{Nonlinear Equations}\label{subsec: nonlinear eq}
A similar treatment can be applied to the \blue{semi-linear} equations taking the form of
\begin{equation}
  \mathcal L u + \mathcal N u = f, ~~\textrm{on } \Omega, \quad \mathcal B u = g ~~ \textrm{on } \partial \Omega,  \label{def: L+N equation }
\end{equation}
where $\mathcal L$, $\mathcal N$ denote the linear and nonlinear operators, respectively. $\mathcal B$ is the boundary operator. $f:\Omega\rightarrow \mathbb{R}$, $g:\partial \Omega \rightarrow \mathbb{R}$. Suppose that $\mathcal L$ splits as $\mathcal L = \mathcal M_d^* \widetilde{\mathcal L} \mathcal M_p$. We then multiply the equation and its boundary condition with the dual variables $\varphi, \psi$ and derive the functional
\begin{equation*}
    \begin{split}
    \mathscr{E}(u, \varphi, \psi ) =    \langle \widetilde{\mathcal L}  \mathcal M_p u , \mathcal M_d \varphi  \rangle_{L^2(\Omega; \mathbb{R}^r )} + \langle \mathcal N(u), \varphi \rangle_{L^2(\Omega)} - \langle f, \varphi\rangle_{L^2(\Omega)} + \blue{\lambda} \langle \mathcal B u, \psi \rangle_{L^2(\partial \Omega)}  &  \\
         - \frac{\blue{\nu}}{2}\left(\|\mathcal M_d \varphi\|^2_{L^2(\Omega; \mathbb{R}^r )} + \blue{\lambda}\|\psi\|_{L^2(\partial \Omega )}^2  \right)  &  .
       \end{split}
\end{equation*}
We can now apply algorithm \eqref{def: intro alg NPDG} with preconditioning matrices $M_d, M_p, M_{bdd}$ mentioned above in \eqref{def: M_d}, \eqref{def: M_p}, \eqref{def: M_bdd} to solve the equation. Related numerical examples and more detailed descriptions of our treatment can be found in Section \ref{subsec: nonlinear elliptic equ} and \ref{subsec: AllenCahn}. 

\subsubsection{Monge-Amp\`ere equation}\label{subsec: discuss on MA}
The algorithm can readily handle some fully nonlinear equation that possesses a saddle point formulation, such as the Monge-Amp\`ere equation. 
\begin{equation}
  |\mathrm{det}(D^2 u(x))| = \frac{\rho_0(x)}{\rho_1(\nabla  u (x))}, ~ \rho_0 \dd x-a.e., \quad    u ~ \textrm{is convex on } \mathbb{R}^d.  \label{insection Monge Ampere}
\end{equation}
Here, $\rho_0, \rho_1$ are probability density functions. $D^2u$ denotes the Hessian matrix of the potential function $u$. This equation takes an equivalent form of
\[ \nabla u_{\sharp} \mu_0 = \mu_1, \quad    u ~ \textrm{is convex on } \mathbb{R}^d, \]
where $\mu_0, \mu_1$ are probability distributions. We assume $\mu_0, \mu_1$ are absolutely continuous with respect to the Lebesgue measure on $\mathbb{R}^d$, with density functions $\rho_0, \rho_1$. And $\sharp$ denotes the ``pushforward'' of probability distribution $\mu_0$ by the map $\nabla u$ in the sense of 
\begin{equation*}
\int_{\mathbb{R}^d} h(\nabla u(x))\,\dd\mu_0(x) = \int_{\mathbb{R}^d} h(y)\,\dd\mu_1(y), \quad \textrm{for any measurable function }  h \textrm{ defined on } \mathbb{R}^d.
\end{equation*}
There is already adequate research on the classical numerical methods for the Monge-Amp\`ere equation. We refer the readers to \citep{benamou2010two,froese2011convergent,benamou2014numerical,neilan2020monge} and the references therein for further discussion.

In this research, we aim to propose a mesh-free algorithm based on the data samples drawn from $\rho_0$ and $\rho_1$ to evaluate $\nabla u(\cdot)$ of the equation. We should first point out that the Monge-Amp\`ere equation is closely related to the following Optimal Transport (OT) problem (also known as the Monge problem) \citep{villani2009optimal, de2014monge},
\begin{align}
    & \min_{\substack{T \in \mathcal M(\mathbb{R}^d, \mathbb{R}^d) \\ T_\sharp \mu_0 = \mu_1}} ~\int_{\mathbb{R}^d} \frac12 \|x-T(x)\|^2\,\dd\mu_0(x).   \label{def: OT Monge}
\end{align}
Here $\mathcal M(\mathbb{R}^d, \mathbb{R}^d)$ denotes the space of measurable maps from $\mathbb{R}^d$ to $\mathbb{R}^d$. We aim at computing for the optimal map $T$ that transports the probability distribution $\rho_0$ to $\rho_1$ by minimizing the $L^2$ transportation cost. One can show that (c.f. Section 3 of \citep{de2014monge}) the optimal map $T_*$ of \eqref{def: OT Monge} exists uniquely as long as $\mu_0, \mu_1$ possess densities $\rho_0, \rho_1$, and there exists a convex function $u: \mathbb{R}^d \rightarrow \mathbb{R}$ such that $T_*(x) = \nabla u(x)$ for $\mu_0$-a.e. $x\in\mathbb{R}^d$. Furthermore, if $\mu_0, \mu_1$ are supported on bounded smooth open sets $X, Y\subset \mathbb{R}^d$, and $\rho_0, \rho_1$ are bounded away from zero and infinity on $X$ and $Y$, then the potential $u$ solves the Monge-Amp\`ere equation \eqref{insection Monge Ampere}.

Given the connection between the Monge-Amp\`ere equation and the OT problem, we mainly focus on computing \eqref{def: OT Monge} instead of \eqref{insection Monge Ampere}. The goal is to compute the OT map $T_*$ (or $\nabla u$). Notice that \eqref{def: OT Monge} is a constrained optimization problem. By denoting $\mathcal C_b(\mathbb{R}^d)$ as the space of bounded continuous functions, it is natural to introduce the Lagrange multiplier $\varphi \in \mathcal C_b(\mathbb{R}^d)$ (also known as the Kantorovich potential of the OT problem) to the constraint $T_\sharp \rho_0 = \rho_1$, and obtain
\begin{align}
  \mathscr{E}(T, \varphi) = & \int_{\mathbb{R}^d} \frac12  \|x - T(x)\|^2\,\dd\mu_0(x) + \int_{\mathbb{R}^d} \varphi(T(x))\,\dd\mu_0(x) - \int_{\mathbb{R}^d} \varphi(y)\,\dd\mu_1(y).         \label{MA loss functional L }  
\end{align}
Upon solving \eqref{def: OT Monge}, we consider the sup-inf saddle point problem
\begin{align}
  \sup_{\varphi\in \mathcal C_b(\mathbb{R}^d)} \inf_{T\in\mathcal M(\mathbb{R}^d, \mathbb{R}^d)} ~ \mathscr{E}(T, \varphi).  \label{saddle Monge with map T}
\end{align}
It is shown in \citep{fan2023neural} that, as long as $\mu_0, \mu_1$ are compactly supported, and $\mu_0$ is absolutely continuous w.r.t. Lebesgue measure, the saddle point $({T}_\star, \varphi_\star)$ of \eqref{saddle Monge with map T} exists, and the map $T_\star(\cdot)$ equals the OT map $T_*(\cdot)$ $\mu_0-$almost surely.

In the computation, we substitute $T, \varphi$ with neural networks $T_\theta, \varphi_\eta$. A natural way of preconditioning this problem is to set $\mathcal M_p = \mathrm{Id}$ and $\mathcal M_d = \mathrm{Id}$ for $M_p(\theta), M_d(\eta)$, i.e.,
\begin{equation}
  \begin{split}\label{pull Id as precond Monge problem}
      M_p(\theta) = & \int_{\mathbb{R}^d} \frac{\partial T_\theta(x)}{\partial \theta}^\top \frac{\partial T_\theta(x)}{\partial \theta}~\rho_0(x)\,\dd x, \\ 
      \quad M_d(\eta; \theta) = & \int_{\mathbb{R}^d} \frac{\partial \varphi_\eta(T_\theta(x))}{\partial \eta} \frac{\partial \varphi_\eta(T_\theta(x))}{\partial \eta}^\top \rho_0(x)\,\dd x.
  \end{split}
\end{equation}

However, a more canonical choice is to set $\mathcal M_p = \mathrm{Id}$ and $\mathcal M_d = \nabla$. To motivate this preconditioning technique, we carry out the following calculation. Suppose $( T_\star, \varphi_\star)$ is the saddle point of the above problem \eqref{saddle Monge with map T}. As $T_\star(\cdot)=T_*(\cdot)$ $\mu_0-$almost surely, one can show that $T_{\star\sharp}\mu_0=\mu_1$, and
\begin{align} 
  T_\star(x) - x + \nabla \varphi_\star(T_\star(x)) = 0, \quad \mu_0-a.s.   \label{optimal equation Tstar, varphistar }
\end{align}
Here we denote $\nabla\varphi_\star(T_\star(x))$ as $ \nabla \varphi_\star(y)|_{y = T_\star(x)}$.

We have
\begin{align}
  \mathscr{E}(T, \varphi) &= \int_{\mathbb{R}^d} \frac{1}{2}\|T(x)-T_\star(x) + T_\star(x) - x\|^2\rho_0(x)\,\dd x + \int_{\mathbb{R}^d} (\varphi(T(x)) - \varphi(T_\star(x)))\rho_0(x)\dd x \nonumber\\
  & = \int_{\mathbb{R}^d} \! \left(\!\frac{1}{2}\|T_\star(x) - x\|^2 + (T(x) - T_\star(x))^\top(T_\star(x) - x) + \frac{1}{2}\|T(x) - T_\star(x)\|^2\!\right)\!\rho_0(x)\dd x \nonumber\\
  & \quad +  \int_{\mathbb{R}^d} \nabla \varphi(T(x))^\top(T(x) - T_\star(x)) \rho_0(x) \dd x \nonumber \\
  & \quad + \int_{\mathbb{R}^d}\frac12(T(x)-T_\star(x))^\top \nabla^2 \varphi(\xi(x)) (T(x)-T_\star(x)) \rho_0(x)\dd x \nonumber \\
  & \blue{ = \int_{\mathbb{R}^d} \left( (T(x)-T_\star(x))^\top (T_\star(x)-x) + \nabla\varphi(T(x))^\top(T(x) - T_\star(x)) \right) \rho_0(x) \dd x} \nonumber \\
  & \blue{ \quad + \int_{\mathbb{R}^d} \frac12 (T(x)-T_\star(x))^\top(I - \nabla^2 \varphi(\xi(x)))(T(x)-T_\star(x)) \rho_0(x)\dd x} \nonumber\\
  & \blue{\quad + \int_{\mathbb{R}^d} \frac12\|T_\star(x)-x\|^2\rho_0(x)\dd x} \nonumber \\
  & \blue{= \int_{\mathbb{R}^d} (T_\star(x)-x+\nabla\varphi(T(x)))^\top(T(x) - T_\star(x)) \rho_0(x)\dd x}    \label{dominating term of L3 }   \\
  & \blue{\quad +  \int_{\mathbb{R}^d} \frac12 (T(x)-T_\star(x))^\top(I - \nabla^2 \varphi(\xi(x)))(T(x)-T_\star(x)) \rho_0(x)\dd x }  \nonumber \\
  & \blue{\quad + \mathrm{Const}.} \nonumber
\end{align}
Here we denote $\xi(x) = (1-s)T(x)+sT_\star(x)$ with $0 \leq s\leq 1$.

Now recall the optimal relation \eqref{optimal equation Tstar, varphistar }, we can reformulate the term \eqref{dominating term of L3 } as
\begin{align}
  & \int_{\mathbb{R}^d} (x-\nabla\varphi_\star(T_\star(x))-x+\nabla\varphi(T(x)))^\top(T(x) - T_\star(x)) \rho_0(x)\,\dd x \nonumber \\
  = & \int_{\mathbb{R}^d} (\nabla\varphi(T(x))-\nabla\varphi_\star(T_\star(x)))^\top(T(x) - T_\star(x)) \rho_0(x)\,\dd x. \label{Linear part of L3 }
\end{align}
Now \eqref{dominating term of L3 } and \eqref{Linear part of L3 } suggest that as $(T, \varphi)$ approaches the optimal solution $(T_\star, \varphi_\star)$, the loss functional $\mathscr{E}(T, \varphi)$ is \blue{roughly the sum of $\Big\langle \nabla\varphi(T(\cdot)) - \nabla\varphi_\star(T_\star(\cdot)), \; T(\cdot) - T_\star(\cdot) \Big\rangle_{L^2(\rho_0)}$ and a quadratic term of $T(\cdot)$.} This suggests setting $T(\cdot) - T_*(\cdot)$ and $\nabla \varphi(T(\cdot))$ as new entities of primal and dual variables in optimization. This yields the preconditioning matrices:
\begin{equation}
  \begin{split}\label{canonical precond for Monge problem}
          M_p(\theta) = & \int_{\mathbb{R}^d} \frac{\partial T_\theta(x)}{\partial \theta}^\top \frac{\partial T_\theta(x)}{\partial \theta}~\rho_0(x)\,\dd x, \\ 
          M_d(\eta ; \theta) = & \int_{\mathbb{R}^d} \frac{\partial}{\partial \eta} (\nabla\varphi_\eta(T_\theta(x)))^\top \frac{\partial}{\partial \eta} (\nabla\varphi_\eta(T_\theta(x))) \rho_0(x)\,\dd x.
  \end{split}
\end{equation}
Applying \eqref{def: intro alg NPDG} to the adversarial training of $T_\theta, \varphi_\eta$ leads to a faster, more robust algorithm for computing the saddle point $(T_\star, \varphi_\star)$ of \eqref{saddle Monge with map T}, where $T_\star$ yields the desired OT map $T_*$, or equivalently, the solution $\nabla u$ to the Monge-Amp\`ere equation. We refer the readers to Section \ref{sec: MA} for details on implementation and numerical examples.

\begin{remark}
It is worth mentioning that the idea of optimizing $\varphi$ with respect to the $H^1$ metric has also been discussed in \citep{jacobs2020fast}, in which the authors introduce a back-and-forth algorithm with $H^1$-preconditioned optimization to deal with the Kantorovich dual problem of \eqref{def: OT Monge}.
\end{remark}

\section{Convergence Analysis of the NPDG flow}\label{section: Numer Anal}
In this section, we provide an \textit{a posteriori} convergence analysis on the time-continuous version of the NPDG algorithm \eqref{def: intro alg NPDG} as $\tau_u, \tau_\varphi \rightarrow 0$ and $\omega\tau_\varphi \rightarrow \gamma>0$. 

Recall that \eqref{def: intro alg NPDG} can be reformulated as
\begin{equation*}
\begin{split}
 &  \left(\left(\begin{array}{c}
          \eta^{n+1}\\
          \xi^{n+1}
        \end{array}
  \right) - 
  \left( \begin{array}{c} 
    \eta^n    \\
    \xi^n
  \end{array}
 \right)\right)\bigg/{\tau_\varphi}  =     
\left( \begin{array}{c} 
M_d(\eta^n)^\dagger \nabla_\eta {\mathscr{E}}(u_{\theta^n}, \varphi_{\eta^n}, \psi_{\xi^n})  \\
M_{bdd}(\xi^n)^\dagger  \nabla_\xi {\mathscr{E}}(u_{\theta^n}, \varphi_{\eta^n}, \psi_{\xi^n})
\end{array} \right), \\
  & \left( 
  \begin{array}{c}
    \widetilde{\varphi}_{n+1} \\
    \widetilde{\psi}_{n+1}
  \end{array}
  \right)  =  \left( \begin{array}{c}
    \varphi_{\eta^{n+1}} \\
    \psi_{\xi^{n+1}}
  \end{array}
  \right) 
  + \omega \tau_\varphi \left(
  \left( \begin{array}{c}
    \varphi_{\eta^{n+1}} \\
    \psi_{\xi^{n+1}}
  \end{array}
  \right) - \left( \begin{array}{c}
    \varphi_{\eta^{n}} \\
    \psi_{\xi^{n}}
  \end{array}
  \right)
  \right)\bigg/\tau_\varphi,  \\
 &  \frac{\theta^{n+1} - \theta^n}{\tau_u}  = ~  - M_p(\theta^n)^\dagger \nabla_\theta {\mathscr{E}}(u_{\theta^n}, \widetilde{\varphi}_{n+1}, \widetilde{\psi}_{n+1}  ).
\end{split}
\end{equation*}
By replacing the finite differences by the time derivatives, as $\tau_u, \tau_\varphi \rightarrow 0$ and $\omega \tau_\varphi \rightarrow \gamma$, we verify that \eqref{def: intro alg NPDG} converges to
\begin{equation}
    \begin{split}\label{PDHG flow}
        \dot \eta_t = & M_d(\eta_t)^\dagger \nabla_\eta \mathscr{E}(u_{\theta_t}, \varphi_{\eta_t}, \psi_{\xi_t}),  \\
        \dot \xi_t = & M_{bdd}(\xi_t)^\dagger \nabla_\xi \mathscr{E}(u_{\theta_t}, \varphi_{\eta_t}, \psi_{\xi_t}), \\
        \dot \theta_t = & - M_p(\theta_t)^\dagger \nabla_\theta \mathscr{E}(u_{\theta_t}, \widetilde \varphi_t, \widetilde \psi_t),
    \end{split} 
\end{equation}
where we denote
\begin{align}
  \left(\begin{array}{c}
          \widetilde \varphi_{t }    \\
           \widetilde \psi_{t }  
     \end{array}\right) = & \left(\begin{array}{c}
           \varphi_{\eta_t }    \\
           \psi_{\xi_t }  
     \end{array}\right) + \gamma \left(\begin{array}{c}
           \dot \varphi_{\eta_t }    \\
           \dot \psi_{\xi_t }  
     \end{array}\right).
     \label{tilde phi tilde phi}
\end{align}
We call the above time-continuous dynamic \eqref{PDHG flow} the \textbf{NPDG flow}. In this section, we analyze the convergence of the numerical solution $\{u_{\theta_t}\}$ along \eqref{PDHG flow}.

\subsection{Natural gradient induces orthogonal projections of Fr\'echet derivatives}
Before our discussion, we need the following lemma. Similar results have already been proved in several references, including \citep{liuPFPE, nurbekyan2023efficient, wu2023parameterized, muller2023achieving, zuo2024numerical}. We restate the lemma here for the sake of completeness.
\begin{lemma}\label{lemma: orthogonal projection }

Given a certain Hilbert space $\mathbb{X}$, we consider a Fr\'echet differentiable functional $\mathscr{F}:\mathbb{X} \rightarrow \mathbb R$. Suppose  $\Theta \subseteq \mathbb{R}^m $ denotes the parameter space, we consider a parametrized family of functions $\{u_\theta\}_{ \theta \in \Theta } $ which belong to $\mathbb{X}$. We denote $D_u \mathscr{F}(u)\in(\mathbb X)^*=\mathbb{X}$ as the Fr\'echet derivative at $u$. Assume that $u_\theta$ is differentiable with respect to $\theta$ and $\frac{\partial u_\theta}{\partial \theta_i}\in\mathbb{X}$ for arbitrary $1\leq i\leq m$, $\theta\in\Theta$. We define the $m\times m$ Gram matrix $M(\theta)$ as
\begin{equation*}
    (M(\theta))_{ij} = \Big\langle \frac{\partial u_\theta}{\partial \theta_i}, \frac{\partial u_\theta}{\partial \theta_j} \Big\rangle_{\mathbb{X}}, \quad 1\leq i, j\leq m. 
\end{equation*}

Furthermore, we denote $F(\theta) = \mathscr{F}(u_\theta)$. Then one can show that 
\begin{itemize}
\item $\nabla_\theta F(\theta)\in \mathrm{Ran}(M(\theta))$,
\item For any $\mathbf{v}\in\mathbb{R}^{m}$ such that $M(\theta)\mathbf{v} = \nabla_\theta F(\theta)$, we can show that $\mathbf{v}$ is the solution to the following least squares problem\footnote{It is worth mentioning that for fixed $x$, $\frac{\partial u_\theta(x)}{\partial\theta}$ is a $k\times m$ matrix.}. 
\begin{align*} 
     \mathbf{v}  
     & \in \underset{\zeta\in \mathbb{R}^m}{\mathrm{argmin} } \left\{ \Big\|D_u \mathscr F(u_\theta) - \frac{\partial u_\theta}{\partial \theta} \zeta \Big\|_{\mathbb{X}}^2  \right\} = \underset{\zeta\in\mathbb{R}^m, \zeta_1,\dots, \zeta_m\in \mathbb{R}}{\mathrm{argmin} } \left\{ \Big\|D_u \mathscr F(u_\theta) - \sum_{i=1}^{m} \zeta_i \frac{\partial u_\theta}{\partial \theta_i} \Big\|_{\mathbb{X}}^2  \right\}.
\end{align*}
One can also verify that
\[ 
    D_u \mathscr F(u_\theta) - \frac{\partial u_\theta}{\partial\theta} \mathbf{v}
\]
as a vector in $\mathbb{X}$, is orthogonal (w.r.t. inner product defined on $\mathbb{X}$) to the subspace spanned by $\{\frac{\partial u_\theta}{\partial\theta_1},\dots, \frac{\partial u_\theta}{\partial\theta_m}\}$. Or equivalently, $\frac{\partial u_\theta}{\partial\theta} \mathbf{v}$
is the orthogonal projection of $D_u \mathscr{F}(u_\theta)$ on $\mathrm{span} \{\frac{\partial u_\theta}{\partial\theta_1},\dots, \frac{\partial u_\theta}{\partial\theta_m}\}$. 
\end{itemize}
\end{lemma}
We defer the proof of this lemma to Appendix \ref{append: subsec proof of NG lemma}.

We should mention that the Moore-Penrose inverse $M(\theta)^\dagger \nabla_\theta F(\theta)$ yields a solution to the least square problem mentioned above. For the convenience of our future discussion, we denote the orthogonal projection (w.r.t. inner product on $\mathbb{X}$) onto $\mathrm{span}  \{\frac{\partial u_\theta}{\partial \theta_1}, \dots, \frac{\partial u_\theta}{\partial \theta_m}\}$ as $\Pi_{\partial_\theta u_\theta}:\mathbb X \rightarrow \mathbb{X}$, we thus have
\[ \frac{\partial u_\theta}{\partial \theta} M(\theta)^\dagger \nabla_\theta F(\theta) = \Pi_{\partial_\theta u_\theta}\left[ D_u \mathscr{ F } (u_\theta)  \right]. \]
Correspondingly, we denote the orthogonal projection onto the orthogonal complement of $\mathrm{span}  \{\frac{\partial u_\theta}{\partial \theta_1}, \dots, \frac{\partial u_\theta}{\partial \theta_m}\}  $ as $\Pi_{\partial_\theta u_\theta^\perp}:\mathbb X \rightarrow \mathbb{X}$, we have,
\[ D_u \mathscr{F}(u_\theta)    - \frac{\partial u_\theta}{\partial \theta} M(\theta)^\dagger \nabla_\theta F(\theta) = \Pi_{\partial_\theta u_\theta^\perp}\left[ D_u \mathscr{ F } (u_\theta)  \right]. \]

\subsection{Convergence analysis of the NPDG flow}\label{subsection: convergence of NPDG for linear PDE }

\blue{Throughout this section, we assume that $\Omega \subset \mathbb{R}^d$ is a bounded open domain with Lipschitz boundary $\partial \Omega$ \citep{grisvard2011elliptic}. That is, $\partial \Omega$ can be locally represented as the graph of a Lipschitz function.}

{\color{black}
\subsubsection{\textit{A posteriori} convergence result for general linear PDEs}}
Recall that we consider the linear equation \eqref{linear PDE} defined on $\mathbb{H}$. We assume $u_*$ as a real solution to \eqref{linear PDE}. We will adopt the notations used in previous Section \ref{subsec: PDHG for PDE} and \ref{subsec: precond PDHG method for PDE}. In our discussion, we always assume that the operator $\widetilde{\mathcal L}$ is bounded from above and below in the sense of
\begin{equation}
  0 < L_0 \leq \inf_{ \textbf{u} \in \primalspc }\frac{\|\widetilde{\mathcal L}  \textbf{u}  \|_{L^2(\Omega)}}{\|\textbf{u}\|_{L^2(\Omega)}} \leq
  \sup_{ \textbf{u} \in \primalspc }\frac{\|\widetilde{\mathcal L} \textbf{u}\|_{L^2(\Omega)}}{\|\textbf{u}\|_{L^2(\Omega)}} \leq L_1 <\infty.  \label{boundness of L, L_0, L_1}
\end{equation}
We denote $L_1 \vee 1 = \max\{L_1, 1\}$ and $L_0 \wedge 1 = \min\{L_0, 1\}$, and 
\begin{equation}
\widetilde{\kappa} = \frac{L_1\vee 1}{L_0 \wedge 1}\label{def: tilde kappa}
\end{equation}
for shorthand.

Suppose that we perform the NPDG flow up to a time $T$. We denote $\alpha, \beta_1, \beta_2\in [0, 1]$ as coefficients quantifying the approximation power of the subspaces spanned by $\left\{(\frac{\partial \mathcal M_p  u_{\theta_t}}{\partial \theta_k}, \sqrt{\lambda}\frac{\partial\mathcal B u_{\theta_t}}{\partial \theta_k})\right\}_{1\leq k \leq m_\theta}$, $\left\{ \frac{\partial \mathcal M_d \varphi_{\eta_t}}{\partial \eta_k}\right\}_{1\leq k \leq m_{\eta}}$, and $\left\{\frac{\partial \psi_{\xi}}{\partial \xi_k}\right\}_{1\leq k \leq m_\xi}$ for $t\in [0, T]$. To be more specific, $\alpha$ is a constant satisfying
{\small\begin{align*}
    \min_{\substack{\zeta\in\mathbb{R}^{m_\theta}\\ \zeta_1, \dots, \zeta_{m_\theta} \in \mathbb{R}}} & \left\{ {\Big\|\sum_{k=1}^{m_\theta} \zeta_k \frac{\partial \mathcal M_p u_{\theta_t}}{\partial \theta_k} - \mathcal M_p (u_{\theta_t} - u_*) \Big\|_{L^2(\Omega)}^2 + \Big\|\sum_{k=1}^{m_\theta} \zeta_k \sqrt{\lambda} \frac{\partial \mathcal B u_{\theta_t}}{\partial \theta} - \sqrt{\lambda}  \mathcal B (u_{\theta_t} - u_*)\Big\|_{L^2(\partial \Omega)}^2} \right\}\\
    &  \quad \quad \quad \quad \quad \quad   \leq  \alpha^2 (\|\mathcal M_p (u_{\theta_t} - u_*)\|_{L^2(\Omega)}^2 + \| \sqrt{\lambda} \mathcal B (u_{\theta_t} - u_*)\|_{L^2(\partial \Omega)}^2), \quad \textrm{for all } t \in [0, T]. 
\end{align*}}
Recall that we define $  \mathbb{L}^2  = L^2(\Omega)\times L^2(\partial\Omega)$, we denote the subspace 
\begin{equation*}
  \partial_\theta \textbf{U}_\theta = \mathrm{span} \left\{\left(\frac{\partial \mathcal M_p  u_\theta}{\partial \theta_1}, \sqrt{\lambda}\frac{\partial\mathcal B u_\theta}{\partial \theta_1}\right),\dots,\left(\frac{\partial \mathcal M_p  u_\theta}{\partial \theta_{m_\theta}}, \sqrt{\lambda}\frac{\partial\mathcal B u_\theta}{\partial \theta_{m_\theta}}\right)\right\} \subset \mathbb{L}^2.
\end{equation*}
Then, $\alpha$ quantifies the upper bound of the relative $\mathbb{L}^2$ norm of $\partial_\theta \textbf{U}_{\theta_t}^\perp$ component of $(\mathcal M_p (u_{\theta_t} - u_*), \mathcal B (u_{\theta_t} - u_*))$ for $t \in  [0, T]$. Here we denote $\partial_\theta \textbf{U}_{\theta_t}^\perp$ as the subspace of $\mathbb{L}^2$ that is orthogonal to $\partial_\theta \textbf{U}_{\theta_t}$ w.r.t. $\mathbb{L}^2$ inner product. Similarly, we denote the subspace
\blue{
  \[  \partial_{\eta, \xi} \boldsymbol{\Phi}_{\eta, \xi} = \mathrm{span}\left\{ \frac{\partial \mathcal M_d \varphi_{\eta}}{\partial \eta_k} \right\}_{ k=1 }^{m_{\eta}} \times   \mathrm{span}\left\{ \sqrt{\lambda}\frac{\partial \psi_{\xi}}{\partial \xi_k} \right\}_{ k = 1 }^{m_\xi} \subset \mathbb{L}^2. \]
}
$\beta_1, \beta_2$ denote the upper bounds of the relative $\mathbb{L}^2$ norms of $\partial_{\eta, \xi} \boldsymbol{\Phi}_{\eta_t, \xi_t}^\perp$ components of $(\widetilde{\mathcal L} \mathcal M_p ( u_{\theta_t} - u_* ), \sqrt{\lambda}\mathcal B (u_{\theta_t} - u_* ))$ and $(\mathcal M_d \varphi_{\eta_t}, \sqrt{\lambda}\psi_{\xi_t})$. The detailed definitions of $\alpha, \beta_1, \beta_2$ can be found later in \eqref{def: alpha}, \eqref{def: beta1} and \eqref{def: beta2}. 

The following Theorem analyzes the convergence of the numerical solution $u_{\theta_t}$ solved from \eqref{PDHG flow} on $[0, T]$.
\begin{theorem}[\textit{A posteriori} convergence analysis of NPDG flow]\label{thm: convergence analysis of NPDG flow} 
Suppose $\{(\theta_t, \eta_t, \xi_t)\}$ solves the NPDG flow \eqref{PDHG flow} on $[0, T]$. Recall that $\alpha, \beta_1, \beta_2$ quantify the approximation quality of neural networks $u_{\theta_t}, \varphi_{\eta_t}, \psi_{\xi_t}$ through $[0, T]$, and $\widetilde{\kappa}$ denotes the condition number \eqref{def: tilde kappa}. Suppose $\alpha + \beta_1 < \frac{1}{\widetilde{\kappa}^2}$, $\beta_2 < 1$, if we further assume that the hyperparameters of the NPDG flow $\gamma, \blue{\nu} > 0$ satisfy 
\begin{equation}
    \left(\frac{1}{\widetilde{\kappa}^2} - (\alpha + \beta_1)\right) \cdot (1-\beta_2) > \frac{ ((1+\beta_1)\gamma\blue{\nu} + \beta_2 + \alpha|1-\gamma\blue{\nu}|)^2}{ 4 \gamma\blue{\nu} }.  \label{ineq condition for posdef of Gamma}
\end{equation} 
Then there exists a constant $r > 0$, such that 
\begin{equation*}
  \|\mathcal M_p (u_{\theta_t} - u_* )\|_{L^2(\Omega; \mathbb{R}^r)}^2 + \lambda \|\mathcal B (u_{\theta_t} - u_*) \|_{L^2(\partial \Omega) }^2 \leq 2  \exp(-rt)  \cdot C_0 \quad \textrm{for } 0 \leq t \leq T.
\end{equation*}
Here $C_0  \geq  0$ is a constant depending on the initial value $(\theta_0, \eta_0, \xi_0)$ of the NPDG flow. We note that $r>0$ is the convergence rate depending on $\widetilde{\mathcal L}$, the hyperparameters $\gamma, \blue{\nu}$, and the relative errors $\alpha, \beta_1, \beta_2$. The explicit form of $r$ is provided in \eqref{def: convergence rate}.
\end{theorem}
The proof of Theorem \ref{thm: convergence analysis of NPDG flow} is provided in Appendix \ref{append: sec proof a posteriori convergence analysis}.

\blue{The dependence of the convergence rate \( r \) on \( \gamma, \nu, \alpha, \beta_1, \beta_2 \) can be significantly simplified as \( \alpha \) and \( \beta_2 \) approach $0$. This assumption is reasonable provided that the tangent spaces associated with the primal and test networks, \( \partial \mathbf U_\theta \) and \( \partial \Phi_{\eta,\xi} \), possess sufficiently strong approximation power for \( \mathbf U_\theta \) and \( \Phi_{\eta,\xi} \), respectively.} Actually, if we assume that $u_\theta(\cdot), \varphi_\eta(\cdot), \psi_\xi(\cdot)$ are linear combinations of basis functions, i.e., 
\begin{equation}
    u_\theta(x) = \sum_{k=1}^{m_\theta} \theta_k u_k(x),  \quad \varphi_\eta(x) = \sum_{k=1}^{m_\eta} \eta_k \varphi_k(x), \quad \psi_\xi(x) = \sum_{k=1}^{m_\xi} \xi_k \psi_k(x),  \label{linear combination}
\end{equation}
with $\theta, \eta, \xi$ serving as the coefficients of the basis functions and $u_k\in H^2(\Omega)$, $\varphi_k\in H_0^1(\Omega)$, $\psi_k\in L^2(\partial \Omega)$. If $u_*$ can further be represented by linear combination of $\{u_k\}_{k=1}^{m_\theta}$, then it holds exactly that $\alpha = \beta_2 = 0$. And the sufficient condition \eqref{ineq condition for posdef of Gamma} for the convergence of NPDG flow reduces to $ \gamma\blue{\nu} \leq \frac{4(\widetilde{\kappa}^{-2} - \beta_1)}{(1+\beta_1)^2}.$
\blue{Furthermore, we have the explicit lower bound of $r$
\begin{equation}\label{explicit convergence rate origin analysis}
  r \geq \frac{\left(4\left(\frac{1}{\widetilde{\kappa}^2} - \beta_1\right) - (1+\beta_1)^2\gamma\nu\right)\cdot\gamma\nu}{8\left((\frac{1}{\widetilde{\kappa}^2  } - \beta_1 )\gamma + \frac{\nu}{(L_1\vee 1)^2}\right)}.
\end{equation}}

\blue{Theorem \ref{thm: convergence analysis of NPDG flow} is \textit{a posteriori} analysis as one may verify whether the estimate in the theorem is valid \textit{after} obtaining the solution $(\theta_t,\eta_t,\xi_t)$, by checking whether condition~\eqref{ineq condition for posdef of Gamma} holds on a finite time interval $[0,T]$.} It is worth mentioning that as $t$ increases, $u_{\theta_t}$ will approach the real solution $u_*$; however, as the approximation gets better, the error term $\left( \mathcal M_p  (u_{\theta_t} - u_*), \mathcal B  (u_{\theta_t} - u_*) \right)$ will erect orthogonally away from the exploration space $\partial_\theta \textbf{U}_{\theta_t}$. Consequently, the quantity $\alpha$ will approach $1$, and so will $\beta_1 \rightarrow 1$. That may prevent condition $\alpha+\beta_1 < \frac{1}{\widetilde{\kappa}^2}$ at a certain time $t$ along the NPDG flow (recall that $\widetilde \kappa > 1$), thus yielding the analysis only applicable on a finite time interval.

{\color{black}
\subsubsection{A refined convergence analysis for elliptic PDEs of divergence form}\label{subsec: numerical analysis with stronger bdd norm}

The main challenge in establishing an \emph{a priori} convergence result that extends to the infinite time horizon stems from the imbalance between the primal function space \( \mathbb{H} \) and the test function space \( \blue{\mathbb{K}^{\mathrm{test}}} \). To be more specific, let us consider the Dirichlet boundary problem associated with Example \ref{ex: 1}
\[ -\Delta u = f, \quad \mathcal B u = g, \]
where $\mathcal B$ denotes the trace operator. Then we set $\mathbb{H}=H^2(\Omega), \mathbb{K}^{test}=H_0^1(\Omega)$. Suppose we compute the equation using $u_\theta, \varphi_\eta, \psi_\xi$ as defined in \eqref{linear combination}. Then, the convergence of the NPDG flow is anticipated as long as the tangent space $\partial_\theta \textbf{U}_\theta=\mathrm{span}\{\nabla u_k, \mathcal B u_k\}$ and $\partial_{\eta, \xi} \boldsymbol{\Phi}_{\eta, \xi}=\mathrm{span}\{\nabla\varphi_k\}, \times \mathrm{span}\{\psi_k\}$ along which the primal function $(\nabla u_\theta, \mathcal B u_\theta)$ and the test functions $(\varphi_\eta, \psi_\xi)$ move can effectively approximate the gradients of $\mathscr{E}(u_{\theta}, \varphi_{\eta}, \psi_{\xi})$. The main issue arises when using $\mathrm{span}\{\nabla \varphi_k\}$ to approximate the term $\nabla(u_\theta - u_*).$ 
Since each $\varphi_k \in \blue{\mathbb{K}^{test}} = H_0^1(\Omega)$, however, $u_\theta - u_*\in \mathbb{H}=H^2(\Omega)$ is not guaranteed to lie in $H_0^1(\Omega)$. Therefore, a significant approximation error
\begin{equation} 
  \|\Pi_{\mathrm{span}\{\nabla\varphi_1, \dots, \nabla\varphi_{m_\eta}\}^\perp}[\nabla(u_\theta - u_*)]\|_{L^2(\Omega)}^2 = \min_{\varphi \in \mathrm{span}\{\varphi_1, \dots, \varphi_{m_\eta}\}}\| \nabla (u_\theta - u_*) - \nabla \varphi \|^2_{L^2(\Omega)},  \label{proj on H_0 error term}
\end{equation}
will arise and it will bound $\beta_1$ away from $0$, and hence impeding the convergence of the NPDG flow. Intuitively, if we assume that $\mathrm{span}\{\varphi_1, \dots, \varphi_{m_\eta}\}$ constitutes a basis of $H_0^1(\Omega)$ as $m_\eta\rightarrow \infty,$ the term \eqref{proj on H_0 error term} can exactly be bounded from above by the fractional Sobolev norm $\|u_\theta - u_*\|^2_{H^{1/2}(\partial \Omega)}$ \citep{gagliardo1957caratterizzazioni, grisvard2011elliptic}.  
This key observation yields a remedy to substitute the original boundary loss term $L^2(\partial \Omega)$ with a stronger norm such as $H^{1/2}(\partial \Omega)$ in order to offset the approximation error caused by the imbalance between the primal and the test functional spaces during natural gradient optimization.

In the following discussion, we establish a modified version of Theorem \ref{thm: convergence analysis of NPDG flow} by incorporating suitable boundary loss. We primarily focus on an important class of linear elliptic equations in divergence form,
\[ 
  -\nabla\cdot(A(x)\nabla u(x))=f(x) ~ \textrm{on } \Omega, \quad u(x)=g(x) ~\textrm{on } \partial \Omega,
\]
where $f\in L^2(\Omega), g\in H^1(\partial \Omega) \subset H^{1/2}(\partial \Omega),$ $A(\cdot):\Omega\rightarrow \mathbb{R}^{d\times d}$ with each entry $A_{ij}(\cdot)\in L^2(\Omega)\cap C^1(\Omega)$. We further assume $A(\cdot)$ is bounded from both above and below, i.e., for any $x\in\mathbb{R}^d$, we always have $\underline{A}\cdot\|\xi\|^2 \leq \xi^\top A(x) \xi \leq \overline{A} \|\xi\|^2$, with the constants $\overline{A} \geq \underline{A} > 0.$ 

Let us denote the Hilbert space $\mathcal X$ such that $H^{3/2}(\partial \Omega)\subseteq\mathcal X\subseteq H^{1/2}(\partial \Omega)$ with continuous inclusion map $\mathcal X \hookrightarrow H^{1/2}(\partial \Omega).$ Then, there exists a constant $C_\mathcal X>0$ such that
\[ \|\mathcal B w\|_{\mathcal X} \geq C_\mathcal X \|\mathcal B w\|_{H^{1/2}(\partial \Omega)}, \quad \textrm{for any } w \in \mathcal X. \]
Denote $\langle\cdot\, , \, \cdot\, \rangle_{\mathcal X}$ as the inner product on $\mathcal X$, we slightly modify the original functional $\mathscr{E}$ by using the $\mathcal X-$boundary norm,
\begin{equation}
\begin{split}\label{def: new loss funcitional with modif bdd err}
  {\mathscr{E}_\mathcal X}(u, \varphi, \psi)   
  \!=\! \langle \widetilde{\mathcal L} \mathcal M_p  u , \mathcal M_d \varphi \rangle_{L^2} \! - \! \langle f, \varphi \rangle_{L^2} \! + \! \lambda \langle \mathcal B u - g, \psi \rangle_{\mathcal X} \! - \! \frac{\blue{\nu}}{2}(\|\mathcal M_d \varphi\|^2_{L^2} \! + \! \lambda \|\psi\|^2_{\mathcal X}).
\end{split}
\end{equation}
Suppose we conduct the NPDG flow \eqref{PDHG flow} 
associated with $\mathscr{E}_{\mathcal X}(u_\theta, \varphi_\eta, \psi_\xi)$. 
We further introduce a useful seminorm $|\cdot|_{H^1(\Omega, A)}$ on $H^1(\Omega)$:
\begin{equation}
  |u|_{H^1(\Omega, A)}^2 = \int_\Omega \nabla u(x)^\top A(x) \nabla u(x) \, \dd x, \quad \textrm{ for } u \in H^1(\Omega).  \label{def: seminorm dot H1(Omega, A)}
\end{equation}
We have the following theorem.
\begin{theorem}[A refined convergence analysis of NPDG flow]\label{thm: convergence analysis on elliptic PDE of divergence type}
  Suppose we pick the hyperparameters $\gamma, \blue{\nu}>0$ such that $\gamma\blue{\nu} < 2$. Assume the coefficient $\lambda$ associated with the boundary loss is sufficiently large, i.e., $\lambda \geq 8\overline{A}\left(\frac{ \gamma(1+\blue{\nu}) }{C_\mathcal X \cdot c_\Omega}\right)^2,$
  where $c_\Omega>0$ is a constant whose value is discussed in detail in Appendix~\ref{append: subsec frac Sobolev spc}. Then we have
  \begin{equation}\label{refined thm bound}
    |u_{\theta_t}-u_*|_{ H_1(\Omega, A)}^2 + \lambda\|u_{\theta_t} - u_*\|_{\mathcal X}^2 \leq  \left( \! e^{-rt} \! \sqrt{E_0} \! + \! \int_0^t \! \frac{1}{2}e^{-r(t-\tau)} \mathrm{Err}(\theta_\tau, \eta_\tau, \xi_\tau, \gamma, \nu, \lambda) \dd \tau \! \right)^2. 
  \end{equation}
 Here, $r$ denotes the convergence rate. It can be shown that the convergence rate 
 \begin{equation}\label{convergence rate of refined theorem}
    r \geq \frac12\cdot\frac{\gamma\blue{\nu}(2-\gamma\blue{\nu})}{\gamma+2\blue{\nu}}.
 \end{equation}
 We set $E_0 = |u_{\theta_0}-u_*|_{ H_1(\Omega, A)}^2 + \lambda\|u_{\theta_0} - u_*\|_{\mathcal X}^2$ as the initial error. $\mathrm{Err}(\theta_t, \eta_t, \xi_t, \gamma, \nu, \lambda)$ represents the summation of the approximation errors associated with the tangential spaces of primal and test networks. The detailed formulation is provided in \eqref{def: Err(t)}.  
\end{theorem}
We prove Theorem \ref{thm: convergence analysis on elliptic PDE of divergence type} in Appendix \ref{append: subsec proof of refined thm}.

As a corollary of Theorem \ref{thm: convergence analysis on elliptic PDE of divergence type}, an \textit{a priori} convergence result can be obtained under the case in which $u_\theta, \varphi_\eta, \psi_\xi$ are linear combinations of basis functions.
\begin{corollary}[\textit{A priori} convergence of NPDG flow]\label{coro: a priori convergence analysis}
  Suppose that $u_\theta, \varphi_\eta, \psi_\xi$ are linear combinations \eqref{linear combination} of suitable basis functions $u_k\in H^2(\Omega)$, $\varphi_k\in H_0^1(\Omega)$, $\psi_k \in \mathcal X$. We keep all the notations and the conditions on $\gamma,\blue{\nu},\lambda$ mentioned in Theorem \ref{thm: convergence analysis on elliptic PDE of divergence type}. We denote
  \begin{equation}
  \begin{split}
    & \mathcal E_u := \min_{\zeta\in\mathbb{R}^{m_\theta}} \int_\Omega \|\sum_{k=1}^{m_\theta}\zeta_k\nabla u_k - \nabla u_*\|^2 \dd x + \frac{\lambda}{\overline{A}}\|\sum_{k=1}^{m_\theta}\zeta_k \mathcal B u_k - g\|_{\mathcal X}^2. \\
    & \mathcal E_{\nabla\varphi} := \min_{\zeta\in\mathbb{R}^{m_\eta}} \int_\Omega \|\sum_{k=1}^{m_\eta}\zeta_k\nabla \varphi_k - \nabla \varphi_*\|^2 \dd x, \quad \mathcal E_{\psi} := \min_{\zeta\in\mathbb{R}^{m_\xi}} \|\sum_{k=1}^{m_\xi}\zeta_k \psi_k - g\|_{\mathcal X}^2.
     \end{split}
  \end{equation}
  Here we denote $\varphi_*\in H_0^1(\Omega)$ as the (weak) solution $\varphi$ to
  \begin{equation*}
    -\nabla\cdot(A(x)\nabla\varphi(x)) = -\nabla\cdot(A(x)\nabla u_*(x)), ~ \textrm{on }\Omega,  \quad  \varphi = 0 ~ \textrm{on }\partial \Omega. 
  \end{equation*}
  Suppose the basis $\{u_k\}$ are picked so that $\mathrm{span}\{\mathcal T u_k\}\subseteq \mathrm{span}\{\varphi_k\}, \mathrm{span}\{\mathcal B u_k\} \subseteq \mathrm{span}\{\psi_k\},$ we then have 
  \begin{equation}\label{thm refine for linear comb u varphi psi bound}
  \begin{split}
    \Big(
        |u_{\theta_t}-u_*|_{ H_1(\Omega, A)}^2
        + \lambda \|u_{\theta_t}-u_*\|_{\mathcal X}^2
    \Big)^{\frac12}
    \;\le\;&
    e^{-rt}\sqrt{E_0} \\
    &+ \frac{1-e^{-rt}}{r}
    \Bigg[
    (3 \vee \gamma)\sqrt{\overline{A}\,\mathcal E_u}
    + \frac{\gamma+3}{\sqrt{2}}
    \sqrt{\overline{A}\,\mathcal E_{\nabla \varphi} + \mathcal E_\psi}
    \Bigg].
  \end{split}
  \end{equation}
  The convergence rate $r > 0$ satisfies the same lower bound as described in \eqref{convergence rate of refined theorem}.
\end{corollary}
The corollary is proved in Appendix \ref{append: coro proof}.

\begin{remark}
It is straightforward to verify the following bound
\begin{equation}\label{thm refine for linear comb u varphi psi bound2}
\begin{split}
    \Big(
        \|\nabla u_{\theta_t}-\nabla u_*\|_{L^2}^2
        \! + \! \lambda \|u_{\theta_t}-u_*\|_{\mathcal X}^2 
    \Big)^{\frac12} &  
    \! \le 
    e^{-rt}\sqrt{\frac{E_0}{\underline{A}}} \\
    & \! + \! \frac{1-e^{-rt}}{r} \!
    \Bigg[
    (3 \vee \gamma)\sqrt{\kappa(A)\,\mathcal E_u}
    \!+\! \frac{\gamma+3}{\sqrt{2}}
    \sqrt{\kappa(A)\,\mathcal E_{\nabla \varphi} + \frac{\mathcal E_\psi}{\underline{A}}}
    \Bigg].
\end{split}
\end{equation}
provided $\xi^\top A(\cdot) \xi \geq \underline{A}\|\xi\|^2$ for $\xi\in\mathbb{R}^d$. Here we denote $\kappa(A):=\overline{A}/\underline{A}\geq 1$.
\end{remark}

A natural choice of the boundary norm $\mathcal X$ is the fractional Sobolev norm $H^{1/2}(\partial \Omega)$. However, the definition \eqref{def: H^1/2 norm} might not be convenient to evaluate. A more practical choice could be the $H^1(\partial \Omega)$ norm, which has been considered in recent reference such as \citep{shao2025solving} for nonlinear equations. In Section \ref{subsec: varcoeff}, we present numerical examples that incorporate the $H^{1}(\partial \Omega)$ boundary norm computed using Monte-Carlo approximations. The numerical evidence provided in Figure \ref{fig: VarCoeff1} suggests improved convergence and accuracy of the algorithm.

\begin{remark}[Hyperparameters $\gamma, \blue{\nu}, \lambda$]
Notice that in the \textit{a posteriori} analysis presented in Theorem~\ref{thm: convergence analysis of NPDG flow}, the convergence rate \( r \) depends not only on the hyperparameters \( \gamma \) and \( \nu \), but also on the relative approximation errors \( \alpha, \beta_1, \beta_2 \), which themselves depend on \( \gamma \) and \( \nu \). This interdependence makes it infeasible to determine suitable values of the hyperparameters \( \gamma \) and \( \nu \) in practice.

In contrast, under the current framework, the convergence rate \( r \) depends solely on \( \gamma,\nu \). As a result, we can explicitly maximize the lower bound $\frac{1}{2}\cdot \frac{\gamma \nu \bigl(2 - \gamma \nu\bigr)}{\gamma + 2\nu},$
which is attained at \( \gamma_* = \frac{2}{\sqrt{3}} \) and \( \nu_* = \frac{1}{\sqrt{3}} \). Consequently, the convergence rate satisfies $r \;\ge\; \frac{2}{3\sqrt{3}}.$

In both Theorem \ref{thm: convergence analysis on elliptic PDE of divergence type} and Corollary \ref{coro: a priori convergence analysis}, the choice of $\lambda>0$ relies on the constant $c_\Omega$ associated with the $H^{1/2}(\partial \Omega)$ norm. It is usually intractable to determine $c_\Omega$ for general $\Omega$. In practice, we pick a rather large $\lambda=10$ to ensure better performance of the method. 
\end{remark}}

{\color{black}
\begin{remark}
The Lyapunov-based proof framework for the NPDG flow developed in this section is inspired by earlier studies in \citep{liu2022primal}, \citep{liu2024numerical} which focus on fully time-implicit schemes for conservation laws and reaction--diffusion equations. We clarify several fundamental differences that distinguish the present theoretical analysis from these previous works in Appendix \ref{apped: subsec_comp_prev_works}. 
\end{remark}}

\section{Algorithm}  \label{sec: alg}

In this section, we provide a detailed description on implementing the NPDG algorithm \eqref{def: intro alg NPDG}. We take the linear PDE \eqref{linear PDE} as an illustrative example.

\subsection{Loss functional and the precondition matrices}
Recall our discussions in Section \ref{subsec: precond PDHG method for PDE}. We introduce the pair of dual neural networks $(\varphi_\eta, \psi_\xi )$ to equation \eqref{linear PDE} and consider the loss functional
\begin{equation*}
\begin{split}    
  \mathscr{E}(u, \varphi, \psi) =& \left( \int_\Omega \widetilde{\mathcal{L}} \mathcal M_p u(x)\cdot\mathcal M_d \varphi(x) - f(x)\varphi(x) \,\dd\mu - \frac{\blue{\nu}}{2} \int_\Omega \mathcal M_d \varphi(x)\cdot \mathcal M_d \varphi(x) \,\dd\mu \right) \\
  & + {\lambda} \left( \int_{\partial \Omega} (\mathcal B u(y) - g(y))\cdot \psi(y) \,\dd\mu_{\partial \Omega} - \frac{\blue{\nu}}{2} \int_{\partial \Omega} \psi(y)\cdot \psi(y) \,\dd\mu_{\partial \Omega} \right).
\end{split}
\end{equation*}
\blue{It is worth noting that the loss functional $\mathscr{E}(u,\varphi,\psi)$ considered here differs slightly from that introduced in~\eqref{preconditioned J loss funcitional }, as we incorporate probability measures $\mu$ and $\mu_{\partial\Omega}$ in the evaluation of the integrals so that the loss can be efficiently approximated using Monte Carlo methods. Moreover, we assume that both measures are absolutely continuous with respect to the Lebesgue measures on $\Omega$ and $\partial\Omega$, respectively, with strictly positive densities, ensuring that the algorithm adequately explores the entire domain and boundary. A convenient choice is to take both $\mu$ and $\mu_{\partial\Omega}$ to be uniform distributions over $\Omega$ and $\partial\Omega$.}

Sometimes, it also helps if we add the $L^2$ boundary loss
\[ \|\mathcal B u - g\|_{L^2(\partial \Omega, \mu_{\partial \Omega})}^2 = \int_{\partial \Omega} (\mathcal B u - g)^2 \,\dd\mu_{\partial \Omega} \]
into $\mathscr{E}(u, \varphi, \psi)$, and consider
\[ \widetilde{\mathscr{E}}(u, \varphi, \psi) = \mathscr{E}(u, \varphi, \psi) + \lambda \|\mathcal B u - g\|_{L^2(\partial \Omega, \mu_{\partial \Omega})}^2. \]
We also recall that the precondition matrices $M_p(\theta)\in\mathbb{R}^{m_\theta\times m_\theta}, M_d(\eta)\in\mathbb{R}^{m_\eta\times m_\eta}, M_{bdd}(\xi)\in\mathbb{R}^{m_\xi \times m_\xi}$ are defined in \eqref{def: M_p}, \eqref{def: M_d} and \eqref{def: M_bdd}.

\subsection{Monte-Carlo approximation }
We apply the Monte Carlo algorithm to approximate the loss function ${\mathscr{E}}(u_\theta; \varphi_\eta, \psi_\xi)$ throughout our computation. Assume that $\{\boldsymbol{X}_i\}_{1\leq i\leq N_{in}}$, $\{\boldsymbol{Y}_j\}_{1\leq j\leq N_{bdd }}$ are samples uniformly drawn from the domain $\Omega$ and its boundary $\partial \Omega$, respectively. 
We compute
\begin{align}
  {\mathscr{E}}(u_\theta; \varphi_\eta, \psi_\xi) \approx & \frac{ 1}{N_{{in}}} \sum_{i=1}^{N_{in} } \widetilde{\mathcal L} \mathcal M_p u(\boldsymbol{X}_i)\cdot\mathcal M_d \varphi(\boldsymbol{X}_i) -f(\boldsymbol{X}_i)\varphi(\boldsymbol{X}_i) - \frac{\blue{\nu}}{2}\mathcal M_d \varphi_\eta(\boldsymbol{X}_i)\cdot \mathcal M_d \varphi_\eta(\boldsymbol{X}_i)  \nonumber\\
  & + \lambda \left( \frac{  1  }{N_{bdd}} \sum_{j=1}^{N_{bdd}} (\mathcal B u(\boldsymbol{Y}_j)-g(\boldsymbol{Y}_j))\psi_\xi(\boldsymbol{Y}_j) - \frac{\blue{\nu}}{2}\psi_\xi(\boldsymbol{Y}_j)\psi_\xi(\boldsymbol{Y}_j)  \right).  \nonumber
\end{align}
Here, ``$\cdot$'' denotes the inner product of vectors. For example, if $\mathcal M_p \! = \! \mathcal M_d \! = \! \nabla$, $\widetilde{\mathcal L}=\mathrm{Id}$, then
\[ \widetilde{\mathcal L} \mathcal M_p u(\boldsymbol{X}_i) \cdot \mathcal M_d \varphi(\boldsymbol{X}_i) = \nabla u(\boldsymbol{X}_i) \cdot \nabla \varphi(\boldsymbol{X}_i), \quad \mathcal M_d \varphi_\eta(\boldsymbol{X}_i)\cdot \mathcal M_d \varphi_\eta(\boldsymbol{X}_i) = \|\nabla \varphi(\boldsymbol{X}_i)\|^2.\]
For general nonlinear PDE, the loss function $\mathscr{E}(u_\theta; \varphi_\eta, \psi_\xi)$ can also be approximated via the Monte-Carlo algorithm. Its gradient $\nabla_\theta \mathscr{E}(u_\theta;\varphi_\eta,\psi_\xi)$ can be computed using auto-differentiation \citep{baydin2018automatic}.

Furthermore, it is also straightforward to evaluate the preconditioning matrices $M_p(\theta)$ via Monte Carlo method, for example $M_p(\theta)$ can be computed as \footnote{Recall that we denote $\frac{\partial}{\partial \theta} \mathcal M_p u_\theta(\boldsymbol{X}_i)$ as the Jacobian of $\mathcal M_p u_\theta$ at $\boldsymbol{X}_i$. For example, if one sets $\mathcal M_p = \nabla$, then $\frac{\partial}{\partial \theta} \mathcal M_p u_\theta(\boldsymbol{X}_i)$ is a $d\times m$ matrix. Similarly, we assume $\frac{\partial u_\theta(\boldsymbol{Y}_j)}{\partial \theta}$ is $1\times d$. }
\begin{align*}
 &  M_p(\theta) \approx \frac{ 1}{N_{{in}}} \sum_{i=1}^{N_{in}}  
 \frac{\partial}{\partial \theta}(\mathcal M_p  u_\theta(\boldsymbol{X}_i))^\top \frac{\partial}{\partial\theta}(\mathcal M_p u_\theta(\boldsymbol{X}_i)) + \frac{ \lambda }{N_{{in}}} \sum_{j=1}^{N_{bdd}}  
 \frac{\partial}{\partial \theta} u_\theta(\boldsymbol{Y}_j)^\top \frac{\partial}{\partial\theta}u_\theta(\boldsymbol{Y}_j) ,\\
\textrm{i.e., } & (M_p(\theta))_{ij} = \frac{ 1 }{N_{{in}}} \sum_{i=1}^{N_{in}} \frac{\partial \mathcal M_p  u_\theta(\boldsymbol{X}_i)}{\partial \theta_i} \cdot \frac{\partial \mathcal M_p u_\theta(\boldsymbol{X}_i) }{\partial\theta_j} + \frac{ \lambda}{N_{bdd}} \sum_{j=1}^{N_{bdd}} \frac{\partial u_\theta(\boldsymbol{Y}_j)}{\partial \theta_i}\frac{\partial u_\theta(\boldsymbol{Y}_j)}{\partial \theta_j}
.  
\end{align*}
It is worth mentioning that we use the \textit{same} set of samples for computation of both the loss function and the preconditioning matrices. \blue{A new set of samples is generated at each iteration of the NPDG algorithm.}

\subsection{Inverting the preconditioning matrices via Krylov iterative solver }
We then solve the least square problem \eqref{derive least square problem} for $\mathbf{v}$. As mentioned in Remark \ref{rk: pseudo inverse}, this is equivalent to \blue{solving} the linear equation 
\begin{equation} 
M_p(\theta_k) \textbf{v} = \nabla_\theta \mathscr{E}(u_{\theta_k}; \varphi_{\eta_k}, \psi_{\xi_k}).  \label{linear equ precondition }
\end{equation}
However, this may suffer from the limitation on scalability: The method always computes and records the entire preconditioning matrix $M_p(\theta_k)$ at each optimization step. For neural networks such as Multilayer Perceptron, $M_p(\theta)$ is generally non-sparse, which suggests that forming this $m\times m$ matrix will occupy immense memory space of the computing resources as the number of parameters of the neural networks increases. For example, in numerical experiment \ref{subsec: varcoeff}, we deal with MLP $u_\theta(\cdot)$ with $d_{in}=50, d_h=256, d_{out}=1 $ and $n_l = 6$, this neural network contains $m_\theta = 279090$ parameters. Forming such $m_\theta \times m_\theta$ matrix is generally infeasible.

As a mitigation, instead of the direct evaluation of the preconditioning matrices, we apply the MINRES algorithm, which is an iterative solver, to solve \eqref{linear equ precondition }. The MINRES iterative solver only requires matrix-vector multiplications that can readily avoid the direct formation of the preconditioning matrices. Similar treatment is also utilized in \citep{dembo1982inexact, martens2010deep, roosta2022newton, rathore2024challenges} and the references therein in optimization problems. The same technique is also used in \citep{wu2023parameterized, jin2024parameterized} to handle the computation of Wasserstein geometric flows. 

We briefly describe how we evaluate $M_p(\theta)\textbf{v}$ for arbitrary vector $\textbf{v}\in\mathbb{R}^m$ under the deep learning framework. Given neural network $u_\theta(\cdot):\mathbb{R}^d\rightarrow \mathbb{R} $ with parameter $\theta\in\mathbb{R}^m$, we make a copy $u_{\theta'}^{\textrm{copy}}(\cdot)$ by inheriting the architecture of $u_\theta(\cdot)$ and by setting $\theta' = \theta.$ We apply auto-differentiation to evaluate $\{\mathcal{M}_p u_\theta(\boldsymbol{X}_i)\}_{i=1}^{N_{in}}$ and $\{\mathcal{M}_p u_{\theta'}^{\textrm{copy}}(\boldsymbol{X}_i)\}_{i=1}^{N_{in}}$, we also evaluate $\{u_{\theta}(\boldsymbol{Y}_j)\}_{j=1}^{N_{bdd}}, \{u_{\theta'}^{\textrm{copy}}(\boldsymbol{Y}_j)\}_{j=1}^{N_{bdd}}$.  Then we compute the scalar 
\begin{equation} 
  \Gamma(\theta', \theta) = \frac{ 1 }{N_{in}} \sum_{i=1}^{ N_{in} }  \mathcal M_p u_{\theta'}^{\textrm{copy}}(\boldsymbol{X}_i)  \cdot  \mathcal M_p u_\theta(\boldsymbol{X}_i) + \frac{ \lambda }{N_{bdd}}\sum_{j=1}^{N_{bdd}}  u_{\theta'}^{\textrm{copy}}(\boldsymbol{Y}_j)u_{\theta}(\boldsymbol{Y}_j).   \label{def: Gamma }
\end{equation}
Now, by applying auto-differentiation again, we take the partial derivative of $\Gamma_{in}(\theta', \theta)$ w.r.t. $\theta$, and making an inner product with $\textbf{v}$, this yields $\partial_\theta \Gamma_{in }(\theta', \theta)^\top \textbf{v}$. Finally, taking the partial derivative w.r.t. $\theta'$ yields $\partial_{\theta'}(\partial_\theta \Gamma(\theta', \theta)\textbf{v})|_{\theta'=\theta} \approx M_p(\theta)\textbf{v}.$ This suggests an effective way of evaluating $M_p(\theta)\textbf{v}$ without forming $M_p(\theta)$ explicitly. We summarize this procedure in the following Algorithm \ref{alg: mat-vec mult}.
\begin{algorithm}
\caption{Evaluating $M_p(\theta)\textbf{v}$}\label{alg: mat-vec mult precond}\label{alg: mat-vec mult}
\begin{algorithmic}[1]
\Require Preconditioning operators $\mathcal M_p, \mathcal M_d$. Neural network $u_\theta(\cdot)$, samples $\{\boldsymbol{X}_i\}_{i=1}^{N_{in}}\subset \Omega$, $\{\boldsymbol{Y}_j\}_{j=1}^{N_{bdd}}\subset\partial\Omega$, vector $\textbf{v}\in\mathbb{R}^m$. 
\State Make a copy $u_{\theta'}^{\textrm{copy}}(\cdot)$ of the given $u_\theta(\cdot)$ with $\theta'=\theta$.
\State Evaluate $\Gamma(\theta', \theta)$ as defined in \eqref{def: Gamma }.
\State Apply auto-differentiation to evaluate $\partial_\theta \Gamma(\theta', \theta)\textbf v$.
\State Apply auto-differentiation to evaluate $\textbf u = \partial_{\theta'}(\partial_\theta \Gamma(\theta', \theta)\textbf{v} )$.
\Ensure $\textbf{u}$
\end{algorithmic}
\end{algorithm}
Similarly, the matrix-vector multiplication involving $M_d(\eta), M_{bdd}(\xi)$ can be computed by using the same technique.

\begin{remark}
  Calculating $M_p(\theta)\mathbf{v}$ can be further simplified by using the finite-difference approximation, which may lead to faster speed and lower memory cost. This technique has been conducted in several Hessian-free optimization algorithms \citep{martens2010deep, knoll2004jacobian, rathore2024challenges}. This possible improvement will serve as the future research directions.
\end{remark}

We adopt the \texttt{scipy.sparse.linalg.minres} solver in SciPy \citep{2020SciPy-NMeth} throughout our implementation. This algorithm involves two key hyperparameters: the maximum number of iterations $n_{\textrm{MINRES}}$, and the tolerance value $tol_{\textrm{MINRES}}$, which determines the acceptable relative residual. For more details on selecting these parameters, please refer to Section \ref{sec: numerical examples }.

\subsection{Sketch of main algorithm}
We summarize the proposed method in Algorithm \ref{alg: NPDG}.
\begin{algorithm}[htb!]
\caption{Natural Primal-Dual Hybrid Gradient method (NPDG)}\label{alg: NPDG}
\begin{algorithmic}[1]
\Require The equation $F(u, \nabla u, \nabla^2 u, \dots) = 0$ on $\Omega$ with (if any) boundary condition $\mathcal B u = g$ on $\partial \Omega$. Preconditioning operators $\mathcal M_p, \mathcal M_d$. The functional $\mathscr{E}(u, \varphi, \psi)$. Stepsizes $\tau_u, \tau_\varphi$, $\tau_\psi$ of the NPDG algorithm; extrapolation coefficient $\omega$; Total iteration number of the NPDG algorithm $N_{iter}$. Number of samples drawn from $\Omega$ and $\partial \Omega$: $N_{in}, N_{bdd}$. Max iteration number $n_{\textrm{MINRES}}$ and tolerance of relative residual $tol_{\textrm{MINRES}}$ of the MINRES algorithm.
\State Initialize the primal neural network $u_\theta(\cdot)$, dual neural network(s) $\varphi_\eta(\cdot)$ and $\psi_\xi(\cdot)$ if the equation is equipped with boundary condition(s).
\For {$iter=1~\textrm{to}~N_{iter}$} 
    \State Set $\eta_0 = \eta, \xi_0 = \xi$.
    \State \blue{Resample points for the Monte-Carlo estimations in each iteration.}
    \State Apply Monte-Carlo algorithm and auto-differentiation to evaluate 
    \[(\textbf{w}_\varphi^\top, \textbf{w}_\psi^\top)^\top = \nabla_{(\eta, \xi)}\mathscr{E}(u_\theta, \varphi_\eta, \psi_\xi).\]
    \State Apply MINRES algorithm ($n_{\textrm{MINRES}}, tol_{\textrm{MINRES}}$) with Algorithm \ref{alg: mat-vec mult} to solve 
    \[M_d(\eta)\textbf{v}_\varphi=\textbf{w}_\varphi, M_{bdd}(\xi)\textbf{v}_\psi=\textbf{w}_\psi.\]
    \State Update $\eta = \eta + \tau_\varphi \textbf{v}_\varphi, ~ \xi = \xi + \tau_\psi \textbf{v}_\psi$.   \Comment{Natural gradient ascent}
    \State Set $\widetilde{\varphi} = \varphi_\eta + \omega(\varphi_\eta - \varphi_{\eta_0})$, $\widetilde{\psi} = \psi_\xi + \omega(\psi_{ \xi } - \psi_{\xi_0})$.  \Comment{Extrapolation in}
    \Statex \hfill  \textit{functional space}
    \State Apply Monte-Carlo algorithm and auto-differentiation to evaluate 
    \[\textbf{w}_u = \nabla_\theta \mathscr{E}(u_\theta, \widetilde{\varphi}, \widetilde{\psi}).\]
    \State Apply MINRES algorithm ($n_{\textrm{MINRES}}$, $tol_{\textrm{MINRES}}$) with Algorithm \ref{alg: mat-vec mult} to solve
    \[ M_p(\theta) \textbf{v}_u  = \textbf{w}_u. \]
    \State Update $\theta = \theta - \tau_u\textbf{v}_u$  \Comment{Natural gradient descent}
\EndFor
\Ensure $u_{\theta}(\cdot)$
\end{algorithmic}
\end{algorithm}

{\color{black}
\begin{remark}
\label{rk: adaptive sampling}
  The probability measures $\mu$ and $\mu_{\partial\Omega}$ used in our implementation need not be uniform. In fact, it is often more appealing to choose these measures adaptively according to the residuals $\big|\mathcal L u_\theta(\cdot)-f(\cdot)\big|$, $\big|\mathcal B u_\theta(\cdot)-g(\cdot)\big|$. Samples from such adaptive distributions can be generated using Markov chain Monte Carlo methods or deep generative models~\citep{tang2023pinns}. However, when the density of $\mu$ is non-constant, applying integration by parts introduces additional score terms associated with $\mu$, which may significantly complicate the construction of the corresponding preconditioning matrices. One possible alternative is to choose $\mathcal M_p=\mathcal L$, $\widetilde{\mathcal L}=\mathrm{Id}$, and $\mathcal M_d=\mathrm{Id}$, thereby avoiding integration by parts altogether. The integration of the NPDG framework with adaptive sampling strategies will remain as an important direction for future research.
\end{remark}
}

\section{Numerical Examples}\label{sec: numerical examples }
In this section, we apply the proposed Natural Primal-Dual Hybrid Gradient (NPDG) algorithm to various types of PDEs, including linear and nonlinear, static, and time-dependent equations. We denote our method as the NPDG algorithm for simplicity.

Throughout numerical experiments, we set neural networks as Multilayer Perceptron (MLP). That is, the fully connected neural network with the input dimension $d_{\textrm{in}}$, the hidden dimension $d_{\textrm{hidden}}$, the output dimension $d_{\textrm{out}}$, and the number of layers $n_{\textrm{MLP}}$. We denote such MLP with activation function $f$ as $\texttt{MLP}_f(d_{\textrm{in}}, d_{\textrm{hidden}}, d_{\textrm{out}}, n_{\textrm{MLP}})$. Readers are referred to Appendix \ref{append: MLP} for further details on MLP. 

\blue{Throughout the experiments in this section, we fix the hyperparameters $\nu = 1$ and $\omega = 1$ for the NPDG algorithm. When applicable, unless otherwise specified, the boundary loss coefficient is always set to $\lambda = 10$.  We always choose the maximum iteration number for the MINRES algorithm as $n_{\textrm{MINRES}}=1000$.}

We compare the proposed algorithm with a series of commonly used deep-learning solvers, namely, Physics-Informed Neural Network (PINN) \citep{raissi2019physics}, Deep Ritz method \citep{yu2018deep}, and primal-dual-type algorithms for PDEs/optimal transport \citep{zang2020weak} \citep{fan2023neural}. We apply Adam \citep{kingma2014adam, Pytorch} and (or) L-BFGS \citep{liu1989limited, Pytorch} algorithms to PINN. When we use the L-BFGS method, we choose $lr=1.0$ as the default. The L-BFGS method does not perform stably with the Deep Ritz and primal-dual type methods. We will only apply the Adam algorithm to these two methods. \blue{To remain consistent with NPDG Algorithm~\ref{alg: NPDG}, we resample at every iteration for all tested algorithms, except for the L-BFGS method, which exhibits severe instability when trained with randomly resampled batches.}

To keep the comparison fair, we keep the same neural network architecture for all the methods tested. We justify the computational efficiency of the proposed methods by summarizing the GPU-time costs of each method for different PDEs with various dimensions in Table \ref{tab: GPU time to accuracy }. The robustness of the proposed method is reflected in the semi-log plots of the relative $L^2$-loss for different equations. Necessary plots are also provided to visualize the numerical results produced by the proposed method.

The Python codes associated with the examples tested in this section can be accessed from the GitHub repository \href{https://github.com/LSLSliushu/NPDG}{https://github.com/LSLSliushu/NPDG}.
 
\subsection{Poisson's equation (10D, 50D)}\label{subsec: Poisson}
We consider the following Poisson's equation defined on the region $\Omega = [0,1]^d$.
\begin{equation}
  -\Delta u = f,~\textrm{on } \Omega, \quad u=g, ~\textrm{on }\partial \Omega,  \label{eq: Laplace}
\end{equation}   
where we define $f(x) = \sum_{k=1}^d  \frac{\pi^2}{4} \sin(\frac{\pi}{2}x_k)$, and $g(x)=\sum_{k=1}^d \sin(\frac{\pi}{2}x_k)$ on $\partial\Omega$. The exact solution of this equation is 
\[u_*(x) = \sum_{k=1}^d \sin(\frac{\pi}{2}x_k).\]
In this example, by multiplying the dual functions $\varphi$ and $\psi$ to the equation $-\Delta u=f$, and its boundary condition $u|_{\partial \Omega}=g$, we introduce the loss functional $\mathscr{E}:H^2(\Omega)\times H^1_0(\Omega)\times L^2(\partial \Omega)\rightarrow \mathbb{R}$ as
\begin{align*}
  \mathscr{E}(u;\varphi, \psi) = &   \int_\Omega (-\Delta u - f)\varphi \dd \mu - \frac{\blue{\nu}}{2}\int_{\Omega}\|\nabla\varphi\|^2\dd \mu + \lambda \left( \! \int_{\partial \Omega} (u - g)\psi \dd \mu_{\partial \Omega} - \frac{\blue{\nu}}{2}\int_{\partial \Omega} \psi^2\dd \mu_{\partial \Omega}\! \right) \\
  = &  \int_\Omega \nabla u\cdot \nabla \varphi -f\varphi\,\dd \mu - \frac{\blue{\nu}}{2}\int_{\Omega}\|\nabla\varphi\|^2\dd \mu    + \lambda \left(\!\int_{\partial \Omega} (u - g)\psi \dd \mu_{\partial \Omega} - \frac{\blue{\nu}}{2}\int_{\partial \Omega} \! \psi^2\dd \mu_{\partial \Omega}\!\right).
\end{align*}
In practice, we discover that it is helpful to add the $L^2(\partial \Omega, \mu_{\partial \Omega})$ loss functional to $\mathscr{E}(u, \varphi, \psi)$. Thus, we obtain
\begin{equation*}
  \widetilde{\mathscr{E}}(u,\varphi, \psi) = \mathscr{E}(u,\varphi, \psi) + \lambda \|\mathcal B u-g\|^2_{L^2(\partial \Omega, \mu_{\partial \Omega})}.
\end{equation*}
In short, we use the functional $\widetilde{\mathscr{E}}$. We substitute $u, \varphi, \psi$ with MLPs with number of layers $n_l$, and $\mathrm{tanh}$ as activation,
\[u_\theta = \texttt{MLP}_{\mathrm{tanh}}(d, 256, 1, n_l), \quad \varphi_\eta = \texttt{MLP}_{\mathrm{tanh}}(d, 256, 1, n_l)\cdot \zeta, \quad \psi_\xi = \texttt{MLP}_{\mathrm{tanh}}(d, 64, 1, n_l).  \] 
Here, we multiply the MLP with the truncation function 
\begin{equation*}
 \zeta(x) = \min_{1\leq k\leq d} ~\{x_k, 1-x_k\},  
\end{equation*}
in order to enforce $\varphi_\eta\in H^1_0(\Omega)$. Furthermore, based on the definition of $\mathscr{E}$, we set 
\[ \mathcal M_p=\mathcal M_d=\nabla    \]
as discussed in Section \ref{section: Numer Anal}. And recall the definition \eqref{def: M_p}, \eqref{def: M_d} and \eqref{def: M_bdd}, we define the preconditioning matrices in the proposed NPDG algorithm as
\[  M_p(\theta) =    \int_\Omega \frac{\partial}{\partial \theta}(\nabla u_\theta(x))\blue{^\top} \frac{\partial}{\partial\theta}(\nabla u_\theta(x)) \,\dd\mu + \lambda  \int_{\partial \Omega} {\frac{\partial  u_\theta(y) }{\partial \theta}}\blue{^\top} \frac{\partial  u_\theta(y) }{\partial\theta} \, \dd\mu_{\partial\Omega}   \]
\[ M_d(\eta) =  \int_\Omega \frac{\partial}{\partial \blue{\eta}}(\nabla \varphi_\eta(x))\blue{^\top} \frac{\partial}{\partial\blue{\eta}}(\nabla \varphi_\eta(x)) \,\dd\mu, \quad M_{bdd}(\xi) = {\lambda}  \int_{\partial \Omega}  {\frac{\partial}{\partial \xi} \psi_\xi(y)}\blue{^\top} \frac{\partial}{\partial \xi}\psi_\xi(y) \dd\mu_{\partial \Omega}. \]
In this example, we pick $N_{in}=2000$ for $d=10$, $N_{in}=4000$ for $d=50$, and $N_{bdd} = 80d$. We set the stepsizes $\tau_u = 0.5 \cdot 10^{-1}$, $\tau_\varphi=\tau_\psi = 0.95 \cdot 10^{-1}$. We test the thresholds $tol_{\textrm{MINRES}} = 10^{-3}, 10^{-4}$ in the algorithm. Throughout this research, we consider the relative $L^2$ error of $u_\theta$ and $\nabla u_\theta.$

In the 10D case, we set $n_l=4$. We plot MINRES iteration numbers at each NPDG step in Figure~\ref{subfig: Laplace Number of MINRES iter vs NPDG iter}. We investigate the effectiveness of our natural({preconditioned})-gradient method by comparing it with the same algorithm using flat gradients. That is, we replace line 7, line 11 in Algorithm \ref{alg: NPDG} by $\eta = \eta + \tau_\varphi \textbf{w}_\varphi, ~ \xi = \xi + \tau_\psi \textbf{w}_\psi$, and $\theta = \theta - \tau_u \textbf{w}_u$. This is demonstrated in Figure~\ref{subfig: compare with no extrapolation, flat grad}. In the same plot, it is also observed that the extrapolation step (line 7 of Algorithm \ref{alg: NPDG}) will slightly enhance the convergence of the proposed algorithm. Furthermore, choosing suitable preconditioning matrices compatible with the mathematical nature of the PDE is crucial for the proposed method. In Figure~\ref{subfig: compare with Md=Mp=Id}, we compare our treatment with the NPDG algorithm with $M_p(\theta), M_d(\eta)$ obtained by setting $\mathcal M_p = \mathcal M_d = \mathrm{Id}$. As reflected in the plot, naïve preconditioning may lead to instabilities in the optimization procedure. \blue{We tested the NPDG algorithms using exponentially growing sample sizes $N_r = 2^k$ and $N_r = 10 \cdot 2^{k-3}$ with $k = 4,6,8,10,12,14$. Figure~\ref{subfig: L2err_v_sample_size} illustrates how the relative $L^2$ error of $u_\theta$, computed by the NPDG algorithm, varies with the sample complexity. The relative error gradually plateaus as the sample size increases beyond $2^{10}=1024$. Based on this observation, throughout the remaining examples in this section, we consistently choose sample sizes in the range of $10^3 \sim 10^4$.}
\begin{figure}[htb!]
  \centering
  \begin{subfigure}[b]{0.24\textwidth}
      \centering
      \includegraphics[width=\textwidth]{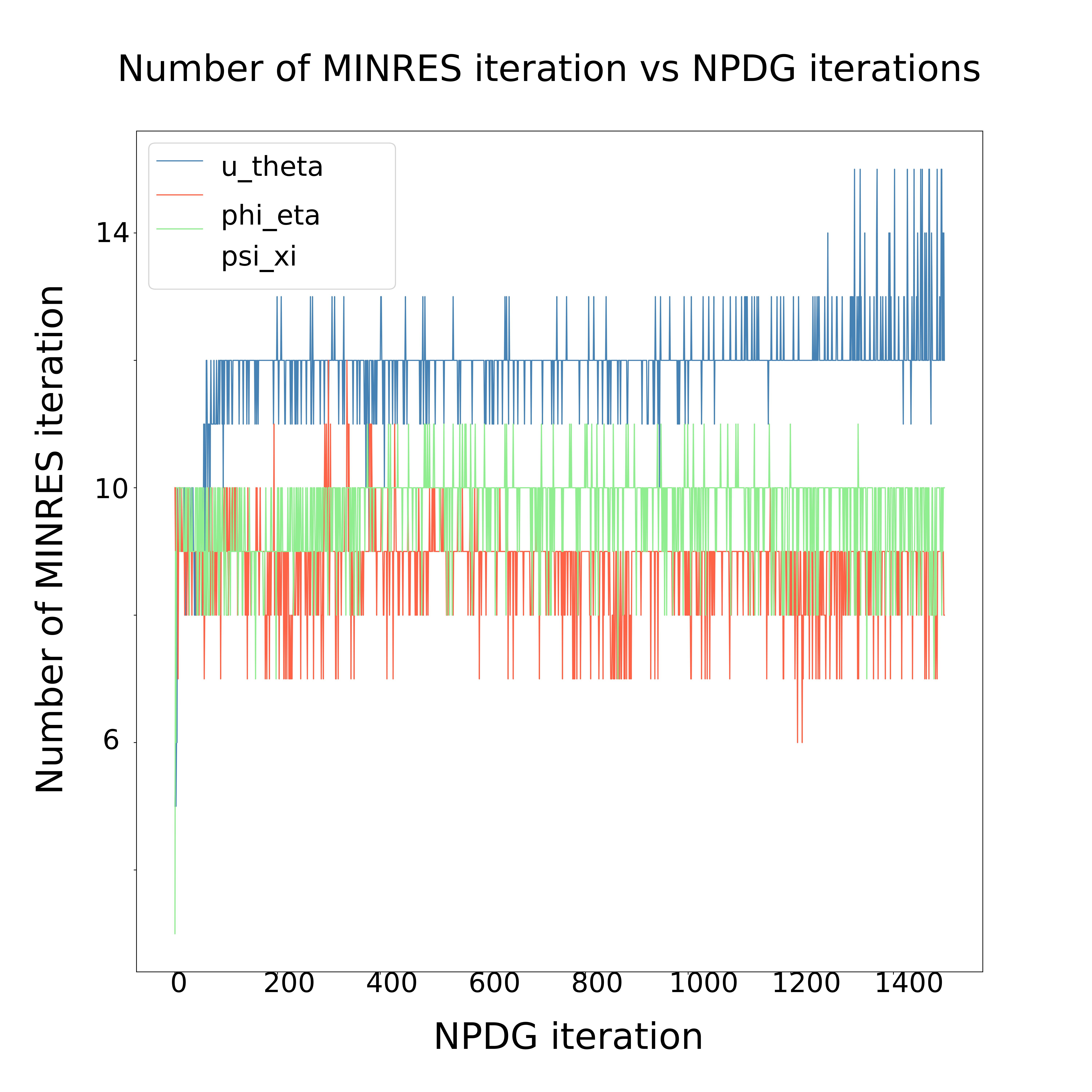}
      \caption{}
      \label{subfig: Laplace Number of MINRES iter vs NPDG iter}
  \end{subfigure}
  \begin{subfigure}[b]{0.24\textwidth}
      \includegraphics[width=\textwidth]{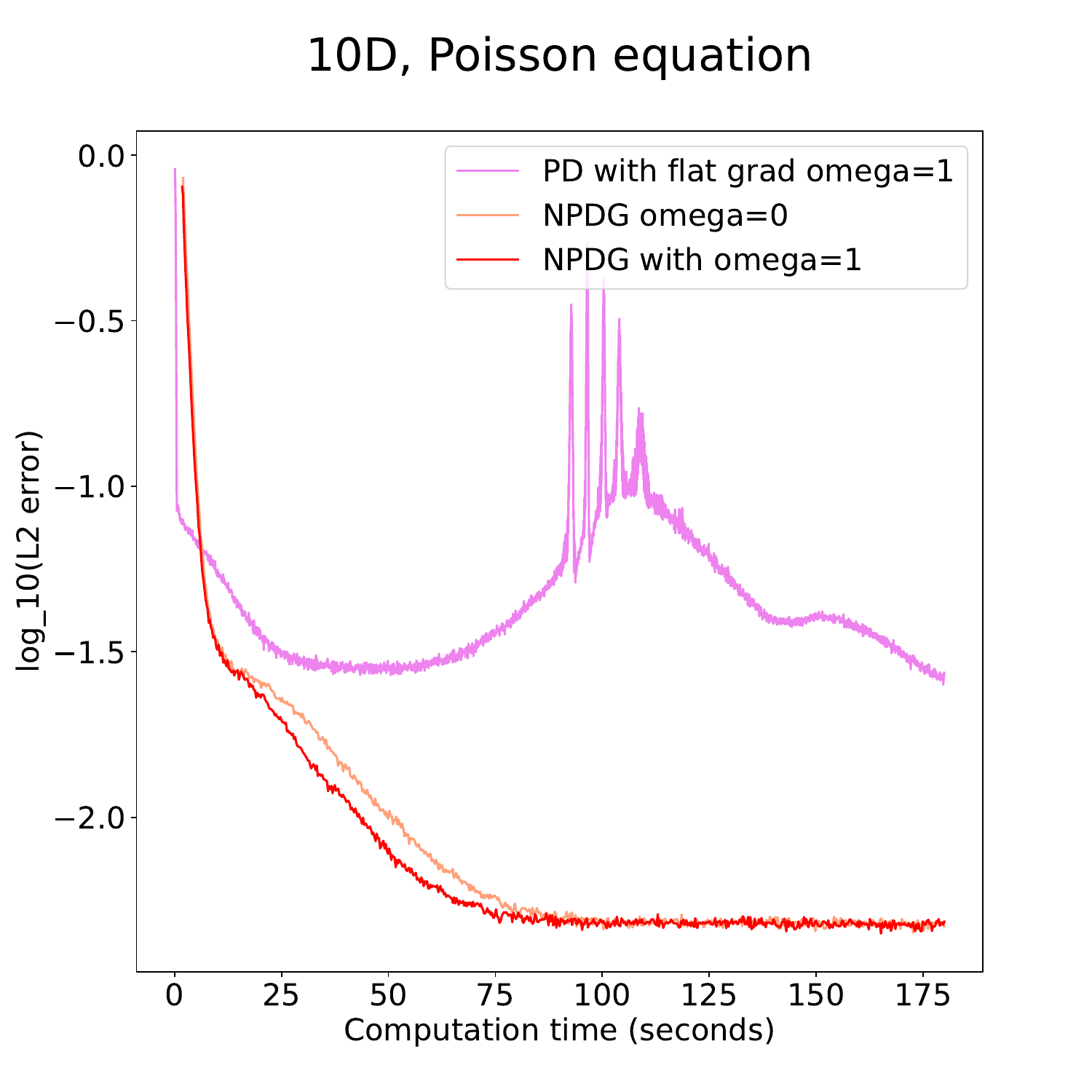}
      \caption{}
      \label{subfig: compare with no extrapolation, flat grad}
  \end{subfigure}
  \begin{subfigure}[b]{0.24\textwidth}
      \includegraphics[width=\textwidth]{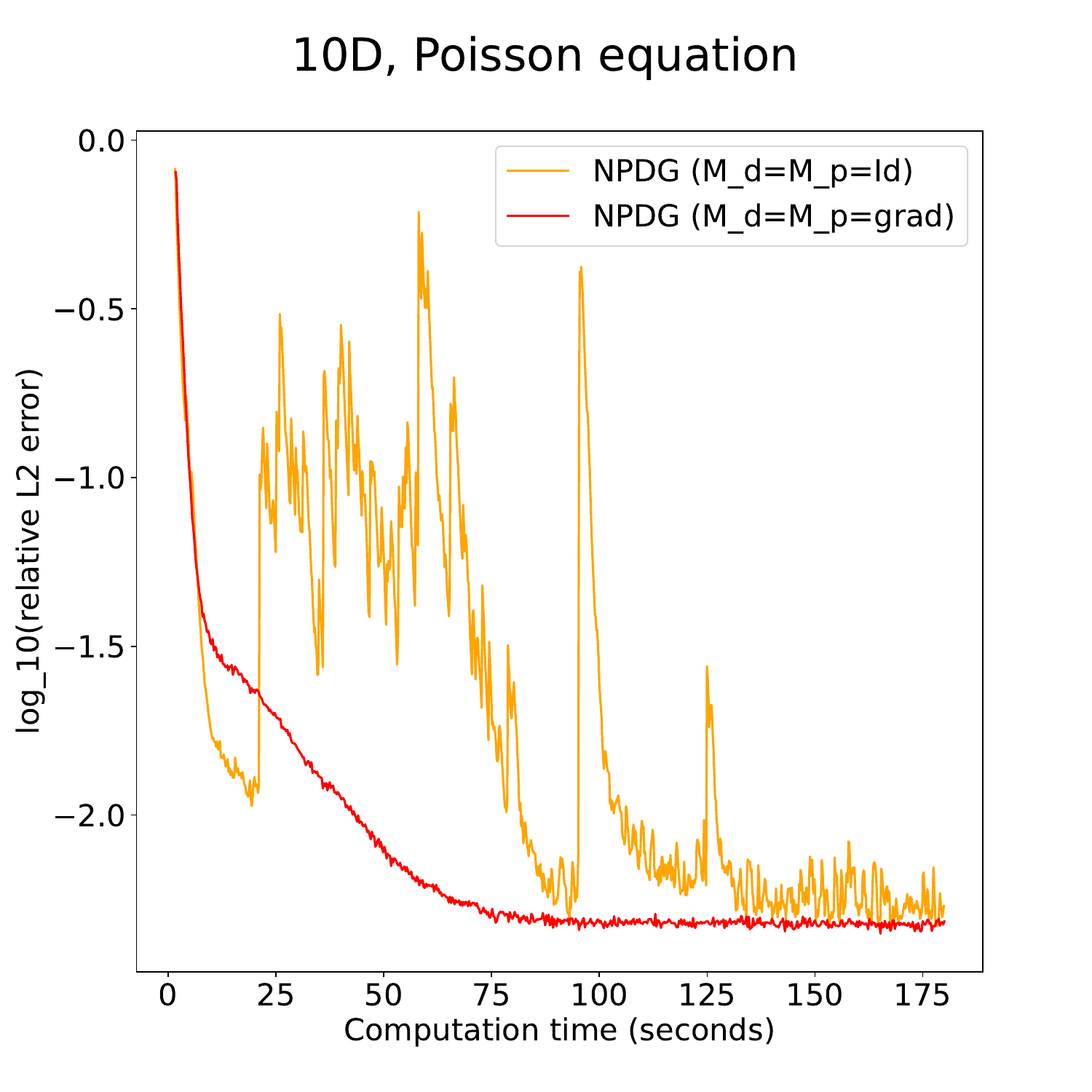}
      \caption{}
      \label{subfig: compare with Md=Mp=Id}
  \end{subfigure}
  \begin{subfigure}[b]{0.24\textwidth}
      \centering
      \includegraphics[width=\textwidth]{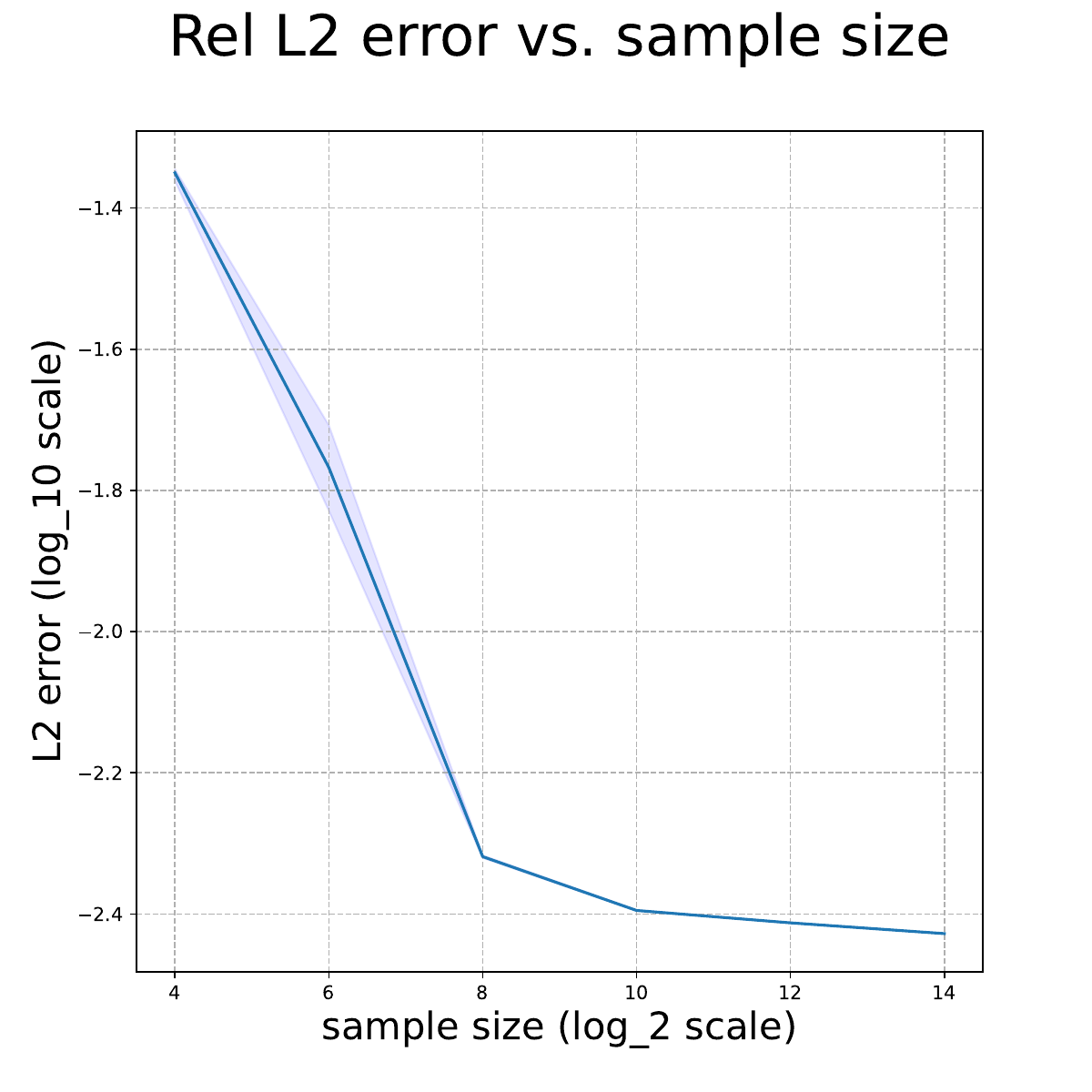}
      \caption{}
      \label{subfig: L2err_v_sample_size}
  \end{subfigure}
  \caption{(10D Poisson equation) \ref{subfig: Laplace Number of MINRES iter vs NPDG iter}: Numbers of iterations required by the MINRES algorithm for updating $\theta, \eta, \xi$ at each NPDG step vs. NPDG iteration. \ref{subfig: compare with no extrapolation, flat grad}: Comparison with the same algorithm using flat gradients instead (pink), and with the same algorithm without extrapolation ($\omega=0$) (light red); \ref{subfig: compare with Md=Mp=Id}: Comparison with our NPDG method, but using $M_p(\theta), M_d(\eta)$ obtained by $\mathcal M_p = \mathcal M_d = \mathrm{Id}$ as our preconditioning (orange). All the plots in these two figures are relative L2 error vs. computational time (seconds). \blue{\ref{subfig: L2err_v_sample_size}: Log-log plot of relative $L^2$ error vs. different sample sizes.}}\label{fig: Laplace2}
\end{figure}

In addition, we also test the same example with $d=50$. We set MLP depth $n_l=6$. We choose the tolerance $tol_{\textrm{MINRES}}=10^{-4}$ to ensure higher accuracy in computing the natural gradient.  We compare the algorithms with the PINN, DeepRitz, and WAN methods. The detailed settings for these three methods are provided in Table \ref{tab: setup for PINN DeepRitz WAN PDAdam}. We run each method up to $8000$ seconds and make semi-log plots of relative error vs. computational time for all the methods tested. Figure \ref{fig: Laplace 50D} presents the associated numerical results. The loss plot \ref{subfig: Laplace semi-log l2 rel err comparison 50D} suggests that our proposed NPDG algorithm converges faster and more stably compared with the algorithms based on the Adam optimizer. 

Furthermore, we record the GPU time spent by each method to achieve a certain accuracy for various dimensions $d=5, 10, 20, 50$. Details are provided in Table \ref{tab: GPU time to accuracy } of Appendix \ref{append: compare among diff methods}. The proposed method performs more efficiently than the other methods as the dimension $d$ increases.
\begin{figure}[htb!]
  \centering
  \begin{subfigure}[b]{0.327\textwidth}
    \centering
    \includegraphics[width=\textwidth]{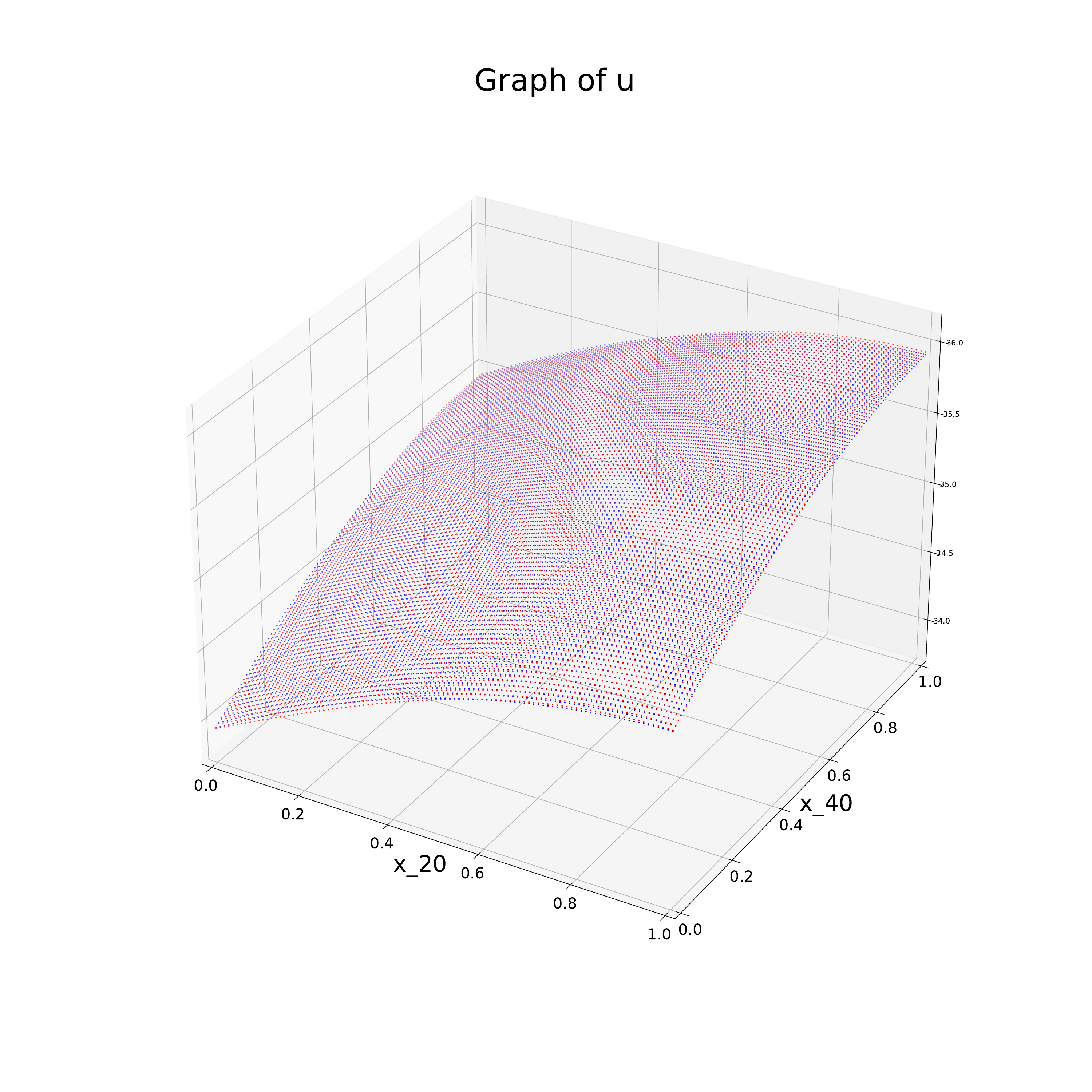}
    \caption{}
    \label{fig: graph on 20 40 plane 50D}
  \end{subfigure}
  \begin{subfigure}[b]{0.327\textwidth}
    \centering
    \includegraphics[width=\textwidth]{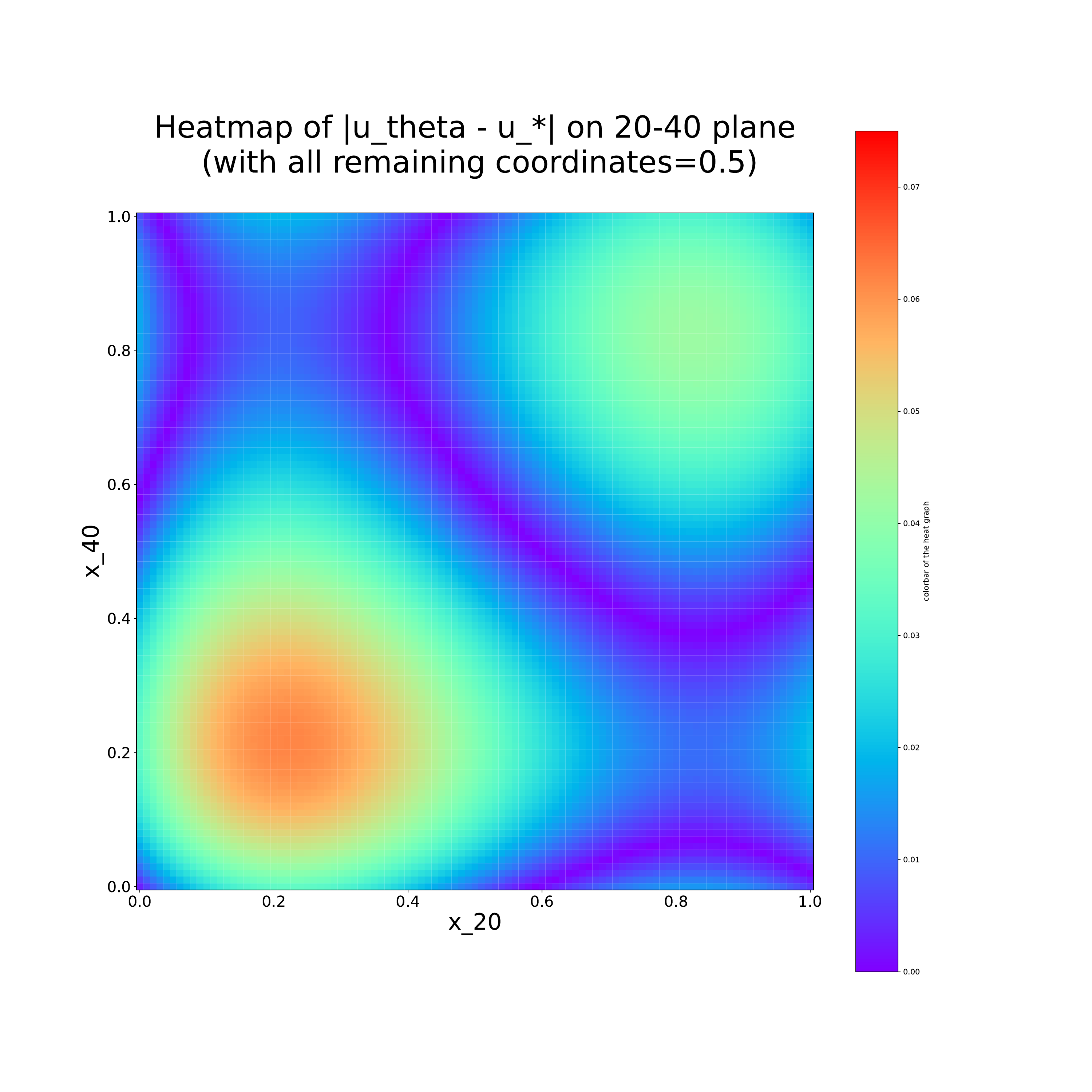}
    \caption{}
    \label{subfig: Laplace error heatmap 50D}
  \end{subfigure}
  \begin{subfigure}[b]{0.327\textwidth}
    \centering
    \includegraphics[width=\textwidth]{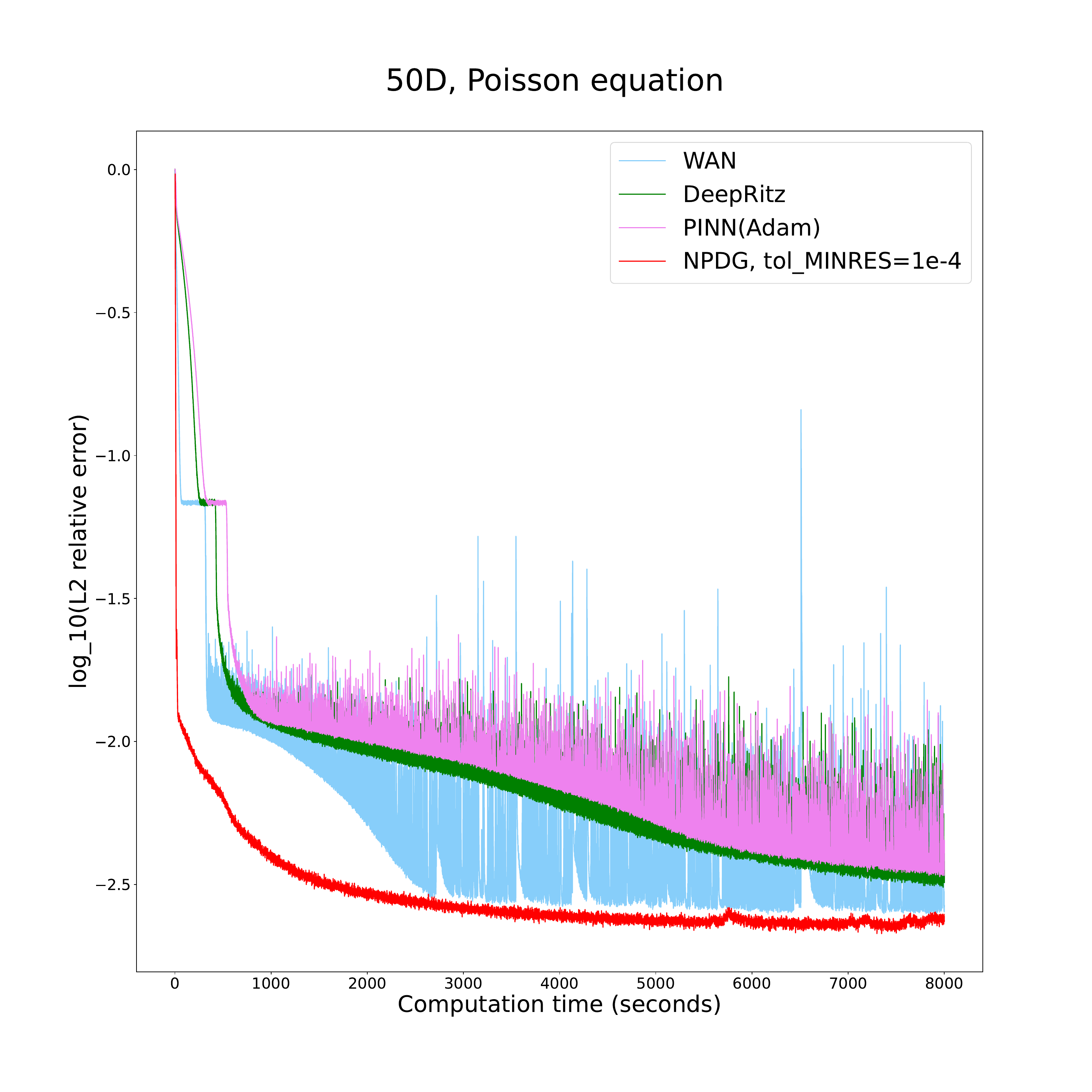}
    \caption{}
    \label{subfig: Laplace semi-log l2 rel err comparison 50D}
  \end{subfigure}
    \caption{(50D Poisson equation) \textbf{Left}: Graph of learned $u_\theta$ and real solution $u_*$ on the 20-40 coordinate plane (with remaining coordinates equal to $\frac12$) in $\mathbb{R}^{50}$; \textbf{Middle}: Heatmap of $|u_\theta(x)-u_*(x)|$ on the same plane. \textbf{Right}: Semi-log plot of relative L2 error vs. computational time (seconds). The values of $\|u_*\|_{L^2(\Omega, \mu)}$ and $\|\nabla u_*\|_{L^2(\Omega, \mu)}$ are provided in Table \ref{tab: real sol and norms}.
  }\label{fig: Laplace 50D}
\end{figure}

\subsection{Elliptic equation with variable coefficients (10D, 20D, 50D)}\label{subsec: varcoeff}
We consider the following elliptic equation with a variable coefficient
\begin{equation}
    -\nabla\cdot(\kappa(x)\nabla u(x)) = f(x), \quad u(y)=g(y) ~\textrm{on } \partial \Omega.  \label{eq: VarCoeff}
\end{equation}
Here we assume $\Omega = [-1, 1]^d$ with even dimension $d$. We set 
\[\kappa(x) = \frac{x^\top \Lambda x + 1}{2}, \quad \textrm{with } \Lambda = \mathrm{diag}(\lambda_0,\lambda_1,\dots,\lambda_0,\lambda_1),\]
where $\lambda_0=1, \lambda_1=4$ appear alternately for $\frac{d}{2}$ times, and choose
\[ f(x) = -\frac{\mathrm{Tr}(\Lambda^{-1})}{2}(x^\top \Lambda x + 1 ) - \|x\|^2, \quad \textrm{and }  g(y)=\frac12 y^\top\Lambda^{-1}y, ~ y\in\partial \Omega.  \]
The solution to this equation is $u_*(x) = \frac{1}{2}x^\top\Lambda^{-1}x$.

Similar to the previous example, we introduce $\varphi, \psi$ to the equation and its boundary condition. Integration by parts yields the functional
\begin{align*}
\mathscr{E}(u,\varphi,\psi) = &  \int_\Omega \kappa(x) \nabla\varphi(x)\cdot\nabla u(x) - f(x)\varphi(x) ~ \dd\mu  - \frac{\blue{\nu}}{2}\int_\Omega \|\nabla \varphi(x)\|^2\,\dd\mu      \\
 & + \lambda \left(\int_{\partial \Omega} (u-g)\psi\,\dd\mu_{\partial \Omega}  - \frac{\blue{\nu}}{2} \int_{\partial \Omega} \psi^2 \,\dd\mu_{\partial \Omega} \right).
\end{align*}
Similarly, we add the boundary loss function to $\mathscr{E}(u, \varphi, \psi)$ to obtain
\begin{equation*}
  \widetilde{\mathscr{E}}(u,\varphi, \psi) = \mathscr{E}(u,\varphi, \psi) + \lambda \|\mathcal B u-g\|^2_{L^2(\partial \Omega, \mu_{\partial \Omega})}.
\end{equation*}
We use $\widetilde{\mathscr{E}}$ in the computation. We set 
\[ \mathcal M_p = \mathcal M_d = \nabla \]
for the preconditioning matrices $M_p(\theta), M_d(\eta), M_{bdd}(\xi)$ as defined in \eqref{def: M_p}, \eqref{def: M_d} and \eqref{def: M_bdd}.

{\color{black} In this example, we also employ the stronger $H^1(\partial \Omega, \mu_{\partial \Omega})$ boundary loss discussed in Section \ref{subsec: numerical analysis with stronger bdd norm}. In our implementation, the boundary integral of the functional $\mathscr{E}(u, \varphi, \psi)$ is now replaced by
\[ \sum_{j=1}^{2d}\int_{S_j} (u-g)\psi + \nabla^{S_j}(u-g)\cdot\nabla^{S_j}\psi \,\dd \mu_{\partial \Omega} - \frac{\nu}{2}\int_{S_j}\psi^2 + \|\nabla^{S_j}\psi\|^2   \,\dd \mu_{\partial \Omega}. \]
Here, for each $j=1,\dots,d-1,$ we denote 
\[S_j^{\pm}:= \Bigl\{(\hat x_1,\dots,\hat x_{j-1},\pm1,\hat x_{j},\dots,\hat x_{d-1})\;:\;\hat x \in[-1,1]^{d-1}\Bigr\}\]
as each face of the cubic region $\Omega$. Then $\partial \Omega=\bigcup_{i=1}^d (S_i^{+} \cup S_i^{-})$. $\nabla^{S_j}$ denotes the gradient of a function restricted to the face $S_j^{\pm}$. The functional is modified correspondingly as
\[ \widetilde{\mathscr{E}}(u, \varphi, \psi) \! = \! \mathscr{E}(u, \varphi, \psi) + \lambda \|\mathcal B u - g\|_{H^{1}(\partial \Omega, \mu_{\partial \Omega})}^2. \] 
Furthermore, to ensure consistency with the strengthened boundary norm, we modify the preconditioning matrices accordingly. In particular, the matrices
$M_p(\theta)$, $M_d(\eta)$ are obtained from
\eqref{def: M_p} and \eqref{def: M_d} by setting
\[
    \mathcal M_p = \mathcal M_d = \sqrt{\kappa(\cdot)}\,\nabla .
\]
Moreover, in the definition \eqref{def: M_p} of the primal preconditioning matrix $\bigl(M_p(\theta)\bigr)_{ij}$, the boundary integration is replaced by
\begin{equation}
  \lambda \sum_{j=1}^{2d}\int_{S_j}
  \frac{\partial u_\theta(y)}{\partial \theta_i}
  \frac{\partial u_\theta(y)}{\partial \theta_j}
  + \frac{\partial}{\partial \theta_i}\!\left(\nabla^{S_j} u_\theta(y)\right)
  \cdot
  \frac{\partial}{\partial \theta_j}\!\left(\nabla^{S_j} u_\theta(y)\right)
  \,\dd \mu_{\partial\Omega} .
  \label{def: precond matrix stronger bdd norm}
\end{equation}
Matrix  $(M_{bdd}(\xi))_{ij}$ is reformulated analogously to \eqref{def: precond matrix stronger bdd norm}, with the partial derivatives $\frac{\partial}{\partial \theta}$ and the function $u_\theta$ replaced by $\frac{\partial}{\partial \xi}$ and $\psi_\xi$, respectively.
}

We test this example with $d=10, 20, 50$. We substitute $u, \varphi, \psi$ with MLPs with $\mathrm{softplus}(\cdot)$ as activation functions. Here, $\mathrm{softplus}(\cdot)$ is a smooth approximation of the ReLU function defined as\footnote{In PyTorch, for numerical stability, the implementation of $\mathrm{softplus}(\cdot)$ reverts to the linear function when $x>\frac{\mathrm{threshold}}{\beta}.$ The default value for the threshold equals $20$.}
\[ \mathrm{softplus}(x) = \frac{1}{\beta} \log(1+\exp(\beta x)) \]
with $\beta = \frac14$. We summarize the neural net architecture of our experiments in Table \ref{tab: VarCoeff experiment basic settings}. Similar to our treatment for the Poisson's equation, we multiply $\varphi_\eta$ by the truncation function $\zeta(\cdot)$ to enforce $\varphi_\eta\in H^1_0(\Omega)$.
\begin{table}[htb!]
{\footnotesize
\begin{center}
\begin{tabular}{|c|c|c|c|c|c|c|}
\hline
        &  \multicolumn{3}{c|}{ Primal \& Dual Neural Networks } &  \multirow{2}{*}{  $N_{in}, N_{bdd}$  }  & \multirow{2}{*}{$\tau_u, \tau_\varphi, \tau_\psi$}  &  \multirow{2}{*}{MINRES tol}    \\ \cline{2-4} 
        &  $u_\theta$  &  $\varphi_\eta$  &  $\psi_\xi$  &    &    &         \\ \hline \hline  
   $d=10$  &  \multirow{2}{*}{$(d, 256, 1, 4)$} & \multirow{2}{*}{$(d, 256, 1, 4)$} & \multirow{2}{*}{$(d, 128, 1, 4)$} & \multirow{2}{*}{$4000, 80d$} & $0.1, 0.19, 0.19$ &  \multirow{2}{*}{ $0.5\cdot 10^{-3}$ } \\ \cline{1-1}\cline{6-6}
 $d=20$ &   &  &  &  & \multirow{2}{*}{ $0.05, 0.095, 0.095$ } &     \\\cline{1-5}\cline{7-7}
 $d=50$ &  $(d, 256, 1, 6)$ & $(d, 256, 1, 6)$ & $(d, 128, 1, 6)$ & $6000, 80d$ &  &  $10^{-4}$   \\ \hline
\end{tabular}
\end{center}
}
\caption{Basic setting of our experiments on computing \eqref{eq: VarCoeff}.}\label{tab: VarCoeff experiment basic settings}
\end{table}

In this example, for all dimensions $d=10, 20, 50,$ the stepsizes $\tau_u, \tau_\varphi,  \tau_\psi$, the number of samples $N_{in}, N_{bdd}$, as well as the tolerance of MINRES, are summarized in Table \ref{tab: VarCoeff experiment basic settings}. We improve the tolerance of the MINRES algorithm from $0.5 \cdot 10^{-3}$ to $10^{-4}$ as the dimension $d$ increases to $50.$ We run the proposed method for 500 and 1000 seconds for $10$D and $20$D problems. For $d=50$, we perform the proposed method \blue{using $L^2(\partial \Omega, \mu_{\partial \Omega})$ boundary loss for $36000$ iterations and the method using $H^1(\partial \Omega, \mu_{\partial \Omega})$ loss for $3000$ iterations.} For all $d=10, 20, 50$, we compare the algorithm with the PINN, DeepRitz, and WAN methods. The detailed settings for these three methods are provided in Table \ref{tab: setup for PINN DeepRitz WAN PDAdam}. We make semi-log/log-log plots of relative error vs. computational time for all methods. The error plots are presented in \blue{Figure \ref{fig: VarCoeff2}}. The plots justify the linear convergence of the proposed method. \blue{The experimental results further demonstrate improved convergence rates and accuracy for the NPDG algorithms, both with and without the use of a stronger boundary norm.} Compared with the other algorithms based on Adam optimizers, the proposed method performs more stably and achieves higher accuracy in this example. We also record the GPU time spent by each method to achieve a certain accuracy. One can find the details in Table \ref{tab: GPU time to accuracy } of Appendix \ref{append: compare among diff methods}. It turns out that only the proposed method can achieve an accuracy such that $\frac{\|u_\theta-u_*\|_{L^2(\Omega, \mu)}}{\|u_*\|_{L^2(\Omega, \mu)}} \leq 0.005$.

For $d=20$, we visualize the solution $u_\theta$ learned by the NPDG algorithm by plotting the graph of $u_\theta$ on the $9-10$ plane while fixing the remaining coordinates to $0$ and $0.5$ for $d=20$ in Figure \ref{fig: VarCoeff1}. The associated heatmaps of $|u_\theta(x)-u_*(x)|$ on the $9-10$ plane are also provided in Figure \ref{fig: VarCoeff1}. To investigate the accuracy of $u_\theta$ over the entire space of $\Omega$, we separate $\Omega=\bigcup_{l=1}^{50} \Omega_l$ into 50 square shells with gradually increasing sizes, $$\Omega_l:=\{x=(x_1,\dots,x_d)^\top \in\mathbb{R}^d| (l-1)/50 \leq |x_k| < l/50, ~1\leq k \leq d\}.$$ We plot the average $L^2$ error of $u_\theta$ computed via different methods on $\Omega_l$ with respect to the size $l/50$ of each square shell $\Omega_l$ in \blue{Figure \ref{subfig: VarCoeff log10(average error on each shell) vs shell size}}.
\begin{figure}[htb!]
\centering
  \begin{subfigure}[b]{0.32\textwidth}
      \includegraphics[width=\textwidth]{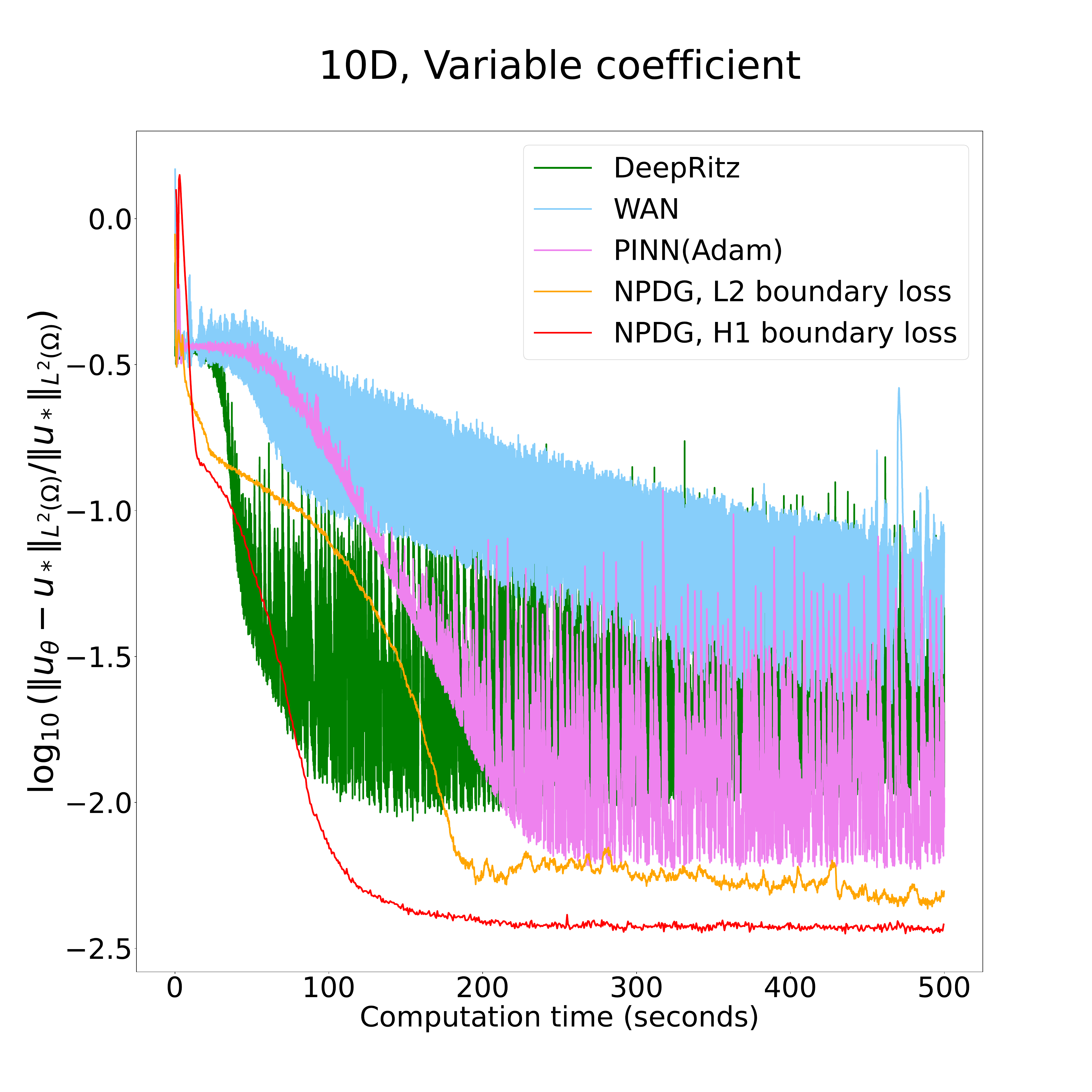}
      \caption{}\label{subfig: VarCoeff log rel l2 err vs comp time 10d}
  \end{subfigure}
  \begin{subfigure}[b]{0.32\textwidth}
      \includegraphics[width=\textwidth]{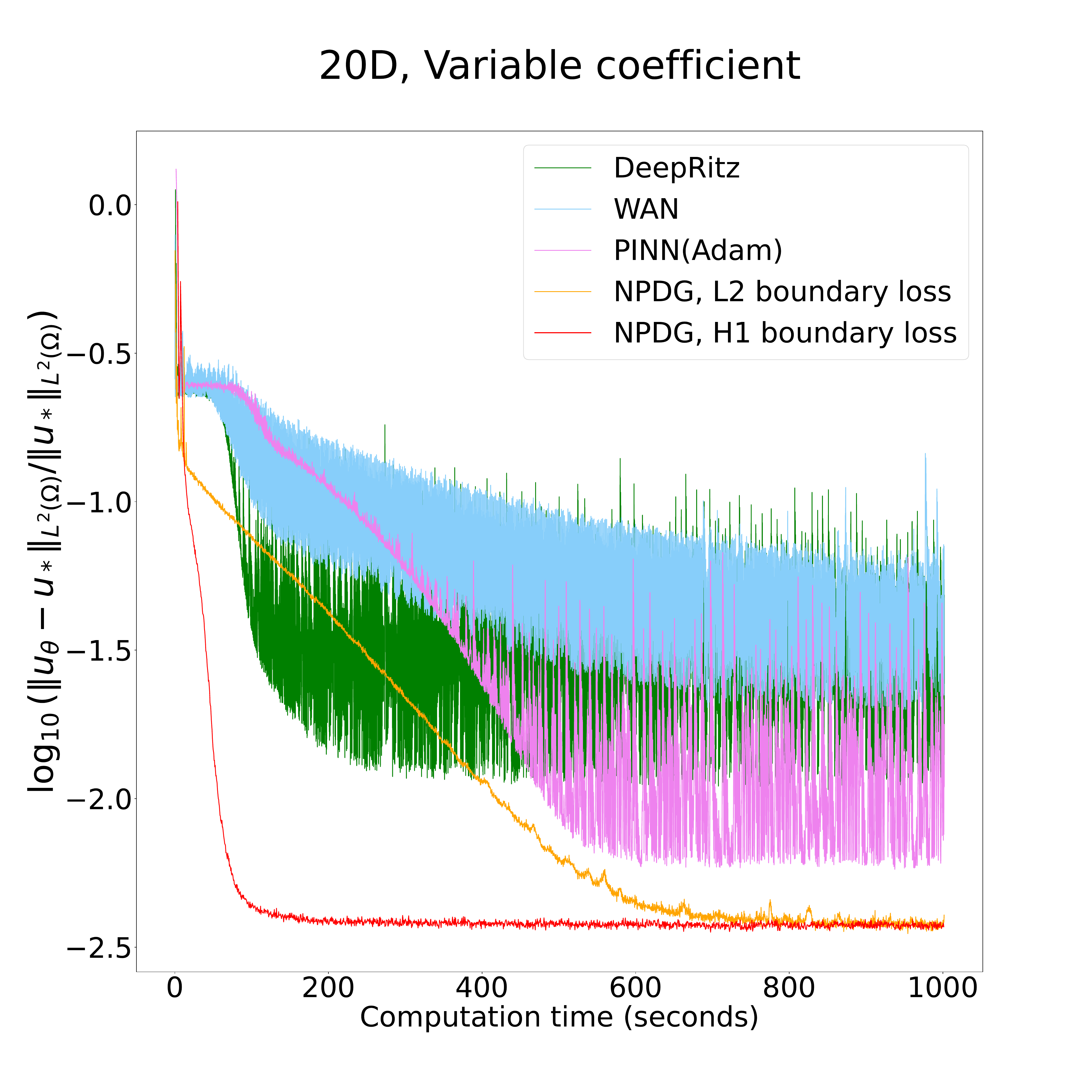}
      \caption{}\label{subfig: VarCoeff log rel l2 err vs comp time 20d}
  \end{subfigure}
  \begin{subfigure}[b]{0.32\textwidth}
      \includegraphics[width=\textwidth]{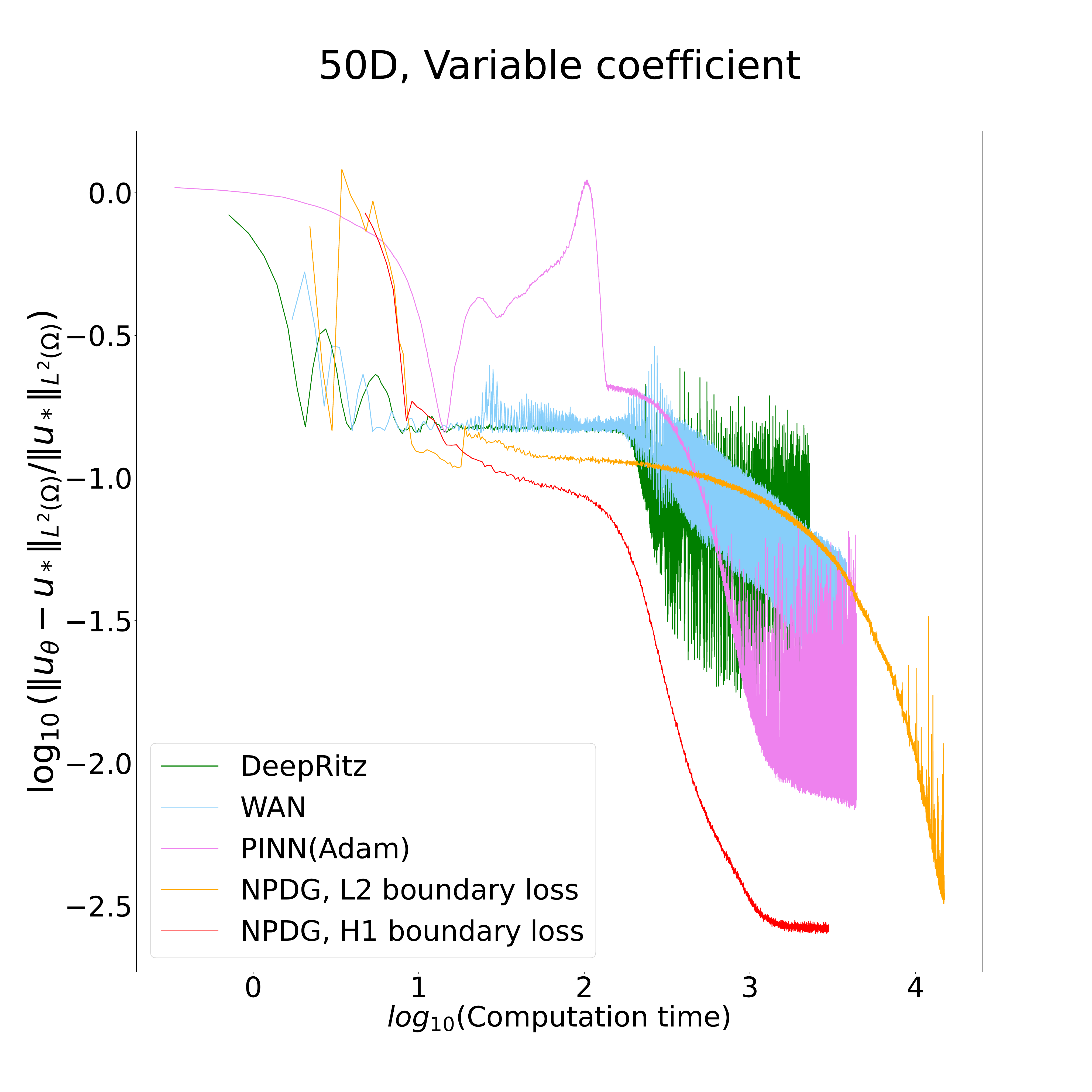}
      \caption{}\label{subfig: VarCoeff log rel l2 err vs log comp time 50d}
  \end{subfigure}
  \begin{subfigure}[b]{0.32\textwidth}
      \includegraphics[width=\textwidth]{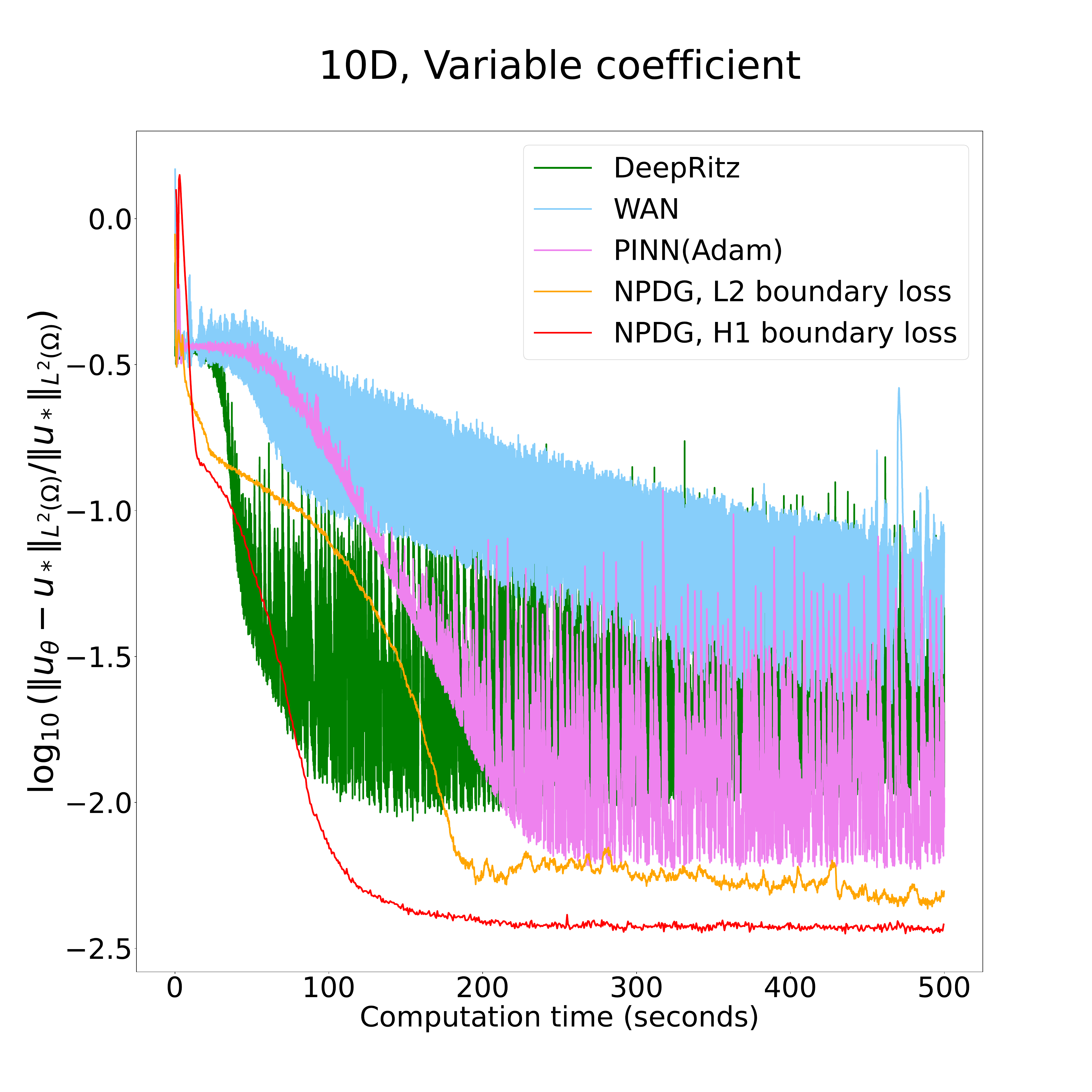}
      \caption{}\label{subfig: VarCoeff log rel H1 err vs comp time 10d}
  \end{subfigure}
  \begin{subfigure}[b]{0.32\textwidth}
      \includegraphics[width=\textwidth]{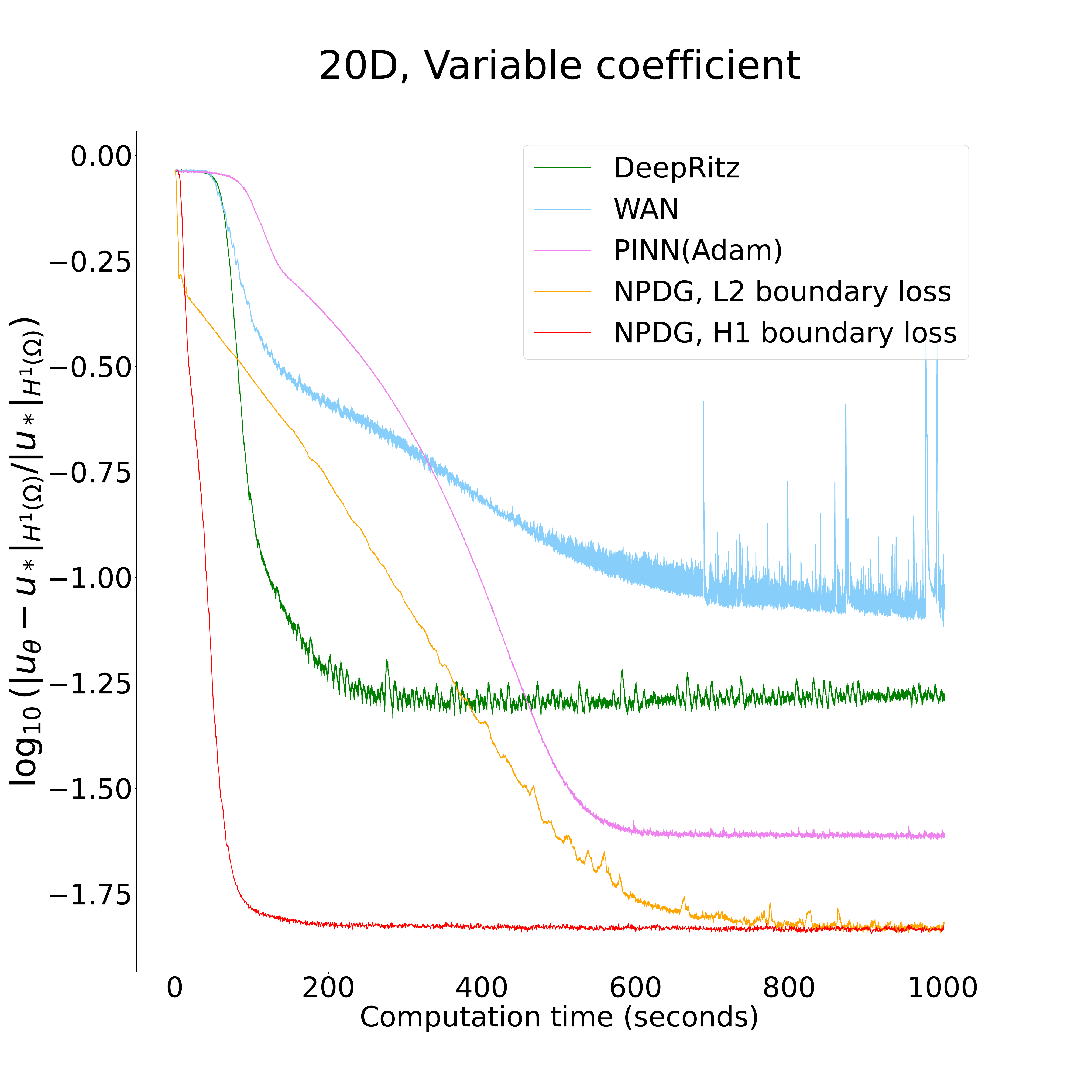}
      \caption{}\label{subfig: VarCoeff log rel H1 err vs comp time 20d}
  \end{subfigure}
  \begin{subfigure}[b]{0.32\textwidth}
      \includegraphics[width=\textwidth]{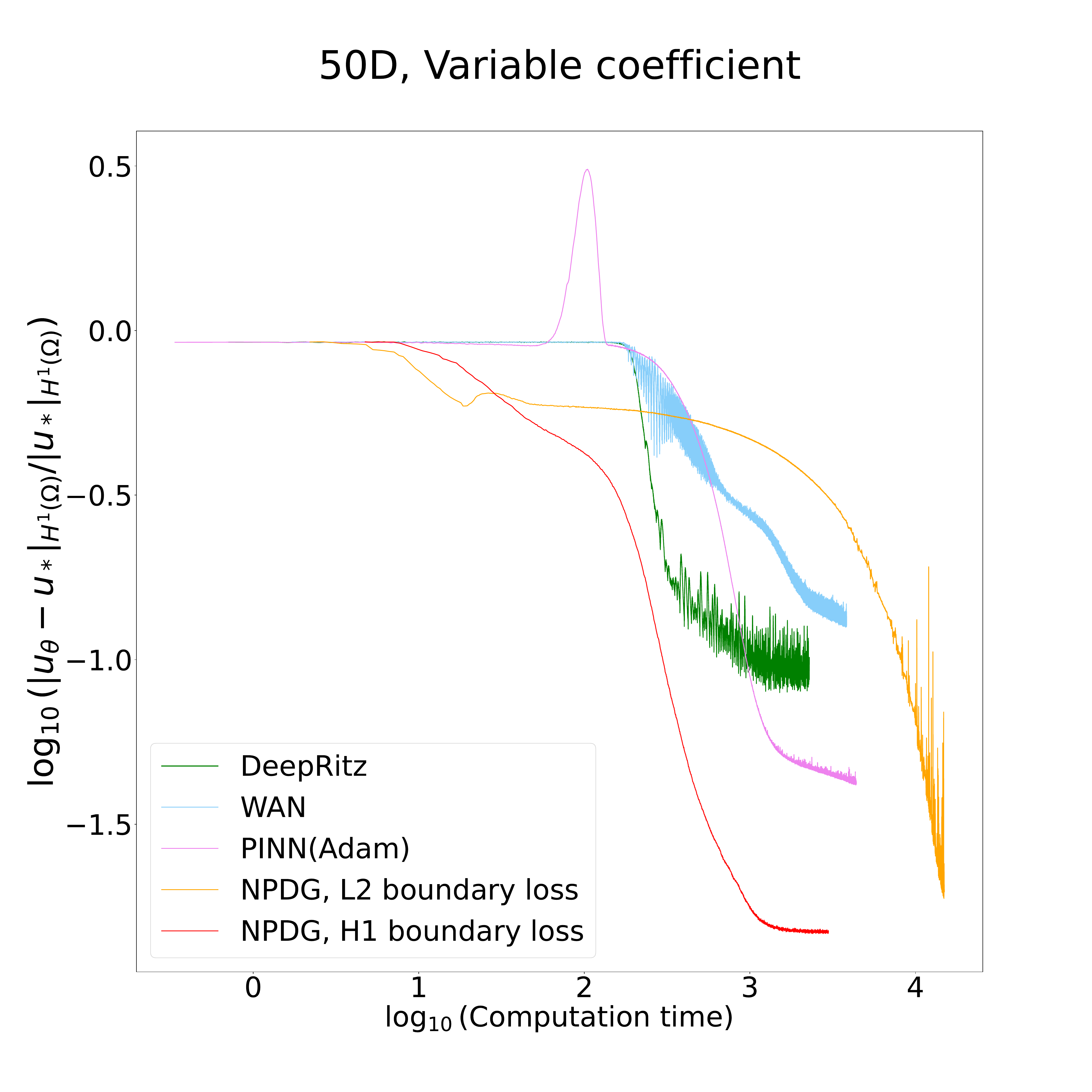}
      \caption{} \label{subfig: VarCoeff log rel H1 err vs log comp time 50d}
  \end{subfigure}
  \caption{\blue{\textbf{Left} column (\ref{subfig: VarCoeff log rel l2 err vs comp time 10d}) (\ref{subfig: VarCoeff log rel H1 err vs comp time 10d}): Semi-log plot (up) of relative L2 error vs. computational time(seconds) and semi-log plot (down) of relative $H^1$ seminorm error ($\frac{\|\nabla u_\theta - \nabla u_*\|_{L^2(\Omega, \mu)}}{\|\nabla u_*\|_{L^2(\Omega, \mu)}}$)
  vs. computational time. Dimension $d=10$; \textbf{Middle} column (\ref{subfig: VarCoeff log rel l2 err vs comp time 20d}) (\ref{subfig: VarCoeff log rel H1 err vs comp time 20d}): The same plots for $d=20$; \textbf{Right} column (\ref{subfig: VarCoeff log rel l2 err vs log comp time 50d}) (\ref{subfig: VarCoeff log rel H1 err vs log comp time 50d}): The same plots (but in Log-log form) for $d=50$. The values of $\|u_*\|_{L^2(\Omega, \mu)}$ and $\|\nabla u_*\|_{L^2(\Omega, \mu)}$ are provided in Table \ref{tab: real sol and norms}.}}\label{fig: VarCoeff2}
\end{figure}

\begin{figure}[htb!]
    \centering
    \begin{subfigure}[b]{0.187\textwidth}
        \includegraphics[width=\textwidth]{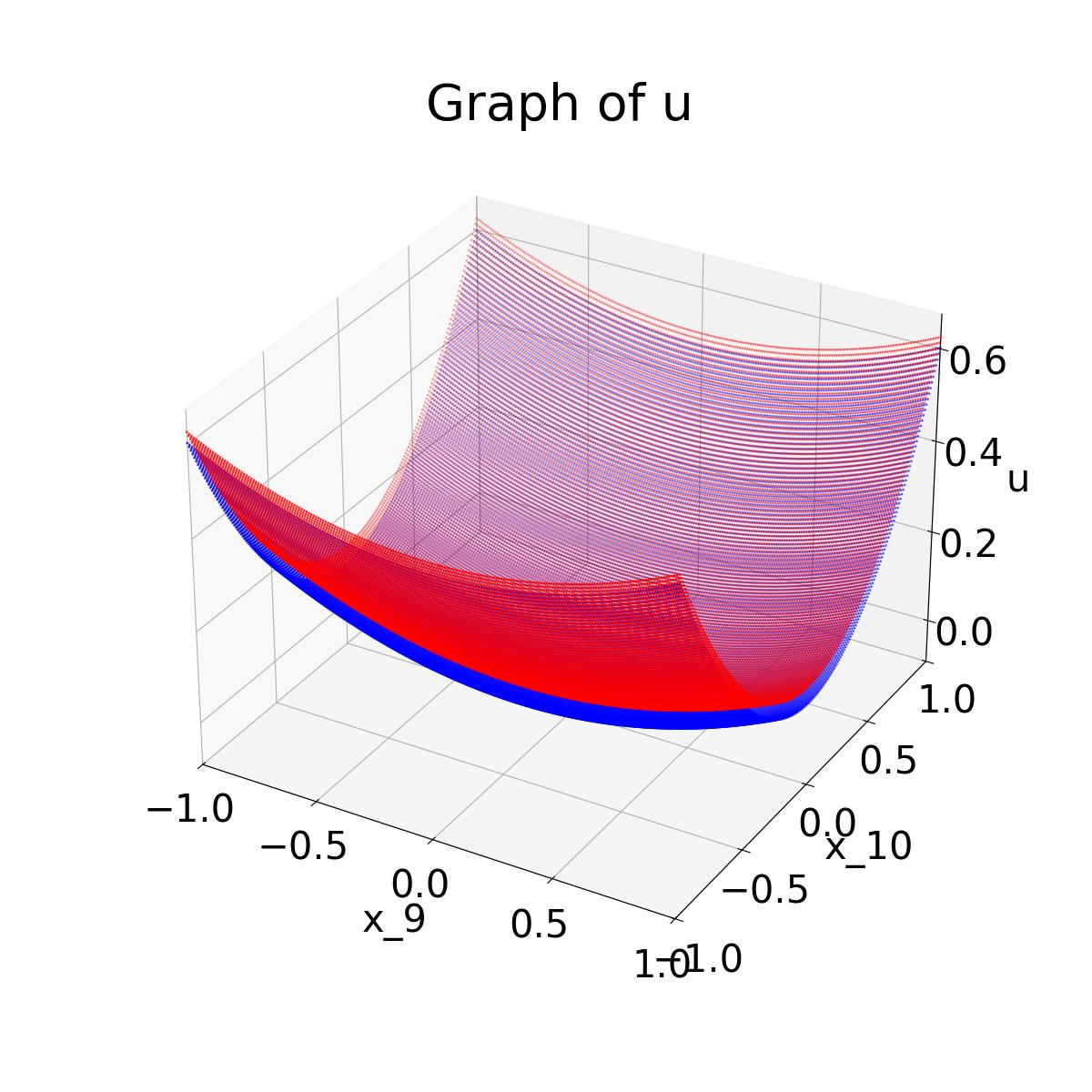}
        \caption{}\label{subfig: VarCoeff graph u 9-10 plt coord = 0.0}
    \end{subfigure}
    \begin{subfigure}[b]{0.187\textwidth}
        \includegraphics[width=\textwidth]{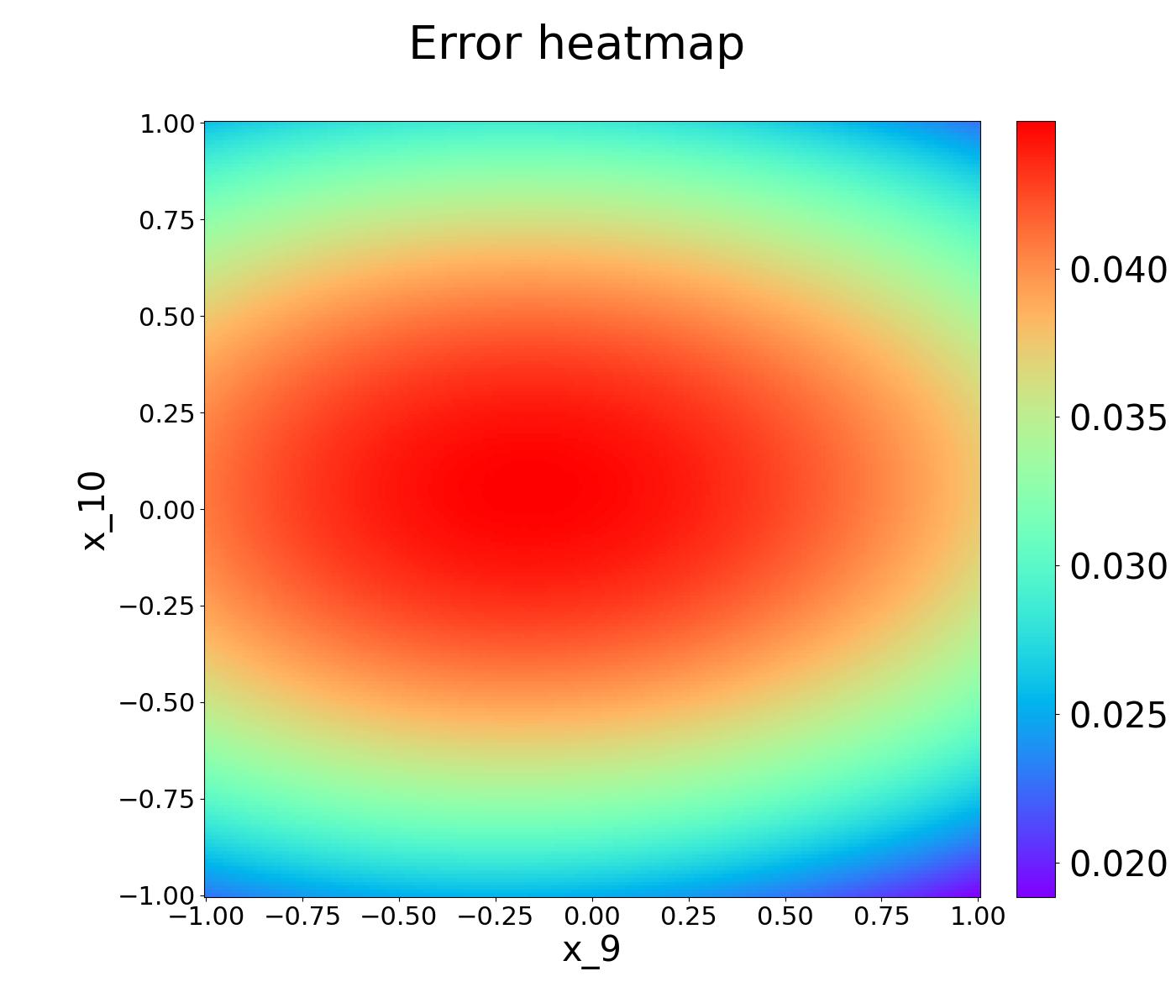}
        \caption{}\label{subfig: VarCoeff error heatmap 9-10 plt coord = 0.0}
    \end{subfigure}
    \begin{subfigure}[b]{0.187\textwidth}
        \includegraphics[width=\textwidth]{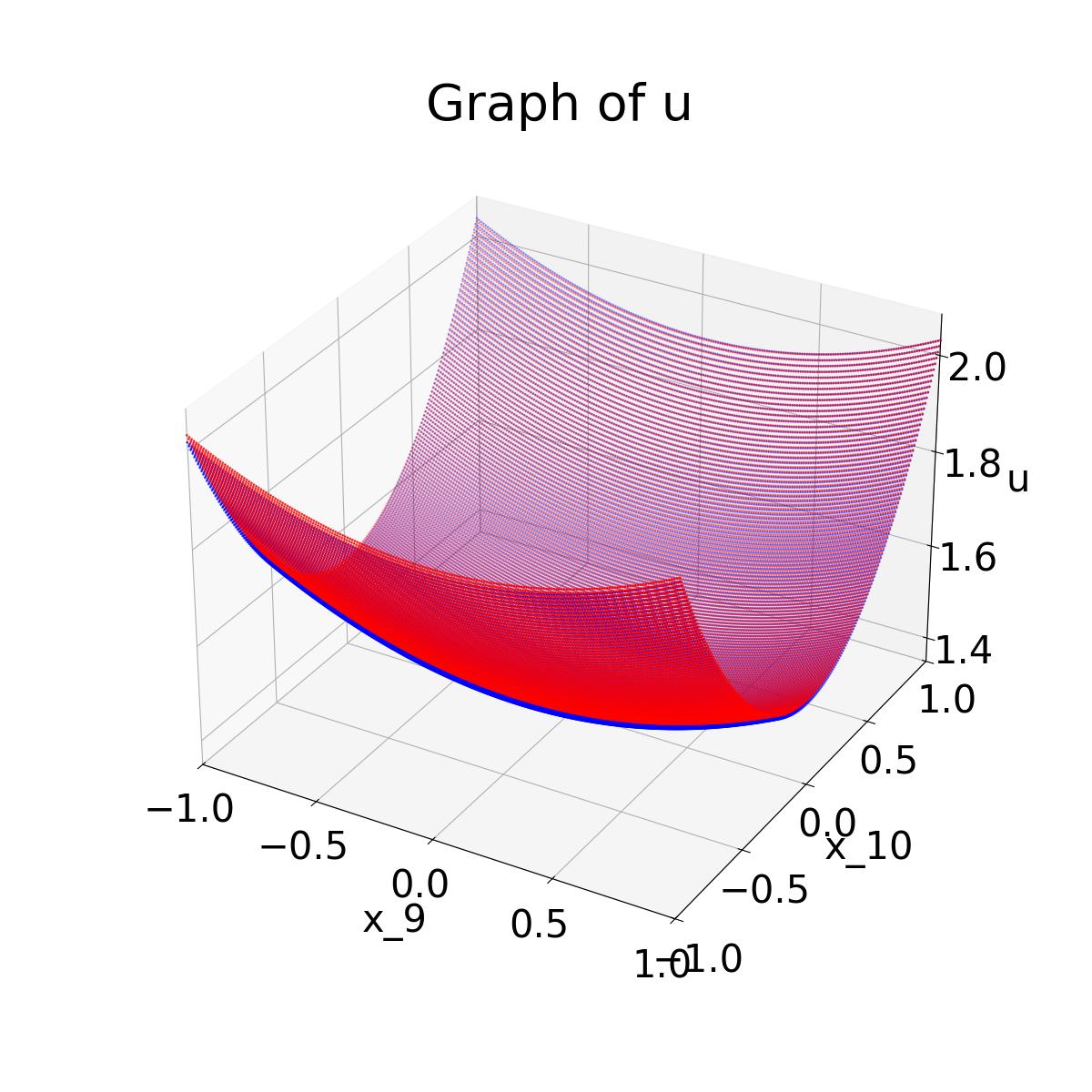}
        \caption{}\label{subfig: VarCoeff graph u 9-10 plt coord = 0.5}
    \end{subfigure}
    \begin{subfigure}[b]{0.187\textwidth}
        \includegraphics[width=\textwidth]{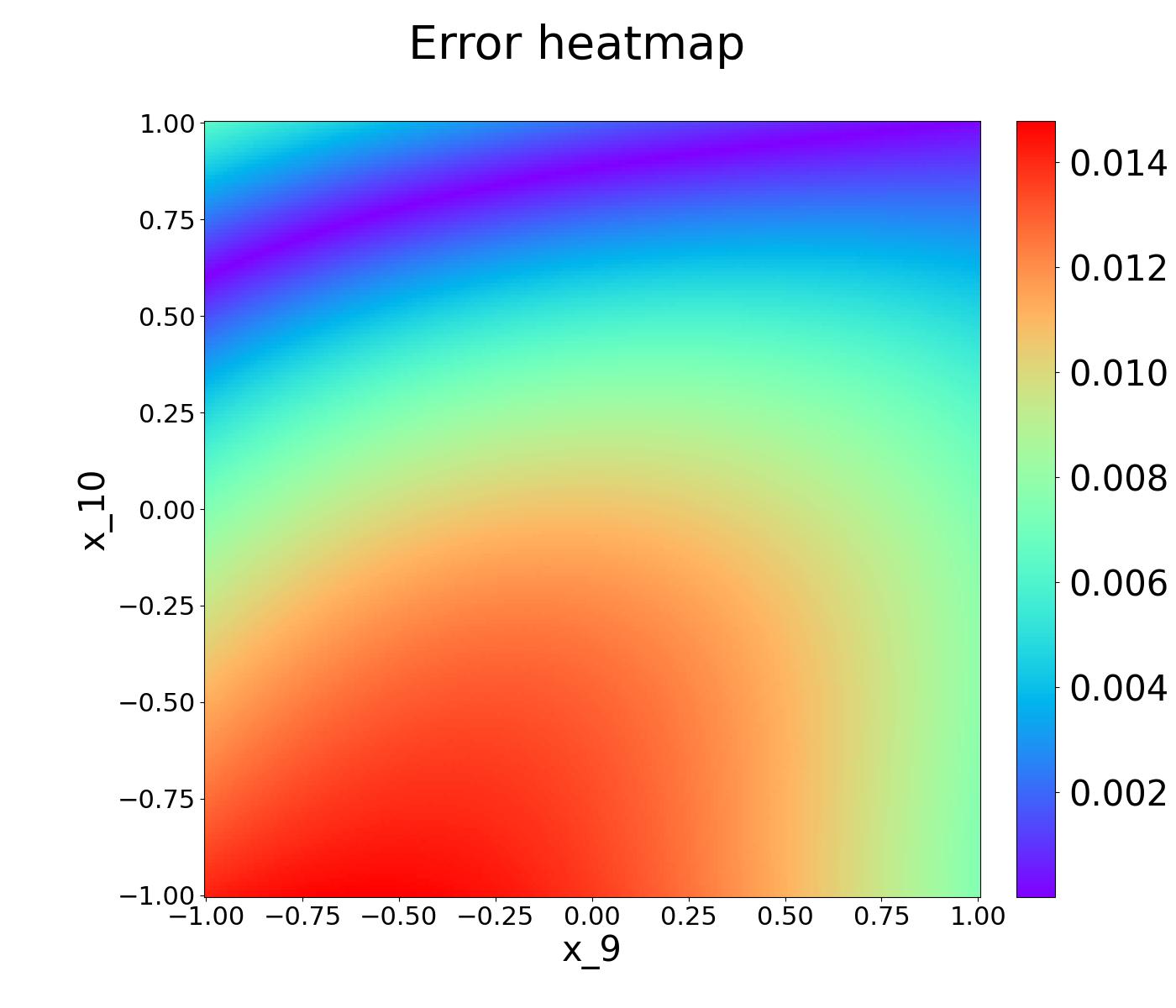}
        \caption{}\label{subfig: VarCoeff error heatmap 9-10 plt coord = 0.5}
    \end{subfigure}
    \begin{subfigure}[b]{0.221\textwidth}
        \includegraphics[width=\textwidth]{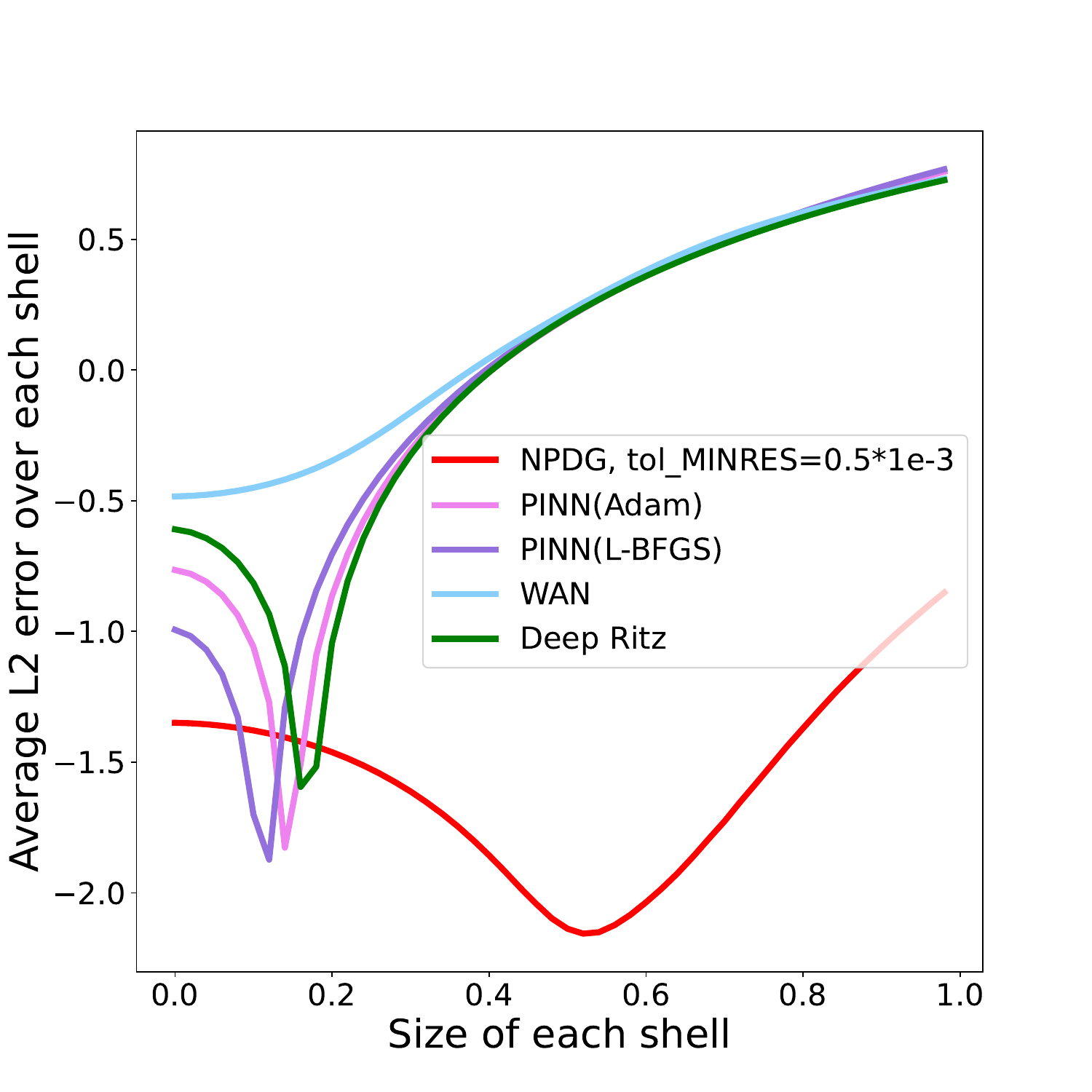}
        \caption{}\label{subfig: VarCoeff log10(average error on each shell) vs shell size}
    \end{subfigure}
    \caption{Plots for $d=20$. \ref{subfig: VarCoeff graph u 9-10 plt coord = 0.0}: Graph of $u_\theta$ obtained by NPDG method (blue) with real solution (red) plotted on $9-10$ plane with remaining coordinates fixed to $0$; \ref{subfig: VarCoeff error heatmap 9-10 plt coord = 0.0} Heatmap of error $|u_\theta(x)-u_*(x)|$ plotted on $9-10$ plane with remaining coordinates fixed to $0$. \ref{subfig: VarCoeff graph u 9-10 plt coord = 0.5}, \ref{subfig: VarCoeff error heatmap 9-10 plt coord = 0.5}: Same plots plotted on $9-10$ plane with remaining coordinates fixed to $0.5$; \ref{subfig: VarCoeff log10(average error on each shell) vs shell size}: Semi-log plot of $\log_{10}\left(\frac{1}{|\Omega_l|}{\| u_\theta -  u_*\|_{L^2(\Omega_l)}}\right)$ vs. size of each square shell $\Omega_l$.}
    \label{fig: VarCoeff1}
\end{figure}

\vspace{0.3cm}
\noindent
\textbf{Different MINRES tolerances}: Slightly improving (i.e., decreasing) the tolerance $tol_{\textrm{MINRES}}$ of the MINRES algorithm yields more accurate directions of the natural gradients and enhances the convergence of the NPDG algorithm. However, selecting $tol_{\textrm{MINRES}}$ too small makes the algorithm sensitive with respect to data stochasticity and thus may introduce instability to the method. This is reflected in Figure \ref{subfig: NPDG diff tol_minres l2 error 20d} and \ref{subfig: NPDG diff tol_minres H1 error 20d}.

\noindent
\textbf{Comparing with L-BFGS optimizer}: We apply the L-BFGS optimizer to PINN and compare its convergence speed with the proposed method. L-BFGS utilizes the second-order information from the loss function in optimization. However, L-BFGS is known to be unstable in stochastic settings--using random batches is not a feasible strategy for L-BFGS method. In this example, we fix the Monte-Carlo samples in the algorithm and optimize the PINN loss function with L-BFGS method. For $d=20$, as shown in Figure \ref{subfig: NPDG vs LBFGS l2 error 20d}, our NPDG algorithm with $tol_{\textrm{MINRES}}=10^{-4}$ converges faster than the L-BFGS method. Moreover, the L-BFGS method faces instability even without data stochasticity. As demonstrated in Figure \ref{subfig: LBFGS l2 error blowup}, the L-BFGS method always blows up given a long enough running time for dimensions $d=20$ and $d=50$.
\begin{figure}
    \centering
    \begin{subfigure}[b]{0.24\textwidth}
        \includegraphics[width=\textwidth]{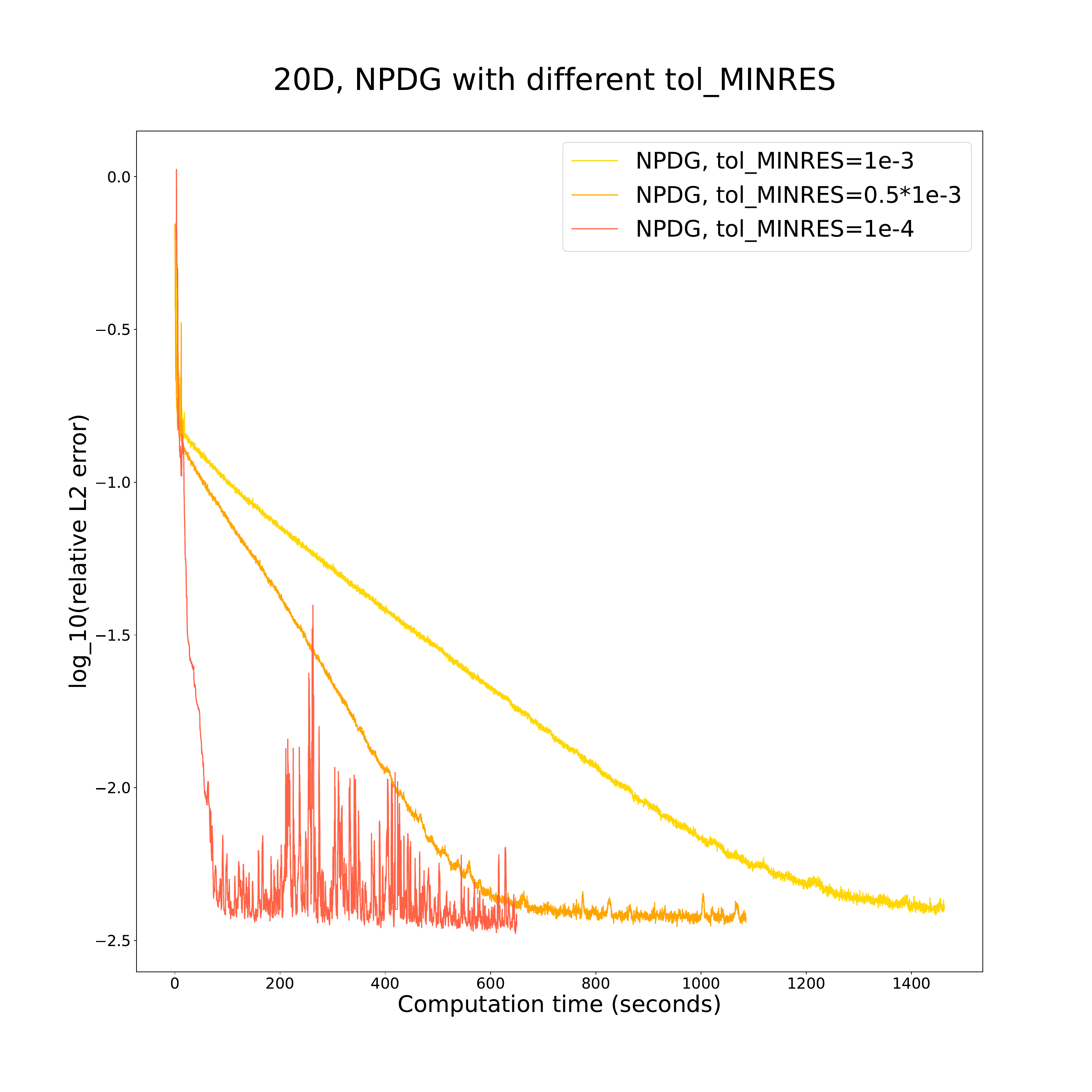}
        \caption{}\label{subfig: NPDG diff tol_minres l2 error 20d}
    \end{subfigure}    
    \begin{subfigure}[b]{0.24\textwidth}
        \includegraphics[width=\textwidth]{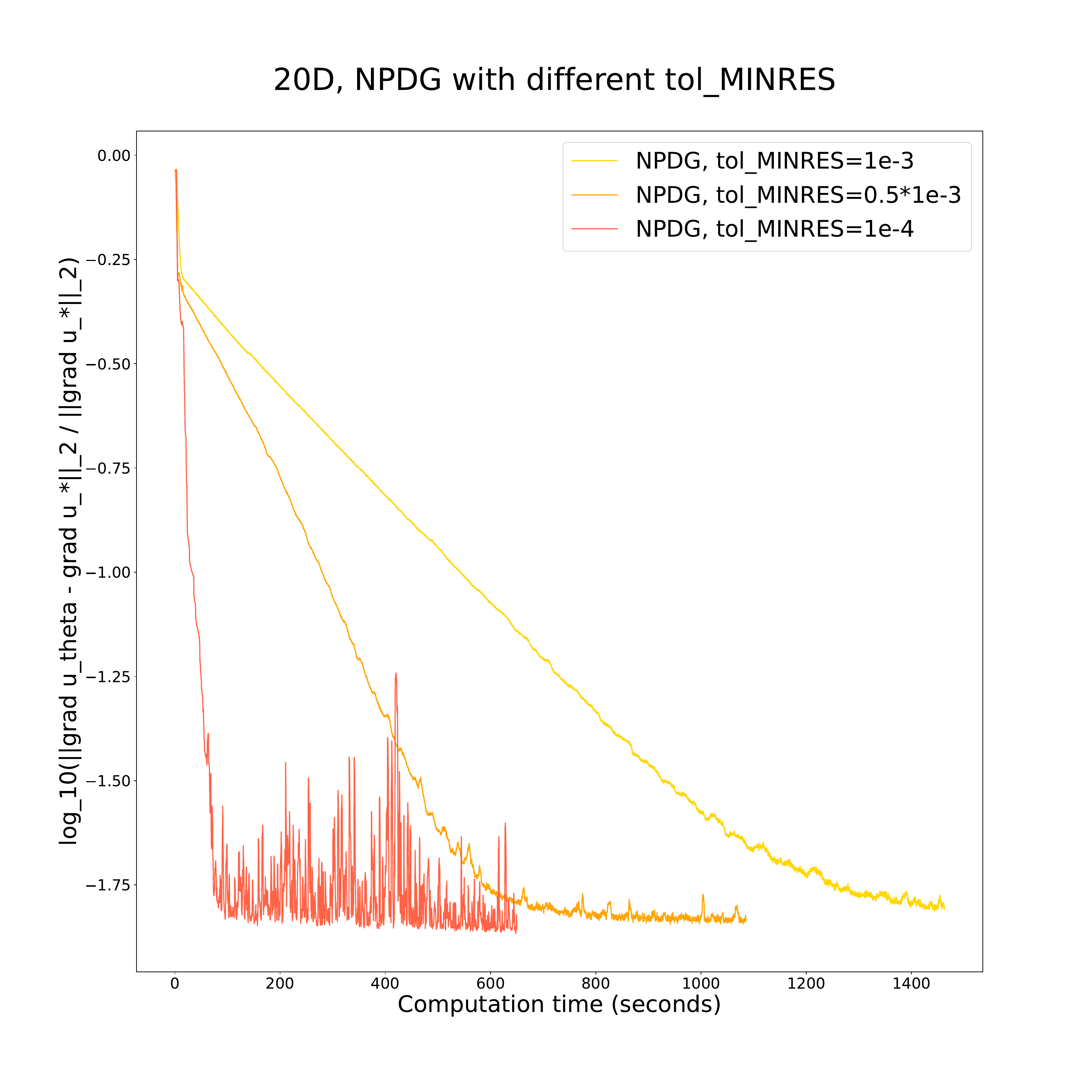}
        \caption{}\label{subfig: NPDG diff tol_minres H1 error 20d}
    \end{subfigure}
    \begin{subfigure}[b]{0.24\textwidth}
        \includegraphics[width=\textwidth]{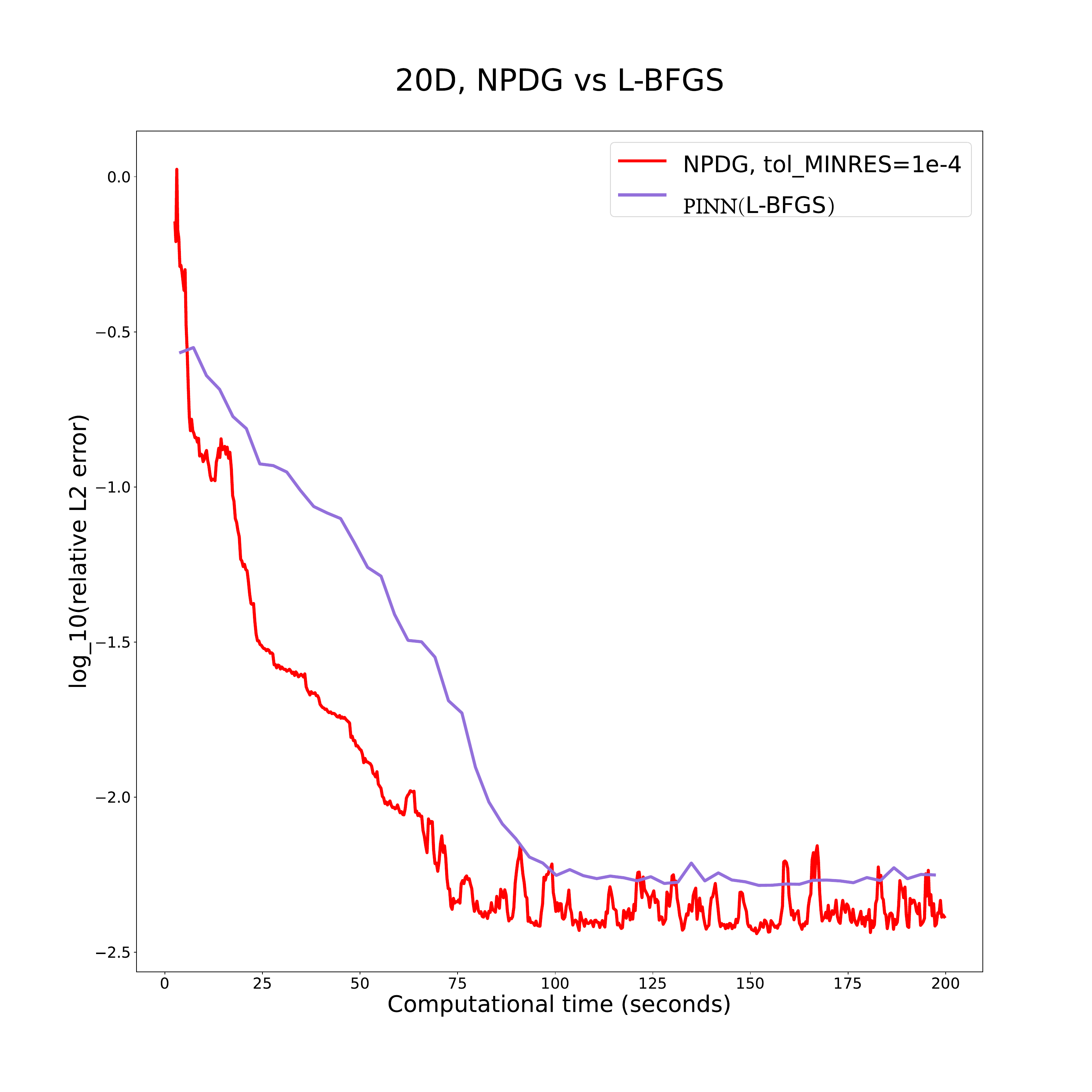}
        \caption{}\label{subfig: NPDG vs LBFGS l2 error 20d}
    \end{subfigure}
    \begin{subfigure}[b]{0.24\textwidth}
        \includegraphics[width=\textwidth]{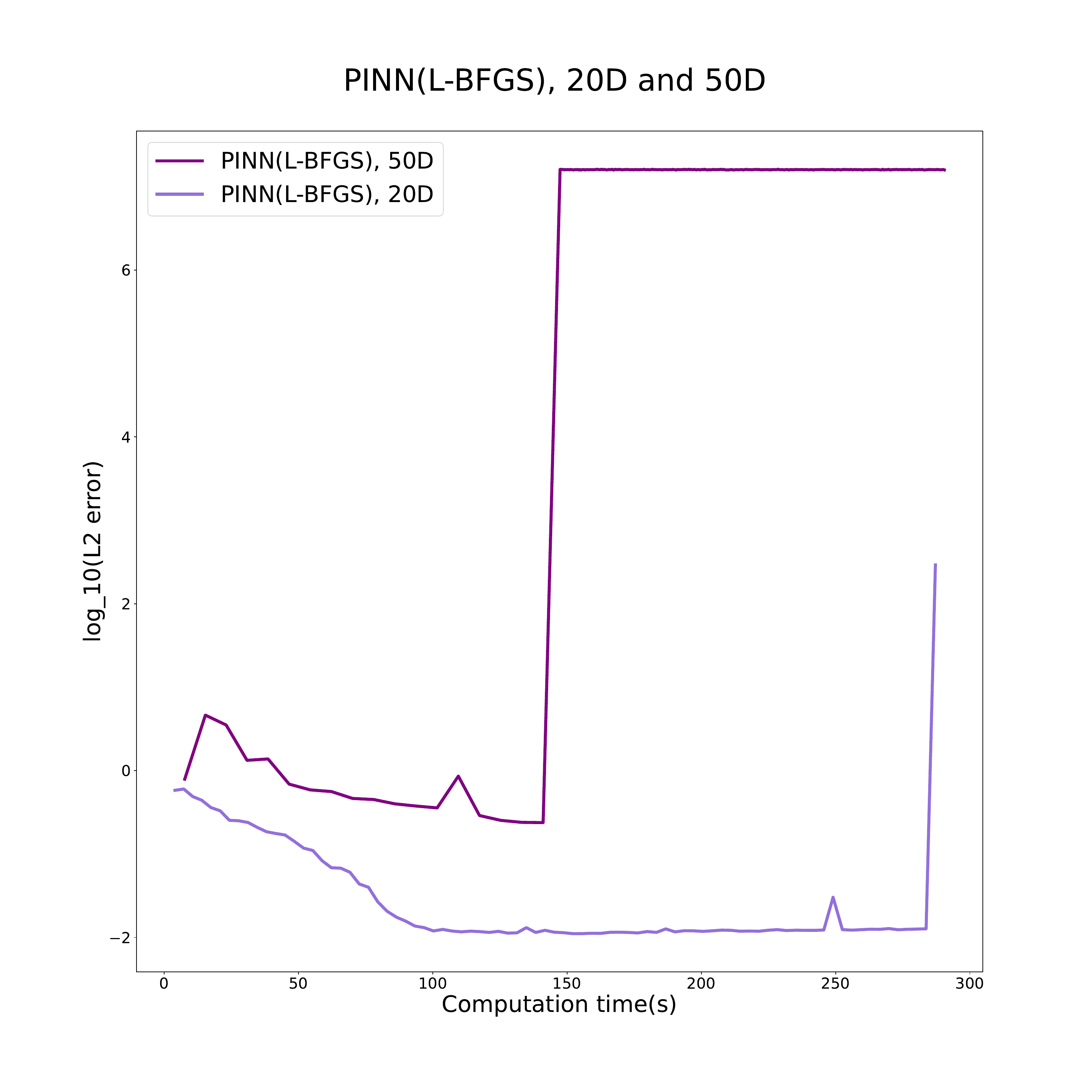}
        \caption{}\label{subfig: LBFGS l2 error blowup}
    \end{subfigure}
    \caption{\color{black}Figures \ref{subfig: NPDG diff tol_minres l2 error 20d}, \ref{subfig: NPDG diff tol_minres H1 error 20d}: Plots of relative error vs. computation time(seconds) with different $tol_{\textrm{MINRES}}=10^{-4}, 0.5 \cdot 10^{-3}, 10^{-3}$. Figure \ref{subfig: NPDG vs LBFGS l2 error 20d}: Plot of relative $L^2$ error vs. computation time (seconds), we compare NPDG with $tol_{\textrm{MINRES}} = 10^{-4}$ to PINN using L-BFGS optimizer. Figure \ref{subfig: LBFGS l2 error blowup}: Long-time behavior of L-BFGS optimizer when applied to 20D and 50D problems.}
    \label{fig:enter-label}
\end{figure}

\subsection{Nonlinear elliptic equation (5D)}  \label{subsec: nonlinear elliptic equ}
We consider the following nonlinear elliptic equation equipped with Dirichlet boundary condition on a $d$-dimensional ball with radius $R=3$
$${B}_{d, R} = \{ x \in \mathbb{R}^d ~ | ~ \|x\| \leq R \}.$$ 
\begin{equation}
  \frac{1}{2}\|\nabla u(x)\|^2 + V(x) = \Delta u(x), \quad u|_{\partial B_{d, R}} = 0.  \label{eq: nonlinear elliptic}
\end{equation}
Here we set
\[ V(x) = -\frac{\pi^2}{8}\sin^2(\frac{\pi}{2}r) - \frac{\pi^2}{4} \cos(\frac{\pi}{2}r) - \frac{\pi(d-1)}{2r} \sin(\frac{\pi}{2}r) \]
with $r=\|x\|$. The solution to this equation is the radial function
\[ u_*(x) = \cos(\frac{\pi}{2}r). \]
Similar to the previous examples, we introduce $\varphi, \psi$ to the equation and its boundary condition. We consider solving $\inf_{u}~\sup_{\varphi, \psi}~\widetilde{\mathscr{E}}(u, \varphi, \psi) := \mathscr{E}(u,\varphi, \psi) + \lambda \|\mathcal B u\|^2_{L^2(\mu_{\partial \Omega})}$ with 
\begin{align*}
\mathscr{E}(u,\varphi,\psi) =  \left(\int_\Omega \nabla\varphi(x)\cdot\nabla u(x) + \frac12 \|\nabla u(x)\|^2\varphi(x) + V(x)\varphi(x) \,\dd\mu - \frac{\blue{\nu}}{2}\int_\Omega \|\nabla \varphi(x)\|^2\,\dd\mu \right) & \\
+ {\lambda} \left(\int_{\partial \Omega} u\psi\,\dd\mu_{\partial\Omega}  - \frac{\blue{\nu}}{2} \int_{\partial \Omega} \psi^2 \,\dd\mu_{\partial\Omega} \right)&.
\end{align*}
It is still unclear what is the optimal way to precondition the nonlinear term in this equation. In our treatment, we only focus on the linear part $\Delta u$ and set 
\[ \mathcal M_p = \mathcal M_d = \nabla \]
for the preconditioning matrices $M_p(\theta), M_d(\eta), M_{bdd}(\xi).$

We test this example with $d=5$, we set %
\[ u_\theta = \texttt{MLP}_{\tanh}(d, 256, 1, 4), ~ \varphi_\eta = \texttt{MLP}_{\tanh}(d, 256, 1, 4), \quad \psi_\xi = \texttt{MLP}_{\tanh}(d, 128, 1, 4).  \]
The stepsizes are chosen as $\tau_u = 0.05, \tau_\varphi=0.095, \tau_\psi=0.095$. We apply Monte-Carlo method to evaluate the loss function, in order to sample uniformly from $B_{d, R}$, we first randomly sample $N_{in} =4000$ points $\rho_1, \dots, \rho_{N_{in}}$ from the interval $[0, R]$ following the density function $p(\rho)= \frac{d+1}{R}\left( \frac{\rho}{R} \right)^d, \rho \in [0, R]$\footnote{This can be done by first sampling $n_\rho$ points $r_1,\dots,r_{n_\rho}$ uniformly from $[0, 1]$ and then transforming each $r_i$ to $\rho_i = r_i^{\frac{1}{d}}\cdot R^{1-\frac{1}{d}}$ for $1\leq i \leq n_\rho$.}. Then we sample $N_{in}$ points $\textbf{w}_1, \dots, \textbf{w}_{N_{in}}$ from the standard Gaussian distribution $\mathcal N(0, I_d)$. Thus, we obtain $N_{in}$ sample points in $B_{d, R}$ by forming $x_i = \rho_i \frac{\textbf{w}_i}{\|\textbf{w}_i\|+e_0}$, $1\leq i\leq N_{in}$. We add $e_0=10^{-8}$ to prevent division by zero. We run the proposed method for $N_{iter} = 10000$ iterations.

In this example, we also test the PINN(Adam/L-BFGS) and WAN methods. The hyperparameters for these methods are provided in Table \ref{tab: setup for PINN DeepRitz WAN PDAdam}.  Log-log plots of the relative error vs. the computation time among the methods are provided in Figure \ref{fig: Nonlinear1}.
\begin{figure}[htb!]
    \centering
    \begin{subfigure}[b]{0.28\textwidth}
        \includegraphics[width=\textwidth]{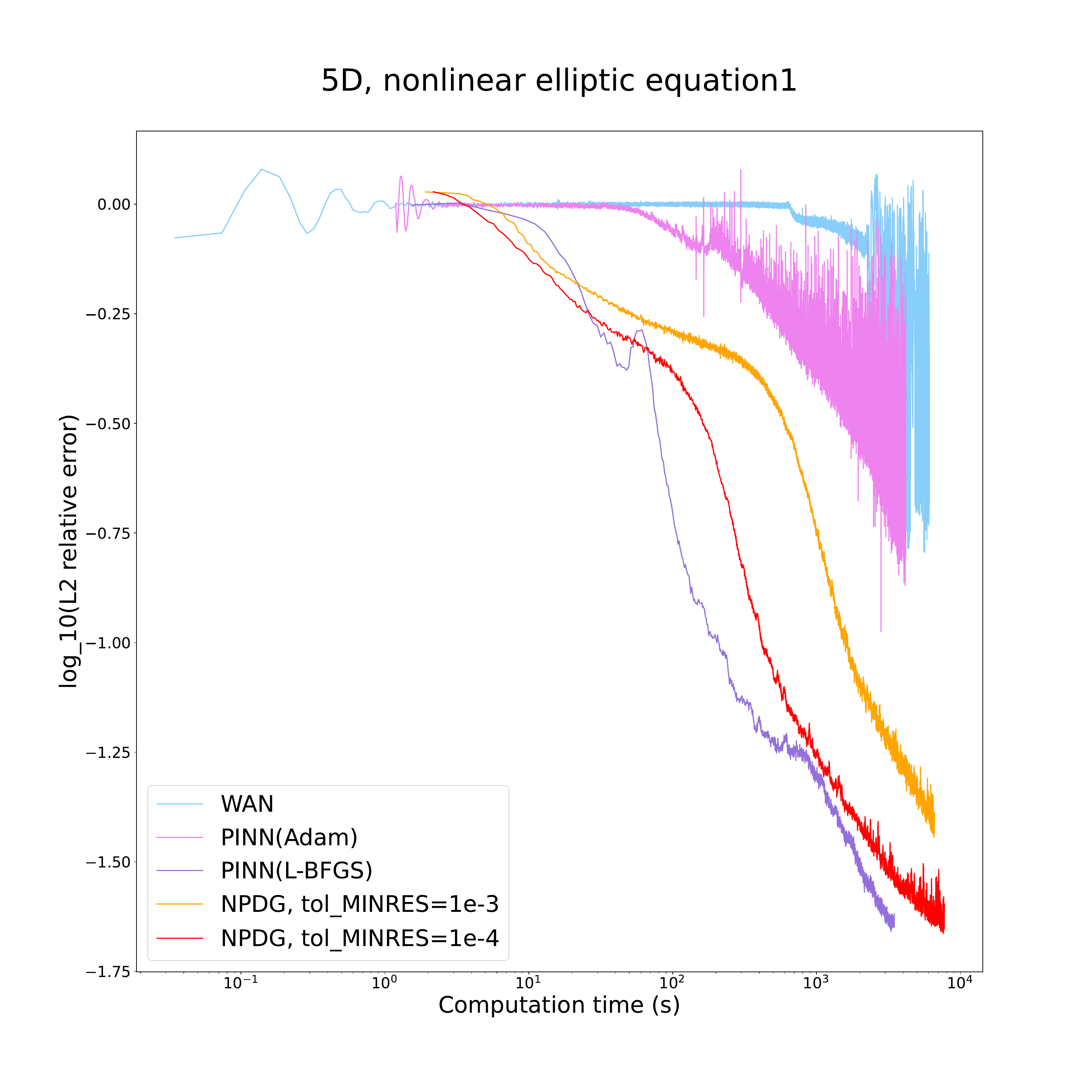}
        \caption{}\label{subfig: Nonlinear log rel l2 err vs log comp time}
    \end{subfigure}
    \begin{subfigure}[b]{0.28\textwidth}
        \includegraphics[width=\textwidth]{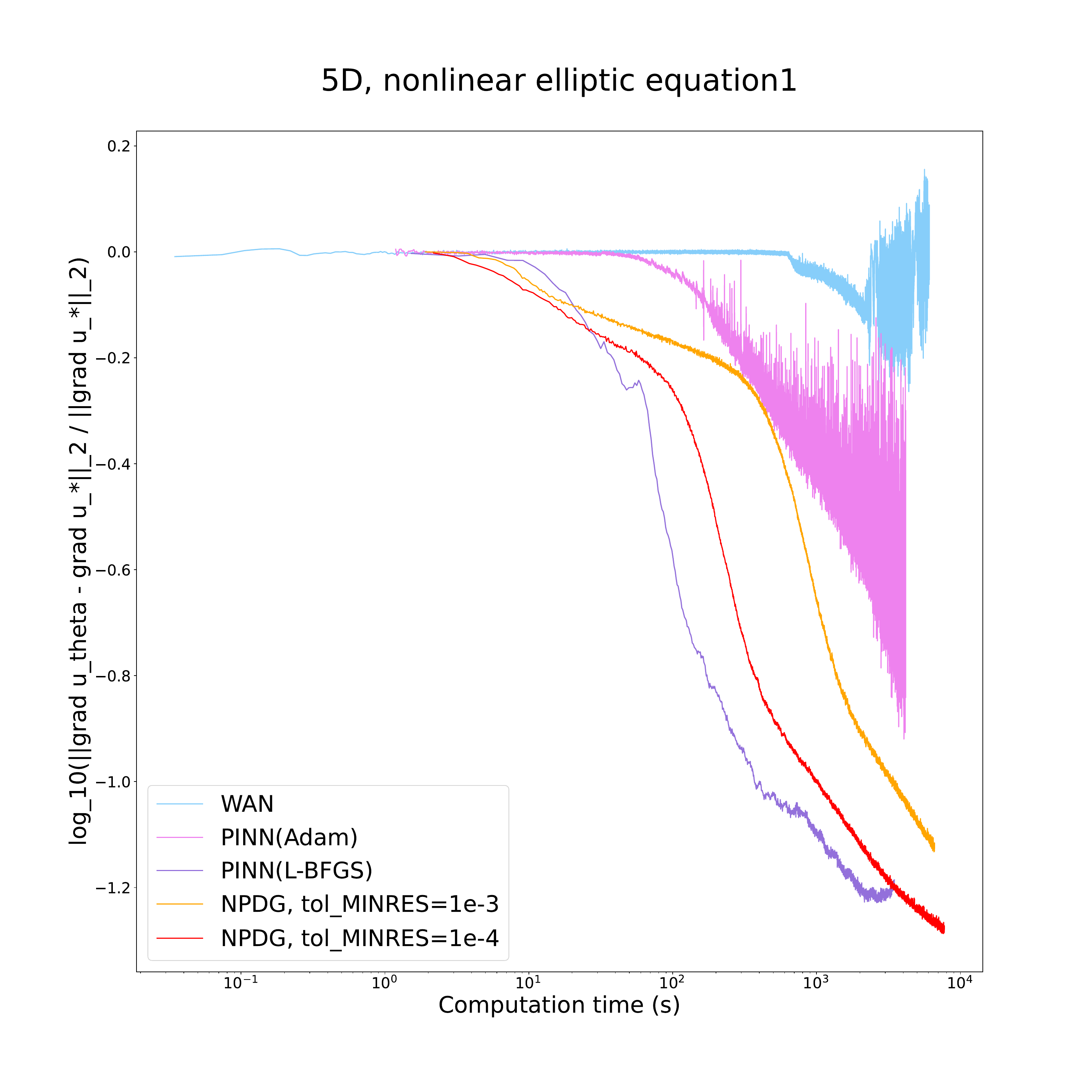}
        \caption{}\label{subfig: Nonlinear log rel H1 err vs log comp time}
    \end{subfigure}
    \begin{subfigure}[b]{0.3\textwidth}
        \includegraphics[width=\textwidth]{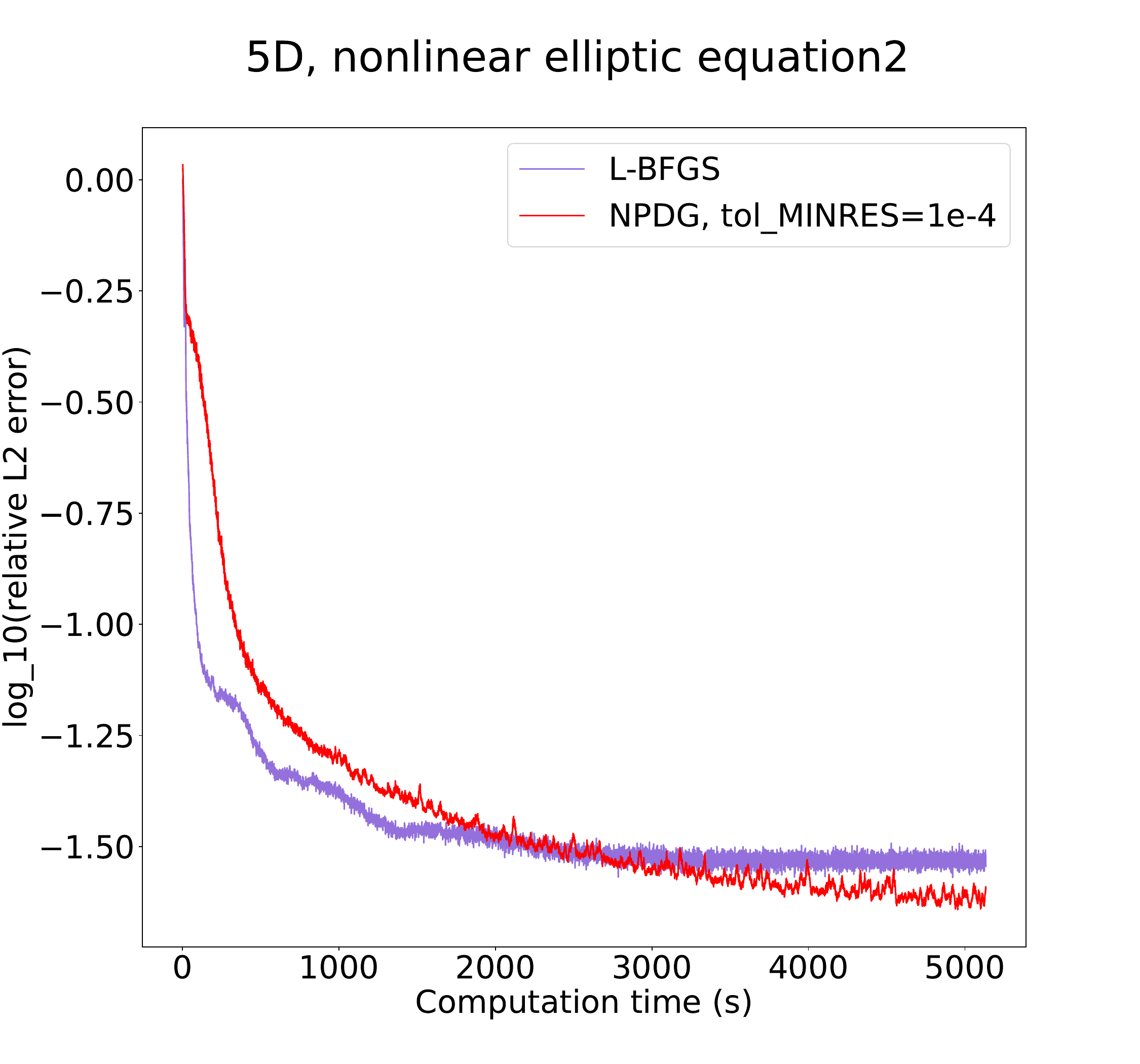}
        \caption{}\label{subfig: Nonlinear mild log rel l2 err vs comp time NPDG and LBFGS}
    \end{subfigure}
    \caption{Equation \eqref{eq: nonlinear elliptic}: \textbf{Left}: Log-log plot of relative $L^2$ error vs. computational time (seconds); \textbf{Middle}: Log-log plot of relative $H^1$ seminorm error vs. computational time (seconds). The values of $\|u_*\|_{L^2(\Omega, \mu)}$ and $\|\nabla u_*\|_{L^2(\Omega, \mu)}$ are provided in Table \ref{tab: real sol and norms}. Equation \eqref{eq: mild nonlinear elliptic}: \textbf{Right}: Semi-log plot of relative L2 error vs. computational time.
    }
    \label{fig: Nonlinear1}
\end{figure}
We plot the graph of $u_\theta$ obtained by the algorithm on the $1-2$ coordinate plane in Figure \ref{subfig:  graph u_theta}. We also plot the heat maps of the error function $|u_{\theta}(\cdot)-u_*(\cdot)|$ on various coordinate planes in Figures \ref{subfig: heatmap err 1-2 plane}-\ref{subfig: heatmap err 4-5 plane}.
\begin{figure}[htb!]
    \centering
    \begin{subfigure}[b]{0.19\textwidth}
        \includegraphics[width=\textwidth]{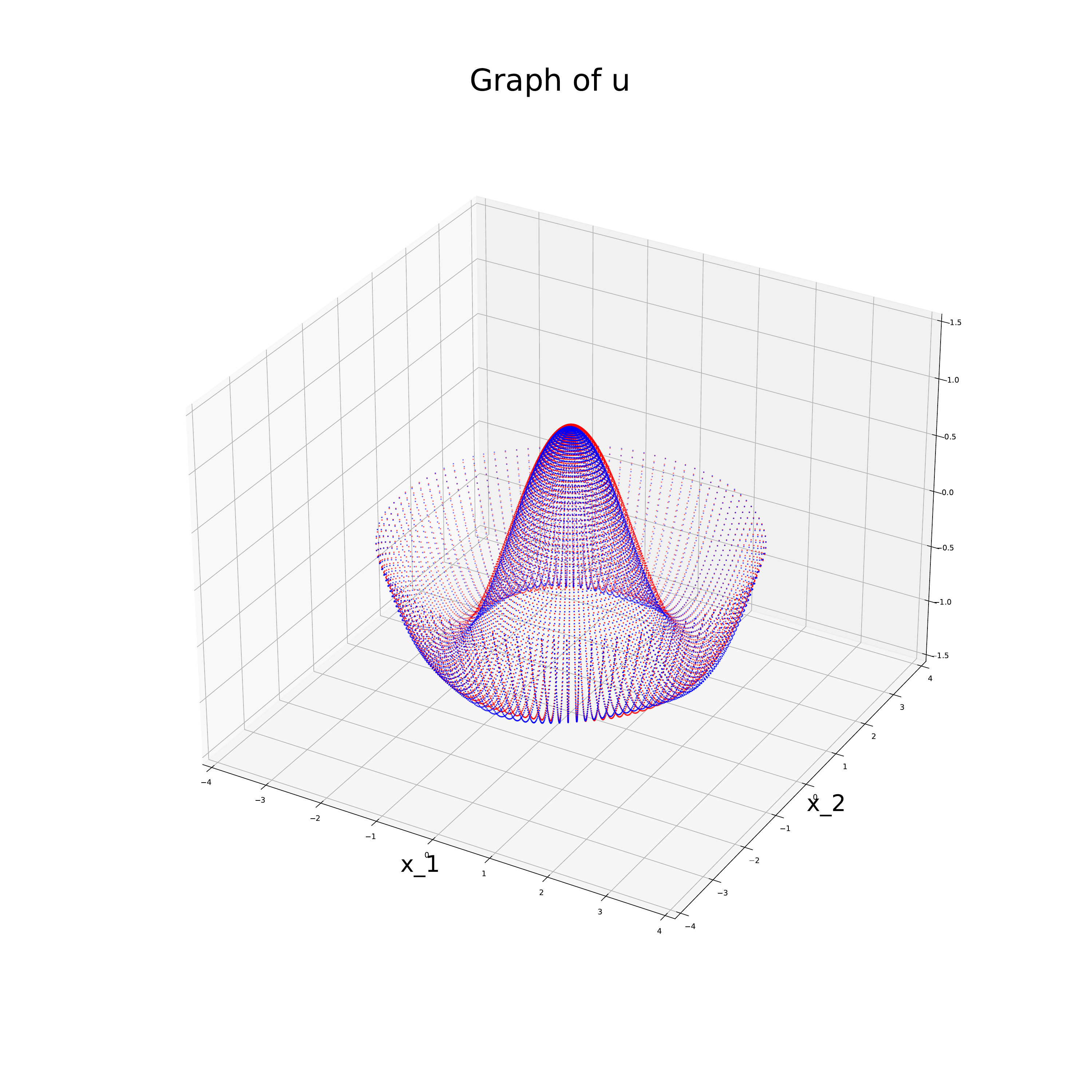}
        \caption{}\label{subfig:  graph u_theta}
    \end{subfigure}
    \begin{subfigure}[b]{0.19\textwidth}
        \includegraphics[trim={0 0 0 3.7cm}, clip, width=\textwidth]{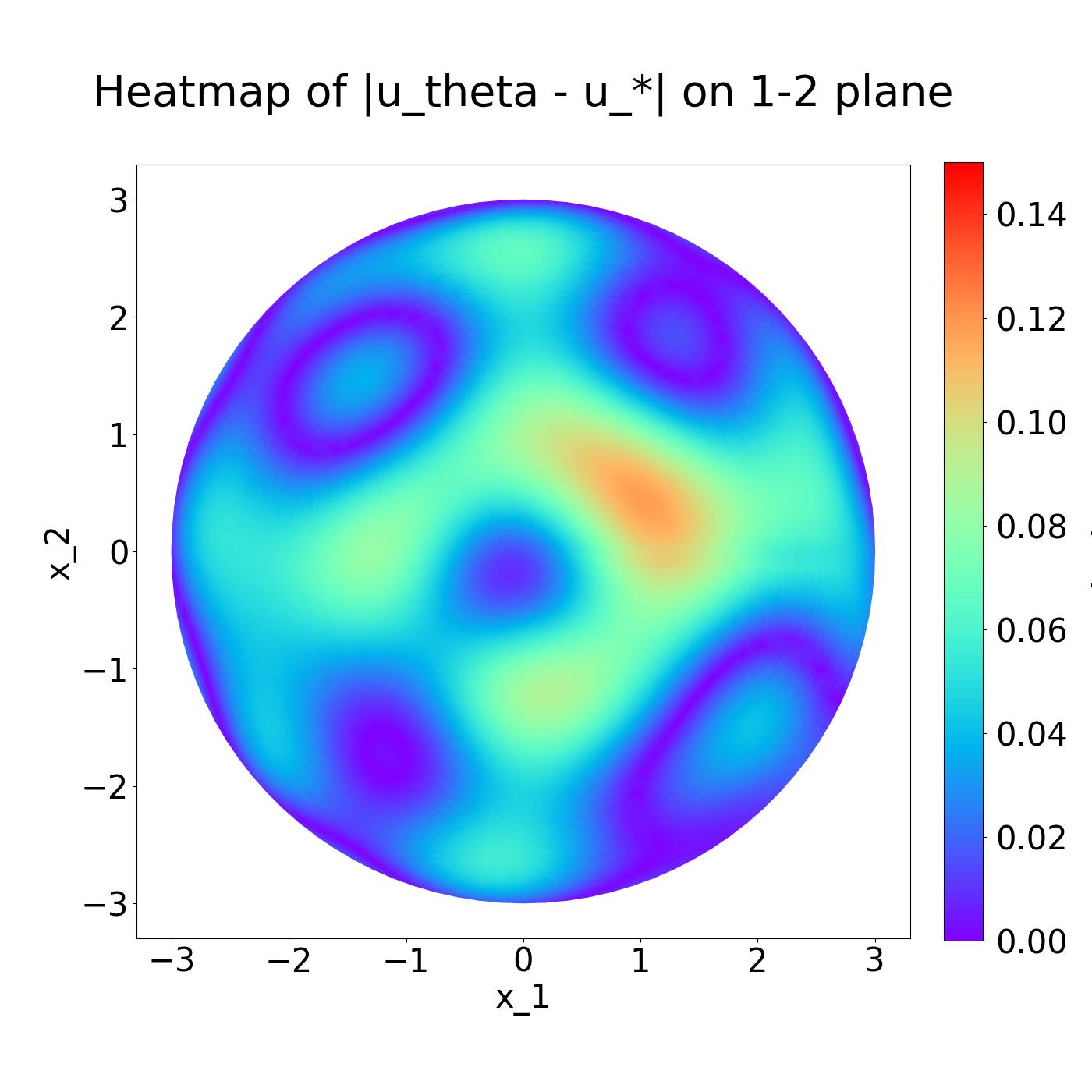}
        \caption{}\label{subfig: heatmap err 1-2 plane}
    \end{subfigure}
    \begin{subfigure}[b]{0.19\textwidth}
        \includegraphics[trim={0 0 0 3.7cm}, clip, width=\textwidth]{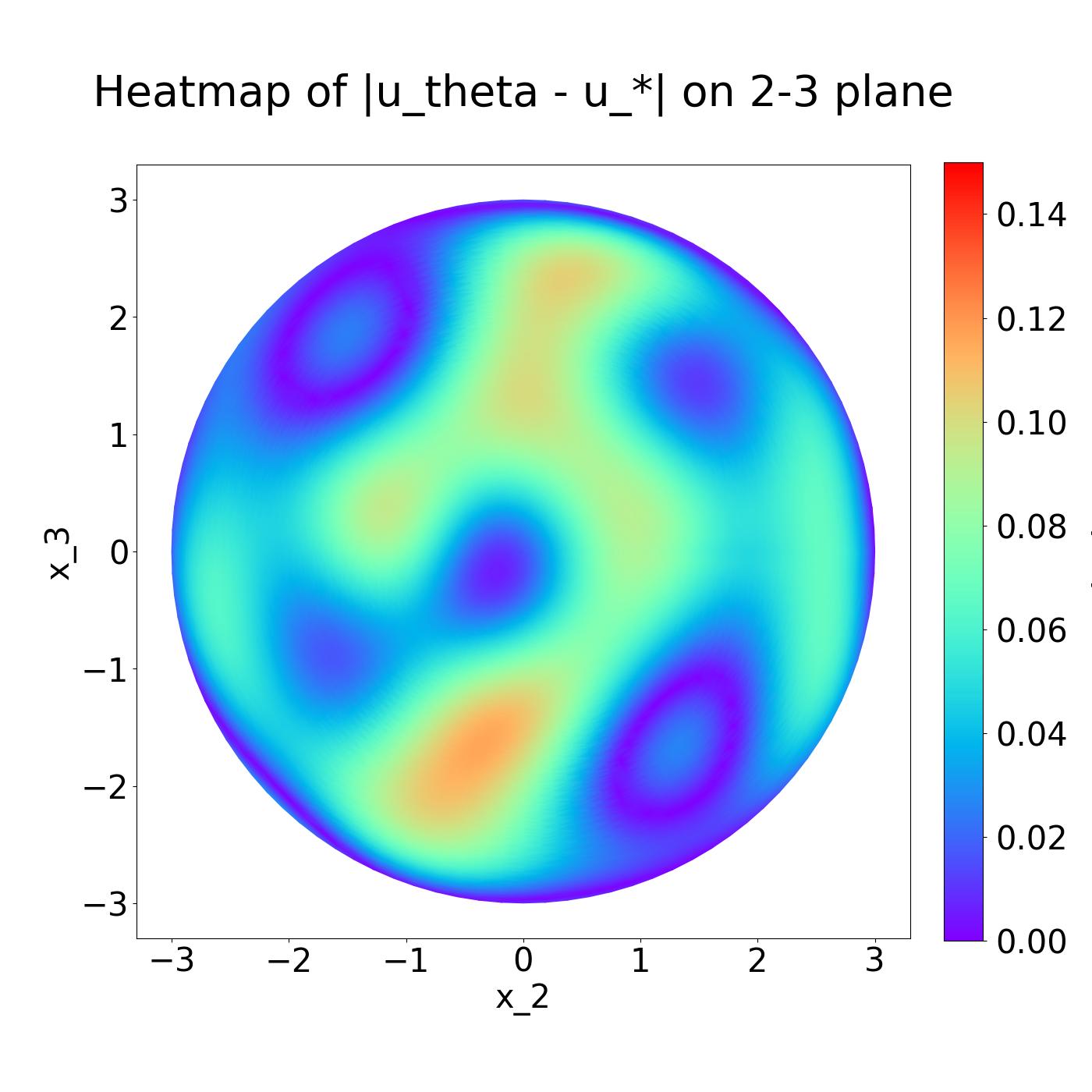}
        \caption{}\label{subfig: heatmap err 2-3 plane}
    \end{subfigure}
    \begin{subfigure}[b]{0.19\textwidth}
        \includegraphics[trim={0 0 0 3.7cm}, clip, width=\textwidth]{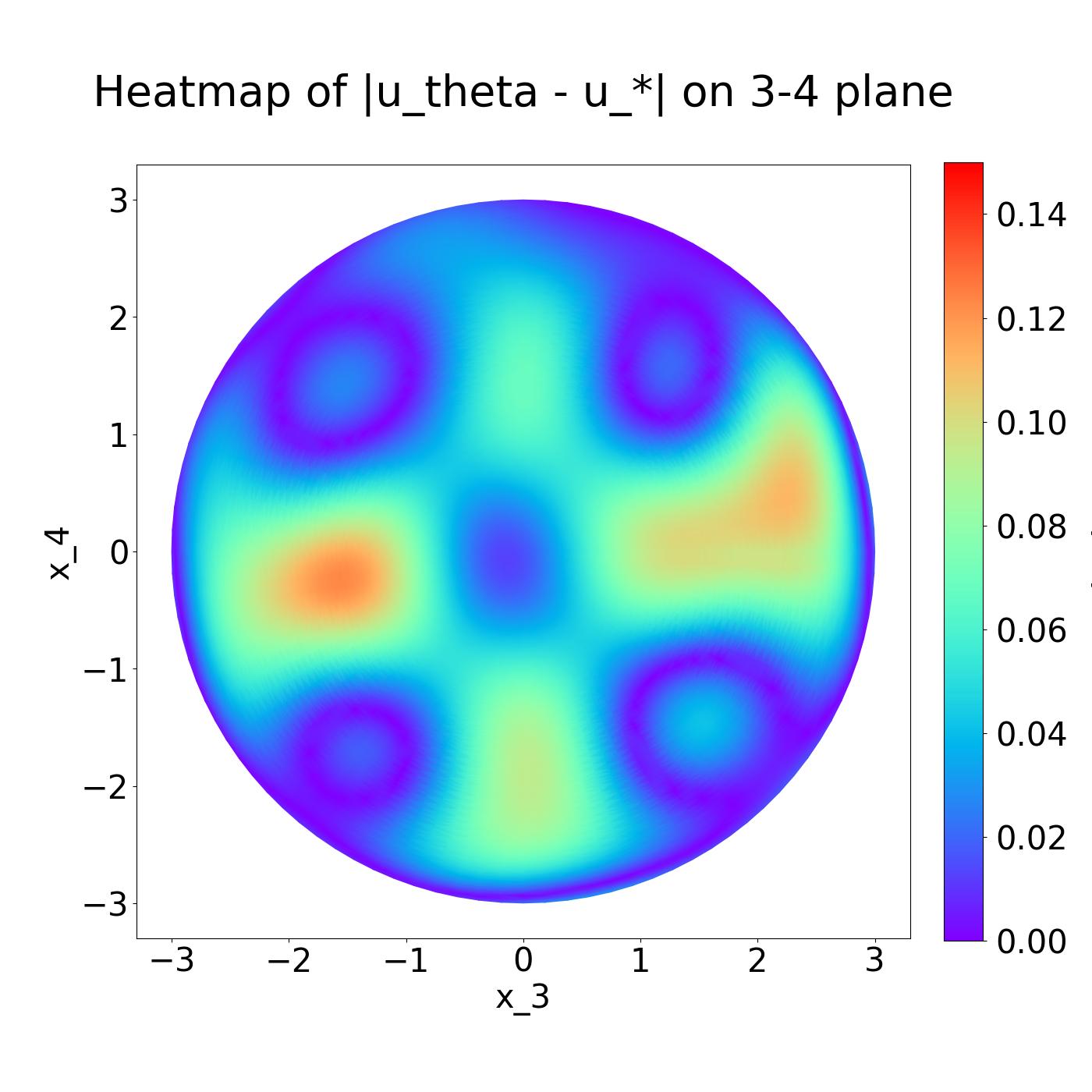}
        \caption{}\label{subfig: heatmap err 3-4 plane}
    \end{subfigure}
    \begin{subfigure}[b]{0.19\textwidth}
        \includegraphics[trim={0 0 0 3.7cm}, clip, width=\textwidth]{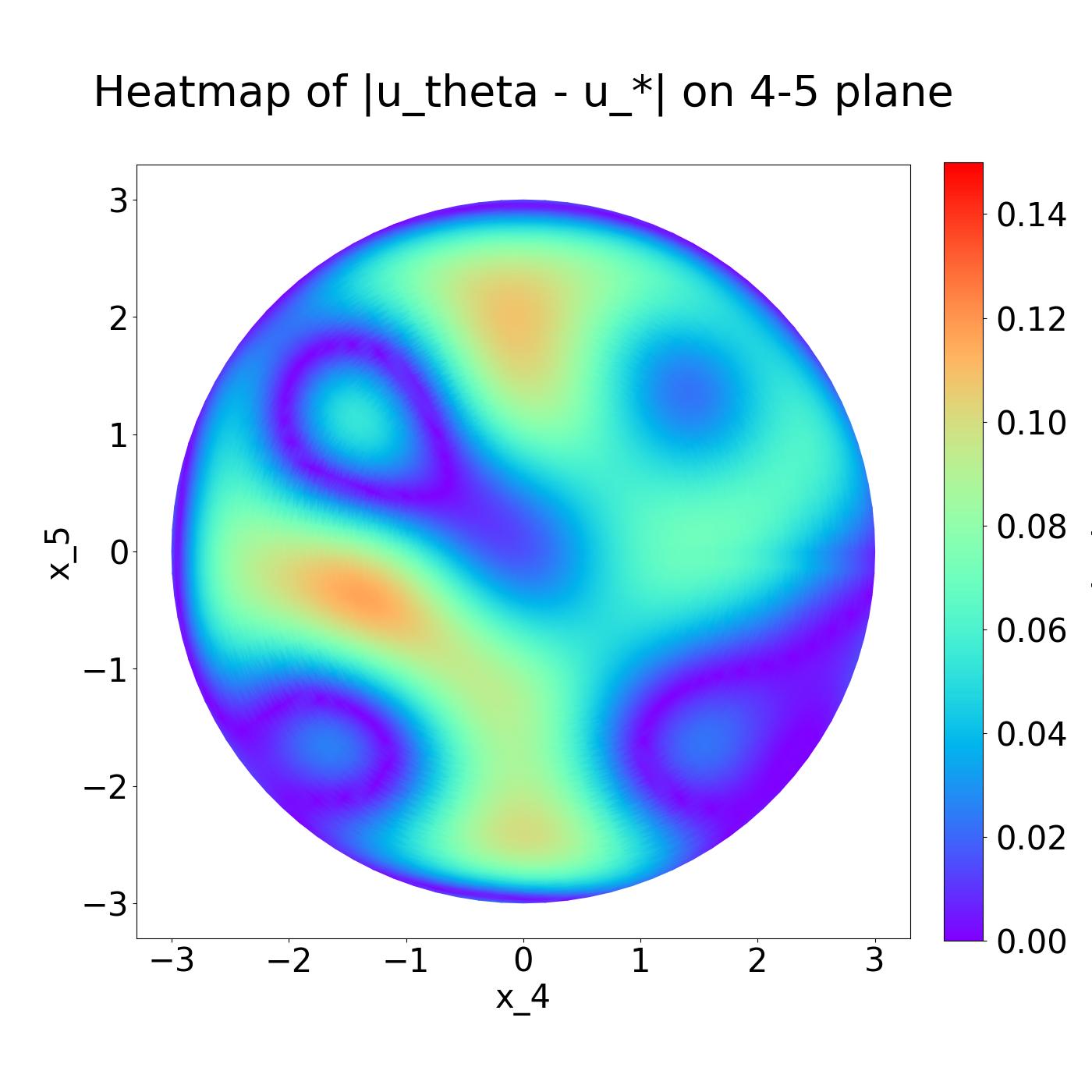}
        \caption{}\label{subfig: heatmap err 4-5 plane}
    \end{subfigure}
    \caption{Figure \ref{subfig:  graph u_theta}: Graph of $u_\theta$ on the $1-2$ coordinate plane (that is, the plane spanned by the first and second components with the remaining coordinates fixed to $0$). The parameter $\theta$ is obtained by the NPDG method after $10000$ iterations; Figures \ref{subfig: heatmap err 1-2 plane}-\ref{subfig: heatmap err 4-5 plane}: Heatmaps of $|u_\theta(\cdot)-u_*(\cdot)|$ plotted on $1-2$, $2-3$, $3-4$, $4-5$ coordinate planes.  }
    \label{fig: Nonlinear2}
\end{figure}
Similar to previous examples, we record the GPU times spent by different methods for achieving certain accuracy in Table \ref{tab: GPU time to accuracy } of Appendix \ref{append: compare among diff methods}. 
Furthermore, we also consider the following equation on the same region $B_{d, R}$ ($d=5$, $R=3$) with a weaker nonlinear term,
\begin{equation}
  \frac{\epsilon_0}{2}\|\nabla u(x)\|^2 + \Delta u(x)  =  V(x), \quad u|_{\partial B_{d, R}} = 0.  \label{eq: mild nonlinear elliptic}
\end{equation}
Here we set $\blue{\epsilon_0} = \frac{1}{10}$ and
\[ V(x) = \frac{\epsilon_0  \pi^2}{8}\sin^2(\frac{\pi}{2}r) - \frac{\pi^2}{4} \cos(\frac{\pi}{2}r) - \frac{\pi(d-1)}{2r} \sin(\frac{\pi}{2}r). \]
The solution to this equation is still $u_*(x) = \cos(\frac{\pi}{2}r)$. We apply the NPDG algorithm with exactly the same neural network architecture and hyperparameters as in \eqref{eq: nonlinear elliptic} to solve equation \eqref{eq: mild nonlinear elliptic}. We also test the L-BFGS optimizer to minimize the PINN loss of equation \eqref{eq: mild nonlinear elliptic}. Figure \ref{subfig: Nonlinear mild log rel l2 err vs comp time NPDG and LBFGS}
 indicates that the proposed method achieves performance that is comparable with L-BFGS in this example.

\subsection{Allen-Cahn equation }  \label{subsec: AllenCahn}
 We have discussed several examples of time-independent PDEs. We now briefly show how the proposed method is applied to resolve the time-implicit, semi-discrete schemes of the time-dependent equations. In this section, we primarily focus on the 1D and 2D Allen-Cahn equations to illustrate the main idea. Future research will explore additional approaches, such as adaptive sampling techniques \citep{wight2020solving} and extensions to higher dimensions. 

We consider the Allen-Cahn equation on a bounded domain $\Omega$ posed with the homogeneous Neumann boundary condition on time interval $[0, T]$.
\begin{equation*}
  \frac{\partial u(x,t)}{\partial t} = \epsilon_0\Delta u(x,t) - \frac{1}{\epsilon_0} W'(u), \quad \frac{\partial u}{\partial \textbf{n}}=0 ~\textrm{on } \partial \Omega, \quad u(\cdot, 0)=u_0(\cdot). 
\end{equation*}
Here we define the double-well potential function $W(u)=\frac{1}{4}(1-u^2)^2$, with $W'(u) = u^3 - u$. \blue{It is well-known that the Allen-Cahn equation can be viewed as the $L^2-$gradient flow of the energy functional $E(u)=\int_\Omega\frac{\epsilon_0}{2}\|\nabla u\|^2 + \frac{1}{2\epsilon_0}W(u) \, \dd x.$} 

In this research, we focus on resolving the time-implicit, semi-discrete numerical scheme of this equation. We divide the time interval into $N_t$ subintervals and consider
\begin{equation*}
  \frac{u^t(x) - u^{t-1}(x)}{h_t} = \epsilon_0 \Delta u^t(x) - \frac{1}{\epsilon_0} W'(u^t(x)), ~ \frac{\partial u^t}{\partial \textbf{n}}=0 ~ \textrm{on } \partial \Omega, 
\end{equation*}
sequentially for $1\leq t \leq N_t$ with $u^0(\cdot)$ set as $u_0(\cdot)$. \blue{One motivation for considering this implicit scheme is its energy stability, in the sense that
$E(u^t)\leq E(u^{t-1}),$ which respects the gradient-flow nature of the equation.
}

The problem boils down to solving $N_t$ consecutive elliptic equations with a cubic term as shown below,
\begin{equation} 
  u^t(x)-\epsilon_0 h_t\Delta u^t(x) + \frac{h_t}{\epsilon_0} ((u^t(x))^3 - u^t(x)) = u^{t-1}(x), ~ \frac{\partial u^t}{\partial \textbf{n} }=0 ~ \textrm{on }\partial \Omega, \quad 1\leq t \leq N_t.    \label{AC equ to elliptic equ with cubic term}
\end{equation}
In the implementation, we substitute the primal function $u$, and the test functions $\varphi, \psi$ with neural networks with $\mathrm{tanh}$ activations,
\[
    u_\theta = \texttt{MLP}_{\mathrm{tanh}}(d, 128, 1, 5), \quad \varphi_\eta = \texttt{MLP}_{\mathrm{tanh}}(d, 128, 1, 5)\cdot \zeta, \quad \psi_\xi = \texttt{MLP}_{\mathrm{tanh}}(d, 64, 1, 5).  
\]
\blue{We adopt different preconditioning strategies based on the magnitude of $\epsilon_0$. A detailed discussion is provided in Appendix~\ref{append: AC}.}

\vspace{0.5cm}

\noindent
\textbf{1D example} We first test the algorithm on the 1D example with $\Omega=[0, 2]$, initial data $u_0(x) = (1-\cos(\pi (x-1)))\cos(\pi (x-1)).$ \blue{We work on two cases in which $\epsilon_0=0.1, T=1, N_t=10$; and $\epsilon_0=0.01, T=0.08, N_t=80$.} We treat the distribution $\mu_{\partial \Omega} = \frac{1}{2}(\delta_0 + \delta_2)$ where $\delta_x$ denotes the Dirac measure\footnote{That is, $\delta_x(E) = 1$ for any measurable set $E\subset \mathbb{R}$ that contains $x$, and $\delta_x(E)=0$ for measurable sets that do not contain $x.$} concentrated on the point $x\in\mathbb{R}$. 

For the algorithm, we set $N_{in}=2000$ and $N_{bdd}=2$. Since $\partial \Omega=\{0,2\}$, one boundary sample is assigned to each endpoint. The boundary loss coefficient is chosen as $\lambda=1$. \blue{We keep $\tau_u = 0.05$, $\tau_\varphi = 0.095$, and $\tau_\psi = 0.095$ for $\epsilon_0 = 0.1$, and choose smaller stepsizes $\tau_u = 0.01$, $\tau_\varphi = 0.02$, and $\tau_\psi = 0.02$ for $\epsilon_0 = 0.01$.}

\blue{In Figure \ref{fig: AC1}, we plot the graphs of our numerical solution $u_{\theta_k}$ obtained at different time nodes $t_k=\frac{k}{N_t}$ ($1\leq k \leq N_t$) with the benchmark solution $\{U^k\}_{k=1}^{N_t}$ solved from time-implicit, finite difference scheme \eqref{implicit, FD} provided in Appendix \ref{append: benchmark}. The semi-log curve of $\sqrt{\frac{1}{N_x}\sum_{i=1}^{N_x} (u_{\theta_k}(x_i) - U_i^k)^2}$ vs. the computation time ($\epsilon_0=0.1$) is presented in Figure~\ref{subfig: AC semi-log MSE vs comptime}. The energy decay plot of $E(u_{\theta_k})$ versus $t_k$ for $\epsilon_0 = 0.01$ is included in Figure~\ref{subfig: AC energy decay eps=0.01}. The numerical solution $u_{\theta_k}$ exhibits monotonic decay of its energy and shows agreement with the benchmark solution.}
\begin{figure}[htb!]
    \centering
    \begin{subfigure}[b]{0.23\textwidth}
        \includegraphics[width=\linewidth]{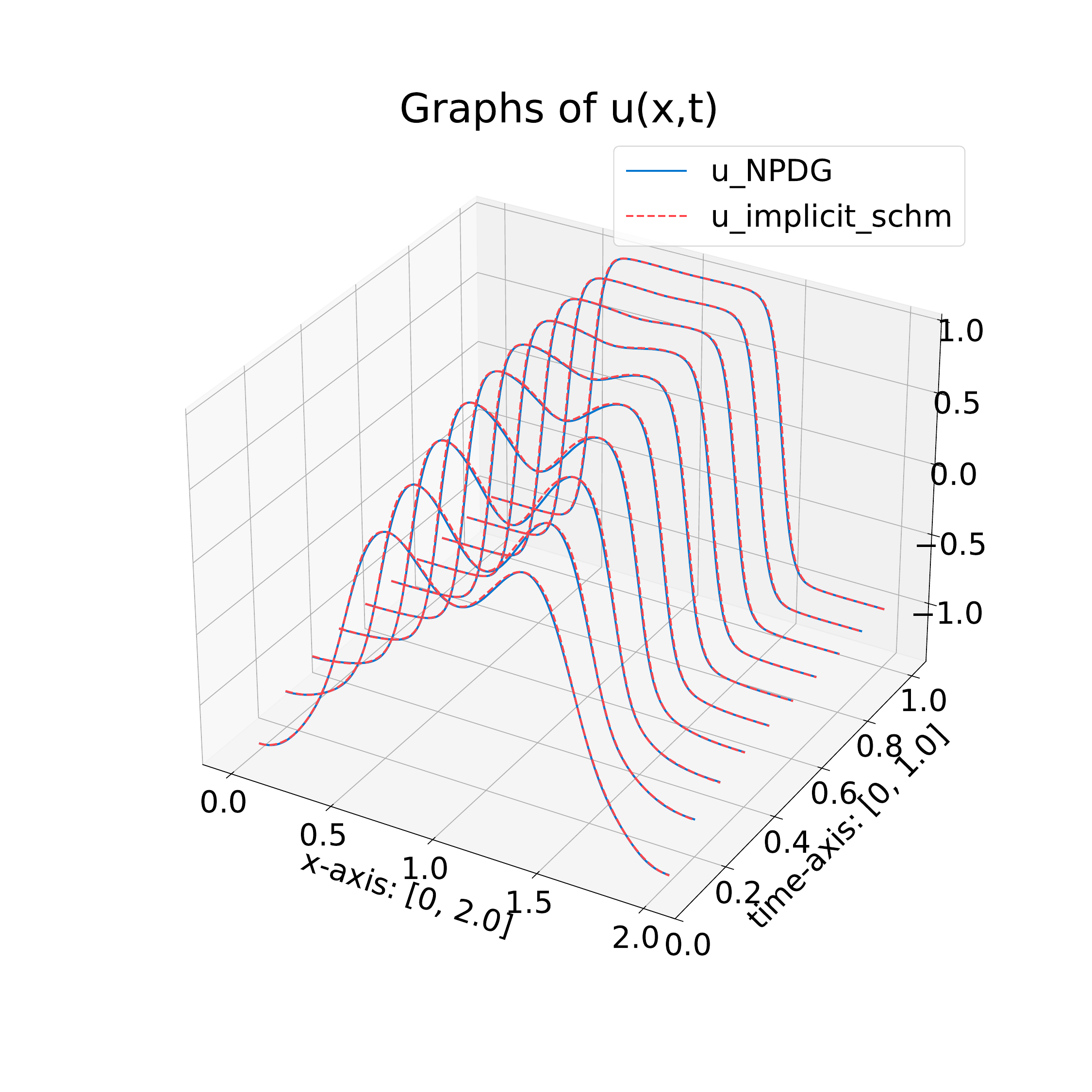}
        \caption{\footnotesize $\epsilon_0=0.1$}\label{subfig: AC graphs of ut}
    \end{subfigure}
    \begin{subfigure}[b]{0.23\textwidth}
        \includegraphics[width=\linewidth]{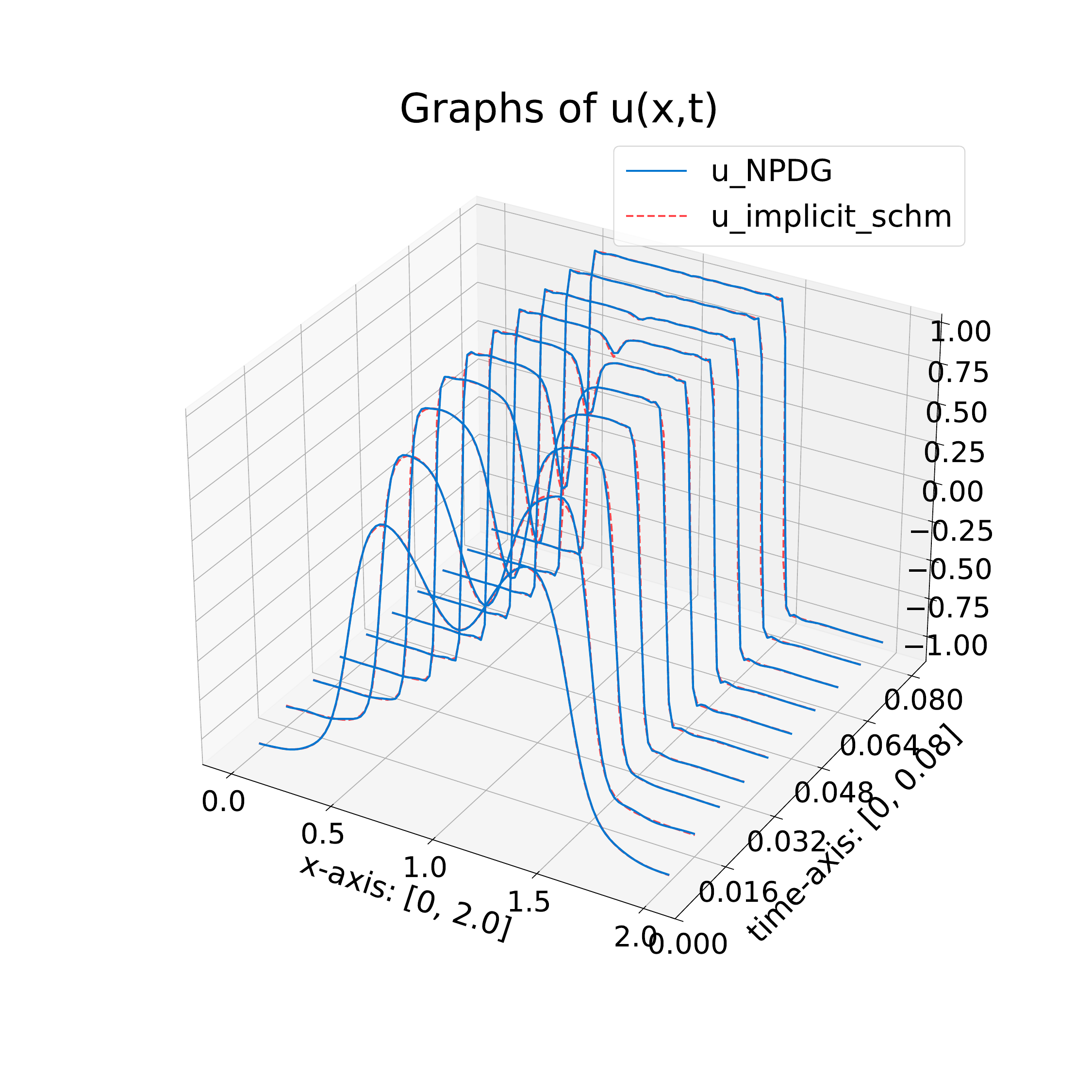}
        \caption{\footnotesize $\epsilon_0=0.01$}\label{subfig: AC graphs of ut eps=0.01}
    \end{subfigure}
    \begin{subfigure}[b]{0.23\textwidth}
        \includegraphics[width=\linewidth]{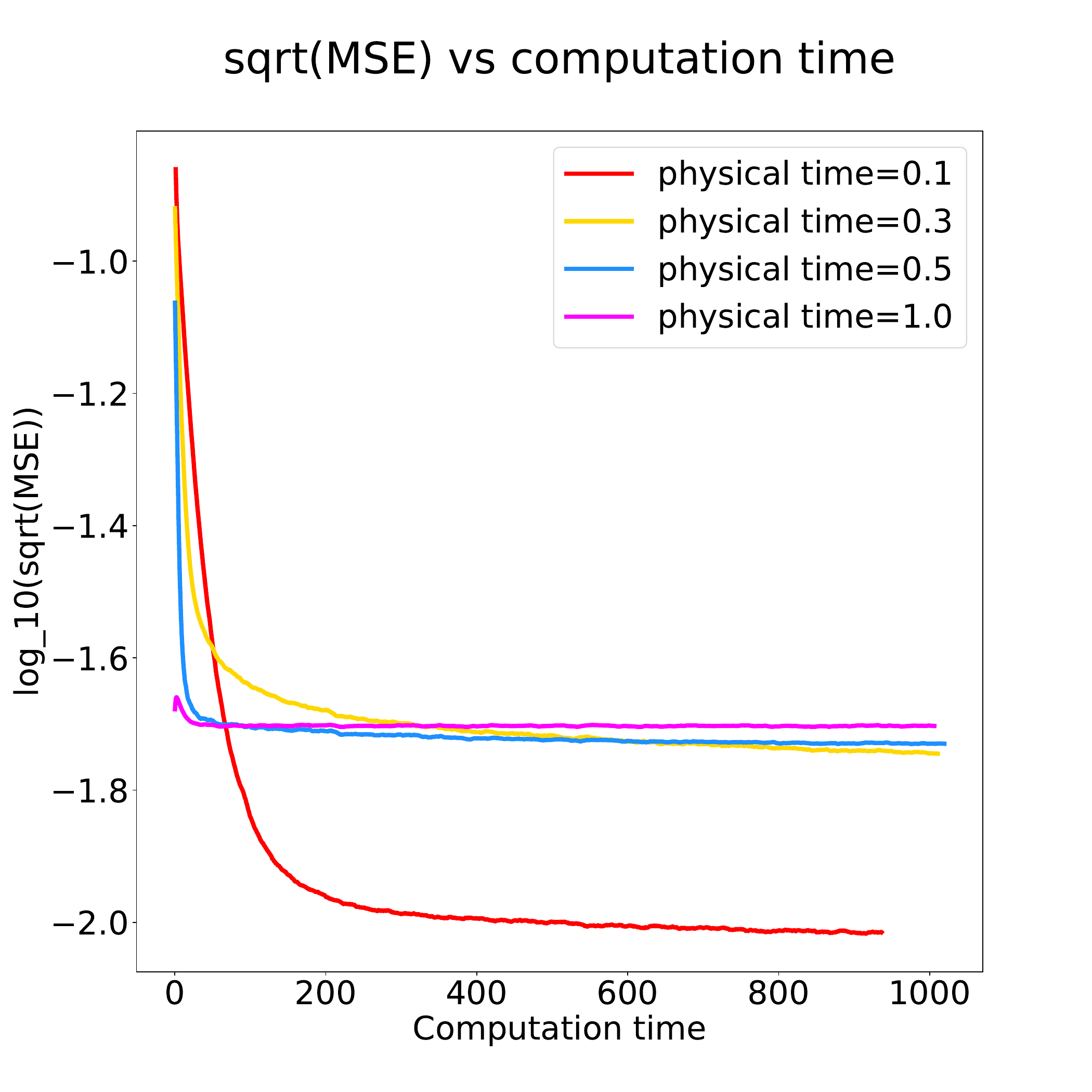}
        \caption{\footnotesize MSE vs. walltime}\label{subfig: AC semi-log MSE vs comptime}
    \end{subfigure}
    \begin{subfigure}[b]{0.23\textwidth}
        \includegraphics[width=\linewidth]{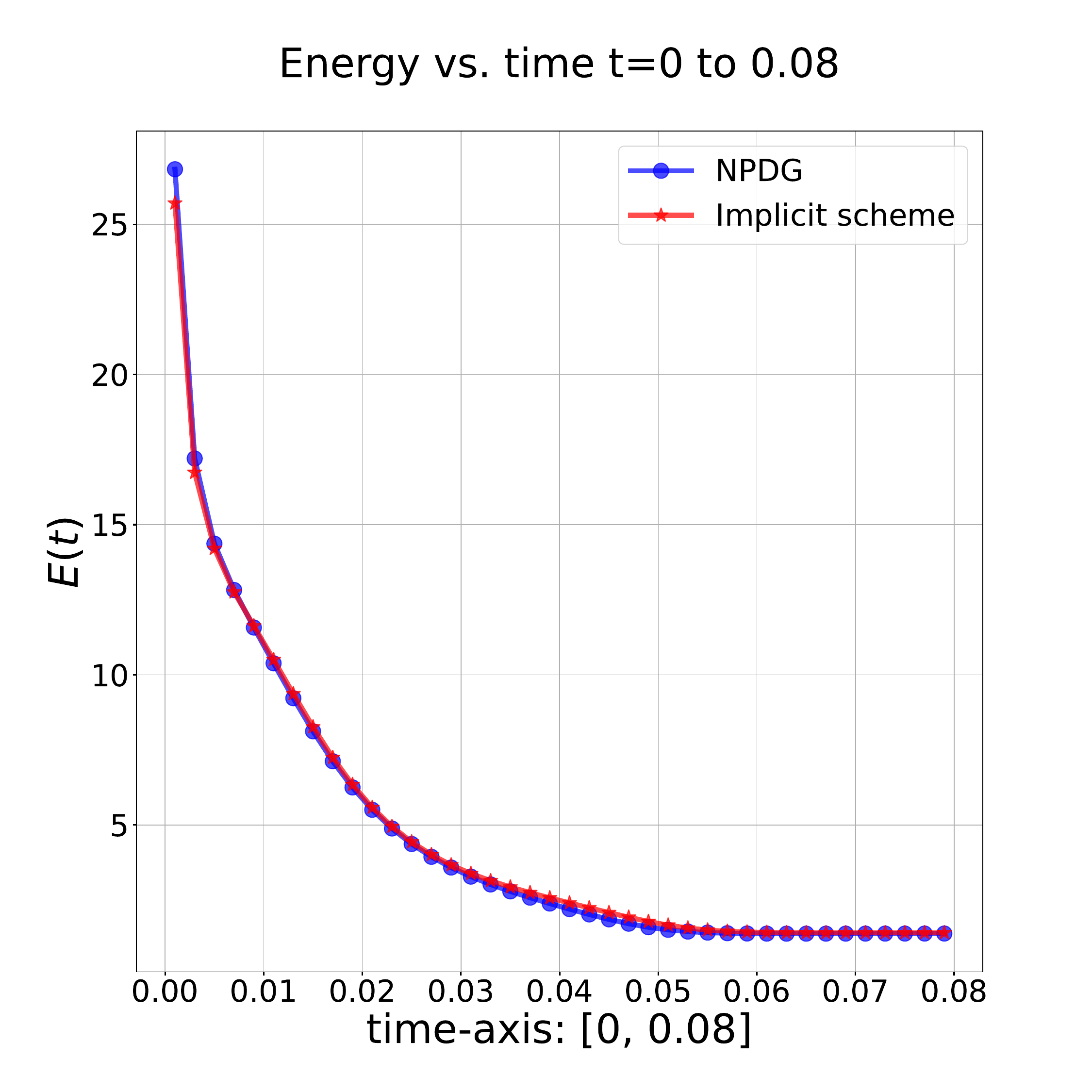}
        \caption{\footnotesize Energy decay}\label{subfig: AC energy decay eps=0.01}
    \end{subfigure}
    \caption{\blue{\ref{subfig: AC graphs of ut}\&\ref{subfig: AC graphs of ut eps=0.01}: The graph of $u_{\theta_k}(\cdot)$ (blue) obtained from the NPDG algorithm for $t_k = \frac{k}{N_t}$, $1\leq k \leq N_t$ together with the benchmark solution (red, dashed line).
    \ref{subfig: AC semi-log MSE vs comptime}: Semi-log plots of the $\sqrt{\textrm{MSE loss}}$ vs. computation time (seconds) at physical time $0.1, 0.3, 0.5, 1.0.$ \ref{subfig: AC energy decay eps=0.01}: Plot of energy $E(u_{\theta_k})$ versus time $t_k$.}}
    \label{fig: AC1}
\end{figure}

\vspace{0.5cm}

\noindent
\textbf{2D example} We further consider a 2D Allen-Cahn equation with $\Omega = [0, 2]^2$, $\epsilon_0=0.1$ and the initial condition $u_0(x) = \tanh\left( -\frac{ \|x-x_0\|-R }{s} \right)$ with $x_0 = (1, 1)^\top$, $R=0.5$ and $s = 0.1$. We set $T=1.5$ and $N_t=15$. We keep the hyperparameters of the NPDG algorithm the same as the previous example except that we set $N_{iter}=1000.$ The numerical results are provided in Appendix \ref{append: benchmark}.

\subsection{ Monge-Amp\`{e}re equation for the $L^2$-Optimal Transport problem  }\label{sec: MA}
In this section, we focus on the computation of the Monge-Ampère equation \eqref{insection Monge Ampere}. A PINN solver for this equation is proposed in \citep{singh2021physics}. Deep learning algorithms from the optimal transport perspective are discussed in \citep{korotin2019wasserstein, makkuva2020optimal, fan2023neural}, among other references.

As discussed in Section \ref{subsec: discuss on MA}, solving the equation is equivalent to solving the $L^2-$optimal transport problem. This can be further reduced to a sup-inf saddle point problem \eqref{saddle Monge with map T}. In this research, we assume that the samples of $\mu_0, \mu_1$ are available. In order to evaluate the functional $\mathscr{E}(T_\theta, \varphi_\eta),$ we generate samples $\{\boldsymbol{X}_i\}_{i=1}^N\sim\mu_0=\rho_0 \dd x$ and $\{\boldsymbol{Y}_i\}_{i=1}^N\sim \mu_1=\rho_1 \dd y$ and apply the Monte-Carlo algorithm,
\[ \mathscr{E}(T_\theta, \varphi_\eta) \approx \frac{1}{N} \sum_{i=1}^N \frac12 \|\boldsymbol{X}_i - T_\theta(\boldsymbol{X}_i)\|^2 + \varphi_\eta(T_\theta(\boldsymbol{X}_i)) - \varphi_\eta(\boldsymbol{Y}_i). \]
By applying Algorithm \ref{alg: mat-vec mult precond}, we calculate the natural (preconditioned) gradients of $\mathscr E(T_\theta, \varphi_\eta)$ with respect to $\theta, \eta$. We then apply the NPDG algorithm \ref{alg: NPDG} to solve the saddle point problem \eqref{saddle Monge with map T} for $T_*(\cdot)$ ($\nabla u(\cdot)$). 

In experiments, we use the Primal-Dual algorithm with the Adam optimizer ({PD-Adam}) proposed in \citep{fan2023neural} as a benchmark for the proposed method. A brief description of this method, as well as its hyperparameters used in all tests, are provided in Appendix \ref{append: PD Adam OT}. We test three numerical examples as a demonstration. The first two examples possess explicit formulas for the OT maps. In the third example, we compute the OT map from standard Gaussian to mixed Gaussian distributions embedded in 10D and 50D spaces. 
In the implementation, we set $T_\theta(\cdot)$, $\varphi_\eta(\cdot)$ as MLP with $\textrm{PReLU}$ activation function 
\begin{equation*}
  \textrm{PReLU}(x) = \begin{cases}
        x,  \quad \textrm{if  } x \geq 0\\
        ax, \quad \textrm{otherwise,}
  \end{cases}
\end{equation*}
where $a\in \mathbb R$ is a learnable parameter. 
The Input Convex Neural Networks (ICNN) architecture \citep{amos2017input} advocated in \citep{makkuva2020optimal} will be considered in future research.

\subsubsection{1D Gaussian to mixed Gaussian}\label{sec: MA1D}
We set $\rho_0 = \mathcal N(0, 1)$, $\rho_1 = \sum_{k=1}^m \lambda_k \mathcal N(\mu_k, \sigma_k^2)$ with $\lambda_k>0$, $\sum_{k=1}^m  \lambda_k = 1$, $\mu_k\in\mathbb{R}$, $\sigma_k > 0$. The optimal transport map takes the explicit form,
\[ T_*(x) = F_1^{-1}(F_0(x)), \quad F_0(x) = \sum_{ k = 1 }^m \frac{\lambda_k}{2}(1 + \mathrm{erf}\left(\frac{x-\mu_k}{\sqrt 2 \sigma_k}\right)), \quad F_1^{-1}(y) = \mathrm{erf}^{-1}(2y-1).  \]
In the example, we consider $m=2$, $\lambda_1 = \frac23, \mu_1 = -1, \sigma_1 = 0.5$; $\lambda_2 = \frac13, \mu_2 = 1, \sigma_2 = 0.5$. We set $T_\theta(\cdot)$ and $\varphi_\eta$ as
\begin{equation*}
  T_\theta = \texttt{MLP}_{\textrm{PReLU}}(1, 50, 1, 3), ~ \varphi_\eta = \texttt{MLP}_{\textrm{PReLU}}(1, 50, 1, 3).
\end{equation*}
We set the sample size $N = 800$, $\omega = 1$, and $\tau_u = \tau_\varphi = 1.5\cdot 10^{-1}$. We perform the NPDG algorithm for $6000$ iterations. Figure \ref{subfig: Monge Ampere 1D semi-log plot sqrt MSE} demonstrates the semi-log plots of the $L^2(\rho_0)$ error $\|T_\theta - T_*\|_{L^2(\rho_0)}$ versus the computation time. We make comparisons among the NPDG algorithms with different preconditioners (\eqref{pull Id as precond Monge problem} and \eqref{canonical precond for Monge problem}), as well as the PD-Adam method.

\subsubsection{5D Gaussian to Gaussian}\label{sec: MA5D}
For $\mu_0, \mu_1\in\mathbb R^5$ and positive-definite symmetric matrices $\Sigma_0, \Sigma_1\in\mathbb R^{5\times 5}$, we set $\rho_0=\mathcal N(\mu_0, \Sigma_0)$, $\rho_1=\mathcal N(\mu_1, \Sigma_1)$.
One can verify that the OT map takes the affine form $T_*(\textbf{x}) = A\textbf{x}+b$ with
\[ A = \sqrt{\Sigma_0}^{-1} (\sqrt{\Sigma_0} \Sigma_1 \sqrt{\Sigma_0})^{1/2} \sqrt{\Sigma_0}^{-1}, \quad b = \mu_1 - A\mu_0. \]

For simplicity, we set $\mu_0=\mu_1=0$ in the test example. The cases in which $\mu_0\neq \mu_1$ can be readily handled by the pre-translating technique introduced in \citep{kuang2017preconditioning}, which reduces the problem to the case in which $\mu_0=\mu_1$. We define 
\[ \Sigma_0 = \mathrm{diag}(1/4,1,1,1), \quad \Sigma_1 = \mathrm{diag}(1,1/4,1)\oplus \left[\begin{array}{cc}
5/8 & 3/8 \\
3/8 & 5/8 
\end{array} \right].
\]
Then the OT map is given by $T_*(x) = \sqrt{\Sigma_0^{-1}\Sigma_1} x,$ with 
\[ 
    \sqrt{\Sigma_0^{-1}\Sigma_1} 
     = \mathrm{diag}(2,1/2,1) 
    \oplus\left[\begin{array}{cc}
          3/4 & 1/4 \\
          1/4 & 3/4
    \end{array}\right].
\]
We set $T_\theta(\cdot)$ and $\varphi_\eta$ as
\begin{equation*}
  T_\theta = \texttt{MLP}_{\textrm{PReLU}} (5, 80, 5, 4),  ~  \varphi_\eta = \texttt{MLP}_{\textrm{PReLU}}(5, 80, 1, 4).
\end{equation*}
We set the sample size $N = 2000$, $\omega = 1$, and $\tau_u = 0.5\cdot 10^{-1}, \tau_\varphi = 0.95\cdot 10^{-1}$. We perform the NPDG algorithm for $20000$ iterations. Similar to the previous example, we present the semi-log plots of $L^2(\rho_0)$ error vs computation time in Figure \eqref{subfig: Monge Ampere 5D semi-log plot sqrt MSE}. The plots of the computed transportation map $T_\theta(\cdot)$ together with $T_*(\cdot)$ are provided in Figure \eqref{subfig: Monge Ampere 5D plot map 12} and \eqref{subfig: Monge Ampere 5D plot map 45}.
\begin{figure}
    \centering
    \begin{subfigure}[b]{0.24\textwidth}
        \includegraphics[width=\textwidth]{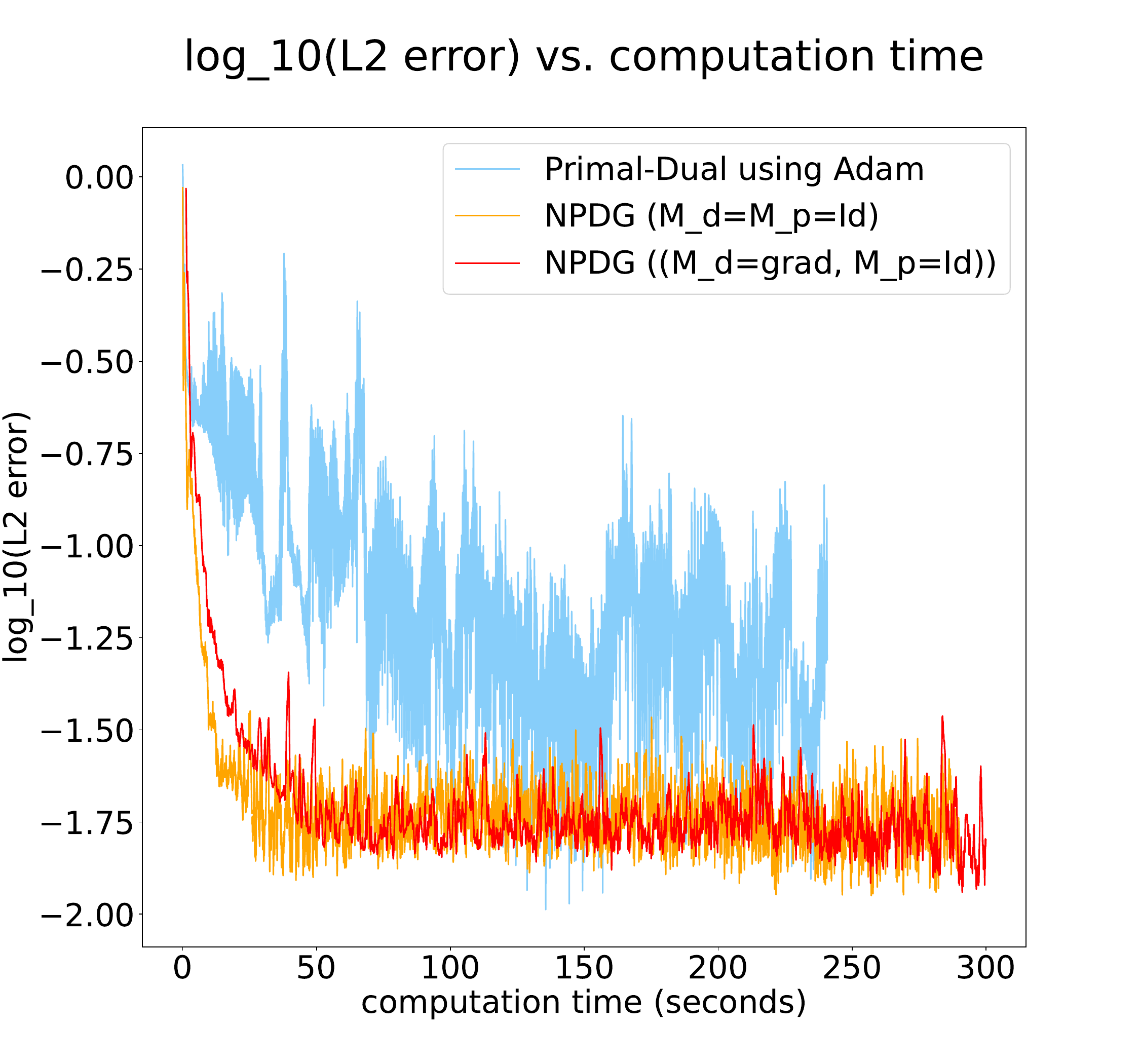}
        \caption{}\label{subfig: Monge Ampere 1D semi-log plot sqrt MSE}
    \end{subfigure}
    \begin{subfigure}[b]{0.23\textwidth}
        \includegraphics[width=\textwidth]{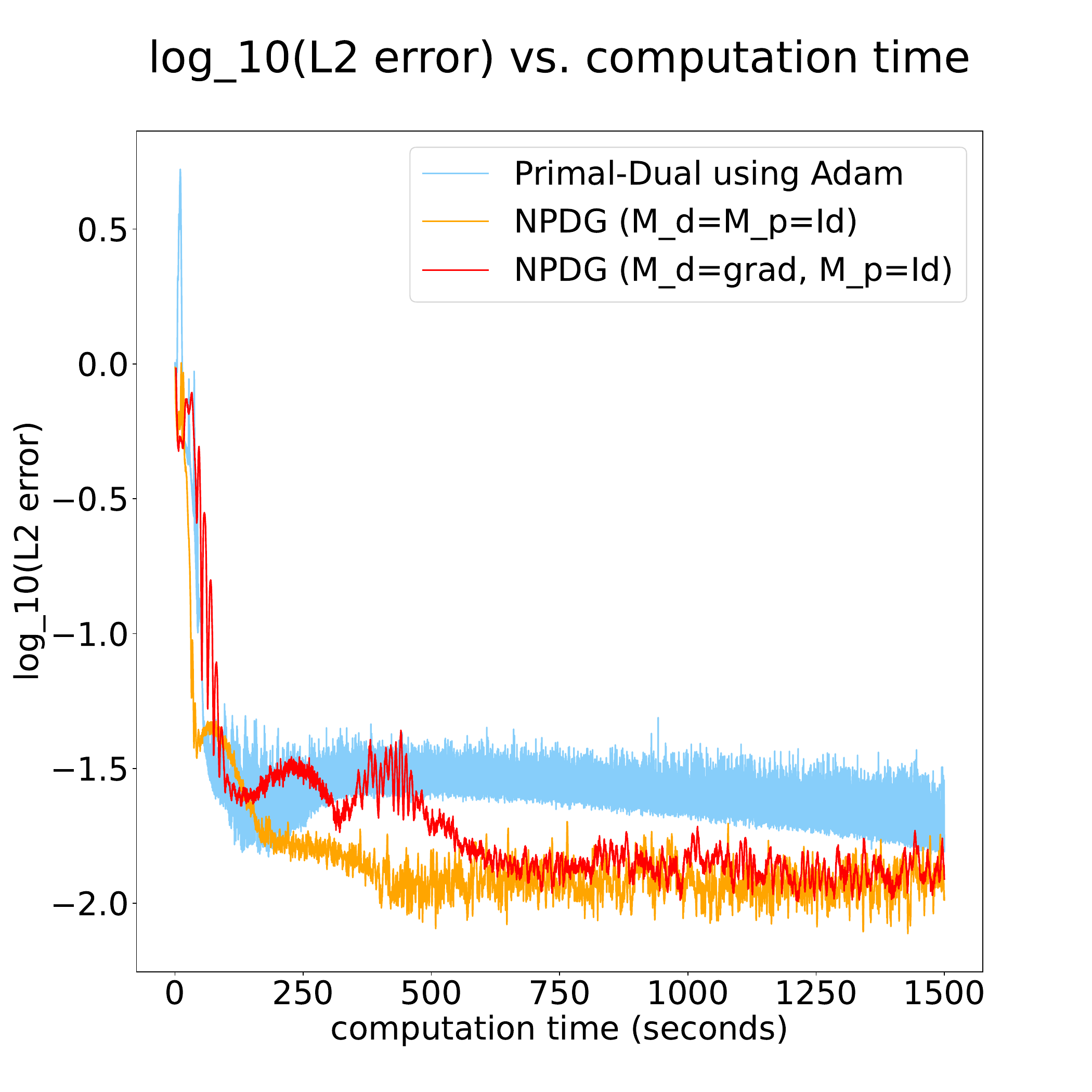}
        \caption{}\label{subfig: Monge Ampere 5D semi-log plot sqrt MSE}
    \end{subfigure}
    \begin{subfigure}[b]{0.23\textwidth}
        \includegraphics[width=\textwidth]{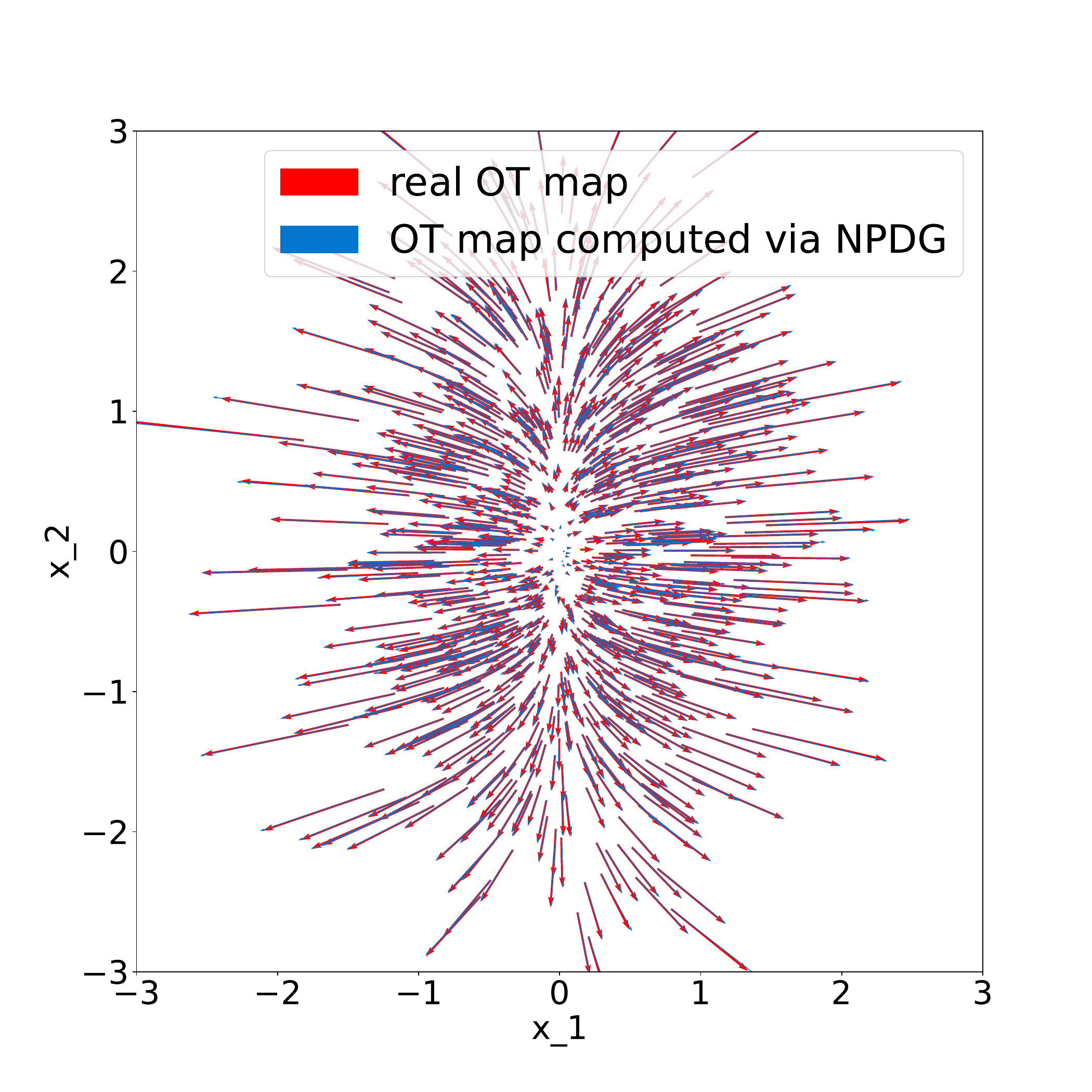}
        \caption{}\label{subfig: Monge Ampere 5D plot map 12}
    \end{subfigure}
    \begin{subfigure}[b]{0.23\textwidth}
        \includegraphics[width=\textwidth]{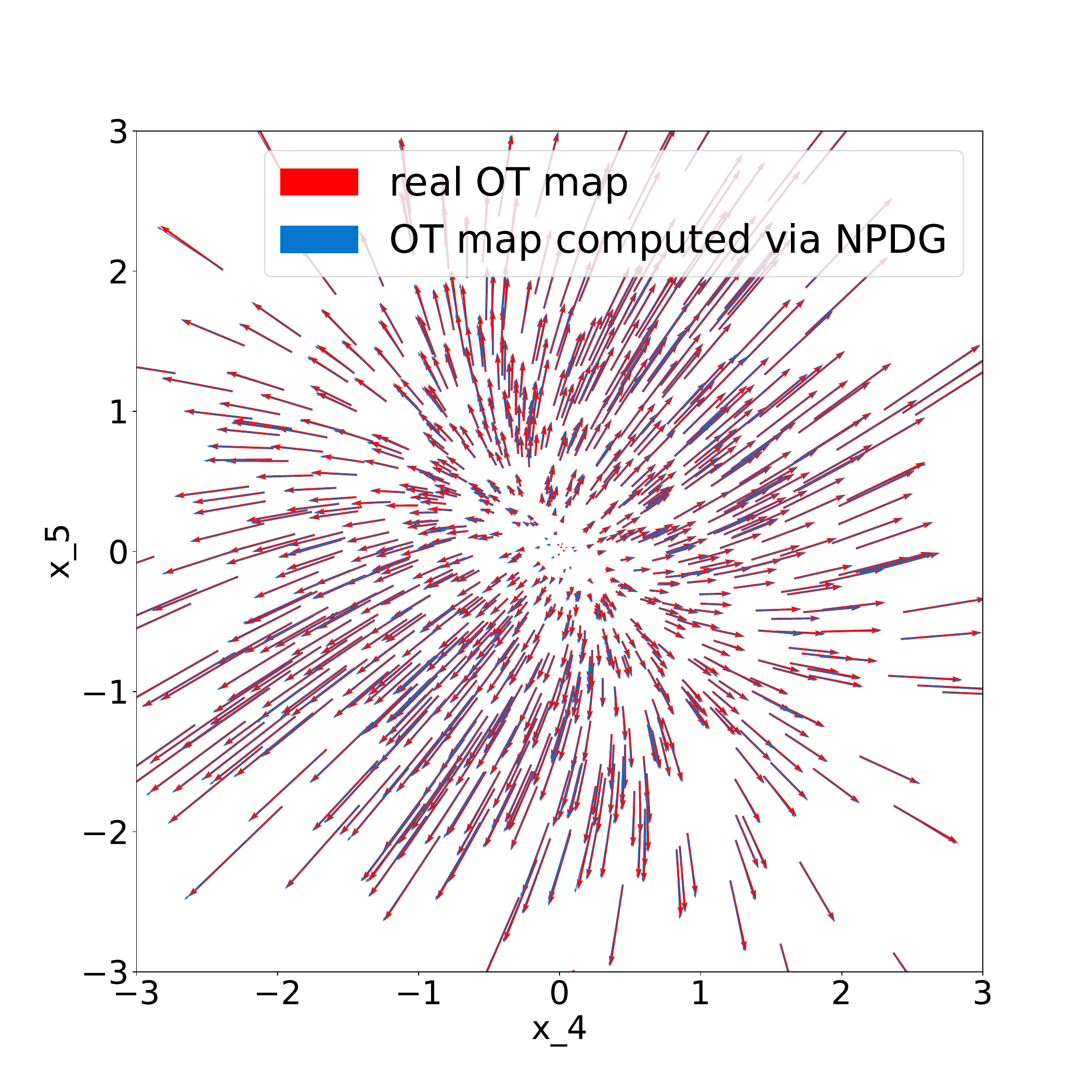}
        \caption{}\label{subfig: Monge Ampere 5D plot map 45}
    \end{subfigure}
    \caption{\textbf{OT problem (1D, 5D)}: \ref{subfig: Monge Ampere 1D semi-log plot sqrt MSE}: Semi-log plots of $\|T_\theta-T_*\|_{L^2(\rho_0)}$ vs computation time (seconds) for the 1D problem discussed in Section \ref{sec: MA1D}; \ref{subfig: Monge Ampere 5D semi-log plot sqrt MSE}: Semi-log plots of $\|T_\theta-T_*\|_{L^2(\rho_0)}$ vs computation time (seconds) for the 5D problem discussed in Section \ref{sec: MA5D};   \ref{subfig: Monge Ampere 5D plot map 12}: Plot of the computed transport map $T_\theta(\cdot)$ (blue) with real OT map $T_*(\cdot)$ (red) on 1-2 plane; \ref{subfig: Monge Ampere 5D plot map 45}: Plot of the computed transport map (blue) with real OT map (red) on 4-5 plane. }
    \label{fig: MA1}
\end{figure}

\subsubsection{High dimensional Gaussian to mixed Gaussian (10D, 50D)}\label{sec: MA10D50D}
We consider the mixed-Gaussian distribution $\sum_{k=1}^8 \lambda_k \mathcal N(\mu_k, \sigma_k^2I)$ defined on $\mathbb{R}^{d}$, where
  \[  \mu_k =   \left( 0, \dots, R\cos\left(\frac{k}{4}\pi\right), \dots, 
  R\sin\left(\frac{k}{4}\pi\right), \dots, 0  \right)^\top \textrm{ with } R=3 , \quad \sigma_k = \frac{4}{25}.
  \]
We assume that the two nonzero entries of $\mu_k$ are located in the $i_0$ and $i_1$ entries. We denote $\rho_a$ as equal mixed-Gaussian
\begin{equation}
  \rho_a = \sum_{k=1}^8 \lambda_k \mathcal N(\mu_k, \sigma_k^2I), \quad \lambda_k=\frac{1}{8},  ~ 1\leq k \leq 8;  \label{equal mixed gaussian lambda}
\end{equation}
we denote $\rho_b$ as a non-equally distributed mixed-Gaussian distribution with
\begin{equation*}  
  \rho_b = \sum_{k=1}^8 \lambda_k \mathcal N(\mu_k, \sigma_k^2I), \quad \lambda_k = \begin{cases}
               \frac{1}{5} \quad k \textrm{ is even,}\\
               \frac{1}{20} \quad k \textrm{ is odd.}
             \end{cases}, ~ 1\leq k \leq 8.
\end{equation*}

Consider $\rho_0=\mathcal N(0, I)$. We compute the optimal transport from $\rho_0$ to $\rho_a$, as well as $\rho_0$ to $\rho_b$, by solving the sup-inf problem \eqref{saddle Monge with map T} using the NPDG algorithm. In the implementation, we always set 
\begin{equation*}
  u_\theta(\cdot) = \texttt{MLP}_{\textrm{PReLU}}(d, 120, d, 6), ~ \varphi_\eta(\cdot) = \texttt{MLP}_{\textrm{PReLU}}(d, 120, 1, 6).
\end{equation*}
We first test the algorithm by setting $d=10$, and $i_0=4, i_1=8$. We choose \eqref{pull Id as precond Monge problem} as preconditioners for NPDG algorithm. we choose the sample size $N = 2000$, $\omega = 1$, and $\tau_u = 0.5\cdot 10^{-2}, \tau_\varphi = 0.95\cdot 10^{-2}$; we perform the NPDG algorithm for $15000$ iterations. We compute the optimal transport maps from $\rho_0$ to $\rho_a$ and $\rho_0$ to $\rho_b$ by applying the NPDG algorithm and the PD-Adam method. We compare the computational results in Figure \ref{fig: MA10D}. The pushforwarded distribution $T_{\theta  \sharp} \rho_0 $ of the proposed method outperforms PD-Adam in terms of homogeneity and shape of the mixed Gaussians.
\begin{figure}[htb!]
    \centering
    \begin{subfigure}[b]{0.24\textwidth}
         \centering
         \includegraphics[width=\textwidth, valign=t]{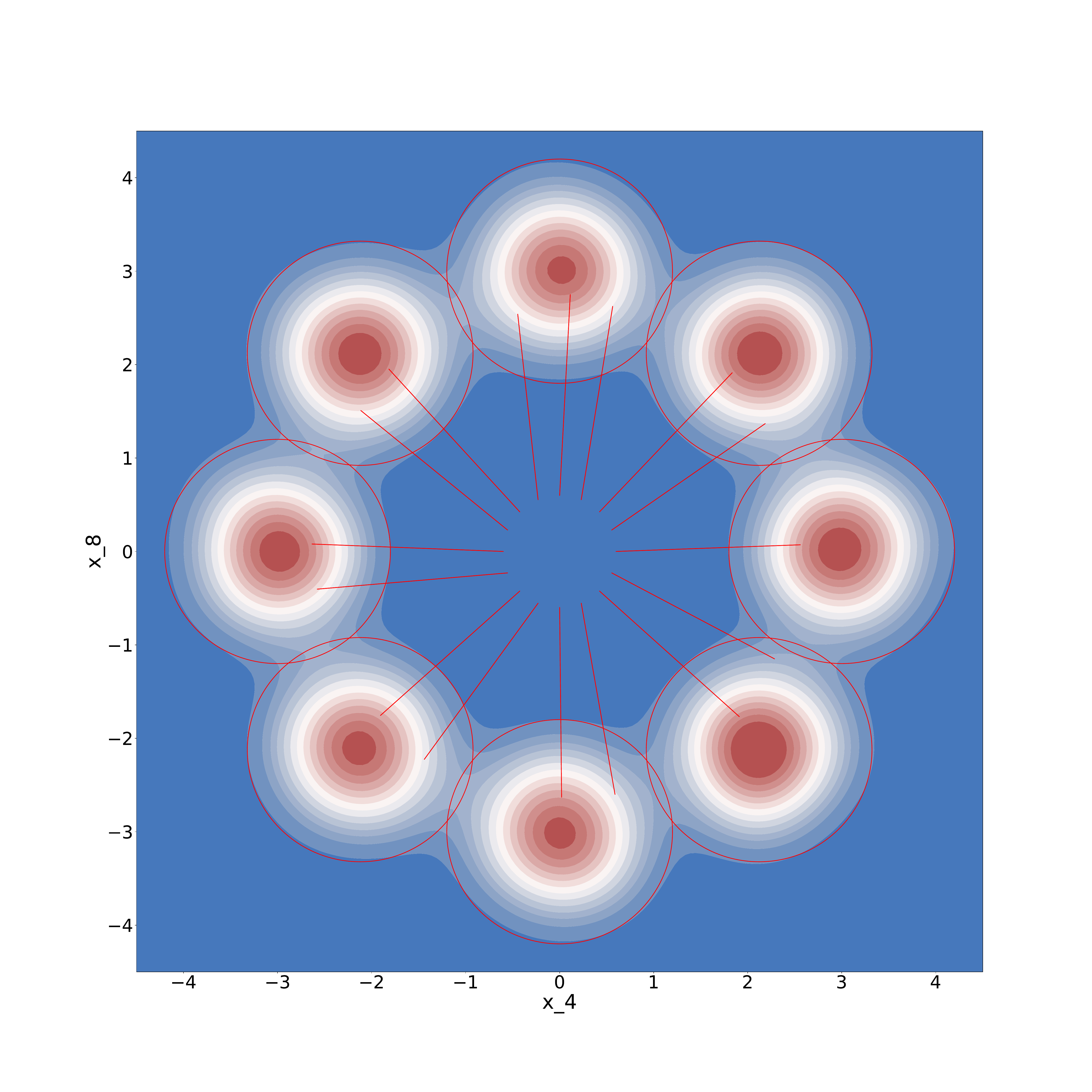}
         \caption{NPDG}\label{subfig: MA10D NPDG equal}
    \end{subfigure}
    \hfill
    \begin{subfigure}[b]{0.24\textwidth}
         \centering
         \includegraphics[width=\textwidth, valign=t]{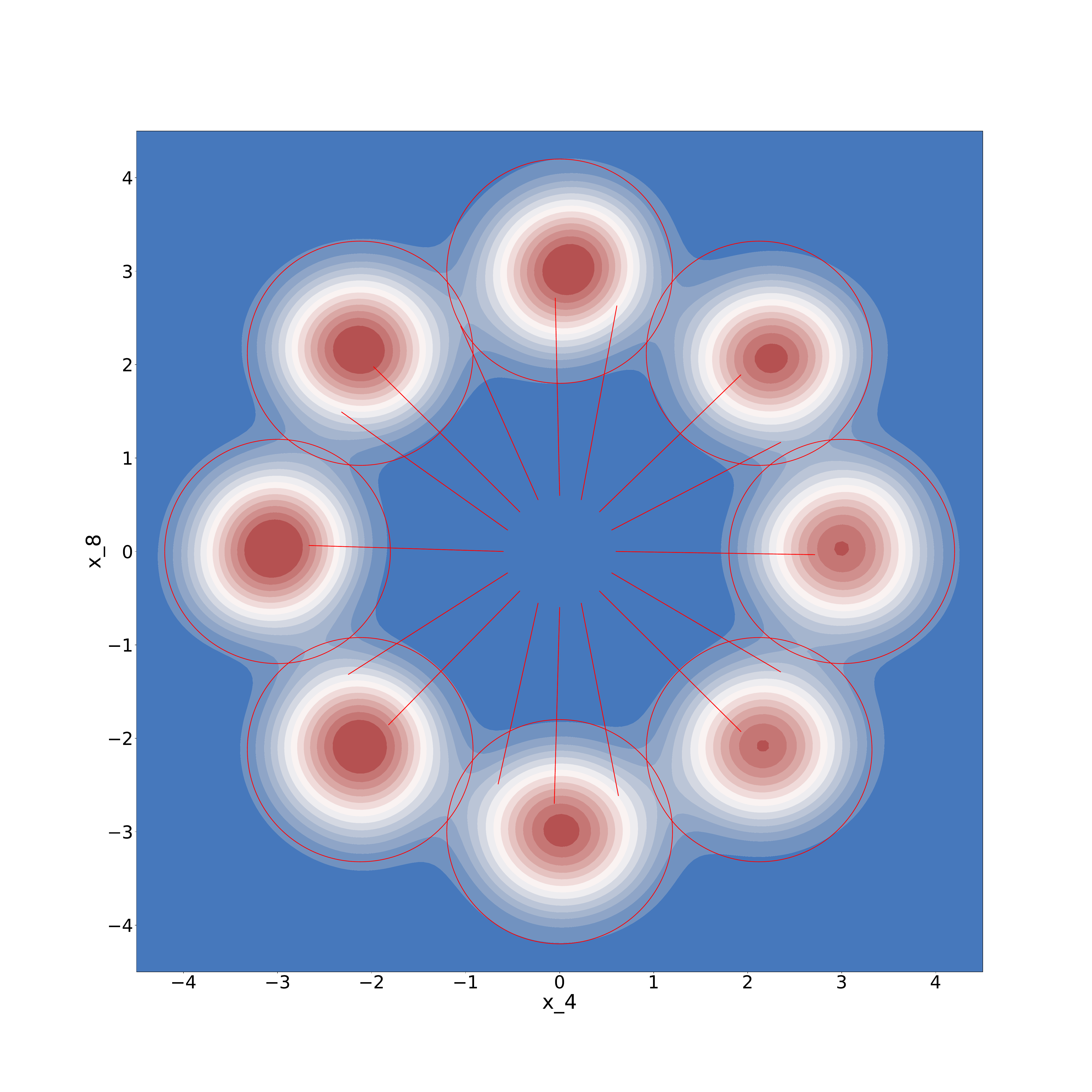}
         \caption{PD-Adam}\label{subfig: MA10D PD adam equal}
    \end{subfigure}
             \centering
         \begin{subfigure}[b]{0.24\textwidth}
         \centering
         \includegraphics[width=\textwidth, valign=t]{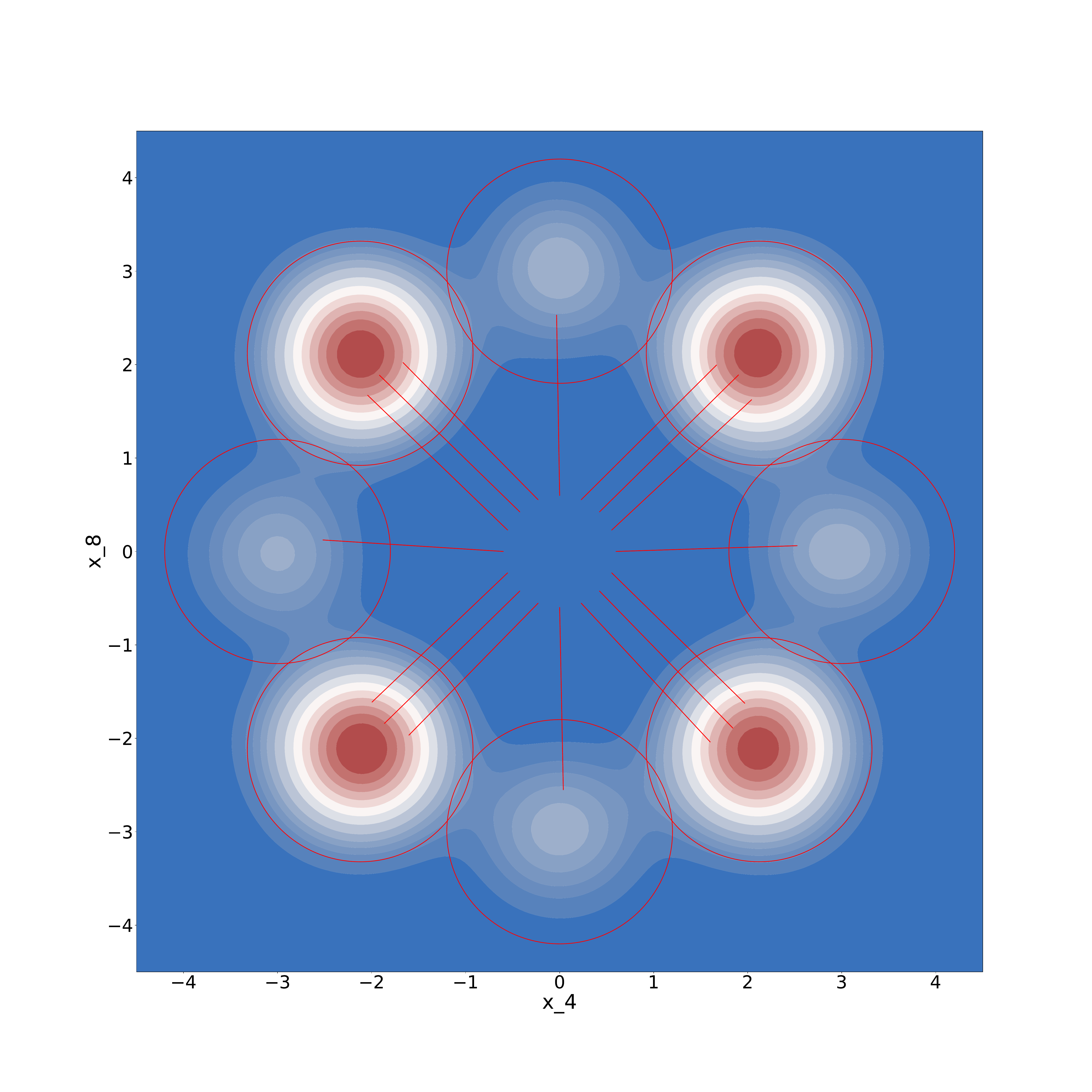}
         \caption{NPDG}\label{subfig: MA10D NPDG nonequal}
     \end{subfigure}
     \hfill
     \begin{subfigure}[b]{0.24\textwidth}
         \centering
         \includegraphics[width=\textwidth, valign=t]{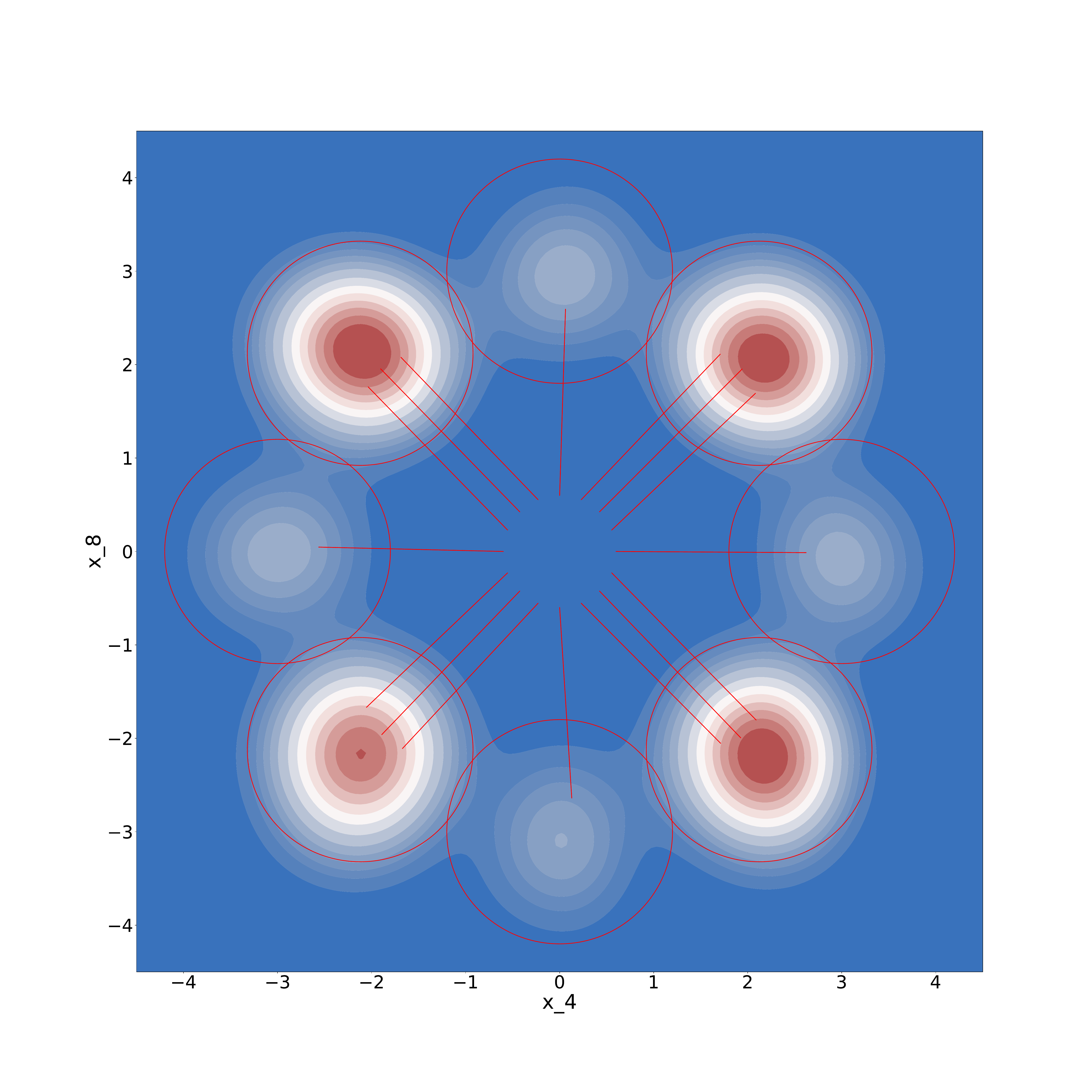}
         \caption{PD-Adam}\label{subfig: MA10D PD adam nonequal}
     \end{subfigure}
    \caption{\textbf{OT problem (10D)}: Plots of the pushforwarded density $T_{\theta\sharp}\rho_0$ by using Kernel Density Estimation (KDE), together with the optimal transport map (red segments). \textbf{Left two figures}: OT from $\rho_0$ to $\rho_a$, \ref{subfig: MA10D NPDG equal}: Numerical result obtained by NPDG, \ref{subfig: MA10D PD adam equal} Numerical result obtained by PD-Adam; \textbf{Right two figures}: OT from $\rho_0$ to $\rho_b$, \ref{subfig: MA10D NPDG nonequal}: Numerical result obtained by NPDG, \ref{subfig: MA10D PD adam nonequal}: Numerical result obtained by PD-Adam. All figures are plotted on the $4-8$ plane.}\label{fig: MA10D}
\end{figure}

We further consider the OT problem with dimension $d = 50$ with $i_0=10, i_1=20$ in which the NPDG algorithm performs more robustly and achieves more accurate solutions compared to the PD-Adam algorithm. We set $tol_{\textrm{MINRES}}=10^{-4}$. We choose the sample size $N = 2000$, the extrapolation coefficient $\omega=5$ and stepsizes $\tau_u = \tau_\varphi = 0.5 \cdot 10^{-2}$. We perform the NPDG algorithm for $20000$ iterations.

We first test the case of transporting $\rho_0$ to equally distributed mixed-Gaussian distribution $\rho_a$. We test the NPDG algorithm with various preconditioning \eqref{pull Id as precond Monge problem}, \eqref{canonical precond for Monge problem}, as well as the PD-Adam method. The results are presented in Figure \ref{fig: MA50D equal gaussians}. It is worth mentioning that upon comparing the transport maps shown in Figure \ref{subfig: MA50D equal gaussian NPDG Id} and \ref{subfig: MA50D equal gaussian NPDG grad}, the more canonical precondition \eqref{canonical precond for Monge problem} yields a solution with higher accuracy. We then test the case of transporting $\rho_0$ to non-equal mixed-Gaussian distribution $\rho_b$. The results are presented in Figure \ref{fig: MA50D nonequal gaussians}. Again, our NPDG algorithm with precondition \eqref{canonical precond for Monge problem} produces the transport map with better quality. Further plots on the numerical solutions can be found in Appendix \ref{append: MA}. The PD-Adam method does not behave as robustly as the NPDG algorithm in this 50D example.  
\begin{figure}[htb!]
     \centering
     \begin{subfigure}{0.24\textwidth}
         \centering
         \includegraphics[width=\textwidth, valign=t]{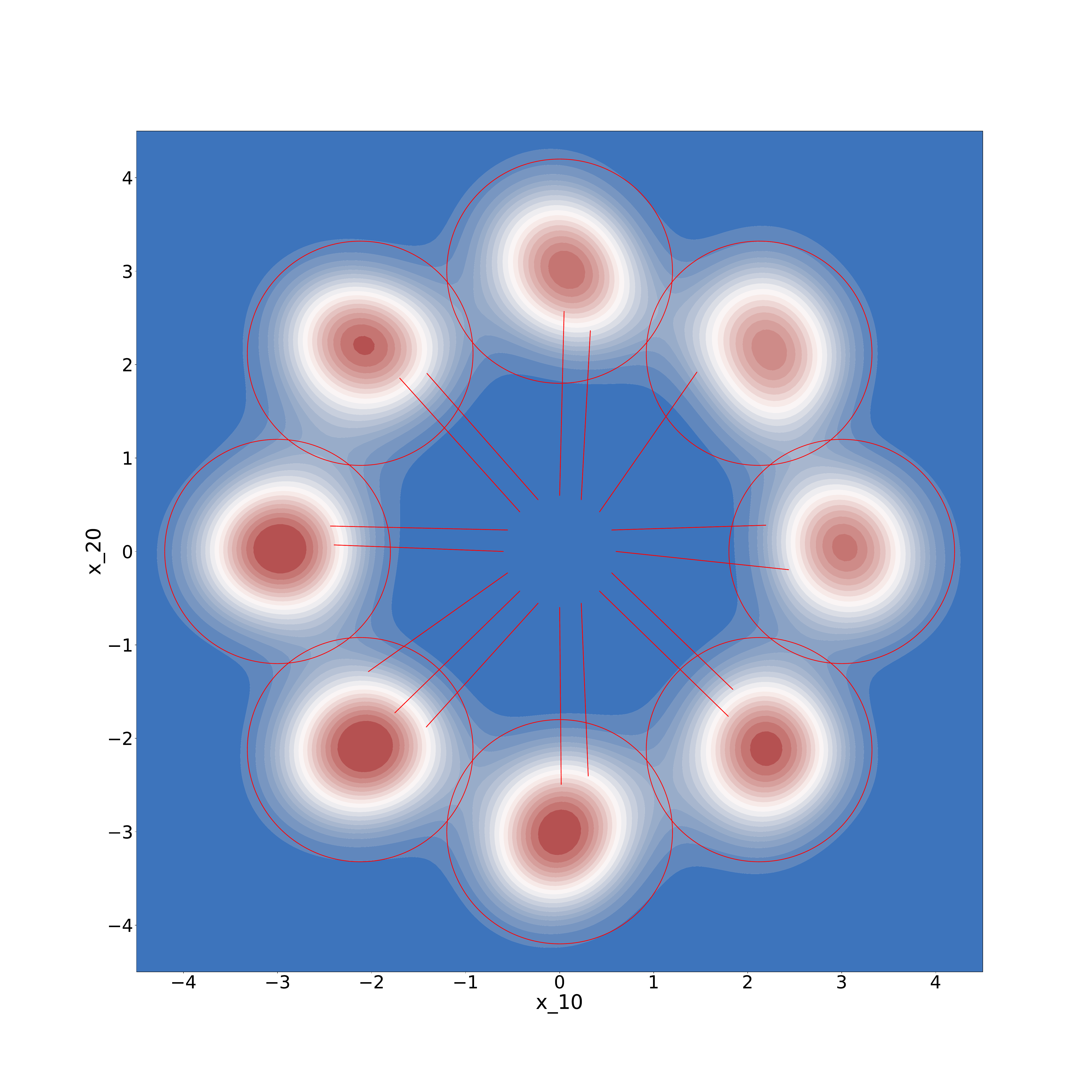}
         \caption{NPDG with \eqref{pull Id as precond Monge problem}}\label{subfig: MA50D equal gaussian NPDG Id}
     \end{subfigure}
     \hfill
     \begin{subfigure}{0.24\textwidth}
         \centering
         \includegraphics[width=\textwidth, valign=t]{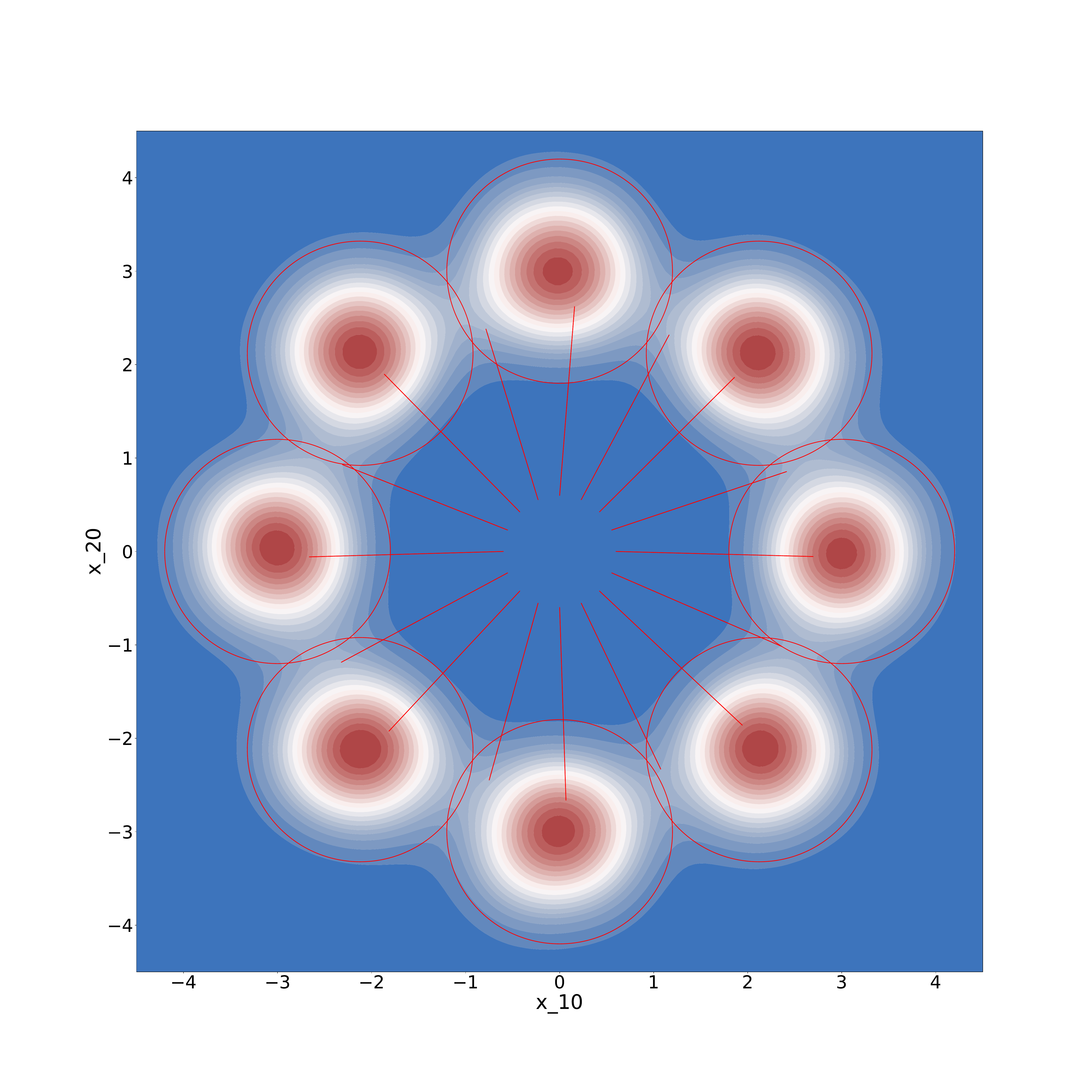}
         \caption{NPDG with \eqref{canonical precond for Monge problem}}\label{subfig: MA50D equal gaussian NPDG grad}
     \end{subfigure}
     \hfill
     \begin{subfigure}{0.24\textwidth}
         \centering
         \includegraphics[width=\textwidth, valign=t]{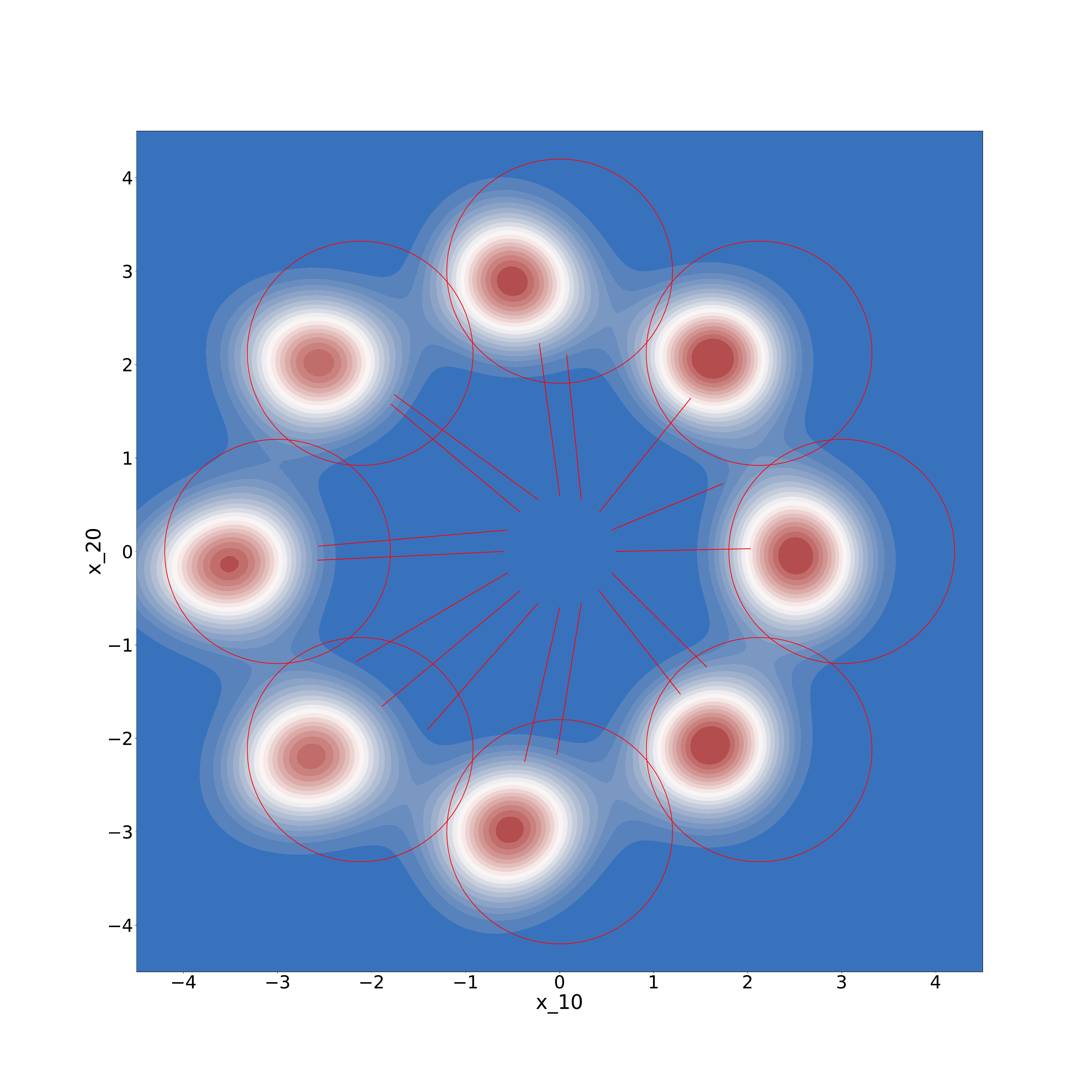}
         \caption{PD-Adam}\label{subfig: MA50D equal gaussian PD adam}
     \end{subfigure}
     \hfill
     \begin{subfigure}{0.24\textwidth}
         \centering
         \includegraphics[width=\textwidth, valign=t]{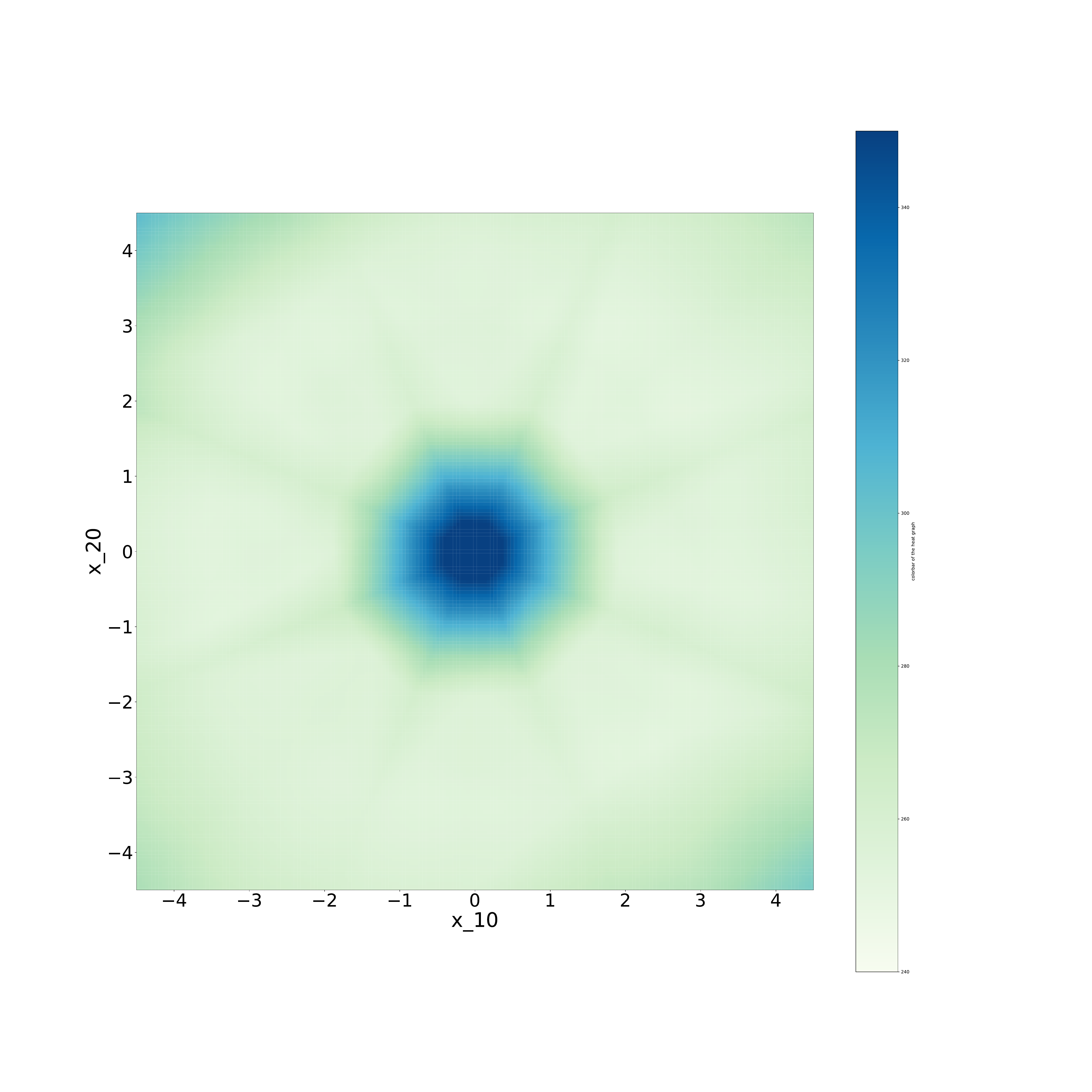}
         \caption{Heat graph of $\varphi_\eta(\cdot)$}\label{subfig: MA50D equal gaussian plot varphi}
     \end{subfigure}
     \caption{\textbf{OT problem from $\rho_0$ to $\rho_a$ (50D)}: Plots of the pushforwarded density $T_{\theta\sharp}\rho_0$ by using Kernel Density Estimation (KDE). \ref{subfig: MA50D equal gaussian NPDG Id}-\ref{subfig: MA50D equal gaussian PD adam}: Numerical results produced by NPDG method and PD-Adam method. \ref{subfig: MA50D equal gaussian plot varphi}: heat graph of the Kantorovich dual function $\varphi_\eta(\cdot)$ learned from NPDG algorithm with precondition \eqref{canonical precond for Monge problem}. All figures are plotted on the $10-20$ coordinate plane.} \label{fig: MA50D equal gaussians}
\end{figure}

\begin{figure}[htb!]
         \centering
         \begin{subfigure}{0.24\textwidth}
         \centering
         \includegraphics[width=\textwidth, valign=t]{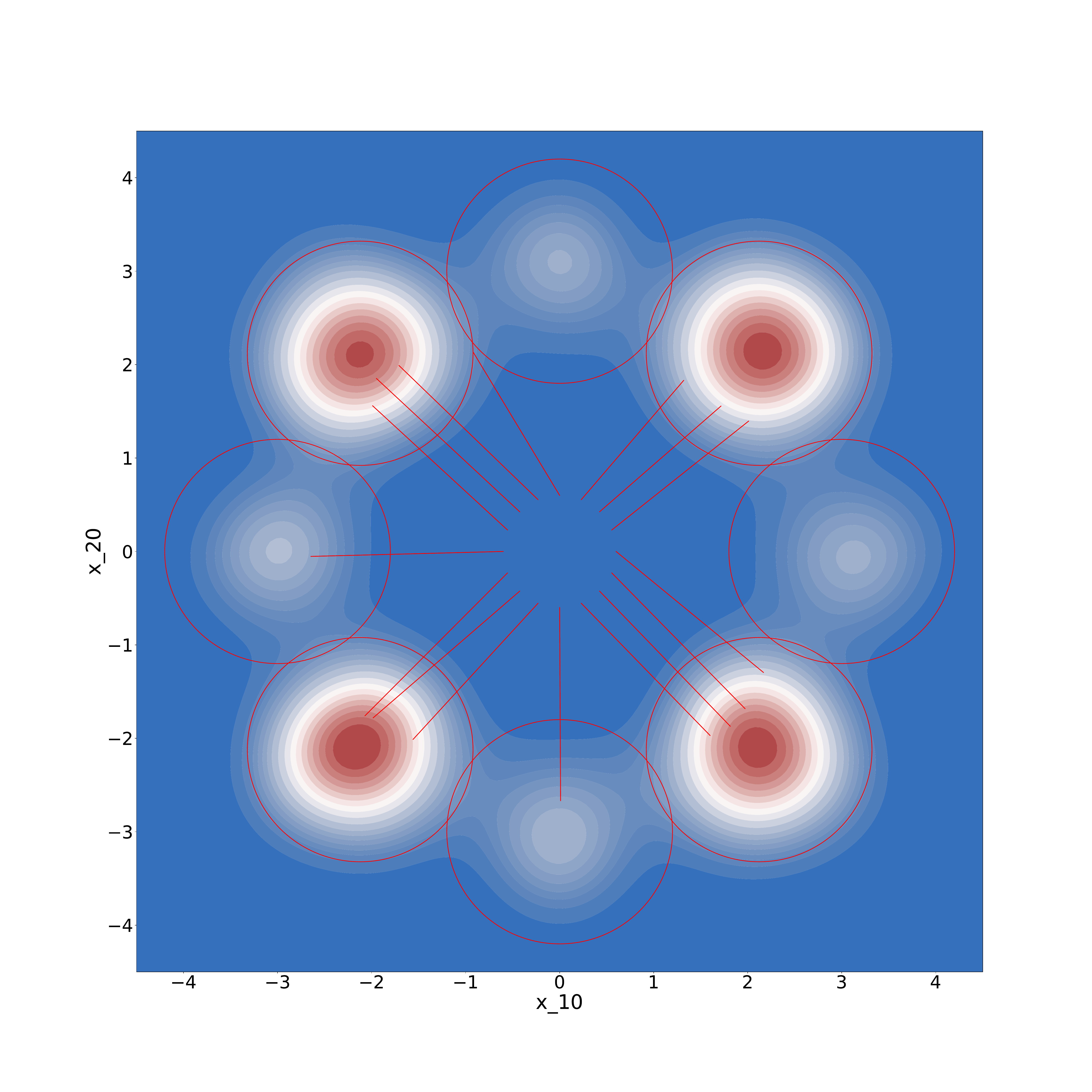}
         \caption{NPDG with \eqref{pull Id as precond Monge problem}}\label{subfig: MA50D nonequal gaussian NPDG Id}
     \end{subfigure}
     \hfill
     \begin{subfigure}{0.24\textwidth}
         \centering
         \includegraphics[width=\textwidth, valign=t]{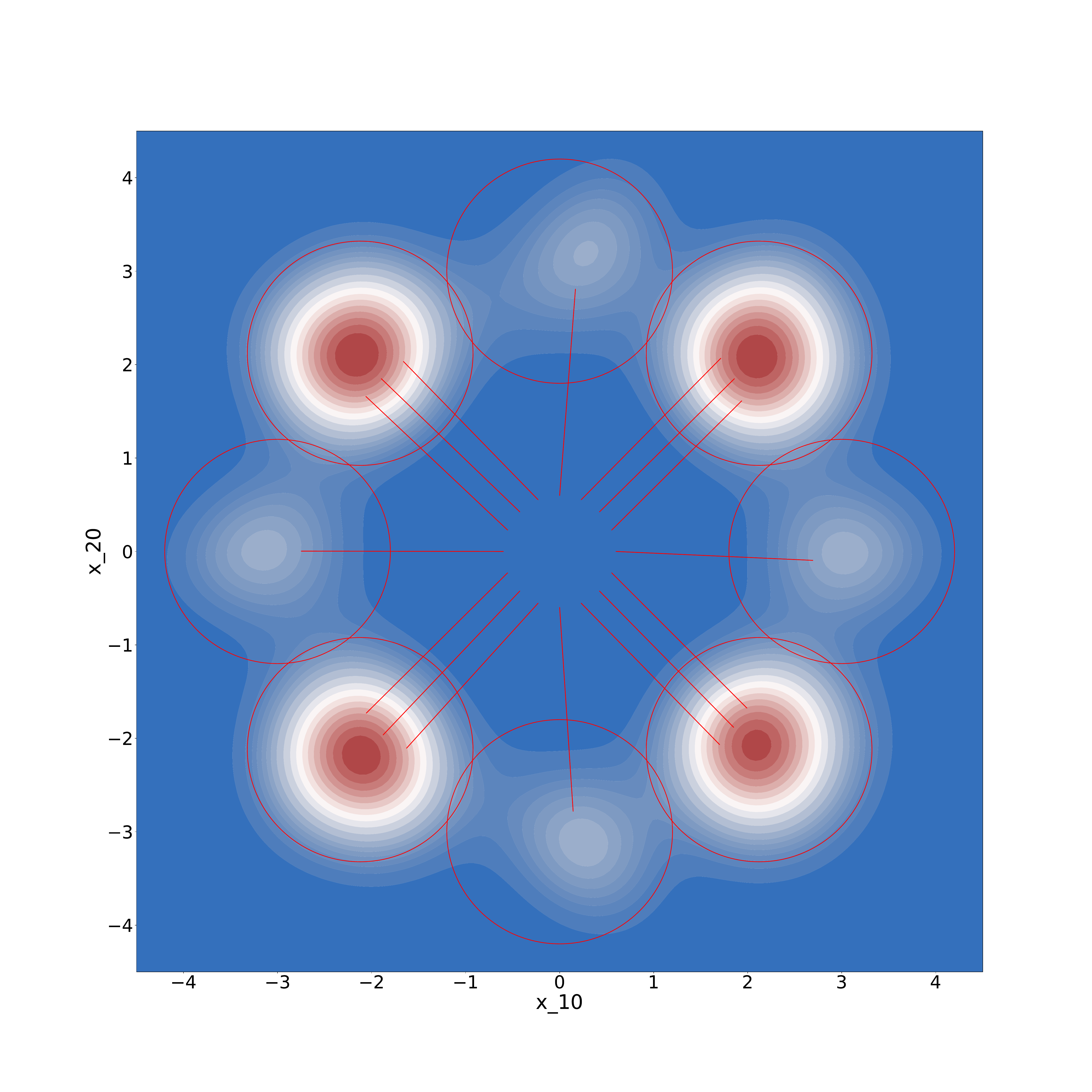}
         \caption{NPDG with \eqref{canonical precond for Monge problem}}\label{subfig: MA50D nonequal gaussian NPDG grad}
     \end{subfigure}
     \hfill
     \begin{subfigure}{0.24\textwidth}
         \centering
         \includegraphics[width=\textwidth, valign=t]{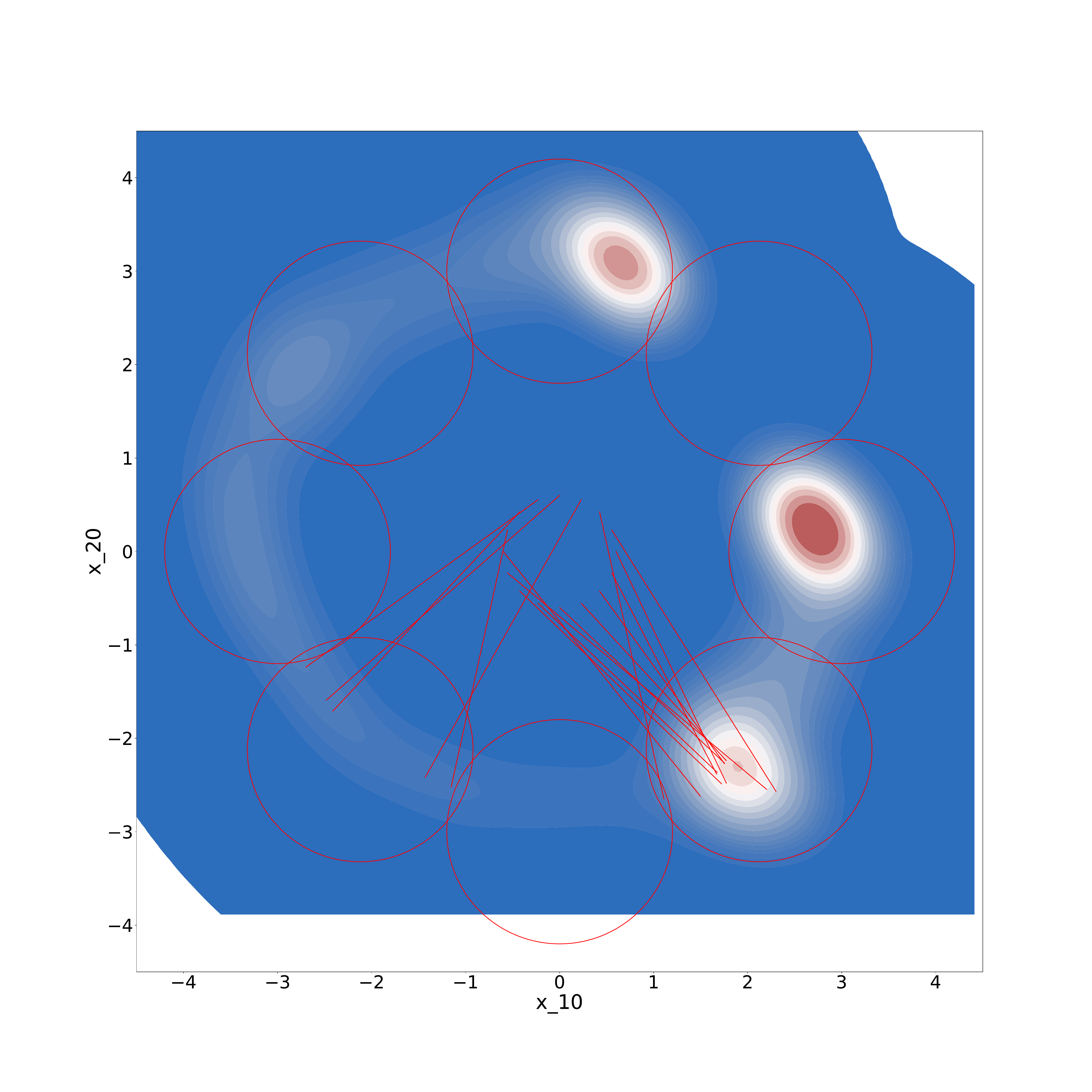}
         \caption{PD-Adam}\label{subfig: MA50D nonequal gaussian PD adam}
     \end{subfigure}
     \begin{subfigure}{0.24\textwidth}
         \centering
         \includegraphics[width=\textwidth, valign=t]{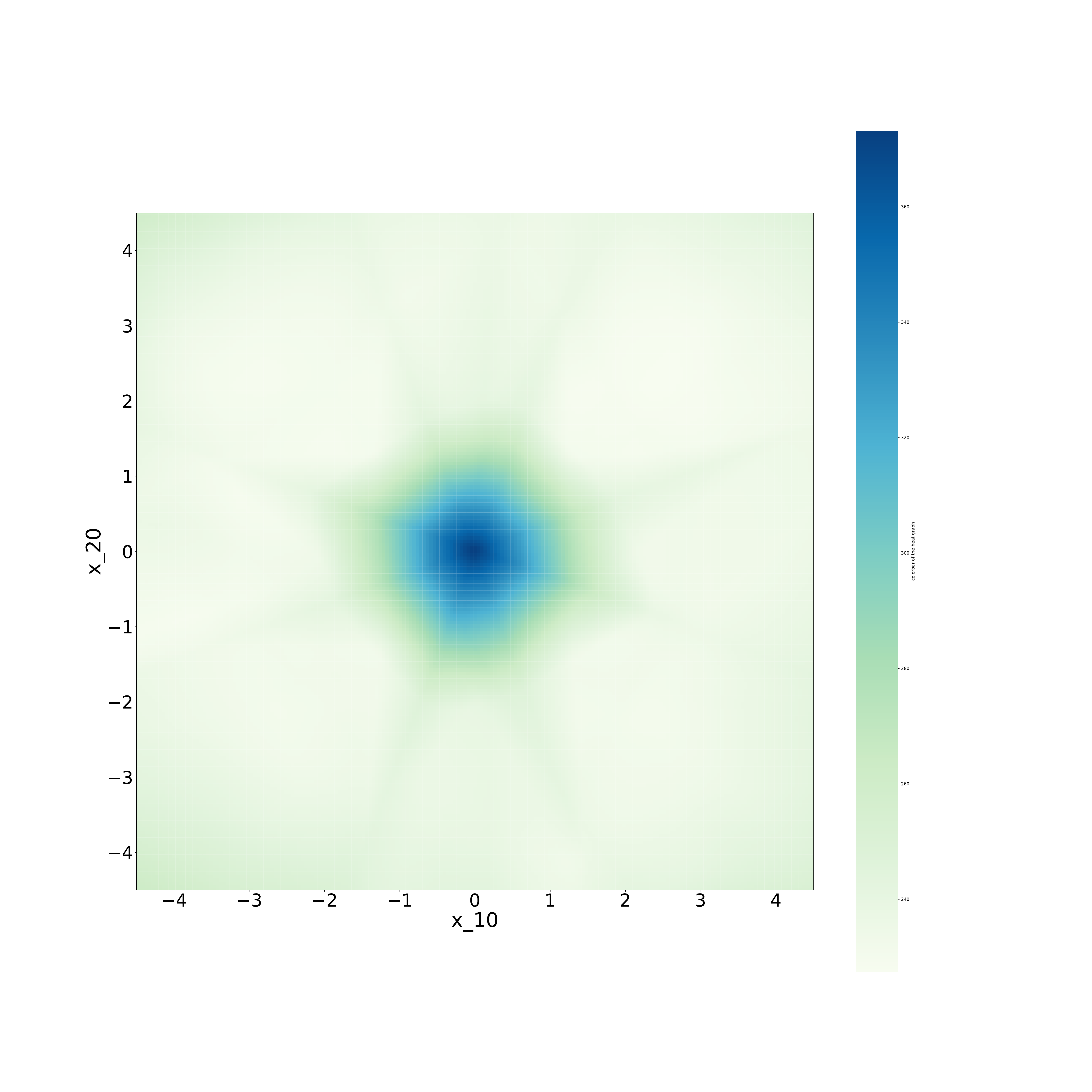}
         \caption{Heat graph of $\varphi_\eta(\cdot)$ }\label{subfig: MA50D nonequal gaussian plot varphi}
     \end{subfigure}
     \caption{\textbf{OT problem from $\rho_0$ to $\rho_b$ (50D)}: Plots of the pushforwarded density $T_{\theta\sharp}\rho_0$ by using Kernel Density Estimation (KDE). \ref{subfig: MA50D nonequal gaussian NPDG Id}-\ref{subfig: MA50D nonequal gaussian PD adam}: Numerical results produced by NPDG method and PD-Adam method. \ref{subfig: MA50D nonequal gaussian plot varphi}: heat graph of the Kantorovich dual function $\varphi_\eta(\cdot)$ learned from NPDG algorithm with precondition \eqref{canonical precond for Monge problem}. All figures are plotted on the $10-20$ coordinate plane.}\label{fig: MA50D nonequal gaussians}
\end{figure}

\section{Discussions}

In this paper, we design a preconditioned adversarial training algorithm called Natural Primal-Dual Hybrid Gradient (NPDG) for solving various PDEs. We incorporate the precondition operators $\mathcal M_p, \mathcal M_d$ in the precondition matrices $M_p(\theta), M_d(\eta)$ for computing the natural gradients. Alternative gradient descent and ascent algorithms, together with suitable extrapolation, are utilized to update the primal and dual neural network parameters. Linear convergence guarantees are established for the time-continuous version of the NPDG algorithm. In practice, we apply the MINRES iterative solver to handle natural gradients efficiently. The proposed algorithm outperforms classical machine learning–based approaches—including PINNs (Adam/LBFGS), the Deep Ritz method, and the Weak Adversarial Network / Primal–Dual Adam algorithm—in terms of convergence speed, robustness, and accuracy across various classes of PDEs, particularly in high-dimensional settings.

{\color{black}
Based on the numerical experiments, several critical questions about the proposed algorithm have arisen. The first concerns the convergence analysis of the time-discrete NPDG algorithm—namely, what are the optimal step sizes $\tau_u, \tau_\varphi, \tau_\psi$? Is it possible to adopt adaptive step sizes? Another crucial aspect that warrants more investigation is reducing the computational burden and improving the accuracy of the method by adopting refined strategies for evaluating natural gradients, such as Kronecker-factored Approximate Curvature (KFAC) \citep{martens2015optimizing, george2018fast, dangel2024kronecker} and randomized Nyström methods \citep{martinsson2020randomized, bioli2025accelerating}.

A primary motivation for developing the preconditioned primal–dual algorithm stems from the weak formulation obtained via integration by parts. However, this approach relies critically on the assumption that the dominant elliptic operator in the equation is of divergence form. Extensions to more general linear PDEs—such as elliptic equations in non-divergence form, as noted in Remark~\ref{rk: L non-divergence form}, are not addressed in the present work and constitute an important direction for future research.
    
Beyond time-dependent reaction–diffusion equations, extending the NPDG method to equation systems involving first-order convection terms, which commonly arise in the modeling of complex fluids, remains a challenging and largely unexplored direction. Other important time-dependent physical equations that are worthy of further investigation include Navier–Stokes equations and Maxwell’s equations.
    
Although the NPDG algorithm has demonstrated satisfactory performance on several classes of nonlinear PDEs, rigorous theoretical guarantees, particularly concerning the convergence of the method, remain open problems and will be pursued in the future work.

In addition to the task of handling various types of PDEs, the proposed research also paves the way for the future application of natural gradient algorithms in adversarial training of neural networks, including Generative Adversarial Networks (GANs) \citep{goodfellow2020generative, arjovsky2017wasserstein} and large-scale optimal transport problems \citep{fan2023neural, korotin2022neural}.
}

\vspace{0.2in}

\noindent\textbf{Acknowledgement:} S. Liu is partially supported by AFOSR YIP award No. FA9550-23-1-0087. S. Liu and S. Osher are partially funded by STROBE NSF STC DMR 1548924, AFOSR MURI FA9550-18-502, and ONR N00014-20-1-2787. W. Li is partially supported by AFOSR YIP award No. FA9550-23-1-0087, NSF DMS-2245097, and NSF RTG: 2038080. The authors would like to thank Prof. Xiaochuan Tian for constructive discussion. They would also appreciate the feedback from the anonymous reviewers that help improve the paper.

\newpage
\appendix

\section{Multiple Layer Perceptron (MLP)}\label{append: MLP}
In this research, we denote a Multiple Layer Perceptron (MLP) with activation function $f$, input dimension $\din$, hidden dimension $\dhidden$, output dimension $\dout$, and number of layers $\numlayer$ as $\texttt{MLP}_f(\din, \dhidden, \dout, \numlayer )$. Such MLP takes the form
\[ \texttt{MLP}_f(\din, \dhidden, \dout, \numlayer)(x) = h_{\numlayer}\circ \dots \circ h_2 \circ h_1(x),\]
where each $h_k(\cdot)$ is defined as
\begin{equation*}
  h_k(x) = \begin{cases}
    f(W_1 x + b_1) \quad \textrm{here } W_1 \in\mathbb{R}^{\dhidden \times \din}, b_1 \in \mathbb{R}^{\dhidden} \quad \textrm{if } k = 1\\
    f(W_k x + b_k) \quad \textrm{here } W_k \in\mathbb{R}^{\dhidden \times \dhidden}, b_k \in \mathbb{R}^{\dhidden} \quad \textrm{if } 2 \leq k \leq \numlayer - 1\\
    W_{\numlayer} x + b_{\numlayer} \quad    \textrm{here } W_{\numlayer}\in\mathbb{R}^{\dout \times \dhidden}, b_{\numlayer} \in \mathbb{R}^{\dout} \quad \textrm{if } k = \numlayer
    \end{cases} .
\end{equation*}
The parameters of the MLP are $(W_{\numlayer}, b_{\numlayer},\dots, W_1, b_1)$. The number of the parameters equals $\dout (\dhidden+1) + (\numlayer-2)\cdot \dhidden(\dhidden+1) + \dhidden (\din+1)$. The activation function $f$ of the MLP is usually chosen as a nonlinear function such as $\mathrm{ReLU(\cdot)}$, $\mathrm{tanh}(\cdot)$, etc\footnote{$\mathrm{ReLU}(x) = \max\{x, 0\}, \mathrm{tanh}(x)=\frac{e^{x}-e^{-x}}{e^x + e^{-x}}.$}.

\section{Proof of the consistency Theorem \ref{thm: consistency inf-sup and sol to PDE}}\label{append: proof of consistency thm}

{\color{black}
\begin{proof}
Without loss of generality, we always assume $\lambda = 1$ for brevity in this proof. We first show 
\begin{equation}
\underset{\varphi\in{\mathbb{K}^{test}}, \psi\in\mathbb{K}^{test}_{\partial \Omega} }{ \mathrm{sup}  } \mathscr{E}(\widehat{u}, {\varphi}, {\psi}) = 0.  \label{lemm1}
\end{equation}
Notice that for arbitrary $u \in \mathbb{H}$, we pick $\varphi = 0,$ $\psi=0.$ Then we obtain $\mathscr{E}(u, 0, 0) = 0$. This yields that $ \underset{\varphi, \psi }{\sup} ~ \mathscr{E}(u, \varphi, \psi) \geq 0$. For brevity, we omit the name of functional spaces for $\varphi, \psi$, and $u$ in this proof. This leads to
\begin{equation}
\inf_{u} \ \sup_{\varphi  , \psi  } ~ \mathscr{E}(u, \varphi, \psi)\geq 0.  \label{inf sup nonnegative}
\end{equation}
On the other hand, since $u_*$ is the solution to \eqref{linear PDE}, we have
\begin{align*}
\mathscr{E}(u_*, \varphi, \psi) & = \langle \mathcal L u_* - f, \varphi \rangle_{L^2(\Omega)} + \langle \mathcal B u_* - g, \psi \rangle_{L^2(\partial \Omega)} - \frac{\blue{\nu}}{2}( \|\mathcal M_d \varphi\|^2_{L^2(\Omega; \mathbb{R}^r)} + \|\psi\|^2_{L^2(\partial\Omega)}  )\\
& = - \frac{\blue{\nu}}{2}( \|\mathcal M_d \varphi\|^2_{L^2(\Omega; \mathbb{R}^r)} + \|\psi\|^2_{L^2(\partial \Omega)}  ),
\end{align*}
with the supremum value 
\[  \underset{\varphi, \psi}{\sup} ~ \mathscr{E}(u_*, \varphi, \psi)=0.  \]
Combining this with \eqref{inf sup nonnegative}, we obtain
\[  \inf_{u} \ \sup_{\varphi, \psi} ~ \mathscr{E}(u, \varphi, \psi) = 0.  \]
Since $(\widehat{u}, \widehat{\varphi}, \widehat{\psi})$ is the solution to the inf-sup problem \eqref{def: inf-sup new E}, we obtain \eqref{lemm1}.

Notice that equation \eqref{lemm1} further yields
\begin{equation*}
\begin{split}
  \langle \widetilde{\mathcal L} \mathcal M_p (\widehat{u} - u_*), \mathcal M_d    \varphi \rangle_{L^2(\Omega; \mathbb{R}^r)} - \frac{\blue{\nu}}{2} \|\mathcal M_d \varphi\|^2_{L^2(\Omega; \mathbb{R}^r)}  &  \\
  + \langle \mathcal B (\widehat{u} - u_*), \psi \rangle_{L^2(\partial \Omega)} - \frac{\blue{\nu}}{2}\|\psi\|^2_{L^2(\partial \Omega)}  &    \leq 0
\end{split}
\end{equation*}
for arbitrary $\varphi \in \blue{\mathbb{K}^{test}}, \ \psi \in \blue{\mathbb{K}^{test}_{\partial \Omega}}$. By setting $\psi=0$ and then $\varphi = 0$ in the above inequality, we obtain
\begin{equation} 
\langle \widetilde{\mathcal L} \mathcal M_p (\widehat{u} - u_*), \mathcal M_d    \varphi \rangle_{L^2(\Omega; \mathbb{R}^r)} - \frac{\blue{\nu}}{2} \|\mathcal M_d \varphi\|^2_{L^2(\Omega; \mathbb{R}^r)} \leq 0, \quad \forall ~ \varphi \in \blue{\mathbb{K}^{test}}   \label{inequality leading to orthogonality varphi}
\end{equation}
and 
\begin{equation}
 \langle \mathcal B (\widehat{u} - u_*), \psi \rangle_{L^2(\partial \Omega)} - \frac{\blue{\nu}}{2}\|\psi\|^2_{L^2(\partial \Omega)} \leq 0, \quad \forall ~ \psi \in \blue{\mathbb{K}^{test}_{\partial \Omega}}. \label{inequality leading to orthogonality psi}
\end{equation}

We first prove that \eqref{inequality leading to orthogonality varphi} leads to
\begin{equation}
  \langle \widetilde{\mathcal L} \mathcal M_p (\widehat{u} - u_*), \mathcal M_d \varphi   \rangle_{L^2(\Omega; \mathbb{R}^r)} = 0, \quad \forall ~ \varphi \in \blue{\mathbb{K}^{test}}. \label{orthogonality result vaprhi}
\end{equation}
Let us suppose \eqref{orthogonality result vaprhi} does not hold, then, there exists $\widetilde{\varphi}\in\blue{\mathbb{K}^{test}}$ such that
\[ \langle \widetilde{\mathcal L} \mathcal M_p (\widehat{u} - u_*), \mathcal M_d \widetilde{\varphi}   \rangle_{L^2(\Omega; \mathbb{R}^r)}  = \alpha \neq 0. \]
This also yields $\mathcal M_d \widetilde \varphi \neq 0$, otherwise, $\langle \widetilde{\mathcal L} \mathcal M_p (u - u_*), \mathcal M_d \widetilde \varphi \rangle_{L^2(\Omega; \mathbb{R}^r)} = 0$ leads to contradiction.

Now, substituting $\varphi = s \widetilde{\varphi}$ in \eqref{inequality leading to orthogonality varphi} leads to
\[  \langle \widetilde{\mathcal L} \mathcal M_p (\widehat{u} - u_*), \mathcal M_d    \varphi \rangle_{L^2(\Omega; \mathbb{R}^r)} - \frac{\blue{\nu}}{2} \|\mathcal M_d \varphi\|^2_{L^2(\Omega; \mathbb{R}^r)} = \alpha s - \frac{\blue{\nu}}{2} \|\mathcal M_d \widetilde{\varphi} \|_{L^2(\Omega; \mathbb{R}^r)}^2 \cdot s^2.   \]
Since $\mathcal M_d \widetilde \varphi \neq 0$, we have $\|\mathcal M_d \varphi\|_{L^2(\Omega; \mathbb{R}^r)}>0$. By setting $s = \frac{\alpha}{\blue{\nu} \|\mathcal M_d \widetilde{\varphi}\|_{L^2(\Omega; \mathbb{R}^r)}^2}$, the above inner product equals
\[ \frac{\blue{\alpha}^2}{2\nu\|\mathcal M_d \widetilde{\varphi}\|_{L^2(\Omega; \mathbb{R}^r)}^2}>0. \]
This is in contradiction to \eqref{inequality leading to orthogonality varphi}. We thus prove \eqref{orthogonality result vaprhi}. 

Using the similar argument, we prove 
\begin{equation}
\langle \mathcal B (\widehat{u} - u_*), \psi \rangle_{L^2(\partial \Omega)} = 0, \quad \forall ~ \psi \in \blue{\mathbb{K}^{test}_{\partial \Omega}}.  \label{orthogonality result psi}
\end{equation}
Now \eqref{orthogonality result vaprhi} further leads to
$\langle \mathcal M_d^* \widetilde{\mathcal L} \mathcal M_p(\widehat{u} - u_*), \varphi \rangle_{L^2(\Omega; \mathbb{R}^r)} = 0,$ for arbitrary $\varphi \in \blue{\mathbb{K}^{test}}.$
That is, 
\[ \langle \mathcal L (\widehat{u} - u_*), \varphi \rangle_{L^2(\Omega)} = 0,  \quad  \forall ~\varphi \in \blue{\mathbb{K}^{test}}. \]
Since $\blue{\mathbb{K}^{test}}$ is dense in $L^2(\Omega)$, this further leads to $\|\mathcal L (\widehat{u} - u_*)\|_{L^2(\Omega)} = 0.$ 

Recall that $\mathcal L u_* = f\in \mathbb{K} \subset L^2(\Omega),$ we thus have $\|\mathcal L \widehat{u} - f\|_{L^2(\Omega)}=0$. We deduce that $\mathcal L \widehat{u} = f,$ a.e. on $\Omega$. Similarly, we also prove that $\mathcal B \widehat{u} = g,$ a.e. on $\partial \Omega$.
\end{proof}

\begin{example}
Consider the Poisson equation $ -\Delta u = f, \quad u|_{\partial \Omega} = g,$ or the linear elliptic equation $-\Delta u + u = f, \quad u|_{\partial \Omega} = g,$ as mentioned in Example \ref{ex: 1} or \ref{ex: 2}, we set the test functional spaces $\blue{\mathbb{K}^{test}} = H_0^1(\Omega)$ and $\blue{\mathbb{K}^{test}_{\partial \Omega}} = L^2(\partial \Omega)$. Since $H_0^1(\Omega)$ is dense in $L^2(\Omega)$, Theorem \ref{thm: consistency inf-sup and sol to PDE} justifies the consistency between the solution to the inf-sup scheme \eqref{def: inf-sup new E} and the solutions to the equations.
\end{example}}

\section{Supplementary proofs and discussions regarding Section \ref{section: Numer Anal}
}\label{append: proof}
In this section, we present the proof to Lemma \ref{lemma: orthogonal projection }, Theorem \ref{thm: convergence analysis of NPDG flow} and provide further discussions regarding the theoretical work. We first prove Lemma \ref{lemma: orthogonal projection }.
\subsection{Proof of Lemma \ref{lemma: orthogonal projection }}\label{append: subsec proof of NG lemma}
\begin{proof}
    We first prove that $\nabla_\theta F(\theta) \in \mathrm{Ran}(M(\theta))$. We can first calculate
    \[ \nabla_\theta F(\theta) = \Big\langle  D_u\mathscr{F}(u_\theta), \frac{\partial u_\theta}{\partial \theta} \Big\rangle_{\mathbb{X}}. \]
    By decomposing $D_u \mathscr{F}(u_\theta)$ as
    \[ D_u \mathscr{F}(u_\theta) = \Pi_{\partial u_\theta} [D_u\mathscr{F}(u_\theta)] + \Pi_{\partial u_\theta^\perp } [D_u\mathscr{F} (u_\theta) ]. \]
    The first term can be written as the linear combination of $\{\frac{\partial u_\theta}{\partial \theta_k}\}_{k=1}^m$, i.e. $\Pi_{\partial u_\theta}[\mathscr{F}(u_\theta)] = \frac{\partial u_\theta}{\partial \theta} \mathbf{u}$ for certain $\mathbf{u}\in \mathbb{R}^m$. The inner product between $\Pi_{\partial u_\theta^\perp } [D_u\mathscr{F} (u_\theta) ]$ and $\frac{\partial u_\theta}{\partial \theta}$ equals $0$. As a result, we have
    \[ \nabla_\theta F(\theta) = \Big\langle \frac{\partial u_\theta}{\partial \theta} \mathbf{u} , \frac{\partial u_\theta}{\partial \theta} \Big\rangle_{\mathbb{X}} = M(\theta) \mathbf{u} \in \mathrm{Ran}(M(\theta)). \]
    On the other hand, we write
    \[ f(\zeta)  =  \Big\|D_u \mathscr{F}(u_\theta) - \frac{\partial u_\theta}{\partial \theta} \zeta \Big\|_{\mathbb{X}}^2 = 
 \zeta^\top M(\theta) \zeta - 2 \zeta^\top \nabla_\theta F(\theta) + \mathrm{Const}. \]
   Recall that $M(\theta)$ is a Gram matrix, it is positive semi-definite, thus $f(\zeta)$ is a convex function. Thus, $\mathbf{v}$ is a minimum of $f(\zeta)$ iff $\nabla f(\zeta) = 0$, which is equivalent to $M(\theta)\mathbf{v} = \nabla_\theta F(\theta).$

   To show the orthogonality, consider arbitrary $\mathbf{w}\in\mathbb{R}^m$, for any $s\in\mathbb{R}$, $f(\mathbf{v}+s\mathbf{w}) \geq f(\mathbf{v})$. This yields
   \[ 0 = \frac{d}{ds} f(\mathbf{v} + s\mathbf{w})\Big|_{s=0} = \Big\langle D_u \mathscr{F}(u_\theta) - \frac{\partial u_\theta}{\partial \theta} \mathbf{v}, \frac{\partial u_\theta}{\partial \theta} \mathbf{w}   \Big\rangle_{\mathbb{X}}  \quad \textrm{for any } \mathbf{w}\in\mathbb{R}^m. \]
    This verifies the fact that $D_u \mathscr{F}(u_\theta) - \frac{\partial u_\theta}{\partial \theta} \mathbf{v}$ is orthogonal to the subspace $\mathrm{span}\{ \frac{\partial u_\theta}{\partial\theta_1},\dots, \frac{\partial u_\theta}{\partial\theta_m} \}$.
\end{proof}

\subsection{Proof of Theorem \ref{thm: convergence analysis of NPDG flow}}\label{append: sec proof a posteriori convergence analysis}

\begin{proof} 
We first recall the functional ${\mathscr{E}}:\mathbb{H} \times \blue{\mathbb{K}^{test}} \times \mathbb{K}_{\partial\Omega}^{test} \rightarrow \mathbb{R}$ defined in \eqref{preconditioned J loss funcitional },
\begin{align*} 
  {\mathscr{E}}(u, \varphi, \psi) = & \langle  \mathcal L u - f, \varphi    \rangle_{L^2(\Omega)} + \lambda \langle \mathcal B u - g, \psi \rangle_{L^2(\partial \Omega)}    \\ 
  & - \frac{\nu}{2}( \|\mathcal M_d \varphi\|^2_{L^2(\Omega; \mathbb{R}^r)} + \|\psi\|^2_{L^2(\partial \Omega)}  ) \nonumber\\  
  = & \Big\langle \mathcal M^*_d \widetilde{\mathcal L} \mathcal M_p (u-u_*), \varphi \Big\rangle_{L^2(\Omega)} + \lambda \Big\langle \mathcal B (u - u_*), \psi \Big\rangle_{L^2(\partial\Omega)}  \\
  & - \frac{\nu}{2}( \|\mathcal M_d \varphi\|_{L^2(\Omega; \mathbb{R}^r)}^2  +  \lambda \|\psi\|_{L^2(\partial \Omega)}^2)  \\
  = & \Big\langle \widetilde{\mathcal L} \mathcal M_p  (u - u_*),  ~  \mathcal M_d \varphi \Big\rangle_{L^2(\Omega; \mathbb{R}^{r})} + \Big\langle \sqrt{\lambda}\mathcal B (u-u_*), \sqrt{\lambda} \psi \Big\rangle_{L^2(\partial \Omega)} \\ 
  & - \frac{\nu}{2}(\|\mathcal M_d \varphi\|_{L^2(\Omega, \mathbb{R}^{r})}^2 + \lambda \| \psi \|_{L^2(\partial\Omega)}^2)\\
  = & \Big\langle 
      \left(\begin{array}{cc}
         \widetilde{\mathcal L}  &  \\
           &   \textrm{Id}
      \end{array}
      \right)
      \left(\begin{array}{c}
        \mathcal M_p (u-u_*) \\
        \sqrt{\lambda}\mathcal B (u-u_*) 
      \end{array}\right), 
      \left(\begin{array}{c}
          \mathcal M_d \varphi \\
          \sqrt{\lambda}\psi
      \end{array}\right) \Big\rangle_{{ \mathbb{L}^2 }} - \frac{\nu}{2}\Big\| \left(\begin{array}{c}
        \mathcal M_d \varphi  \\
        \sqrt{\lambda} \psi \end{array}\right) \Big\|_{{ \mathbb{L}^2 }}^2.
\end{align*}
We now substitute $u, \varphi, \psi$ with parametrized functions $u_\theta, \varphi_\eta, \psi_\xi$, with $\theta\in\Theta_\theta\subseteq\mathbb{R}^{m_\theta}, \eta\in\Theta_{\eta}\subseteq\mathbb{R}^{m_\eta}, \xi\in\Theta_{\xi} \subseteq \mathbb{R}^{m_\xi}$. Recall that we define as $\widehat{E}(\theta; \eta, \xi) = {\mathscr{E}}(u_\theta; \varphi_\eta, \psi_\xi)$. In our discussion, we assume that $\mathcal M_p(u_\theta - u_*)$, $\mathcal B(u_\theta - u_*)$, $\mathcal M_d \varphi_{\eta}$ and $\psi_\xi$ are differentiable w.r.t. parameters $\theta, \eta, \xi$; and $\frac{\partial }{\partial \theta} (\mathcal M_p (u_\theta - u_*)) \in \widetilde{\mathbb H}$, $\frac{\partial }{\partial \eta} (\mathcal M_d \varphi_\eta) \in \blue{\widetilde{\mathbb{K}}^{test}}$, and $\frac{\partial }{\partial \xi} (\sqrt{\lambda} \psi_\xi) \in \blue{\mathbb{K}^{test}_{\partial \Omega}}$ for arbitrary $\theta\in\Theta_{\theta}, \eta\in\Theta_{ \eta }, \xi\in\Theta_{ \xi }$.

Now recall the preconditioning matrices introduced in \eqref{def: M_d}, \eqref{def: M_bdd} and \eqref{def: M_p}, they can be formulated as: 
\begin{align*}
  (M_p(\theta))_{ij} & = \Big\langle \frac{\partial }{\partial \theta_i}  \left( 
  \begin{array}{c}
       \mathcal M_p (u_\theta - u_*)    \\
       \sqrt{\lambda} \mathcal B  (  u_\theta - u_* ) 
  \end{array}
  \right) , ~\frac{\partial}{\partial \theta_j}\left( 
  \begin{array}{c}
      \mathcal M_p (u_\theta - u_*)     \\
      \sqrt{\lambda} \mathcal B  (u_\theta - u_*) 
  \end{array}
  \right)     
  \Big\rangle_{{ \mathbb{L}^2 }}
  \\
  (M_d(\eta))_{ij} & =  \Big\langle \frac{\partial }{\partial \eta_i}(\mathcal M_d \varphi_\eta) , ~ \frac{\partial }{\partial \eta_j}(\mathcal M_d \varphi_\eta)   \Big\rangle_{L^2(\Omega;  \mathbb{R}^{r})} 
  \\
  (M_{bdd}(\xi))_{ij} &  =  \Big\langle \frac{\partial }{\partial \xi_i}(\sqrt{\lambda}\psi_\xi), \frac{\partial }{\partial \xi_j } (\sqrt{ \lambda } \psi_\xi ) \Big\rangle_{L^2(\partial\Omega)}. 
\end{align*}
To alleviate our notation, we denote $M_{d, bdd}(\eta, \xi) = M_d(\eta)\oplus M_{bdd}(\xi).
$ We further denote
\begin{equation}\label{def: U Phi}
\begin{split}
    \textbf{U}_\theta = \left(\begin{array}{c}
         \mathcal M_p (u_\theta - u_*) \\
         \sqrt{\lambda} \mathcal B  (u_\theta - u_*)
      \end{array}\right)  \in  \widetilde{\mathbb{H}} \times \mathbb{K}_{\partial \Omega} \subseteq   { \mathbb{L}^2 }, \quad  
      \boldsymbol{\Phi}_{ \eta,  \xi  } = \left(\begin{array}{c}
           \mathcal M_d \varphi_\eta \\
           \sqrt{\lambda} \psi_\xi
      \end{array}      
      \right) \in \blue{\widetilde{\mathbb{K}}^{test}} \times \blue{\mathbb{K}^{test}_{\partial \Omega}} \subseteq { \mathbb{L}^2 }.
\end{split}
\end{equation}
By slightly abusing the notation, we denote
$\widetilde{\mathscr{E}}: { \mathbb{L}^2 } \times { \mathbb{L}^2 } \rightarrow  \mathbb{R}$ as
\[ {\mathscr{E}}(\textbf{U}_\theta, \boldsymbol{\Phi}_{\eta, \xi}) = \Big\langle (\widetilde{\mathcal L}\oplus \mathrm{Id})\textbf{U}_{\theta}, \boldsymbol{\Phi}_{\eta, \xi}\Big\rangle_{{ \mathbb{L}^2 }} - \frac{\blue{\nu}}{2}\|\boldsymbol{\Phi}_{\eta, \xi }\|_{{ \mathbb{L}^2 }}^2, \]
which is equal to the previous functional ${\mathscr{E}}(u_\theta, \varphi_\eta, \psi_\xi )$.

Notice that \eqref{tilde phi tilde phi} is denoted as $\boldsymbol{\Phi}_{\eta_t, \xi_t}+\gamma \dot{\boldsymbol{\Phi}}_{\eta_t, \xi_t}$ by using our new notation, the NPDG flow \eqref{PDHG flow} can be formulated as
\begin{equation}
  \begin{split}\label{PDHG flow simplify form}
      {(\dot\eta_t^\top, \dot\xi_t^\top)^\top} = & M_{d,bdd}(\eta_t, \xi_t)^\dagger \nabla_{\eta, \xi }  {\mathscr{E}}(\textbf{U}_{\theta_t}, \boldsymbol{\Phi}_{\eta_t, \xi_t })  \\ 
      \dot \theta_t = & - M_p(\theta_t)^{\dagger} \nabla_\theta {\mathscr{E}}(\textbf{U}_{\theta_t}, \boldsymbol{\Phi}_{\eta_t, \xi_t } + \gamma \dot{\boldsymbol{\Phi}}_{\eta_t, \xi_t }). 
  \end{split}
\end{equation}
Now suppose $(\theta_t, \eta_t, \xi_t )$ solves \eqref{PDHG flow simplify form}; we compute
\begin{equation} 
  \dot{\boldsymbol{\Phi}}_{\eta_t, \xi_t} = \frac{\partial \boldsymbol{\Phi }_{\eta_t \xi_t }}{\partial (\eta, \xi ) } M_{d, bdd}(\eta_t, \xi_t )^\dagger \nabla_{\eta, \xi } {\mathscr{E}}(\textbf{U}_{\theta_t}, \boldsymbol{\Phi}_{\eta_t, \xi_t } ).  \label{dynamic of varphi_pt }
\end{equation}
By treating $\mathbb{X}={ \mathbb{L}^2 }$ and $\mathscr{F}(\cdot)$ as ${\mathscr{E}}(\textbf{U}_\theta, \cdot)$ in Lemma \ref{lemma: orthogonal projection }, the right-hand side of \eqref{dynamic of varphi_pt } is nothing but the orthogonal projection of $D_{\boldsymbol{\Phi}}{\mathscr{E}}(\textbf{U}_{\theta_t}, \boldsymbol{\Phi}_{\eta_t, \xi_t })=(\widetilde{\mathcal L}\oplus \mathrm{Id})\textbf{U}_{\theta_t} - \blue{\nu} \boldsymbol{\Phi}_{\eta_t, \xi_t }$ onto the tangent space $\partial_{\eta, \xi } \boldsymbol{\Phi}_{\eta_t, \xi_t }$, that is,
\begin{align*}
  \frac{\partial \boldsymbol{\Phi}_{\eta_t, \xi_t }}{\partial (\eta, \xi)} M_{d, bdd}(\eta_t, \xi_t )^\dagger \nabla_{\eta, \xi} {\mathscr{E}}(\textbf{U}_{\theta_t}, \boldsymbol{\Phi}_{\eta_t, \xi_t} ) = &\Pi_{\partial_{\eta, \xi} \boldsymbol{\Phi}_{\eta_t, \xi_t }}[ D_{\boldsymbol{\Phi}} {\mathscr{E}}(\textbf{U}_{\theta_t}, \boldsymbol{\Phi}_{\eta_t, \xi_t })  ] \\
  = & \Pi_{\partial_{\eta, \xi} \boldsymbol{\Phi}_{\eta_t, \xi_t }}[ (\widetilde{\mathcal L}\oplus \mathrm{Id})\textbf{U}_{\theta_t} - \nu \boldsymbol{\Phi}_{\eta_t , \xi_t } ].
\end{align*}
Similarly
\begin{equation} 
    \dot{\textbf{U}}_{\theta_t} = - \frac{\partial \textbf{U}_{\theta_t}}{\partial \theta} M_p(\theta_t)^\dagger \nabla_{\theta} {\mathscr{E}}(\textbf{U}_{\theta_t}, \boldsymbol{\Phi}_{\eta_t, \xi_t} + \gamma \dot{\boldsymbol{\Phi}}_{\eta_t, \xi_t }).  \label{dynamic of u_theta_t }
\end{equation}
By denoting $\widetilde{\mathcal L}^*$ as the adjoint operator\footnote{In the sense that
\[ \Big\langle \widetilde{ \mathcal{ L } } {v}, w \Big\rangle_{L^2(\Omega;\mathbb{R}^{r})} = \Big\langle v, \widetilde{\mathcal L}^* w \Big\rangle_{L^2(\Omega; \mathbb{R}^{r})}, \quad \forall~ v \in \widetilde{\mathbb{H}}, ~  w \in \widetilde{\mathbb{K}}^{test}. \]
} of $\widetilde{\mathcal L}$, we have
\begin{align*} 
  {\mathscr{E}}(\textbf{U}, \boldsymbol{\Phi}) 
  = \Big\langle (\widetilde{\mathcal L} \oplus \mathrm{Id})  \textbf{U}, \boldsymbol{\Phi} \Big\rangle_{{ \mathbb{L}^2 }} - \frac{\blue{\nu}}{2}\|\boldsymbol{\Phi}\|_{{ \mathbb{L}^2 }}^2 = \Big\langle \textbf{U} , (\widetilde{\mathcal L}^* \oplus \mathrm{Id}) \boldsymbol{\Phi} \Big\rangle_{{ \mathbb{L}^2 }} - \frac{\blue{\nu}}{2}\|\boldsymbol{\Phi}\|_{{ \mathbb{L}^2 }}^2.
\end{align*}
Then, the right-hand side of \eqref{dynamic of u_theta_t } equals
\begin{equation*}
  - \Pi_{\partial_\theta\textbf{U}_{\theta_t}} [ D_{\textbf{U}}{\mathscr{E}}(\textbf{U}_{\theta_t}, {\boldsymbol{\Phi}}_{\eta_t, \xi_t } + \gamma \dot{\boldsymbol{\Phi}}_{\eta_t, \xi_t })   ] = - \Pi_{\partial_\theta \textbf{U}_{\theta_t}}[ (\widetilde{\mathcal L}^*\oplus \mathrm{Id})(\boldsymbol{\Phi}_{\eta_t, \xi_t } + \gamma \dot{\boldsymbol{\Phi}}_{\eta_t, \xi_t }) ],
\end{equation*}
Thus the corresponding dynamic of \eqref{PDHG flow simplify form} in the functional space can be formulated as
\begin{equation}\label{NPDG flow on functional spaces}
    \begin{split}
        & \dot{\boldsymbol{\Phi}}_{\eta_t, \xi_t} = \Pi_{\partial_{\eta, \xi} \boldsymbol{\Phi}_{\eta_t, \xi_t}}[ (\widetilde{\mathcal L}\oplus \mathrm{Id})\textbf{U}_{\theta_t} - \blue{\nu} \boldsymbol{\Phi}_{\eta_t, \xi_t  } ],  \\
        & \dot{\textbf{U}}_{\theta_t} = - \Pi_{\partial_\theta \textbf{U}_{\theta}}[ (\widetilde{\mathcal L}^* \oplus \mathrm{Id})(\boldsymbol{\Phi}_{\eta_t, \xi_t } + \gamma \dot{\boldsymbol{\Phi}}_{\eta_t, \xi_t }) ].
    \end{split}
\end{equation}
We now consider the Lyapunov functional
\begin{align}
  \mathcal I(\textbf{U}, \boldsymbol{\Phi}) = & \frac12 (\|\mathcal M_p (u - u_* )\|_{L^2(\Omega; \mathbb{R}^r)}^2 + \lambda \|\mathcal B (u - u_*) \|_{L^2(\partial \Omega) }^2 + \|\mathcal M_d \varphi\|_{L^2(\Omega; \mathbb{R}^r)}^2 + \lambda \|\psi\|_{L^2(\partial  \Omega)}^2  )  \nonumber  \\
  = & \frac12 \|\textbf{U}\|_{{ \mathbb{L}^2 }}^2 + \frac12 \|\boldsymbol{\Phi}\|_{{ \mathbb{L}^2 }}^2.  \label{Lyapunov func }
\end{align}
We shall study the decay of this Lyapunov functional along $\{(\textbf{U}_{\theta_t}, \boldsymbol{\Phi}_{\eta_t, \xi_t})\}$. We calculate
\begin{equation}\label{calculate dI dt}
\begin{split}
  \frac{d}{dt}\mathcal I (\textbf{U}_{\theta_t}, \boldsymbol{\Phi}_{\eta_t, \xi_t }) & = \bigla \ut, \dotut \bigra_{{ \mathbb{L}^2 }} + \bigla \boldsymbol{\Phi}_{\eta_t, \xi_t}, \dot{\boldsymbol{\Phi}}_{\eta_t, \xi_t} \bigra_{{ \mathbb{L}^2 }} \\
  & = \underbrace{\bigla \ut, - \Pi_{\partial_\theta \textbf{U}_{\theta_t}}[ (\widetilde{\mathcal L}^*\oplus \mathrm{Id})(\boldsymbol{\Phi}_{\eta_t, \xi_t  } + \gamma \dot{\boldsymbol{\Phi}}_{\eta_t, \xi_t  }) ]  \bigra_{{ \mathbb{L}^2 }}}_{(1)} \\ 
  &  \quad + \underbrace{\bigla \varphit, ~ \Pi_{\partial_{\eta, \xi} \boldsymbol{\Phi}_{\eta_t, \xi_t }}[ (\widetilde{\mathcal L}\oplus \mathrm{Id})\textbf{U}_{\theta_t} - \blue{\nu} \boldsymbol{\Phi}_{\eta_t, \xi_t   } ] \bigra_{{ \mathbb{L}^2 }}}_{(2)}
\end{split}
\end{equation}

We further compute (1) as:
\begin{equation}\label{estimation of term (1)}
\begin{split}
  (1) = & ~  - \bigla \ut, \quad \Pi_{\partial\textbf{U}_{\theta_t}}[ (\widetilde{\mathcal L}^*\oplus \mathrm{Id})(\boldsymbol{\Phi}_{\eta_t, \xi_t  } + \gamma \Pi_{\partial\boldsymbol{\Phi}_{\eta_t, \xi_t }}[ (\widetilde{\mathcal L}\oplus \mathrm{Id})\textbf{U}_{\theta_t} - \blue{\nu} \boldsymbol{\Phi}_{\eta_t, \xi_t   } ] ) ]  \bigra_{{ \mathbb{L}^2 }} \\
  = & ~ - \bigla \Pi_{\partial\textbf{U}_{\theta_t}}[\ut], \quad (\widetilde{\mathcal L}^*\oplus \mathrm{Id})(\boldsymbol{\Phi}_{\eta_t, \xi_t  } + \gamma (\widetilde{\mathcal L} \oplus \mathrm{Id})\ut - \gamma\blue{\nu} \boldsymbol{\Phi}_{\eta_t, \xi_t }  )    \bigra_{{ \mathbb{L}^2 }}  \\
  & ~ + \bigla \Pi_{\partial \ut} [ \ut ] , \quad  \gamma(\widetilde{\mathcal L}^* \oplus \mathrm{Id}) \Pi_{\partial \boldsymbol{\Phi}_{\eta_t, \xi_t}^{\perp}} ((\widetilde{\mathcal L}\oplus \mathrm{Id}) \ut - \blue{\nu} \boldsymbol{\Phi}_{\eta_t, \xi_t} )   \bigra_{{ \mathbb{L}^2 }}\\
  = & ~ \underbrace{- \bigla \ut, \quad  (\widetilde{\mathcal L}^*\oplus \mathrm{Id})((1-\gamma\blue{\nu})\boldsymbol{\Phi}_{\eta_t, \xi_t  } + \gamma (\widetilde{\mathcal L} \oplus \mathrm{Id})\ut  ) \bigra_{{ \mathbb{L}^2 }}}_{(A)} \\
  & ~ + \underbrace{\bigla \Pi_{\partial\textbf{U}_{\theta_t}^{\perp}}[\ut], \quad (\widetilde{\mathcal L}^*\oplus \mathrm{Id})((1-\gamma\blue{\nu} )\boldsymbol{\Phi}_{\eta_t, \xi_t  } + \gamma (\widetilde{\mathcal{L} } \oplus \mathrm{Id})\ut)  \bigra_{{ \mathbb{L}^2 }}}_{(R1)} \\
  & ~ + \underbrace{\gamma\bigla (\widetilde{\mathcal L} \oplus \mathrm{Id}) ~\Pi_{\partial \ut} [ \ut ] , \quad \Pi_{\partial \boldsymbol{\Phi}_{\eta_t, \xi_t}^{\perp}} [ (\widetilde{\mathcal L}\oplus \mathrm{Id}) \ut - \blue{\nu} \boldsymbol{\Phi}_{\eta_t, \xi_t} ]     \bigra_{{ \mathbb{L}^2 }}}_{(R2)}.
\end{split}
\end{equation}
For the second equality, we use the fact that the orthogonal projection $\Pi_{\partial \ut}$ is self-adjoint on ${ \mathbb{L}^2 }$. 

Furthermore, the term (2) equals
\begin{equation}\label{estimation of term (2)}
\begin{split}
  (2) = & ~ \bigla \varphit,  \quad  (\widetilde{\mathcal L}\oplus \mathrm{Id})\textbf{U}_{\theta_t} - \blue{\nu} \boldsymbol{\Phi}_{\eta_t, \xi_t   }     \bigra_{{ \mathbb{L}^2 }} + \bigla \varphit,  \quad  \Pi_{\partial \boldsymbol{\Phi}_{\eta_t, \xi_t}^\perp}[ (\widetilde{\mathcal L}\oplus \mathrm{Id})\textbf{U}_{\theta_t} - \blue{\nu} \boldsymbol{\Phi}_{\eta_t, \xi_t }]    \bigra_{{ \mathbb{L}^2 }}  \\
  = & ~ \underbrace{\bigla \varphit,  \quad  (\widetilde{\mathcal L}\oplus \mathrm{Id})\textbf{U}_{\theta_t} - \blue{\nu} \boldsymbol{\Phi}_{\eta_t, \xi_t   }     \bigra_{{ \mathbb{L}^2 }}}_{(B)} + \underbrace{\bigla \Pi_{\partial\boldsymbol{\Phi}_{\eta_t, \xi_t}^\perp} [\varphit],  \quad  (\widetilde{\mathcal L}\oplus \mathrm{Id})\textbf{U}_{\theta_t} - \blue{\nu} \boldsymbol{\Phi}_{\eta_t, \xi_t }    \bigra_{{ \mathbb{L}^2 }}}_{(R3)}
\end{split}
\end{equation}

Then one can calculate
\begin{align}
   & ~ ~ (A) + (B)  \nonumber  \\
   & = -\bigla \ut, (\widetilde{\mathcal L}^*\oplus \mathrm{Id})((1-\gamma\blue{\nu})\boldsymbol{\Phi}_{\eta_t, \xi_t  } + \gamma  (\widetilde{\mathcal L}\oplus \mathrm{Id})\textbf{U}_{\theta_t} )  ) \bigra_{{ \mathbb{L}^2 }} \! + \! \bigla \varphit,  (\widetilde{\mathcal L}\oplus \mathrm{Id})\textbf{U}_{\theta_t} - \blue{\nu} \boldsymbol{\Phi}_{\eta_t, \xi_t   }     \bigra_{{ \mathbb{L}^2 }}  \nonumber  \\
  & = - \gamma \bigla \ut, (\widetilde{\mathcal L}^* \oplus \mathrm{Id})(\widetilde{\mathcal L} \oplus \mathrm{Id}) \ut \bigra_{{ \mathbb{L}^2 }} + \gamma\blue{\nu} \bigla \varphit, (\widetilde{\mathcal L} \oplus \mathrm{Id}) \ut  \bigra_{{ \mathbb{L}^2 } } - \blue{\nu}\bigla \varphit, \varphit  \bigra_{{ \mathbb{L}^2 } }  \label{calculate  A+B  }
\end{align}
Recall the assumption \eqref{boundness of L, L_0, L_1}, we have: 
\begin{align*}
  \|(\widetilde{\mathcal L} \oplus \mathrm{Id}) \textbf{U}\|_{{ \mathbb{L}^2 }}^2 & = \|\widetilde{\mathcal L} \textbf{u}\|_{L^2(\Omega; \mathbb{R}^r)}^2 + \|w\|_{L^2(\partial \Omega)}^2 \\
  & \leq L_1^2 \|\textbf{u}\|_{L^2(\Omega; \mathbb{R}^r)}^2 + \|w\|_{L^2(\partial \Omega)}^2 \\
  & \leq (L_1^2 \vee 1)
  \cdot (\|\textbf{u}\|_{L^2(\Omega; \mathbb{R}^r)}^2 + \|w\|_{L^2(\partial \Omega)}^2) = (L_1^2 \vee 1) \cdot \|\textbf{U}\|_{{ \mathbb{L}^2 }}^2. 
\end{align*}
That is, $\|(\widetilde{\mathcal L} \oplus \mathrm{Id}) \textbf{U}\|_{{ \mathbb{L}^2 }} \leq (L_1 \vee 1) \cdot \|\textbf{U}\|_{{ \mathbb{L}^2 }}$. Similarly, we have $\|(\widetilde{\mathcal{L}} \oplus \mathrm{Id}) \textbf{U} \|_{{ \mathbb{L}^2 }} \geq (L_0 \wedge 1) \|\textbf{U}\|_{{ \mathbb{L}^2 }}.$  

We can verify that \eqref{calculate  A+B  } yields
\begin{equation}\label{est (A)+(B)}
  (A) + (B) \leq - \gamma (L_0\wedge 1)^2\|\textbf{U}_{\theta_t}\|_{{ \mathbb{L}^2 }}^{\blue{2}} + \gamma \blue{\nu} (L_1 \vee 1) \|\boldsymbol{\Phi}_{\eta_t, \xi_t }\|_{{ \mathbb{L}^2 }}\cdot \|\textbf{U}_{\theta_t}\|_{{ \mathbb{L}^2 }} - \blue{\nu} \|\boldsymbol{\Phi}_{\eta_t, \xi_t }\|_{{ \mathbb{L}^2 }}^2.        
\end{equation}
Moreover, by Cauchy-Schwarz inequality, we  estimate the remainder terms $(R1), (R2), (R3)$ as
\begin{align*}
  (R1) \leq &  ~  \|\Pi_{\partial \ut^\perp}[\ut]\|_{{ \mathbb{L}^2 }}  \cdot  ((L_1\vee 1)|1-\gamma\blue{\nu}|\|\boldsymbol{\Phi}_{\eta_t, \xi_t }\|_{{ \mathbb{L}^2 }} + \gamma (L_1 \vee 1)^2 \|\ut\|_{{ \mathbb{L}^2 }}) \\
  \leq &  ~  \alpha \|\ut\|_{{ \mathbb{L}^2 }} \cdot ((L_1\vee 1)|1-\gamma\blue{\nu}|\|\boldsymbol{\Phi}_{\eta_t, \xi_t }\|_{{ \mathbb{L}^2 }} + \gamma (L_1 \vee 1)^2 \|\ut\|_{{ \mathbb{L}^2 }})\\
  = &  ~\alpha \cdot (L_1\vee 1) \cdot |1-\gamma \blue{\nu} | \cdot \|\ut\|\cdot \|\boldsymbol{\Phi}_{\eta_t, \xi_t }\|_{{ \mathbb{L}^2 }} + \alpha    \cdot \gamma \cdot (L_1\vee 1)^2 \cdot \|\ut\|_{{ \mathbb{L}^2 }}^2 . 
\end{align*}
\begin{align*}
  (R2) \leq &  ~  \gamma \cdot (L_1\vee 1) \|\Pi_{\partial \ut} [\ut] \|_{{ \mathbb{L}^2 }} \cdot \|\Pi_{\partial \boldsymbol{\Phi}_{\eta_t, \xi_t}^{\perp}} [ (\widetilde{\mathcal L}\oplus \mathrm{Id}) \ut - \blue{\nu} \boldsymbol{\Phi}_{\eta_t, \xi_t} ] \|_{{ \mathbb{L}^2 }} \\
  \leq & ~  \gamma \cdot (L_1 \vee 1)  \cdot  \|\ut\|_{{ \mathbb{L}^2 }} \cdot ( 
  \|\Pi_{\partial \boldsymbol{\Phi}_{\eta_t, \xi_t }^\perp }[(\widetilde{\mathcal L}\oplus \mathrm{Id}) \ut]\|_{{ \mathbb{L}^2 } }  +  \blue{\nu} \| \Pi_{\partial \boldsymbol{\Phi}_{\eta_t, \xi_t }} [\boldsymbol{\Phi}_{\eta_t, \xi_t }] \|_{{ \mathbb{L}^2 }} )\\
  \leq & ~  \gamma \cdot (L_1 \vee 1) \cdot \|\ut\|_{{ \mathbb{L}^2 }} \cdot ( \beta_1 \|(\widetilde{\mathcal L}  \oplus  \mathrm{Id}) \ut \|_{{ \mathbb{L}^2 }} + \blue{\nu} \beta_2 \|\boldsymbol{\Phi}_{\eta_t, \xi_t}\|_{{ \mathbb{L}^2 }} )  \\
  \leq & ~ \gamma \cdot (L_1 \vee 1)^2 \cdot \beta_1 \|\ut\|_{{ \mathbb{L}^2 }}^2 + \gamma \blue{\nu} \cdot (L_1 \vee 1) \cdot \beta_2 \cdot \|\ut\|_{{ \mathbb{L}^2 }}  \cdot    \| \boldsymbol{\Phi}_{\eta_t, \xi_t} \|_{{ \mathbb{L}^2 }}.
\end{align*}
\begin{align*}
  (R3) \leq & ~  \|\Pi_{\partial\boldsymbol{\Phi}_{\eta_t, \xi_t}^\perp} [\varphit]\|_{{ \mathbb{L}^2 } } \cdot \|(\widetilde{\mathcal L}\oplus \mathrm{Id})\textbf{U}_{\theta_t} - \blue{\nu} \boldsymbol{\Phi}_{\eta_t, \xi_t}\|_{ { \mathbb{L}^2 } } \\
  \leq & ~ \beta_2 \cdot \|\boldsymbol{\Phi}_{\eta_t, \xi_t}\|_{{ \mathbb{L}^2 }  } \cdot ((L_1 \vee 1) \|\ut\|_{{ \mathbb{L}^2 }} + \blue{\nu} \|\boldsymbol{\Phi}_{\eta_t, \xi_t}\|_{{ \mathbb{L}^2 }}  ) \\
  = & ~  \beta_2 \cdot (L_1 \vee 1) \cdot \|\ut\|_{{ \mathbb{L}^2 }} \cdot \|\boldsymbol{\Phi}_{\eta_t, \xi_t}\|_{ { \mathbb{L}^2 } } + \beta_2 \cdot \blue{\nu} \cdot \|\boldsymbol{\Phi}_{\eta_t, \xi_t}\|_{{ \mathbb{L}^2 }}^2.
\end{align*}
Here, we denote
\begin{align}
\alpha = &  \max_{t\in[0, T]} ~ \frac{\|\Pi_{\partial \ut^\perp}[\ut]\|_{{ \mathbb{L}^2 }}}{\|\ut\|_{{ \mathbb{L}^2 }}}; \label{def: alpha}\\
 \beta_1 = &  \max_{t\in[0, T]} ~ \frac{\|\Pi_{\partial \boldsymbol{\Phi}_{\eta_t, \xi_t }^\perp }[(\widetilde{\mathcal L}\oplus \mathrm{Id}) \ut]\|_{{ \mathbb{L}^2 }}}{\| (\widetilde{\mathcal L}\oplus \mathrm{Id}) \ut \|_{{ \mathbb{L}^2 }}};  \label{def: beta1}\\
 \beta_2 = &  \max_{t\in[0, T]} ~ \frac{\|\Pi_{\partial \boldsymbol{\Phi}_{\eta_t, \xi_t }^{\perp  }  }[\boldsymbol{\Phi}_{\eta_t, \xi_t}]\|_{{ \mathbb{L}^2 }}}{\|\boldsymbol{\Phi}_{\eta_t, \xi_t }\|_{{ \mathbb{L}^2 }}} .  \label{def: beta2}
\end{align}
It is not hard to tell that $0 \leq \alpha, \beta_1, \beta_2 \leq 1$. 

Now, recall \eqref{calculate dI dt} and \eqref{calculate  A+B  }, together with the estimates on the remainder terms $(R1), (R2), (R3)$ we obtain
\begin{align*}
 \frac{d}{dt}\mathcal I (\textbf{U}_{\theta_t}, \boldsymbol{\Phi}_{\eta_t, \xi_t  } ) \leq & - \gamma \cdot ((L_0\wedge 1)^2 - (L_1 \vee 1)^2 (\alpha + \beta_1) )  \cdot \|\ut\|_{{ \mathbb{L}^2 }}^2    \nonumber \\
 & + (L_1 \vee 1) \cdot ((1+\beta_1)\gamma\blue{\nu} +  \beta_2  + \alpha|1-\gamma\blue{\nu}| ) \cdot  \| \varphit\|_{{ \mathbb{L}^2 }} \cdot \|\ut\|_{{ \mathbb{L}^2 }}  \nonumber\\
 &  -  \blue{\nu} \cdot (1 - \beta_2) \cdot \|\boldsymbol{\Phi}_{\eta_t, \xi_t }\|_{{ \mathbb{L}^2 }}^2.  \nonumber  \\
 \leq & - \left[\|\ut\|_{{ \mathbb{L}^2 }}, \|\boldsymbol{\Phi}_{\eta_t, \xi_t}\|_{{ \mathbb{L}^2 }}\right] \underbrace{\left[ \begin{array}{cc}
     \Gamma_{\textbf{U}\textbf{U}} & \Gamma_{\boldsymbol{\Phi}\textbf{U}}/2  \\
     \Gamma_{\boldsymbol{\Phi}\textbf{U}}/2 & \Gamma_{\boldsymbol{\Phi}\boldsymbol{\Phi}}
 \end{array} \right]}_{\Gamma} \left[\begin{array}{c}
       \|\ut\|_{{ \mathbb{L}^2 }} \\
       \|\boldsymbol{\Phi}_{\eta_t, \xi_t }\|_{{ \mathbb{L}^2 }}
 \end{array}\right].  
\end{align*}

Here we denote 
\begin{equation*}
\begin{split}
  & \Gamma_{\textbf{U}\textbf{U}} = \gamma \cdot ((L_0\wedge 1)^2 - (L_1 \vee 1)^2(\alpha + \beta_1)), \quad  \Gamma_{\boldsymbol{\Phi}\boldsymbol{\Phi}} = \blue{\nu} (1-\beta_2), \\ 
  & \Gamma_{\boldsymbol{\Phi}\textbf{U}} = -(L_1 \vee 1)\cdot ((1+\beta_1)\gamma\blue{\nu} + \beta_2 + \alpha|1-\gamma\blue{\nu}|). 
\end{split}
\end{equation*}
Since we assumed that $\frac{1}{\widetilde{\kappa}^2} > \alpha + \beta_1$, this yields $\Gamma_{\textbf{U}\textbf{U}}>0$; and $\beta_2 < 1$ yields $\Gamma_{\boldsymbol{\Phi}\boldsymbol{\Phi}}>0$; moreover, \eqref{ineq condition for posdef of Gamma} is equivalent to $\mathrm{det}(\Gamma) = \Gamma_{\textbf{U}\textbf{U}}\Gamma_{\boldsymbol{\Phi}\boldsymbol{\Phi}}-\frac14\Gamma_{\boldsymbol{\Phi}\textbf{U}}^2>0$. In conclusion, these lead to the fact that $\Gamma$ is positive definite. Further, we denote the smaller eigenvalue of $\Gamma$ as
\begin{equation}
  r = \frac12 \left(\Gamma_{\textbf{U}\textbf{U}}+\Gamma_{\boldsymbol{\Phi}\boldsymbol{\Phi}}- \sqrt{(\Gamma_{\textbf{U}\textbf{U}} - \Gamma_{\boldsymbol{\Phi}\boldsymbol{\Phi}})^2  +  \Gamma_{\boldsymbol{\Phi}\textbf{U}}^2}\right).  \label{def: convergence rate}
\end{equation}
Thus, $r > 0$,
and we obtain
\begin{equation*}
    \frac{d}{dt}\mathcal I (\textbf{U}_{\theta_t}, \boldsymbol{\Phi}_{\eta_t, \xi_t  } ) \leq - r \cdot \mathcal I (\textbf{U}_{\theta_t}, \boldsymbol{\Phi}_{\eta_t, \xi_t  }), \quad t\in [0, T].
\end{equation*}
Applying the Gr\"{o}nwall's inequality yields
\begin{equation*}
  \mathcal I(\textbf{U}_{\theta_t}, \boldsymbol{\Phi}_{\eta_t, \xi_t  }) \leq \exp(-rt)\cdot \mathcal I(\textbf{U}_{\theta_0}, \boldsymbol{\Phi}_{\eta_0, \xi_0}),
\end{equation*}
for $t\in[0, T]$. Recall definition \eqref{Lyapunov func }, we have proven the theorem
\begin{equation*}
  \|\mathcal M_p (u_{\theta_t} - u_* )\|_{L^2(\Omega; \mathbb{R}^r)}^2 + \lambda \|\mathcal B (u_{\theta_t} - u_*) \|_{L^2(\partial \Omega) }^2 \leq 2  \exp(-rt)  \cdot \mathcal I(\textbf{U}_{\theta_0}, \boldsymbol{\Phi}_{\eta_0, \xi_0}), \quad 0 \leq t \leq T. 
\end{equation*}
\end{proof}

{\color{black}
\subsection{Some definitions related to the fractional Sobolev space}\label{append: subsec frac Sobolev spc}
We give a brief definition of the fractional Sobolev space $H^{1/2}(\partial \Omega)$ and state a useful result to be used in the next section characterizing its norm.
\begin{definition}[$H^{1/2}(\partial \Omega)$ space]
  For open bounded domain $\Omega\subset\mathbb{R}^d$ with Lipschitz boundary $\partial \Omega$, we define 
  \begin{equation*}
      H^{1/2}(\partial \Omega) = \left\{u \in L^2(\partial\Omega)\,:\,\|u\|_{H^{1/2}(\partial\Omega)} < \infty\right\},
  \end{equation*}
  where we define the fractional Sobolev norm (also known as the Sobolev-Slobodeckii norm using general $L^p$ norm, see \citep{gagliardo1957caratterizzazioni} and the references therein for more details)
  \begin{equation}\label{def: H^1/2 norm}
    \|u\|^2_{H^{1/2}(\partial\Omega)}
    := \|u\|_{L^2(\partial\Omega)}^2
    + |u|_{H^{1/2}(\partial\Omega)}^2,
  \end{equation}
  with seminorm
  \begin{equation*}
    |u|_{H^{1/2}(\partial\Omega)}^2
    := \int_{\partial\Omega}
    \int_{\partial\Omega}
    \frac{|u(x) - u(y)|^2}{\|x - y\|^{d}}
    \, \mathrm{d}s_x \, \mathrm{d}s_y.
  \end{equation*}
\end{definition}
One can then define the bounded, surjective trace operator $\mathcal{B}:H^1(\Omega)\rightarrow H^{1/2}(\partial \Omega)$ by extending it from the smooth function space \citep{mclean2000strongly, evans2022partial}. We have the following theorem:
\begin{theorem}\label{char of H12 norm inducing from H1}
  The $H^{1/2}(\partial \Omega)$ norm is equivalent to the following norm $\|\cdot\|_\star$ up to a constant:
  \begin{equation} 
    \|g\|_{\star} := \inf_{u\in H^1(\Omega), \mathcal{B} u = g} \|u\|_{H^1(\Omega)}, \quad \textrm{ for any } g\in \mathcal{B}(H^{1}(\Omega))=H^{1/2}(\partial \Omega).  \label{def: star norm}
  \end{equation}
  That is, there exist constants $C_{\Omega} > c_{\Omega} > 0$ only depending on $\Omega$ s.t. for any $g\in H^{1/2}(\partial \Omega),$
  \[ C_\Omega\|g\|_{\star} \geq \|g\|_{H^{1/2}(\partial \Omega)} \geq c_\Omega \|g\|_{\star}.
  \]
\end{theorem}
We refer the readers to Theorem 1.I in \citep{gagliardo1957caratterizzazioni} as a proof. More comprehensive discussions regarding this result can be found in \citep{mclean2000strongly, pechstein2013boundary}.}

{\color{black}
\subsection{Proof of Theorem \ref{thm: convergence analysis on elliptic PDE of divergence type}.}\label{append: subsec proof of refined thm}
As we focus on the Dirichlet boundary problem, we always treat $\mathcal B$ as the trace operator throughout this section. Before working on this proof, we slightly modify the definitions of several notations that are used in the previous proof of Theorem \ref{thm: convergence analysis of NPDG flow} for the current result. In this case, we have $\mathbb{H}=H^2(\Omega), \widetilde{\mathbb{H}}=H^1(\Omega; \mathbb{R}^d), \widetilde{\mathbb{K}}=H^1(\Omega;\mathbb{R}^d), \mathbb{K}=L^2(\Omega), \, \mathbb{K}_{\partial \Omega}=\mathcal X$ and $\widetilde{\mathbb{K}}^{test}=L^2(\Omega; \mathbb{R}^d), \mathbb{K}^{test}=H_0^1(\Omega), \mathbb{K}^{test}_{\partial \Omega}=\mathcal X.$ The operators $\mathcal M_p, \widetilde{\mathcal L}, \mathcal M_d$ are defined as 
\[ { \mathcal M_p}: \mathbb{H}\rightarrow \widetilde{ \mathbb{H} }, u \mapsto \sqrt{A(\cdot)}\nabla u(\cdot), \quad \widetilde{\mathcal L} = \mathrm{Id}: \widetilde{\mathbb{H}}\rightarrow \widetilde{\mathbb{K}}, \quad \mathcal M_d: \blue{\widetilde{\mathbb{K}}^{test}} \rightarrow \blue{\mathbb{K}^{test}}, \varphi\mapsto \sqrt{A(\cdot)}\nabla \varphi(\cdot).  \]
In this proof, we define $\mathbb{L}^2 := L^2(\Omega; \mathbb{R}^d)\times \mathcal X$. We keep the notations of $\textbf{U}_\theta, \boldsymbol{\Phi}_{\eta, \xi}$ as 
\begin{equation*}
\begin{split}
    \textbf{U}_\theta = \left(\begin{array}{c}
         \mathcal M_p (u_\theta - u_*) \\
         \sqrt{\lambda} \mathcal B  (u_\theta - u_*)
      \end{array}\right)\in\mathbb{L}^2, \quad  
      \boldsymbol{\Phi}_{ \eta,  \xi  } = \left(\begin{array}{c}
           \mathcal M_d \varphi_\eta \\
           \sqrt{\lambda} \psi_\xi
      \end{array}      
      \right)\in\mathbb{L}^2,
\end{split}
\end{equation*}
and define the preconditioning matrices 
\begin{align*}
  (M_p(\theta))_{ij} & = \Big\langle \frac{\partial \textbf{U}_\theta}{\partial \theta_i} , ~\frac{\partial \textbf{U}_\theta}{\partial \theta_j}\Big\rangle_{{ \mathbb{L}^2 }}, \quad 
  (M_d(\eta))_{ij} =  \Big\langle \frac{\partial \mathcal M_d \varphi_\eta}{\partial \eta_i} , ~ \frac{\partial \mathcal M_d \varphi_\eta}{\partial \eta_j}   \Big\rangle_{L^2(\Omega;  \mathbb{R}^d)},  \\ 
  (M_{bdd}(\xi))_{ij} & =  \Big\langle \frac{\partial (\sqrt{\lambda}\psi_\xi)}{\partial \xi_i}, \frac{\partial (\sqrt{ \lambda } \psi_\xi)}{\partial \xi_j } \Big\rangle_{\mathcal X}.
\end{align*}
Again, we assume the differentiability of $\textbf{U}_\theta, \boldsymbol{\Phi}_{\eta, \xi}$ w.r.t. parameters $\theta, \eta, \xi$; and $\frac{\partial }{\partial \theta} (\mathcal M_p (u_\theta - u_*)) \in \widetilde{\mathbb H}$, $\frac{\partial }{\partial \eta} (\mathcal M_d \varphi_\eta) \in \blue{\widetilde{\mathbb{K}}^{test}}$, and $\frac{\partial }{\partial \xi} (\sqrt{\lambda} \psi_\xi) \in \blue{\mathbb{K}^{test}_{\partial \Omega}}$ for arbitrary $\theta, \eta,$ and $\xi.$

Recall that $\Pi_{\partial \textbf{U}_\theta}:\mathbb{L}^2\rightarrow\mathbb{L}^2$ denotes the orthogonal projection onto $\mathrm{span}\{\partial_{\theta_k}\textbf{U}_\theta\}$ w.r.t. the $\mathbb{L}^2$ inner product, while similarly, $\Pi_{\partial \boldsymbol{\Phi}_{\eta, \xi}} = \Pi_{\partial \boldsymbol{\varphi}_\eta} \oplus \Pi_{\partial \boldsymbol{\psi}_\xi}$ with $\Pi_{\partial \boldsymbol{\varphi}_\eta}, \Pi_{\partial \boldsymbol{\psi}_\xi}$ denote the orthogonal projections onto $\partial \boldsymbol{\varphi}_\eta = \mathrm{span}\{\partial_{\eta_k}\mathcal M_d \varphi_\eta\}$ and $\mathrm{span}\{\partial_{\xi_k} \psi_\xi\}$ w.r.t. $L^2(\Omega;\mathbb{R}^d)$ and $\mathcal X$ inner products, respectively. 

Before we present the proof, we need the following two lemmas.
\begin{lemma}\label{lemma: elliptic eq assoc variational problem}
For a given $w\in H^2(\Omega)$, the following variational problem admits a unique minimizer $\widehat{\varphi}\in H_0^1(\Omega),$
\begin{equation} 
  \min_{\varphi \in H_0^1(\Omega)} \;  |\varphi - w|_{H^1(\Omega, A)}^2.  
  \label{variational problem assoc. proj}
\end{equation}
Denote $\mathcal T:H^2(\Omega)\rightarrow H_0^1(\Omega), w\mapsto \widehat{\varphi}$, then $\mathcal T$ is a linear operator. Furthermore, $\sqrt{A(\cdot)}(\nabla\widehat{\varphi}-\nabla w)$ is orthogonal to all $\sqrt{A(\cdot)}\nabla\phi$ with $\phi\in H_0^1(\Omega)$ w.r.t. $L^2(\Omega; \mathbb{R}^d)$ inner product. 
\end{lemma}
\begin{proof}
The existence and uniqueness of the minimizer of \eqref{variational problem assoc. proj} is a standard result in calculus of variations. Readers are referred to \citep{evans2022partial} (c.f. Theorem 2 and Theorem 3 in Chap. 8) for a proof. It can be verified that $\mathcal T w=\widehat{\varphi}$ is the weak solution to the Euler-Lagrange equation, which is a linear elliptic equation:
\begin{equation}
    -\nabla\cdot(A(x)\nabla\varphi(x)) = -\nabla\cdot(A(x)\nabla w(x)), ~ \textrm{on }\Omega, \quad  \varphi = 0 ~ \textrm{on }\partial \Omega. \label{EL of the variational problem}
\end{equation}
This yields that $\int_\Omega \nabla(\widehat{\varphi}(x)-w(x))^\top A(x) \nabla\phi(x) \dd x = 0$ for arbitrary $\phi\in H_0^1(\Omega).$ This verifies the orthogonality assertion of the Lemma.
Conversely, any weak solution $\varphi$ of \eqref{EL of the variational problem} is a minimizer of \eqref{variational problem assoc. proj}. See Section 8.2.3 of \citep{evans2022partial} for a detailed discussion. The equivalence between the minimizer of the variational problem and the solution to Euler-Lagrange equation verifies the linearity of $\mathcal T.$

\end{proof}

\begin{lemma}\label{lemma: bounding the proj norm with H12 norm}
  For arbitrary $w\in H^1(\Omega),$ denote $\widehat{\varphi} = \mathcal T w = \underset{\varphi\in H_0^1(\Omega)}{\mathrm{argmin}} \, |w - \varphi|_{ H_1(\Omega, A)},$ we have the inequality
  \begin{equation}
    |w - \widehat{\varphi}|_{ H_1(\Omega, A)} \leq \sqrt{\overline{A}} \|\mathcal{B} w\|_\star.
  \end{equation}
\end{lemma}
\begin{proof}
  Let us first consider the variational problem
  \begin{equation} 
    \min_{\phi\in H_0^1(\Omega)} \|\phi + w\|_{H^1(\Omega)}^2.  \label{variational problem for H12 norm}
  \end{equation}
  By using the similar arguments for proving Lemma \ref{lemma: elliptic eq assoc variational problem}, one can show that there exists unique $\phi_*\in H_0^1(\Omega)$ that minimizes \eqref{variational problem for H12 norm}. 

  On the other hand, it is straightforward to verify the equivalence between \eqref{variational problem for H12 norm} and \eqref{def: star norm}. Therefore, the optimal value for \eqref{variational problem for H12 norm} yields $\|\phi_* + w\|_{H^1(\Omega)} = \|\mathcal{B} w\|_\star.$ 

  Furthermore, we have 
  \begin{align*} 
    |w - \widehat{\varphi}|_{H^1(\Omega, A)}^2 = \min_{\varphi\in H_0^1(\Omega)} |w - \varphi|_{ H_1(\Omega, A)}^2 \leq |w + \phi_*|_{ H_1(\Omega, A)}^2, 
  \end{align*}
  and 
  \begin{equation*}
      |w + \phi_*|_{ H_1(\Omega, A)}^2 = \int_\Omega \nabla (w + \phi_*)^\top A(x) \nabla(w + \phi_*) \dd x \leq \overline{A}\|w+\phi_*\|_{H^1(\Omega)}^2 = \overline{A}\|\mathcal{B} w\|_{\star}^2.
  \end{equation*}
  Combining the two inequalities proves the assertion.

\end{proof}

We are now ready to prove the result:
\begin{proof}
  Recall that $\{(\theta_t, \eta_t, \xi_t)\}_{t\geq 0}$ denotes the solution to the NPDG flow associated with the functional $\mathscr{E}_{\mathcal X }(u_\theta, \varphi_\eta, \psi_\xi)$ defined in \eqref{def: new loss funcitional with modif bdd err}. We again consider the Lyapunov functional 
  \[ \mathcal I(\Utheta, \Phietaxi) = \frac12(\|\Utheta\|^2_{\mathbb{L}^2} + \|\Phietaxi\|_{\mathbb{L}^2}^2), \]
  defined in \eqref{Lyapunov func }. Using almost the identical derivation demonstrated in the proof of Theorem \ref{thm: convergence analysis of NPDG flow}, we obtain the time derivative of $\mathcal I(\Utheta, \Phietaxi)$ as
  \begin{equation}\label{decompose dI/dt for est}
  \begin{split}
      \frac{d}{dt} I(\Utheta, \Phietaxi) = &  (1) + (2)\\
                                         = & [(\textrm{A}) + (\textrm{R1}) + (\textrm{R2})] + [(\textrm{B}) + (\textrm{R3})],
  \end{split}
  \end{equation}
  where the terms (1) and (2) take the same forms as in \eqref{calculate dI dt}. They are further decomposed as (A)+(R1)+(R2) as in \eqref{estimation of term (1)} and (B)+(R3) as in \eqref{estimation of term (2)} respectively. 

  We first estimate the remainder terms (R1) and (R2) in \eqref{estimation of term (1)}. Instead of introducing the relative errors $\alpha, \beta_1, \beta_2$, we keep  the approximation errors in the present estimation. For (R1), we have
  \begin{equation}\label{est (R1)}
  \begin{split}
     (R1) \leq & \|\Pi_{\parUt^\perp}[\Uthetat]\|_{\mathbb{L}^2}\cdot \|\Phietaxit+\gamma\Pi_{\partial \Phietaxit}[\Uthetat - \blue{\nu}\Phietaxit]\|_{\mathbb{L}^2}\\
     \leq & \err(\mathbf U_{\theta_t} \!\! \mid \! \partial \mathbf U_{\theta_t}) 
     \cdot ((1+\gamma\blue{\nu})\|\Phietaxit\|_{\mathbb{L}^2} + \gamma \|\Uthetat\|_{\mathbb{L}^2})\\
     \leq & 2\;((1+\gamma\blue{\nu})\vee\gamma)\;\cdot\err(\mathbf U_{\theta_t} \!\! \mid \! \partial \mathbf U_{\theta_t}) \cdot \sqrt{\mathcal I(\Uthetat, \Phietaxit)}.
  \end{split}
  \end{equation}
  The first inequality is due to the Cauchy-Schwarz inequality and $\widetilde{\mathcal L}=\mathrm{Id}$.
  Here, we denote the approximation error 
  $\err(\mathbf U_\theta \mid \partial \mathbf U_\theta):=\|\Pi_{\partial \textbf{U}_\theta^\perp}[\textbf{U}_\theta]\|_{\mathbb{L}^2}$, which can be formulated as\footnote{\blue{For matrix $A\in \mathbb{R}^{d\times d},$ we denote the vector norm $\|\textbf{x}\|_{A}^2:=\textbf{x}^\top A \textbf{x}$ for $\textbf{x}\in\mathbb{R}^d$.}}
  \begin{align}\label{def: err paru to u}
    \err(\mathbf U_\theta \mid \partial \mathbf U_\theta)^2 = \min_{\zeta \in \mathbb{R}^{m_\theta}} \int_\Omega \|\nabla\langle \partial_\theta u_\theta(x), \zeta\rangle -\nabla(u_\theta-u_*)\|_{A(x)}^2 \, \dd x + \lambda \|\langle \partial_\theta u_\theta, \zeta\rangle - (u_\theta - u_*)\|_{\mathcal X}^2,
  \end{align}
  where we denote $\langle \partial_\theta u_\theta(x), \zeta\rangle = \sum_{k=1}^{m_\theta}\zeta_k \partial_{\theta_k}u_\theta(x)\in \parUt.$
  For the term (R2) in \eqref{estimation of term (1)}, we have
  \begin{equation}\label{est (R2)}
    \begin{split}
      (R2) & = \gamma \cdot \langle \Pi_{\parUt}[\Uthetat], \; \Pi_{\parPhit^\perp} [\Uthetat - \blue{\nu} \Phietaxit] \rangle_{\mathbb{L}^2} \\
      & = \gamma \cdot \langle \Pi_{\parUt}[\Uthetat], \; \Pi_{\parPhit^\perp} [\Uthetat] \rangle_{\mathbb{L}^2} - \gamma \blue{\nu} \cdot \langle \Pi_{\parUt}[\Uthetat], \; \Pi_{\parPhit^\perp} [\Phietaxit] \rangle_{\mathbb{L}^2}\\
      & \leq \gamma \cdot \|\Uthetat\|_{\mathbb{L}^2} \cdot \|\Pi_{\parPhit^\perp} [\Uthetat]\|_{\mathbbL} + \gamma\blue{\nu} \cdot \|\Uthetat\|_{\mathbbL} \cdot \|\Pi_{\parPhit}[\Phietaxit]\|_{\mathbbL}.
    \end{split}
  \end{equation}
  To estimate $\|\Pi_{\parPhit^\perp} [\Uthetat]\|_{\mathbbL}$ in the first term above, we have
  \begin{equation*}
    \|\Pi_{\parPhit^\perp} [\Uthetat]\|_{\mathbbL}^2 = \|\Pi_{\parvarphit^\perp} [\mathcal M_p (u_{\theta_t} - u_*)]\|_{L^2(\Omega; \mathbb{R}^d)}^2 + \|\Pi_{\parpsit^\perp}[u_\theta - u_*]\|_{\mathcal X}^2.
  \end{equation*}
  The estimation of $\|\Pi_{\parvarphit^\perp} [\mathcal M_p (u_{\theta_t} - u_*)]\|_{L^2(\Omega; \mathbb{R}^d)}^2$ requires more effort as it accounts for the approximation of using elements in $H_0^1(\Omega)$ to approximate the vector in $H^1(\Omega)$, thus yielding non-negligible discrepancy. To deal with this term, we can decompose
  \[ 
      \mathcal M_p (u_{\theta_t} - u_*) = \sqrt{A(\cdot)}\nabla(u_\theta - u_*) = \sqrt{A(\cdot)}(\nabla(u_\theta - u_*) - \nabla \widehat{\varphi}) + \sqrt{A(\cdot)} \nabla \widehat{\varphi}(\cdot).
  \]
  Here we denote $\widehat{\varphi} = \mathcal T(u_\theta-u_*)\in H_0^1(\Omega),$ with the operator $\mathcal T:H^1(\Omega)\rightarrow H_0^1(\Omega)$ defined in Lemma \ref{lemma: elliptic eq assoc variational problem}.

  Then we have 
  \begin{align} 
    \Pi_{\parvarphit^\perp} [\mathcal M_p(u_{\theta_t} - u_*)] = & \Pi_{\parvarphit^\perp}[\sqrt{A(\cdot)}(\nabla(u_\theta - u_*) - \nabla \widehat{\varphi})] + \Pi_{\parvarphit^\perp}[\sqrt{A(\cdot)}\nabla\widehat{\varphi}]\nonumber\\
    = & \sqrt{A(\cdot)}(\nabla(u_\theta - u_*) - \nabla \widehat{\varphi}) + \Pi_{\parvarphit^\perp}[\sqrt{A(\cdot)}\nabla\widehat{\varphi}].\label{decomposition of Pi_phi[Mp(u)]}
  \end{align}
  The second equality is due to Lemma \ref{lemma: elliptic eq assoc variational problem}, which asserts 
  that $\sqrt{A(\cdot)}(\nabla(u_\theta - u_*) - \nabla \widehat{\varphi})$ is orthogonal to each $\partial_{\eta_k} \mathcal M_p \varphi_{\eta_t}\in \parvarphit$ for $k=1, \dots, m_\eta$. 
  
  Furthermore, we have
  \begin{align*}  
    \Pi_{\parvarphit^\perp}[\sqrt{A(\cdot)}\nabla\widehat{\varphi}] = &  (\mathrm{Id}-\Pi_{\parvarphit})[\sqrt{A(\cdot)}\nabla\widehat{\varphi}]\\
    = & \sqrt{A(\cdot)}\nabla\widehat{\varphi} - \Pi_{\parvarphit}[\sqrt{A(\cdot)}\nabla\widehat{\varphi}]  \in\{\sqrt{A(\cdot)}\nabla \phi\,|\,\phi\in H_0^1(\Omega)\}.
  \end{align*}
  Thus, the two vectors in \eqref{decomposition of Pi_phi[Mp(u)]} are orthogonal to each other in $L^2(\Omega; \mathbb{R}^d)$ thanks to Lemma \ref{lemma: elliptic eq assoc variational problem}. We can then compute
  \begin{equation*}
      \|\Pi_{\parvarphit^\perp}[\mathcal M_p(u_{\theta_t}-u_*)]\|_{L^2(\Omega;\mathbb{R}^d)}^2 = |(u_{\theta_t} - u_*) - \widehat{\varphi}|_{H^1(\Omega, A)}^2 + \|\Pi_{\parvarphit^\perp}[\sqrt{A(\cdot)}\nabla\widehat{\varphi}]\|_{L^2(\Omega; \mathbb{R}^d)}^2.
  \end{equation*}
  Here we denote the seminorm $|\cdot|_{H^1(\Omega, A)}$ as defined in \eqref{def: seminorm dot H1(Omega, A)}. Lemma \ref{lemma: bounding the proj norm with H12 norm} leads to the estimation 
  \[ |(u_{\theta_t} - u_*) - \widehat{\varphi}|_{H^1(\Omega, A)} \leq \sqrt{\overline{A}}\|\mathcal{B} [u_{\theta_t} - u_*]\|_{\star} \leq \frac{\sqrt{\overline{A}}}{c_\Omega} \|u_{\theta_t} - u_*\|_{H^{1/2}(\partial \Omega)}. \] 
  The second inequality is due to Theorem \ref{char of H12 norm inducing from H1}.

  On the other hand, we denote 
  \begin{equation} \label{def: err parphi to u}
    \err(u_{\theta_t} \!\! \mid \!\parvarphit):=\|\Pi_{\parvarphit^\perp}[\sqrt{A(\cdot)}\nabla\widehat{\varphi}]\|_{L^2(\Omega; \mathbb{R}^d)} = \inf_{\zeta\in\mathbb{R}^{m_\eta}} \, |\langle \partial_\eta \varphi_{\eta_t}, \zeta \rangle - \widehat{\varphi}|_{H^1(\Omega, A)},
  \end{equation}
  with $\langle \partial_\eta \varphi_\eta(x), \zeta\rangle = \sum_{k=1}^{m_\eta}\zeta_k \partial_{\eta_k}\varphi_\eta(x)\in \parvarphit$.
  And analogously, 
  \begin{equation} 
    \err(\mathcal{B} u_{\theta_t} \!\! \mid \! \parpsit):= \|\Pi_{\parpsit^\perp}[u_{\theta_t} - u_*]\|_{\mathcal X} = \inf_{\zeta\in\mathbb{R}^{m_\xi}} \, \| \langle \partial_\xi \psi_{\xi_t}, \zeta \rangle - (u_{\theta_t} - u_*) \|_{\mathcal X},  \label{def: err bdd psi to u}
  \end{equation}
  with $\langle \partial_\xi \psi_\xi(x), \zeta\rangle = \sum_{k=1}^{m_\xi}\zeta_k \partial_{\xi_k}\psi_\xi(x)\in \parpsit$. 
  
  As a result, we can bound $\|\Pi_{\parPhit^\perp} [\Uthetat]\|_{\mathbbL}^2$ as
  \begin{equation} 
    \|\Pi_{\parPhit^\perp} [\Uthetat]\|_{\mathbbL}^2 \leq \left(\frac{\sqrt{\overline{A}}}{c_\Omega}\right)^2\|u_{\theta_t} - u_*\|^2_{H^{1/2}(\partial \Omega)} + \err(u_{\theta_t} \!\! \mid \!\parvarphit)^2 + \lambda \, \err(\mathcal{B} u_{\theta_t} \!\! \mid \! \parpsit)^2.   \label{bound Pi_Phi_U}
  \end{equation}

  We define
  \begin{align}
    \textrm{err}(\varphietat\!\!\mid\!\parvarphit) = &  \min_{\zeta\in \mathbb{R}^{m_\eta}} |\langle \partial_\eta \varphi_{\eta_t}, \zeta \rangle - \varphi_{\eta_t}|_{H^1(\Omega, A)}.
    \label{def: err parphi to phi} \\
    \textrm{err}(\psixit\!\!\mid\!\parpsit) = &  \min_{\zeta\in\mathbb{R}^{m_\xi}} \|\langle \partial_\xi \psi_{\xi_t }, \zeta \rangle - \psi_{\xi_t}\|_{\mathcal X}. \label{def: err parpsi to psi}
  \end{align}
  Then, $\|\Pi_{\parPhit}[\Phietaxit]\|_{\mathbb{L}^2}^2 = \textrm{err}(\varphietat\!\!\mid\!\parvarphit)^2 + \lambda \, \textrm{err}(\psixit\!\!\mid\!\parpsit)^2.$

  As a result, the remainder term (R2) in \eqref{est (R2)} can be bounded by
  \begin{equation}\label{final est (R2)}
  \begin{split}
    (R2) \leq \sqrt{2} \gamma (1+\blue{\nu}) \|\Uthetat\|_{\mathbbL}
    \cdot \Big( & \left(\frac{\sqrt{\overline{A}}}{c_\Omega}\right)^2\|u_{\theta_t} - u_*\|^2_{H^{1/2}(\partial \Omega)} + \err(u_{\theta_t} \!\! \mid \!\parvarphit)^2 \\ 
    & + \lambda \, \err(\mathcal{B} u_{\theta_t} \!\! \mid \! \parpsit)^2 + \textrm{err}(\varphietat\!\!\mid\!\parvarphit)^2 + \lambda \, \textrm{err}(\psixit\!\!\mid\!\parpsit)^2 \Big)^{\frac12}.
  \end{split}
  \end{equation}

  The remainder term (R3) in \eqref{estimation of term (2)} can be bounded using
  \begin{equation}\label{est (R3)}
  \begin{split}
    (R3) \leq &  \|\Pi_{\parPhit}[\Phietaxit]\|_{\mathbbL} \cdot \|\Uthetat - \blue{\nu} \Phietaxit\|_{\mathbbL} \\
    \leq & (\textrm{err}(\varphietat\!\!\mid\!\parvarphit)^2 + \lambda \, \textrm{err}(\psixit\!\!\mid\!\parpsit)^2 )^{\frac12}  \cdot  (\|\Uthetat\|_{\mathbbL} + \blue{\nu} \| \Phietaxit\|_{\mathbbL}) \\
    \leq & \frac{1 \vee \blue{\nu}}{2} (\textrm{err}(\varphietat\!\!\mid\!\parvarphit)^2 + \lambda \, \textrm{err}(\psixit\!\!\mid\!\parpsit)^2 )^{\frac12} \cdot \sqrt{\mathcal I(\Uthetat, \Phietaxit)}.
  \end{split}
  \end{equation}

  We have now established estimations for the remainder terms (R1), (R2), (R3) in \eqref{decompose dI/dt for est}. The term (A)+(B) can be estimated using the same technique presented in \eqref{est (A)+(B)}. However, before estimating (A)+(B), recall that $(\textrm{A}) = -\gamma\|\Uthetat\|_{\mathbbL}^2 - (1-\gamma\blue{\nu})\langle \Uthetat, \Phietaxit  \rangle_{\mathbbL},$
  we shall separate (A) as
  \[ (\textrm{A}) = \underbrace{-\frac{\gamma}{2}\|\Uthetat\|_{\mathbbL}^2 - (1-\gamma\blue{\nu})\langle \Uthetat, \Phietaxit  \rangle_{\mathbbL}}_{(\textrm{A'})} - \frac{\gamma}{2}\|\Uthetat\|_{\mathbbL}^2, \]
  where the $- \frac{\gamma}{2}\|\Uthetat\|_{\mathbbL}^2$ term will be used to offset the boundary term $\|u_{\theta_t} - u_*\|_{H^{1/2}(\partial \Omega)}$ arising in the estimation \eqref{final est (R2)} of (R2).
  
  Recall that $\widetilde{\mathcal L}=\mathrm{Id}$, hence $L_1=L_0=1,$ we have
  \begin{equation}\label{est (A')+(B)}
  \begin{split}
    (\textrm{A'})+(\textrm{B}) = \left[\|\ut\|_{{ \mathbb{L}^2 }}, \|\boldsymbol{\Phi}_{\eta_t, \xi_t}\|_{{ \mathbb{L}^2 }}\right] \left[ \begin{array}{cc}
     -\frac{\gamma}{2} & \frac{\gamma\blue{\nu}}{2}\\
     \frac{\gamma\blue{\nu}}{2} & - \blue{\nu} 
     \end{array} \right] \left[\begin{array}{c}
           \|\ut\|_{{ \mathbb{L}^2 }} \\
           \|\boldsymbol{\Phi}_{\eta_t, \xi_t }\|_{{ \mathbb{L}^2 }}
     \end{array}\right] \leq &  -r\cdot(\|\Uthetat\|_{\mathbbL}^2 + \|\Phietaxit\|_{\mathbbL}^2) \\
     = &  -2r\cdot \mathcal I(\Uthetat, \Phietaxit).  
 \end{split}
 \end{equation}
 \vspace{-0.1cm}
 Here we denote the larger eigenvalue of the above $2\times 2$ matrix as $-r$ with
 \begin{equation}
   r = \frac{\gamma + 2\blue{\nu}}{4} - \sqrt{\left(\frac{\gamma + 2\blue{\nu}}{4}\right)^2 - \left(\frac{\gamma\blue{\nu}}{2} - \frac{(\gamma\blue{\nu})^2}{4}\right)}.  \label{new convergence rate}
  \end{equation}
  \vspace{-0.1cm}
  One can verify that $r>0$ as long as $\gamma \blue{\nu}< 2.$ Furthermore, $r > \frac{\frac{\gamma\blue{\nu}}{2} - \frac{(\gamma\blue{\nu})^2}{4}}{2 (\frac{\gamma + 2\blue{\nu}}{4}) } = \frac12\cdot\frac{\gamma\blue{\nu}(2-\gamma\blue{\nu})}{\gamma+2\blue{\nu}}.$
  
  Finally, we combine our estimations \eqref{est (A')+(B)}, \eqref{est (R1)}, \eqref{est (R2)}, and \eqref{est (R3)} together to obtain
  \begin{equation}\label{total est}
    \begin{split}
      \frac{d}{dt}\mathcal I_t = &  [(\textrm{A'}) + (\textrm{B})] - \frac{\gamma}{2}\|\Uthetat\|_{\mathbbL}^2 + (\textrm{R1}) + (\textrm{R2}) + (\textrm{R3}) \\
      \leq & - 2r\cdot \mathcal I_t - \frac{\gamma}{2}\|\Uthetat\|_{\mathbbL}^2  \\
      & + 2\;((1+\gamma\blue{\nu})\vee\gamma)\;\cdot\err(\mathbf U_{\theta_t} \!\! \mid \! \partial \mathbf U_{\theta_t}) \cdot \sqrt{\mathcal I_t} \\
      & + \sqrt{2} \gamma (1+\blue{\nu}) \|\Uthetat\|_{\mathbbL}
        \cdot \Big( \left(\frac{\sqrt{\overline{A}}}{c_\Omega}\right)^2\|u_{\theta_t} - u_*\|^2_{H^{1/2}(\partial \Omega)} + \err(u_{\theta_t} \!\! \mid \!\parvarphit)^2 \\
      & \quad \quad \quad \quad \quad  + \lambda\,\err(\mathcal{B} u_{\theta_t} \!\! \mid \! \parpsit)^2 + \textrm{err}(\varphietat\!\!\mid\!\parvarphit)^2 + \lambda\,\textrm{err}(\psixit\!\!\mid\!\parpsit)^2 \Big)^{\frac12}\\
      & + \frac{1 \vee \blue{\nu}}{2} (\textrm{err}(\varphietat\!\!\mid\!\parvarphit)^2 + \lambda\,\textrm{err}(\psixit\!\!\mid\!\parpsit)^2 )^{\frac12} \cdot \sqrt{\mathcal I_t}\\
      \leq &  - 2r\cdot \mathcal I_t + \textrm{Err}(\theta_t, \eta_t, \xi_t, \gamma, \nu, \lambda) \cdot \sqrt{\mathcal I_t}  \\
      & - \frac{\gamma}{2}\|\Uthetat\|_{\mathbbL}^2 + \sqrt{2} \gamma (1+\blue{\nu}) \|\Uthetat\|_{\mathbbL} \cdot \frac{\sqrt{\overline{A}}}{c_\Omega}\|u_{\theta_t} - u_*\|_{H^{1/2}(\partial \Omega)},
    \end{split}
  \end{equation}
  where we denote $\mathcal I_t = \mathcal I(\Uthetat, \Phietaxit)$ for brevity. For the second inequality, we use the fact that $\|\Uthetat\|_{\mathbbL}\leq \sqrt{\mathcal I_t}$, and $(a^2+b^2)^{\frac12}\leq |a|+|b|$ for any $a, b\in\mathbb{R}$. The quantity $\textrm{Err}(\theta_t, \eta_t, \xi_t, \lambda)$ collects all the approximation errors
  \begin{equation}\label{def: Err(t)}
        \begin{split}
        \mathrm{Err}(\theta_t,\eta_t,\xi_t,\gamma,\nu,\lambda)
        = {} & 2\bigl((1+\gamma\blue{\nu})\vee\gamma\bigr)
              \cdot \err\!\bigl(\mathbf U_{\theta_t}\!\mid\!\partial\mathbf U_{\theta_t}\bigr)
        \\
        & + \sqrt{2}\,\gamma(1+\blue{\nu})
          \Bigl(
              \err\!\bigl(u_{\theta_t}\!\mid\!\parvarphit\bigr)^2
            + \lambda\,\err\!\bigl(\mathcal B u_{\theta_t}\!\mid\!\parpsit\bigr)^2
        \\[-0.2em]
        & \hspace{6.7em}
            + \err\!\bigl(\varphietat\!\mid\!\parvarphit\bigr)^2
            + \lambda\,\err\!\bigl(\psixit\!\mid\!\parpsit\bigr)^2
          \Bigr)^{\frac12}
        \\
        & + \frac{1\vee\blue{\nu}}{2}
          \Bigl(
              \err\!\bigl(\varphietat\!\mid\!\parvarphit\bigr)^2
            + \lambda\,\err\!\bigl(\psixit\!\mid\!\parpsit\bigr)^2
          \Bigr)^{\frac12}.
        \end{split}
  \end{equation}
  Recall that $\err(\mathbf U_{\theta_t} \!\! \mid \! \partial \mathbf U_{\theta_t}), \err(u_{\theta_t} \!\! \mid\! \parvarphit), \err(u_{\theta_t} \!\!\mid\! \parpsit), \textrm{err}(\varphietat\!\!\mid\!\parvarphit), \textrm{err}(\psixit\!\!\mid\!\parpsit)$ are defined in \eqref{def: err paru to u}, \eqref{def: err parphi to u}, \eqref{def: err bdd psi to u}, \eqref{def: err parphi to phi}, \eqref{def: err parpsi to psi} respectively. In the following discussion, we denote $\textrm{Err}(\theta_t, \eta_t, \xi_t, \gamma, \nu, \lambda)$ as $\textrm{Err}(\theta_t, \eta_t, \xi_t)$ for simplicity.
  
  The last terms in \eqref{total est} yield
  \begin{equation}\label{remaining bdd estimation}
      \begin{split}
          & - \frac{\gamma}{2}\|\Uthetat\|_{\mathbbL}^2 + \sqrt{2} \gamma (1+\blue{\nu}) \|\Uthetat\|_{\mathbbL} \cdot \frac{\sqrt{\overline{A}}}{c_\Omega}\|u_{\theta_t} - u_*\|_{H^{1/2}(\partial \Omega)} \\
          = & -\frac{\gamma\|\Uthetat\|_{\mathbbL}}{2}\left(\|\Uthetat\|_{\mathbbL} - 2\sqrt{2}\gamma(1+\blue{\nu})\frac{\sqrt{\overline{A}}}{c_\Omega}\|u_{\theta_t} - u_*\|_{H^{1/2}(\partial \Omega)}\right) \\
          \leq & -\frac{\gamma\|\Uthetat\|_{\mathbbL}}{2}\left(\sqrt{\lambda}\|u_{\theta_t} - u_*\|_{\mathcal X} - 2\sqrt{2}\gamma(1+\blue{\nu})\frac{\sqrt{\overline{A}}}{c_\Omega}\|u_{\theta_t} - u_*\|_{H^{1/2}(\partial \Omega)}\right).
      \end{split}
  \end{equation}
  Now recall that the boundary norm $\|\cdot\|_\mathcal X$ satisfies that for arbitrary $w\in H^1(\Omega),$
  \[ \|\mathcal B w\|_{\mathcal X} \geq C_\mathcal X\|\mathcal B w\|_{H^{1/2}(\partial \Omega)}, \]
  together with fact that
  \[ 
        \sqrt{\lambda} \geq \frac{2\sqrt{2}\gamma(1+\blue{\nu})\sqrt{\overline{A}}}{ C_\mathcal X \cdot c_\Omega},
  \]
  we know the quantity \eqref{remaining bdd estimation} $\leq 0$.

  In conclusion, \eqref{total est} yields 
  \[ \frac{d\mathcal I_t}{dt} \leq -2r\cdot \mathcal I_t + \textrm{Err}(\theta_t, \eta_t, \xi_t) \cdot \sqrt{\mathcal I_t}, \]
  which further leads to $\frac{d}{dt}(e^{2rt}\mathcal I_t)\leq e^{2rt}\cdot\textrm{Err}(\theta_t, \eta_t, \xi_t) \cdot \sqrt{\mathcal I_t}.$ 
  
  Now, denoting $\mathcal J_t = e^{2rt}\mathcal I_t$ yields 
  \[ \frac{d\mathcal J_t}{dt}\leq e^{rt}\cdot\textrm{Err}(\theta_t, \eta_t, \xi_t) \cdot \sqrt{\mathcal J_t}. \]
  As long as $\mathcal J_t > 0,$ we have
  \[ \frac{d}{dt}\sqrt{\mathcal J_t}\leq\frac{e^{rt}}{2}\textrm{Err}(\theta_t, \eta_t, \xi_t). \]
  Integration on the interval $[0, t]$ yields
  \begin{equation}\label{final bound}
    \mathcal I_t \leq  \left(e^{-rt}\sqrt{\mathcal I_0} + \int_0^t \frac{1}{2}e^{-r(t-\tau)} \textrm{Err}(\theta_\tau, \eta_\tau, \xi_\tau) \dd \tau \right)^2.
  \end{equation}
  This proves the result.

\end{proof}

\subsection{Proof of Corollary \ref{coro: a priori convergence analysis}}\label{append: coro proof}

\begin{proof}
  Notice that $u_\theta$ is the linear combination of basis functions $\{u_k\}_{k=1}^{m_\theta},$ 
  we have
  \begin{equation}
  \begin{split}
      \overline{A} \, \mathcal E_u  \geq & \inf_{\zeta\in\mathbb{R}^{m_\theta}} \int_\Omega \|\sum_{k=1}^{m_\theta}\zeta_k\nabla u_k - \nabla u_*\|_{A(x)}^2 \dd x + \lambda \|\sum_{k=1}^{m_\theta}\zeta_k\mathcal B u_k - g\|_\mathcal X^2 \\
      = & \inf_{\zeta\in\mathbb{R}^{m_\theta}} \int_\Omega \|\sum_{k=1}^{m_\theta}\zeta_k\nabla u_k - \nabla (u_\theta - u_*)\|_{A(x)}^2 \dd x + \lambda \|\sum_{k=1}^{m_\theta}\zeta_k\mathcal B u_k - (\mathcal B u_\theta - g)\|_\mathcal X^2\\
      \overset{\eqref{def: err paru to u}}{=} & \err(\mathbf U_{\theta_t} \!\! \mid \! \partial \mathbf U_{\theta_t})^2.
  \end{split} 
  \end{equation}
  This is 
  \begin{equation}
      \err(\mathbf U_{\theta_t} \!\! \mid \! \partial \mathbf U_{\theta_t})\leq \sqrt{\overline{A} \, \mathcal E_{u}}.
  \end{equation}
  Recall the approximation error $\err(u_{\theta_t}\!\! \mid\! \parvarphit)$ defined in \eqref{def: err parphi to u}. Notice that $\widehat{\varphi}=\mathcal T(u_{\theta_t} - u_*)$, as $\mathcal T$ is linear, this yields $\widehat{\varphi} = \sum_{k=1}^{m_\theta}\theta_k \mathcal T u_k - \mathcal T u_*$. We thus have
  \begin{equation*}
  \begin{split}
    \err(u_{\theta_t}\!\! \mid\! \parvarphit)^2 = & \inf_{\zeta\in\mathbb{R}^{m_\eta}} |\langle \partial_\eta \varphi_{\eta_t}, \zeta \rangle - \widehat{\varphi}|_{H^1(\Omega, A)}^2 \\
    = &  \inf_{\zeta\in\mathbb{R}^{m_\eta}} \int_\Omega \|\sum_{k=1}^{m_\eta}\zeta_k\nabla\varphi_k(x) - \left(\sum_{k=1}^{m_\theta}\theta_k \nabla \mathcal T(u_k)(x) - \nabla \mathcal T(u_*)(x)\right)\|^2_{A(x)} \dd x.
  \end{split}
  \end{equation*}
  Now, as we have assumed that $\mathrm{span}\{\mathcal T u_k\}\subseteq \mathrm{span}\{\varphi_k\},$ we obtain
  \begin{equation}
    \overline{A} \, \mathcal E_{\nabla\varphi} \geq \inf_{\zeta\in\mathbb{R}^{m_\eta}} \int_\Omega \|\sum_{k=1}^{m_\eta} \zeta_k \nabla\varphi_k(x) - \nabla \mathcal T(u_*)(x) \|^2_{A(x)}\dd x = \err(u_{\theta_t}\!\! \mid\! \parvarphit)^2,
  \end{equation}
  which leads to
  \begin{equation}
      \err(u_{\theta_t}\!\! \mid\! \parvarphit)\leq \sqrt{\overline{A} \, \mathcal E_{\nabla\varphi}}.
  \end{equation}
  Moreover, as $\mathrm{span}\{\mathcal B u_k\} \subseteq \mathrm{span}\{\psi_k\},$ we have 
  \begin{equation*}
    \err(u_{\theta_t} \!\!\mid\! \parpsit) = \sqrt{\mathcal E_\psi}.
  \end{equation*}
  Furthermore, as $\varphi_\theta \in \mathrm{span}\{\varphi_k\}, \psi_\xi \in \mathrm{span}\{\psi_k\}$, we obtain
  \begin{equation*}
      \textrm{err}(\varphietat\!\!\mid\!\parvarphit) = 0, \; \textrm{err}(\psixit\!\!\mid\!\parpsit) = 0.
  \end{equation*}
  Plugging these estimations into \eqref{refined thm bound} of Theorem \ref{thm: convergence analysis on elliptic PDE of divergence type} yields the result.
\end{proof}

}

{\color{black}
\subsection{Comparison with previous works}\label{apped: subsec_comp_prev_works}

The Lyapunov-based proof framework for the NPDG flow developed in this work is inspired by earlier studies in \citep{liu2022primal}, \citep{liu2024numerical}. Nevertheless, we shall clarify that there are several fundamental differences that distinguish the present theoretical analysis from these previous works.

First, the methods in \citep{liu2022primal, liu2024numerical} are formulated within a finite-difference setting, where the primal and dual variables \(u\) and \(\varphi\) are approximated by their values on discrete grid points. In contrast, our approach employs intact parametrized functions as computational models. This leads to substantially different convergence analyses: the former relies primarily on spectral estimations of the preconditioned matrices arising from spatial discretization, whereas the latter exploits the projection properties between primal and test functional spaces induced by natural gradient structures.

Moreover, in \citep{liu2022primal, liu2024numerical}, the discretized differential operator is incorporated into the preconditioner without splitting, namely, $\mathcal L = \mathcal M_d^*\widetilde{\mathcal L} \mathcal M_p$ with $\mathcal M_d = \mathrm{Id}, \widetilde{\mathcal L}=\mathrm{Id}, \mathcal M_p = \mathcal L,$ and the dual variable is allowed to vary freely without boundary constraints. By contrast, our framework accommodates general operator-splitting strategies based on integration by parts. This naturally enforces the choice of the test space \(\mathbb{K} = H_0^1(\Omega)\), leading to a more canonical formulation but also necessitating a more delicate treatment of boundary error term and a more refined convergence analysis as demonstrated in Theorem \ref{thm: convergence analysis on elliptic PDE of divergence type}.}

\section{Basic settings for the methods tested in Section \ref{sec: numerical examples }}

We provide the loss function, as well as the hyperparameters of the three methods PINN, Deep Ritz, and WAN tested in experiments in the following Table \ref{tab: setup for PINN DeepRitz WAN PDAdam}. In the following Table \ref{tab: real sol and norms}, we summarize the real solutions and their norms for equation \eqref{eq: Laplace}, \eqref{eq: VarCoeff} and \eqref{eq: nonlinear elliptic} tested in our experiments. 
{
\begin{table}[htb!]
{\tiny
\begin{center}
\begin{tabular}{|c|c|c|c|c|}
\hline
\multicolumn{2}{|c|}{}  &  PINN &  Deep Ritz    &  WAN/Primal-Dual using Adam  \\ \hline \hline 
\multirow{4}{*}{\begin{tabular}[c]{@{}c@{}} Poisson  \eqref{eq: Laplace} \\ ($d=50$) \end{tabular}}      &  \begin{tabular}[c]{@{}c@{}} loss \\ function \end{tabular}    &   \begin{tabular}[c]{@{}c@{}} $\int_\Omega |-\Delta u_\theta-f|^2\dd x$\\ $+\lambda \int_{\partial\Omega} |u_\theta-g|^2\dd\sigma$  \end{tabular}      &  \begin{tabular}[c]{@{}c@{}} $\int_\Omega \frac12\|\nabla u_\theta\|^2-fu_\theta \dd x$\\ $+\lambda \int_{\partial\Omega} |u_\theta-g|^2\dd\sigma$  \end{tabular}   & \begin{tabular}[c]{@{}c@{}} $\log\left( |\int_\Omega \nabla u_\theta\cdot\nabla\varphi_\eta - f \varphi_\eta \dd x |^2 \right) $\\ $ - \log \left( \int_\Omega  \varphi_\eta^2 \dd x \right) $  \\ $+\lambda \int_{\partial\Omega} |u_\theta-g|^2\dd\sigma$   \end{tabular}     \\ \cline{2-5} 
                    & $\lambda$ &      $10^4$   &  $10^4$       &   $10^4$  \\ \cline{2-5}
                    & $lr$ & $lr = 10^{-4}$         &  $lr = 10^{-4}$       &   \begin{tabular}[c]{@{}c@{}} $\tau_\theta = 0.5\cdot 10^{-3}$ \\  $\tau_\eta = 0.5 \cdot 10^{-2}$ \end{tabular}      \\ \cline{2-5} 
                    & $N_{iter}$ & \multicolumn{3}{c|}{\begin{tabular}[c]{@{}c@{}} Iterate till GPU time reaches $200$s ($d=10$)/$8000$s ($d=50$)  \\  \end{tabular} }\\ \cline{2-5}
                    &  $(N_{in}, N_{bdd})$ &      $(4000,80d)$   &  $(4000,80d)$       &  $(10000,60d)$ \\ \cline{2-5}
                    & NN & \multicolumn{3}{c|}{$u_\theta = \texttt{MLP}_{\mathrm{tanh}}(d, 256, 1, 6), \quad \varphi_\eta = \texttt{MLP}_{\mathrm{tanh}}(d, 256, 1, 6)\cdot \zeta$}\\  \hline 
\multirow{5}{*}{\begin{tabular}[c]{@{}c@{}} VarCoeff \eqref{eq: VarCoeff} \\ ($d=10, 20, 50$)  \end{tabular}} &  \begin{tabular}[c]{@{}c@{}} loss\\ function \end{tabular}  &  \begin{tabular}[c]{@{}c@{}} $\int_\Omega |-\nabla\cdot(\kappa \nabla u_\theta)-f|^2\,\dd x$\\ $+\lambda \int_{\partial\Omega} |u_\theta-g|^2\,\dd\sigma$  \end{tabular}    & \begin{tabular}[c]{@{}c@{}} $\int_\Omega \kappa\|\nabla u_\theta\|^2\,\dd x$\\ $+\lambda \int_{\partial\Omega} |u_\theta-g|^2\, \dd\sigma$       \end{tabular}      &  \begin{tabular}[c]{@{}c@{}} $\log\left(|\int_\Omega \kappa \nabla u_\theta\cdot\nabla \varphi_\eta\,\dd x|^2\right)$\\ $-\log\left(\int_\Omega \varphi_\eta^2\,\dd x\right)$ \\ $+\lambda \int_{\partial\Omega} |u_\theta-g|^2\,\dd\sigma$    \end{tabular}  \\ \cline{2-5} 
                          & $\lambda$ &      $10^4$   &  $10^3$       &   $10^4$  \\ \cline{2-5}
                          &  $lr$      &     $lr=10^{-4}$  &   $lr=0.5\cdot 10^{-3}$ 
                          & \begin{tabular}[c]{@{}c@{}} 
                          $(d=10)$ \begin{tabular}[c]{@{}c@{}} $\tau_\theta=0.5\cdot 10^{-2}$\\ $\tau_\eta=0.5\cdot 10^{-1}$ \end{tabular} 
                          \\ 
                          $(d=20, 50)$ \begin{tabular}[c]{@{}c@{}} $\tau_\theta=0.5\cdot 10^{-3}$\\ $\tau_\eta=0.5\cdot 10^{-2}$ \end{tabular}    \end{tabular} 
                          
                          \\ \cline{2-5}
                          &  \begin{tabular}[c]{@{}c@{}} $N_{iter}$  
                          \end{tabular}  &  \multicolumn{3}{c|}{Iterate till GPU time reaches $500$s ($d=10$)/$1500$s ($d=20$) }  \\  \cline{3-5}
                          &  \begin{tabular}[c]{@{}c@{}}   \end{tabular}  &  14000 ($d=50$)  &  10000 ($d=50$) &  12000 ($d=50$) \\ \cline{2-5} 
                           &  $(N_{in}, N_{bdd})$ &      $(4000,80d)$     &  $(4000,80d)$       &   $(4000,80d)$ 
                          \\ \cline{2-5}
                          &  NN  &  \multicolumn{3}{c|}{ 
                          \begin{tabular}{c} 
                          $u_\theta = \texttt{MLP}_{\mathrm{softplus}}(d, 256, 1, 4), \quad \varphi_\eta = \texttt{MLP}_{\mathrm{softplus}}(d, 256, 1, 4)\cdot \zeta$ for $d=10,20$\\
                          $u_\theta = \texttt{MLP}_{\mathrm{softplus}}(d, 256, 1, 6), \quad \varphi_\eta = \texttt{MLP}_{\mathrm{softplus}}(d, 256, 1, 6)\cdot \zeta$ for $d=50$
                          \end{tabular}
                          }  \\    \hline   
\multirow{5}{*}{\begin{tabular}[c]{@{}c@{}} Nonlinear Elliptic \\ \eqref{eq: nonlinear elliptic} $d=5$  \end{tabular} } &  \begin{tabular}[c]{@{}c@{}} loss\\ function \end{tabular}  & \begin{tabular}[c]{@{}c@{}} $\int_\Omega |\frac12 \|\nabla u_\theta\|^2 + V - \Delta u_\theta|^2\dd x$\\ $+\lambda \int_{\partial\Omega} u_\theta^2\dd\sigma$  \end{tabular} &  N.A.  &  \begin{tabular}[c]{@{}c@{}} $\log(|\int_\Omega \nabla u_\theta \cdot \nabla\varphi_\eta$ \\  $+  \frac12 \|\nabla u_\theta\|^2 \varphi_\eta + V \varphi_\eta \dd x|^2)$\\ $-\log\left( \int_\Omega \varphi_\eta^2 \dd x \right)$ \\ $ + \lambda \int_{\partial \Omega} u_\theta^2\dd\sigma $ \end{tabular}    \\ \cline{2-5}
                          & $\lambda$ & $10^4$ &  N.A.  &  $ 10^3 $  \\ \cline{2-5}
                          & $lr$ & $10^{-4}$ & N.A. & $0.5\cdot 10^{-3}, 0.5\cdot 10^{-2}$ \\ \cline{2-5}
                          &  $N_{iter}$  &  20000  &  N.A.  &  20000  \\    \cline{2-5}
                          &  $(N_{in}, N_{bdd})$ & $(4000, 40d)$ & N.A. & $(4000, 40d)$ \\ \cline{2-5}
                          & NN &  \multicolumn{3}{c|}{$u_\theta = \texttt{MLP}_{\tanh}(d, 256, 1, 4), \quad \varphi_\eta = \texttt{MLP}_{\tanh}(d, 256, 1, 4)\cdot \zeta$ } \\    \hline 
\end{tabular}
\end{center}
}
\caption{Loss functions and hyperparameters of the different methods tested in our experiments.}\label{tab: setup for PINN DeepRitz WAN PDAdam}
\end{table}}

\begin{table}[htb!]
\footnotesize
\centering
\renewcommand{\arraystretch}{1.2}
\setlength{\tabcolsep}{4pt}

\begin{tabular}{c|c|c|c|c}
\toprule
 & Domain $\Omega$ & Solution $u_*$
 & $\|u_*\|_{L^2(\Omega, \mu)}$
 & $\|\nabla u_*\|_{L^2(\Omega, \mu)}$ \\
\midrule

\makecell[c]{Poisson\\(50)}
& \makecell[c]{$[0,1]^d$\\ $|\Omega|=1$}
& $\sum_{k=1}^d \sin\!\bigl(\tfrac{\pi}{2}x_k\bigr)$
& \makecell[l]{$5d:\;3.2566$\\ $10d:\;6.4402$\\ $20d:\;12.8066$\\ $50d:\;31.9052$}
& \makecell[l]{$5d:\;2.4836$\\ $10d:\;3.5124$\\ $20d:\;4.9673$\\ $50d:\;7.8539$}
\\
\midrule

\makecell[c]{VarCoeff\\(51)}
& \makecell[c]{$[-1,1]^d$\\ $|\Omega|=2^d$}
& \makecell[c]{$\tfrac12 x^\top\Lambda^{-1}x$\\ $\lambda_0=1,\;\lambda_1=4$}
& \makecell[l]{$10d:\;1.0969$\\ $20d:\;2.1392$\\ $50d:\;5.2647$}
& \makecell[l]{$10d:\;1.4434$\\ $20d:\;2.0412$\\ $50d:\;3.2275$}
\\
\midrule

\makecell[c]{Nonlinear\\Elliptic\\(53)}
& \makecell[c]{$B_{d,3}$\\ $|\Omega|=\dfrac{\pi^{d/2}3^d}{\Gamma(\frac d2+1)}$}
& $\cos\!\bigl(\tfrac{\pi}{2}\|x\|\bigr)$
& $5d:\;0.6285$
& $5d:\;1.2218$
\\

\bottomrule
\end{tabular}

\caption{Solutions and their $L^2(\Omega, \mu)$ norms used for benchmarking.}
\label{tab: real sol and norms}
\end{table}

\section{Comparison among different methods}\label{append: compare among diff methods}
In the following Table \ref{tab: GPU time to accuracy }, we test four different methods with various step sizes on different equations. The step sizes used for each method are summarized below.
\begin{itemize}
    \item \textbf{NPDG} ($\tau_u, \tau_\varphi, \tau_\psi$): A.($1.5\cdot 10^{-1}, 1.5\cdot 10^{-1}, 1.5\cdot 10^{-1}$), B.($10^{-1}, 10^{-1}, 10^{-1}$), C.($0.5\cdot 10^{-1}, 0.95\cdot 10^{-1}, 0.95\cdot 10^{-1}$), D.($0.5\cdot 10^{-1}, 0.5\cdot 10^{-1}, 0.5\cdot 10^{-1}$); \\
    We fix $tol_{\textrm{MINRES}}=10^{-3}$ for $d=5, 10, 20$, and $tol_{\textrm{MINRES}}=10^{-4}$ for $d=50$.
    \item \textbf{PINN(Adam)} ($lr$): A.($0.5\cdot 10^{-2}$) B.($10^{-3}$)	C.($0.5\cdot 10^{-3}$) D.($10^{-4}$) E.($0.5\cdot 10^{-4}$);
    \item \textbf{DeepRitz} ($lr$): A.($0.5\cdot 10^{-2}$) B.($10^{-3}$)	C.($0.5\cdot 10^{-3}$) D.($10^{-4}$) E.($0.5\cdot 10^{-4}$);
    \item \textbf{WAN} ($\tau_\theta, \tau_\eta$): A.($0.5\cdot 10^{-2}, 0.5 \cdot 10^{-1}$), B.($10^{-3}, 10^{-2}$), C.($0.5\cdot 10^{-3}, 0.5 \cdot 10^{-2}$), D.($10^{-4}, 10^{-3}$), E.($0.5\cdot 10^{-4}, 0.5\cdot 10^{-3}$).
\end{itemize}
We record the computation time (seconds) spent by each method to achieve accuracy $\delta$ in Table \ref{tab: GPU time to accuracy }, we only present the time for the most efficient step size(s). \blue{When applying NPDG to VarCoeff, we adopt the $H^1(\partial \Omega, \mu_{\partial \Omega})$ boundary loss as described in Section \ref{subsec: varcoeff}.}

\begin{table}[htb!]
\label{tab: GPU time comparison PINN DeepRitz WAN}
{\small
\begin{center}
\begin{tabular}{|c|c|c|c|c|c|c|}
\hline
           Equ     & $\delta^*$ &  $d$     &  PINN(Adam) &  Deep Ritz    &  WAN  &{NPDG} \\ \hline \hline
\multirow{4}{*}{Poisson}  & \multirow{4}{*}{$0.005$}    &  5D  &    $26.22$ {\tiny(A)}    &  $\textbf{25.11}$ {\tiny (A)}   &  $51.14$ {\tiny (B)}  &  $68.87$ {\tiny  (A)}  \\ \cline{3-7} 
                    &   &  10D  &   $44.83$ {\tiny (A)}    &  $43.45$ {\tiny (B)}   & $51.65$ \tiny{(C)}  &  $\textbf{40.98}$ \tiny{(B)} \\ \cline{3-7} 
                    &      & 20D &  $160.82$   {\tiny (C)}     &  $183.49$ {\tiny (B)}       &   $460.12$ {\tiny (D)}   &  $\textbf{110.42}$ {\tiny (B)}  \\ \cline{3-7}
                    &         &       50D       &  $1989.06$ {\tiny (C)}   &  $1452.29$ {\tiny (B)}    &   $2117.24$ {\tiny (D)}  &      \textbf{821.24} {\tiny (C)}        \\ \hline
\multirow{6}{*}{VarCoeff} & \multirow{3}{*}{$0.01$} & 10D &  --  &  $\textbf{105.2}$ {\tiny (C)} &   --  &  $\blue{151.80}$ {\tiny (C) }     \\ \cline{3-7}
                       &    & 20D &    --     &  $\textbf{228.55}$ {\tiny (C)}  &  --    &  $\blue{309.05}$ {\tiny (C)}  \\ \cline{3-7}
                       &    &  50D     &     ${774.70}$ {\tiny (D)}     & --       &    --    &  \blue{\textbf{419.15}} {\tiny (C)}  \\ \cline{2-7}
                       & \multirow{3}{*}{$0.005$} & 10D &    --      &   --     &   --    & $\blue{\textbf{231.49}}$ { \tiny (C) } \\ \cline{3-7}
                       &    & 20D &    --     & --       &   --   & $\blue{\textbf{382.59}}$ {\tiny (C)} \\ \cline{3-7}
                       &    &  50D     &   --  &  --   &  --  & $\blue{\textbf{674.18}}$ {\tiny (C)} \\ \hline
\multirow{2}{*}{Nonlinear Elliptic}  &  $0.1$   & 5D & $2805.92$ {\tiny (B)} &    N.A.  &  $1130.76$ {\tiny (C)} 
&  $\textbf{1086.35}$  {\tiny (C)}  \\   \cline{2-7}
                                     &  $0.05$   & 5D & -- & N.A. &  --  & $\textbf{1894.89  }  $ {\tiny (C)}  \\  \hline
\end{tabular}
\end{center}
}
\caption{GPU time (seconds) spent by different methods upon achieving the designated accuracy $\delta$. The uppercase letters inside each parenthesis indicate the optimal learning rate(s) used in the algorithm. We apply the Monte-Carlo method with 
sample size $10^5$ to evaluate the relative $L^2$ error of $u_\theta$. ``--'' denotes that the method does not achieve the designated accuracy in a given time.\\}\label{tab: GPU time to accuracy }
\end{table}

{\color{black}
\section{Algorithmic details for the Allen-Cahn equation}\label{append: AC}
In this section, we provide more details on the NPDG algorithm for the Allen-Cahn equation discussed in Section \ref{subsec: AllenCahn}. Given the time implicit scheme of the equation, our goal is to solve the semi-linear elliptic equation \eqref{AC equ to elliptic equ with cubic term} at each time step.

{\color{black}
\subsection{Handling Different Regimes of $\epsilon_0$}\label{append: AC diff epsilon_0}
When the positive diffusion coefficient $\epsilon_0$ is bounded away from $0$, equation \eqref{AC equ to elliptic equ with cubic term} is dominated by the linear Laplacian operator. We can further tame the nonlinear term $W'(u)=u^3-u$ by subtracting its linear approximation at the equilibrium state $\bar{u}=\pm 1$, i.e., we consider $R(u) = W'(u)-(W'(\bar{u}) + W''(\bar{u})(u-\bar{u}))$. One can verify $W''(\bar{u})=W''(\pm 1)=2.$ We then absorb the linear term $W''(\bar{u})u$ of $W'(\bar{u}) + W''(\bar{u})(u-\bar{u})$ to the linear portion of \eqref{AC equ to elliptic equ with cubic term} to obtain
\begin{equation*} 
  \underbrace{((1+\frac{h_t W''(\bar{u})}{\epsilon_0})\mathrm{Id} - h_t\epsilon_0\Delta)}_{\mathcal D}u + \frac{h_t}{\epsilon_0}R(u) = u^{t-1} - \underbrace{\frac{h_t}{\epsilon_0} (W'(\bar{u}) - W''(\bar{u})\bar{u})}_{\mathrm{Const}}.
\end{equation*}
\vspace{-0.14cm}
It is reasonable to precondition on the linear differential operator $\mathcal D$ for this equation. We introduce the operators
\begin{equation*}
  \mathcal M_p = \mathcal M_d : u \mapsto \left( \begin{array}{c}
       \sqrt{1+h_t W''(\bar{u})/\epsilon_0}  u  \\
       \sqrt{\epsilon_0 h_t} \nabla u
  \end{array} \right).
\end{equation*}
It is not difficult to verify that $\langle \mathcal M_p u, \mathcal M_d \varphi \rangle_{L^2} = \langle \mathcal D u, \varphi \rangle_{L^2}$ for arbitrary $\varphi\in H_0^1(\Omega)$. Thus, we introduce $\varphi\in H^1_0(\Omega), \psi\in {L^2(\partial \Omega)} $ for the equation and its boundary condition and design the loss functional 
\begin{equation*}
\begin{split}
    \mathscr{E}(u,\varphi, \psi|~u^{t-1}) = 
        & \int_\Omega (u-u^{t-1} + \frac{h_t}{\epsilon_0}W'(u))\varphi + \epsilon_0 h_t \nabla u \cdot \nabla \varphi\,\dd\mu(x) \\
        & - \frac{\blue{\nu}}{2} \left( \left(1-\frac{h_t}{\epsilon_0}W''(\bar{u})\right) \int_\Omega  \varphi^2\,\dd\mu(x) - \epsilon_0 h_t \int_\Omega \|\nabla \varphi\|^2 ~ \dd\mu(x) \right) \\
        & + \lambda \left( \int_{\partial \Omega} \frac{\partial u}{\partial \textbf{n} } \psi\,\dd\mu_{\partial \Omega}(y) - \frac{\blue{\nu}}{2} \int_{\partial \Omega} \psi^2\,\dd\mu_{\partial \Omega} \right).
\end{split}
\end{equation*}
In practice, we found that it makes the optimization more stable if we add the residual loss
\begin{equation*}
    \mathscr{E}_{\textrm{Res}}(u|u^{t-1}) = \int_{\Omega} \Big| u - u^{t-1} - \epsilon_0 h_t \Delta u + \frac{h_t}{\epsilon_0}W'(u) \Big|^2\,\dd\mu(x) + \lambda \int_{\partial \Omega} \Big| \frac{\partial u}{\partial \textbf{n}} \Big|^2 \,\dd\mu_{\partial \Omega},  
\end{equation*}
as a regularization term to $\mathscr{E}(u,\varphi, \psi)$, and consider
\[ \widetilde{\mathscr{E}}(u,\varphi, \psi\,|\,u^{t-1}) = \mathscr{E}(u,\varphi, \psi\,|\,u^{t-1}) + \mathscr{E}_{\textrm{Res}}(u,\varphi, \psi\,|\,u^{t-1} ). \] 
Correspondingly, the precondition matrices are set as 
\begin{equation}\label{def: precond matrices for AC large epsilon}
\begin{split}
  M_p(\theta) = &  \left(1+\frac{h_t W''(\bar{u})}{\epsilon_0} \right) \int_\Omega {\frac{\partial u_\theta}{\partial \theta}}^\top  \frac{\partial u_\theta}{\partial \theta}\,\dd\mu(x)+{h_t\epsilon_0} \int_\Omega \frac{\partial}{\partial \theta}(\nabla u_\theta)^\top  \frac{\partial}{\partial\theta}(\nabla u_\theta)\,\dd\mu(x)   \\
  &  +  \lambda  \int_{\partial \Omega} \frac{\partial}{\partial \theta}(\partial_{\textbf{n}} u_\theta)^\top \frac{\partial}{\partial \theta}(\partial_{\textbf{n}} u_\theta)\,\dd\mu_{\partial \Omega}(y),  
\end{split}
\end{equation}
\[ M_d(\eta) = {\left(1+\frac{h_t W''(\bar{u})}{\epsilon_0} \right)} \int_\Omega \frac{\partial \varphi_\eta}{\partial \eta}^\top  \frac{\partial \varphi_\eta}{\partial \eta}\,\dd\mu(x) + h_t\epsilon_0  \int_\Omega \frac{\partial}{\partial \eta}(\nabla \varphi_\eta)^\top \frac{\partial}{\partial\eta}(\nabla \varphi_\eta)\,\dd\mu(x), \]
\[ M_{bdd}(\xi) = {\lambda}  \int_{\partial \Omega}  {\frac{\partial}{\partial \xi} \psi_\xi}^\top \frac{\partial}{\partial \xi}\psi_\xi \, \dd\mu_{\partial \Omega}(y). \]
We apply this treatment to both the 1D and 2D examples in Section \ref{subsec: AllenCahn} with $\epsilon_0=0.1$.}

In the presence of strong reaction and weak diffusion, the parameter $\epsilon_0$ approaches $0$. Equation \eqref{AC equ to elliptic equ with cubic term} is dominated by the nonlinear term $\frac{1}{\epsilon_0}W'(u)$. Under such regime, we change our strategy and consider the test functions $\varphi\in L^2(\Omega), \psi\in L^2(\partial \Omega)$, together with the functional
\begin{equation}
\begin{split}
    \mathscr{E}(u, \varphi, \psi \, |\, u^{t-1}) = &  \int_\Omega (u-u^{t-1}-\epsilon_0h_t\Delta u+\frac{h_t}{\epsilon_0}W'(u))\varphi\,\dd\mu_\Omega(x) - \frac{\nu}{2}\int_\Omega\varphi^2\dd\mu_\Omega(x) \\
    & + \lambda \left( \int_{\partial \Omega} \frac{\partial u}{\partial \textbf{n} } \psi\,\dd\mu_{\partial \Omega}(y) - \frac{\blue{\nu}}{2} \int_{\partial \Omega} \psi^2\,\dd\mu_{\partial \Omega} \right)
\end{split}
\end{equation}
By dropping the Laplacian term in the Jacobian of $u-u^{t-1}-\epsilon_0h_t\Delta u+\frac{h_t}{\epsilon_0}W'(u)$, we approximate the Jacobian operator using $\mathcal G:=\mathrm{Id}+\frac{h_t}{\epsilon_0}W''(u):u\mapsto u + \frac{h_t}{\epsilon_0}W''(u)u.$ By incorporating $\mathcal G$ as the preconditioning operator, we obtain
\begin{equation}
  M_p(\theta) = \int_\Omega (1+\frac{h_t}{\epsilon_0}W''(u_\theta))^2\frac{\partial u_\theta}{\partial\theta}^\top\frac{\partial u_\theta}{\partial\theta}\,\dd \mu_\Omega(x) + \lambda  \int_{\partial \Omega} \frac{\partial}{\partial \theta}(\partial_{\textbf{n}} u_\theta)^\top \frac{\partial}{\partial \theta}(\partial_{\textbf{n}} u_\theta)\,\dd\mu_{\partial \Omega}(y).
\end{equation}
Meanwhile, the preconditioning matrices \( M_d(\eta) \) and \( M_{\mathrm{bdd}}(\xi) \) are derived by considering the identity operators on \( L^2(\Omega) \) and \( L^2(\partial \Omega) \). Consequently, $M_d(\eta)\;=\;\int_\Omega\frac{\partial \varphi_\eta}{\partial \eta}^{\!\top} \frac{\partial \varphi_\eta}{\partial \eta}\,\dd \mu_\Omega,$ while \( M_{\mathrm{bdd}}(\xi) \) remains the same as defined in \eqref{def: precond matrices for AC large epsilon}.

This treatment is adopted in the 1D example presented in
Section~\ref{subsec: AllenCahn} with $\epsilon_0 = 0.01$. An advantage of this treatment is that it avoids integration by parts, thereby allowing greater flexibility in the choice of the error-adaptive measure $\mu_\Omega$. In our implementation, we take $\mu_\Omega \;=\; \tfrac12\,\mathcal U[0,2] \;+\; \tfrac12\,\mathcal U[0.8,1.2],$ 
which places additional sampling weight near $x=1$ and thus enables a more accurate resolution of the solution profile in that region. Here, $\mathcal U[a,b]$ denotes the uniform distribution on the interval $[a,b]$.}

\subsection{Benchmark solution \& Comparison}\label{append: benchmark}
We use the numerical solution $\{U^k\}_{k=1}^{N_t}$ solved from the following time-implicit, finite difference scheme
\begin{align} 
   & \frac{U^k_i-U^{k-1}_i}{h_t} = \epsilon_0\frac{U^k_{i+1}-2U^k_{i}+U^k_{i-1}}{h_x^2} - \frac{1}{\epsilon_0}({U^k_i}^3 - U^k_i), \label{implicit, FD} \\
   & U_{-1}^k=U_0^k,~ U_{N_x+1}^k=U_{N_x}^k, \quad \forall ~0 \leq i \leq N_x, \quad \textrm{for }~ 1\leq k\leq N_t, \nonumber
\end{align}
as the benchmark for $u_\theta$ computed from the NPDG algorithm. In our computation, we set $N_x=400$, $h_x=2/N_x$, $U^0_i=u_0(\frac{2i}{N_x})$.

For the 2D example, in Figure \ref{fig: AC2}, we plot the graphs of the neural network solution $u_{\theta_k}$ together with the numerical solution $\{U_{ij}^k\}$ obtained via the time-implicit finite difference scheme. The semi-log curves for $\sqrt{\textrm{MSE loss}}$ versus training time are provided in Figure \ref{fig: AC2}. Further comparison plots and the heatmaps of the pointwise error $|u_{\theta_k}(\cdot)-U^k|$ are presented in Figure~\ref{fig: AC3}.
\begin{figure}[htb!]
    \centering
    \begin{subfigure}[b]{0.26\textwidth}
        \includegraphics[width=\textwidth]{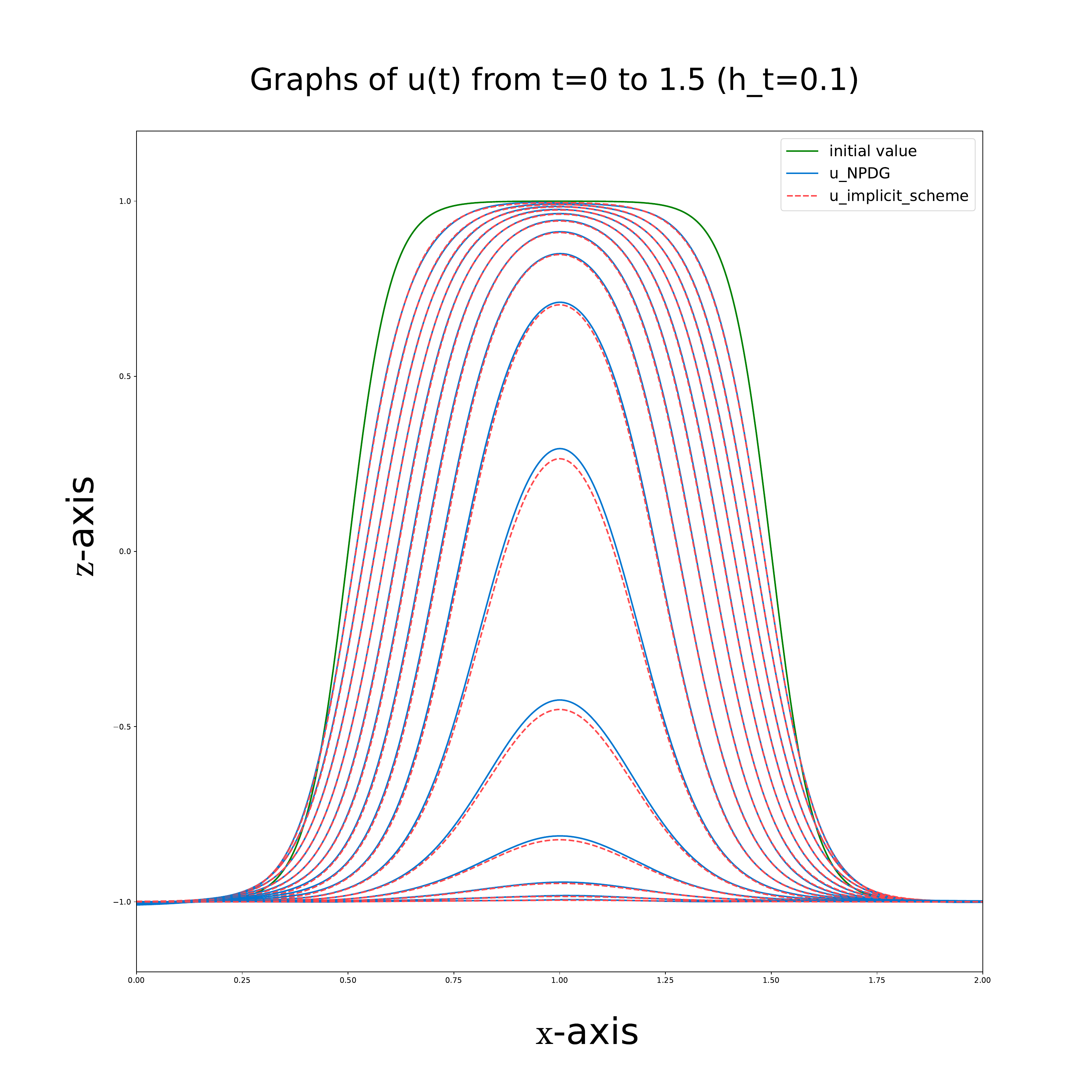}
        \caption{}\label{subfig: AC2 graphs on dim0}
    \end{subfigure}
        \begin{subfigure}[b]{0.26\textwidth}
        \includegraphics[width=\textwidth]{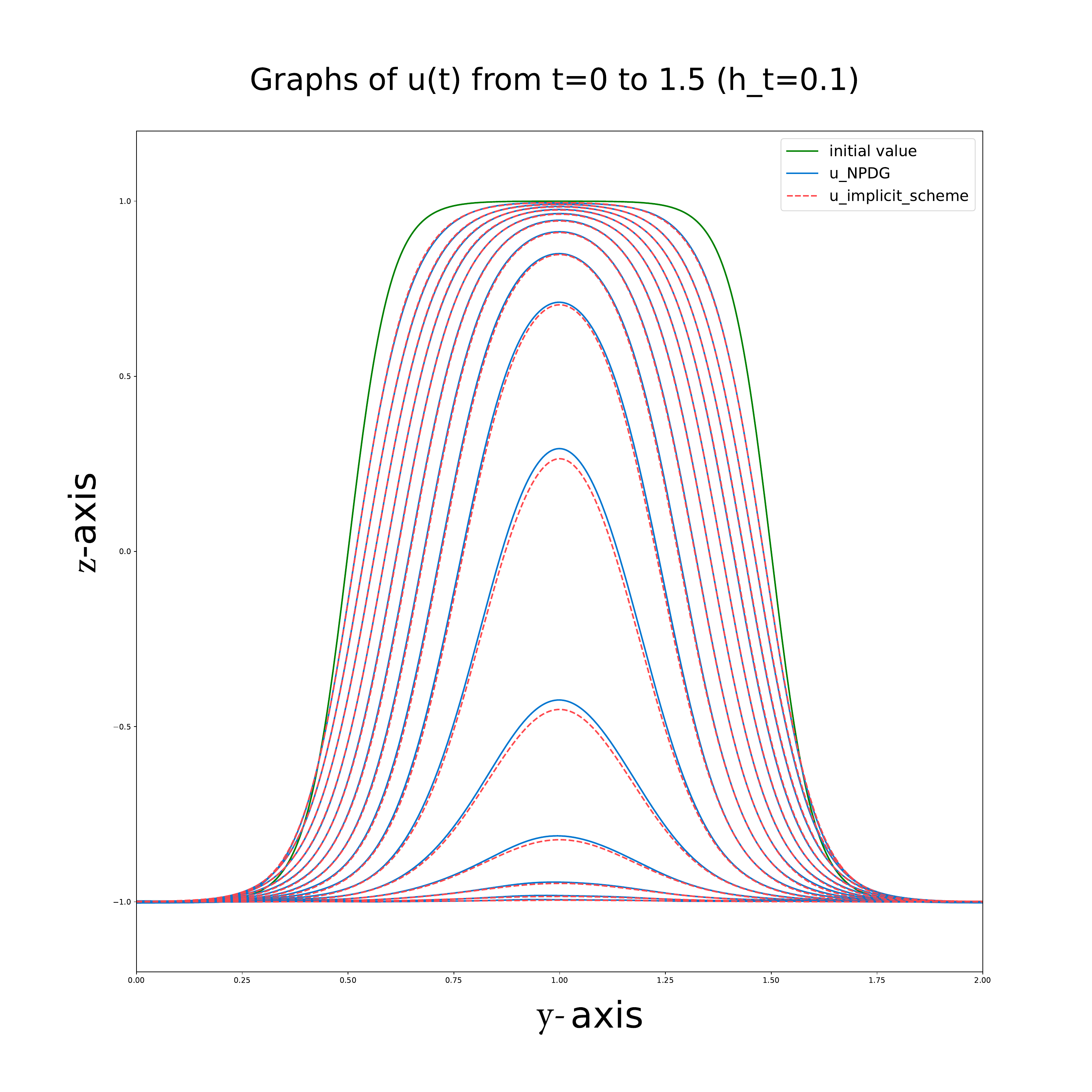}
        \caption{}\label{subfig: AC2 graphs on dim1}
    \end{subfigure}
        \begin{subfigure}[b]{0.26\textwidth}
        \includegraphics[width=\textwidth]{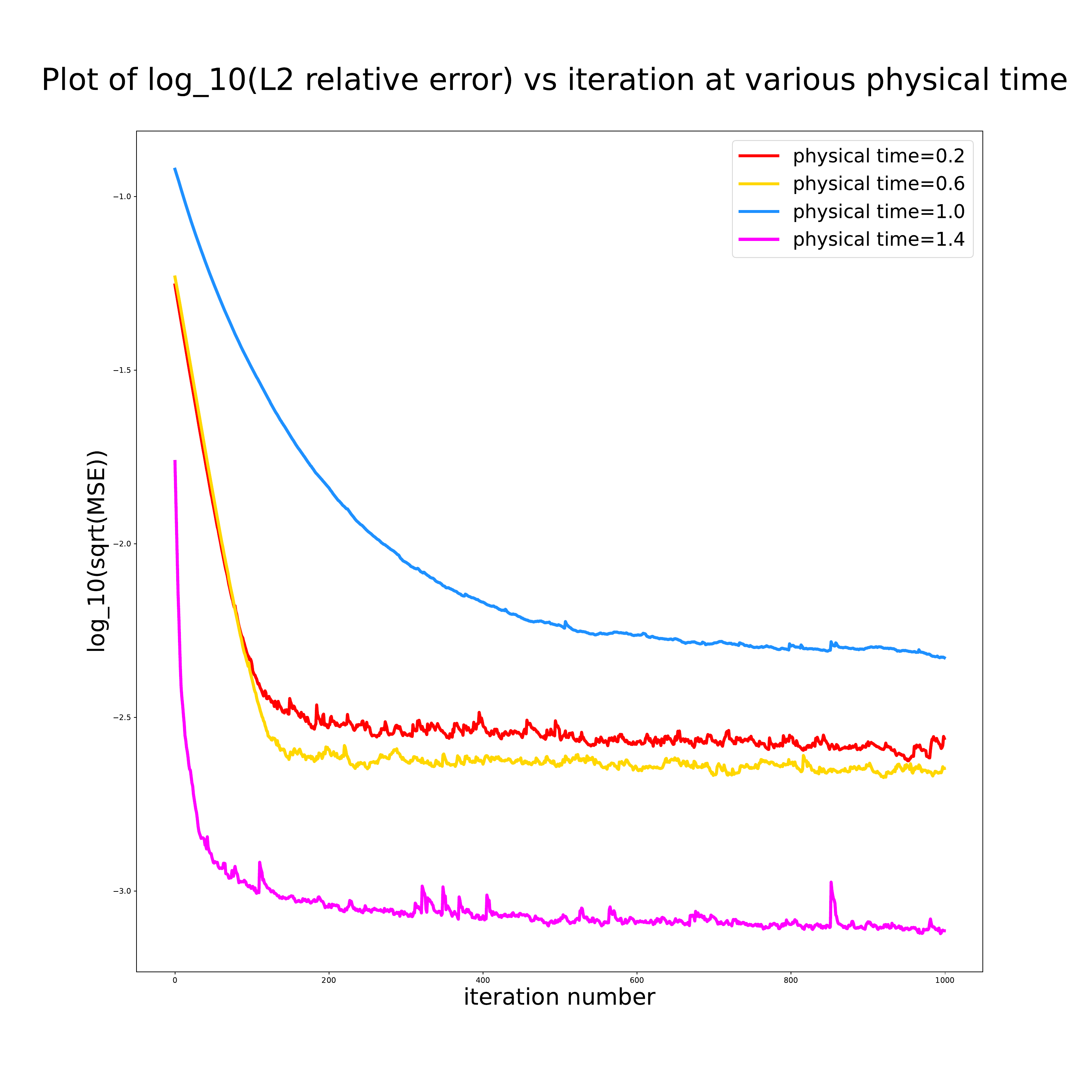}
        \caption{}\label{subfig: log MSE vs iteration}
    \end{subfigure}
    \caption{\ref{subfig: AC2 graphs on dim0} \& \ref{subfig: AC2 graphs on dim1}: Comparison of neural network solution $u_\theta(\cdot)$ (blue) and finite difference solution $U^k$ (red) along the $x$ and $y$ axis at time $t_k$, $1\leq k \leq 15$. \ref{subfig: log MSE vs iteration}: Semi-log plots of $\sqrt{\textrm{MSE loss}}$ vs computation time (seconds) at physical time $0.2, 0.6, 1.0, 1.4$. }
    \label{fig: AC2}
\end{figure}
\begin{figure}[htb!]
    \centering
    \begin{subfigure}[b]{0.18\textwidth}
        \includegraphics[trim={0 0 0 8cm},clip, width=\textwidth]{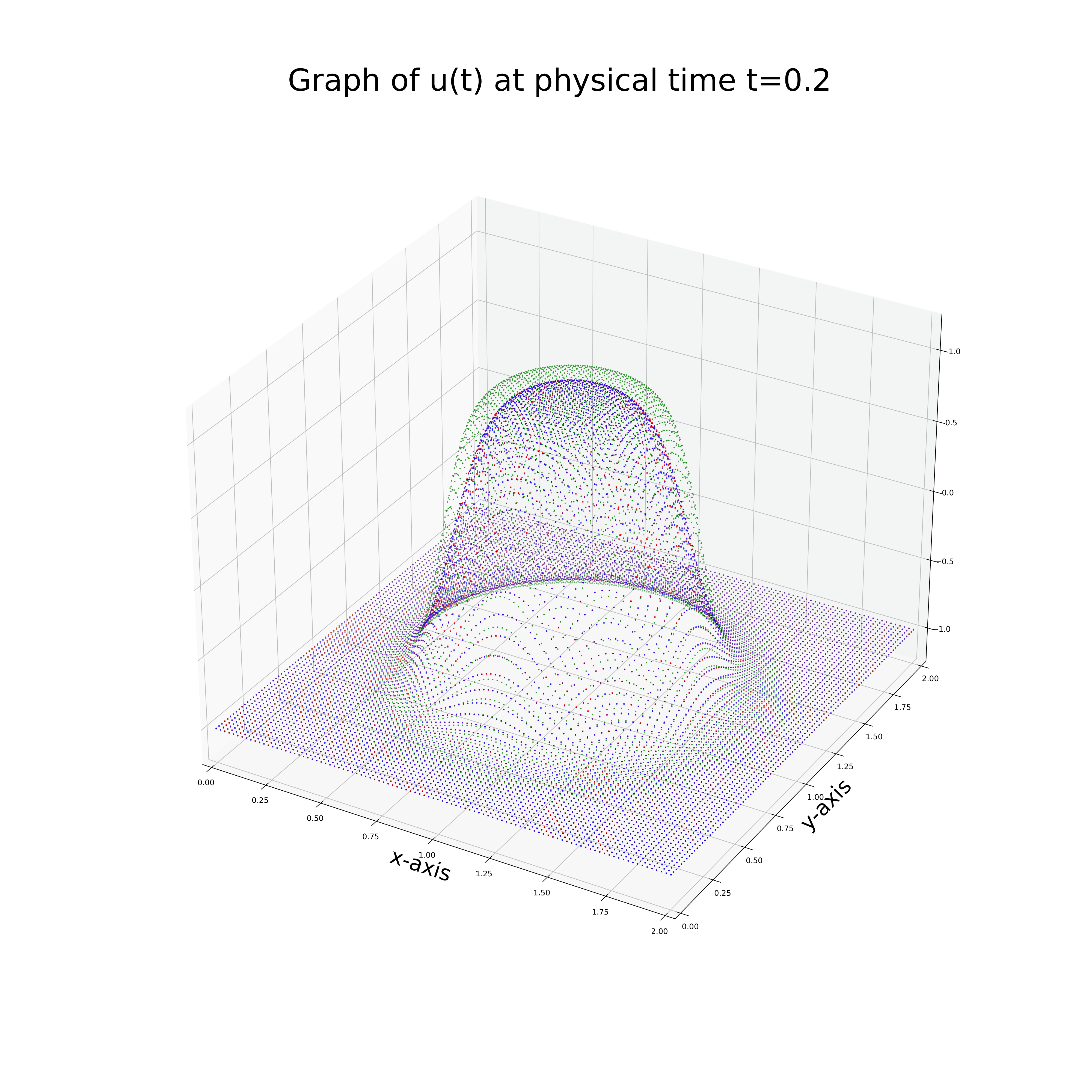}
    \end{subfigure}
        \begin{subfigure}[b]{0.18\textwidth}
        \includegraphics[trim={0 0 0 8cm},clip, width=\textwidth]{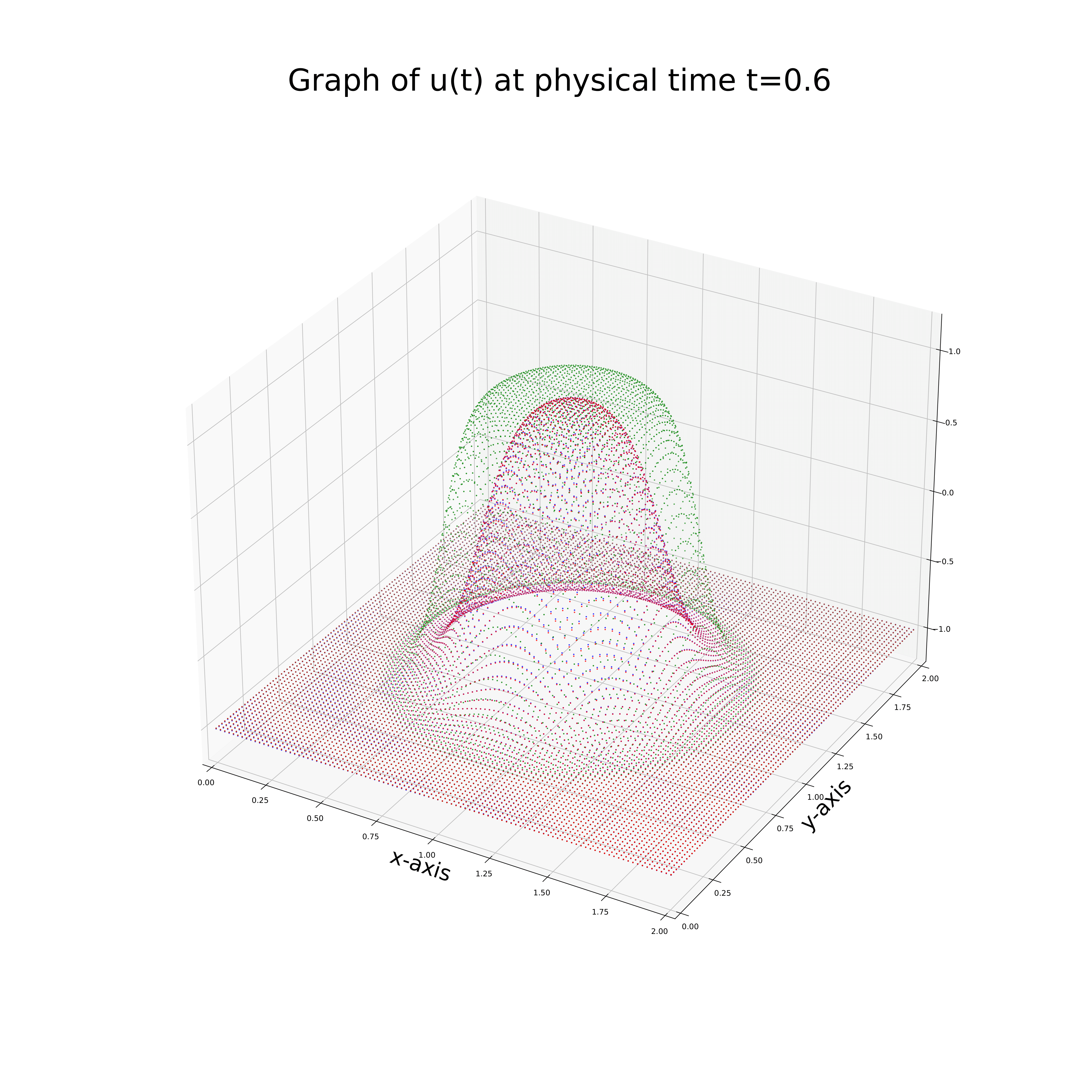}
    \end{subfigure}
        \begin{subfigure}[b]{0.18\textwidth}
        \includegraphics[trim={0 0 0 8cm},clip, width=\textwidth]{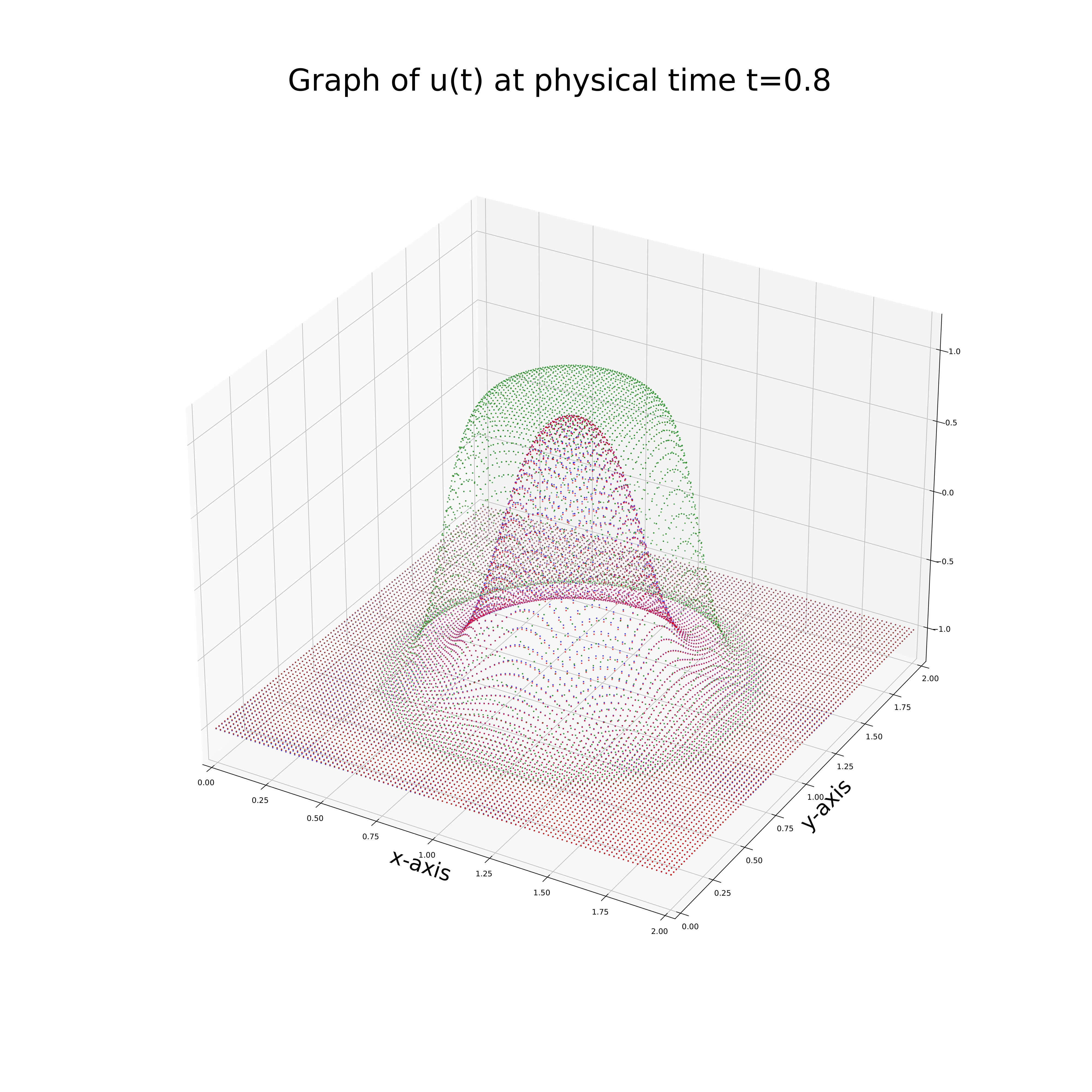}
    \end{subfigure}
        \begin{subfigure}[b]{0.18\textwidth}
        \includegraphics[trim={0 0 0 8cm},clip, width=\textwidth]{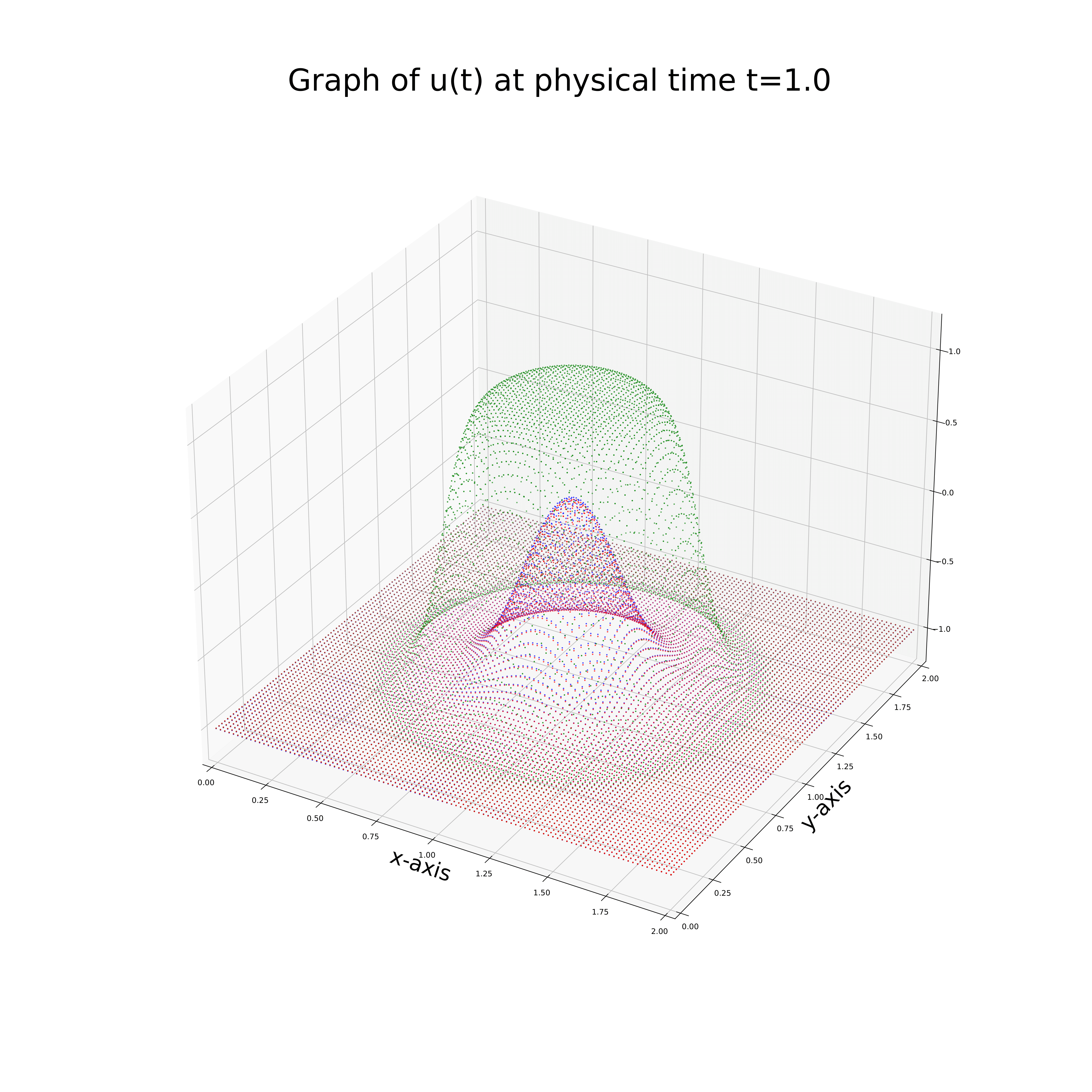}
    \end{subfigure}
        \begin{subfigure}[b]{0.18\textwidth}
        \includegraphics[trim={0 0 0 8cm},clip, width=\textwidth]{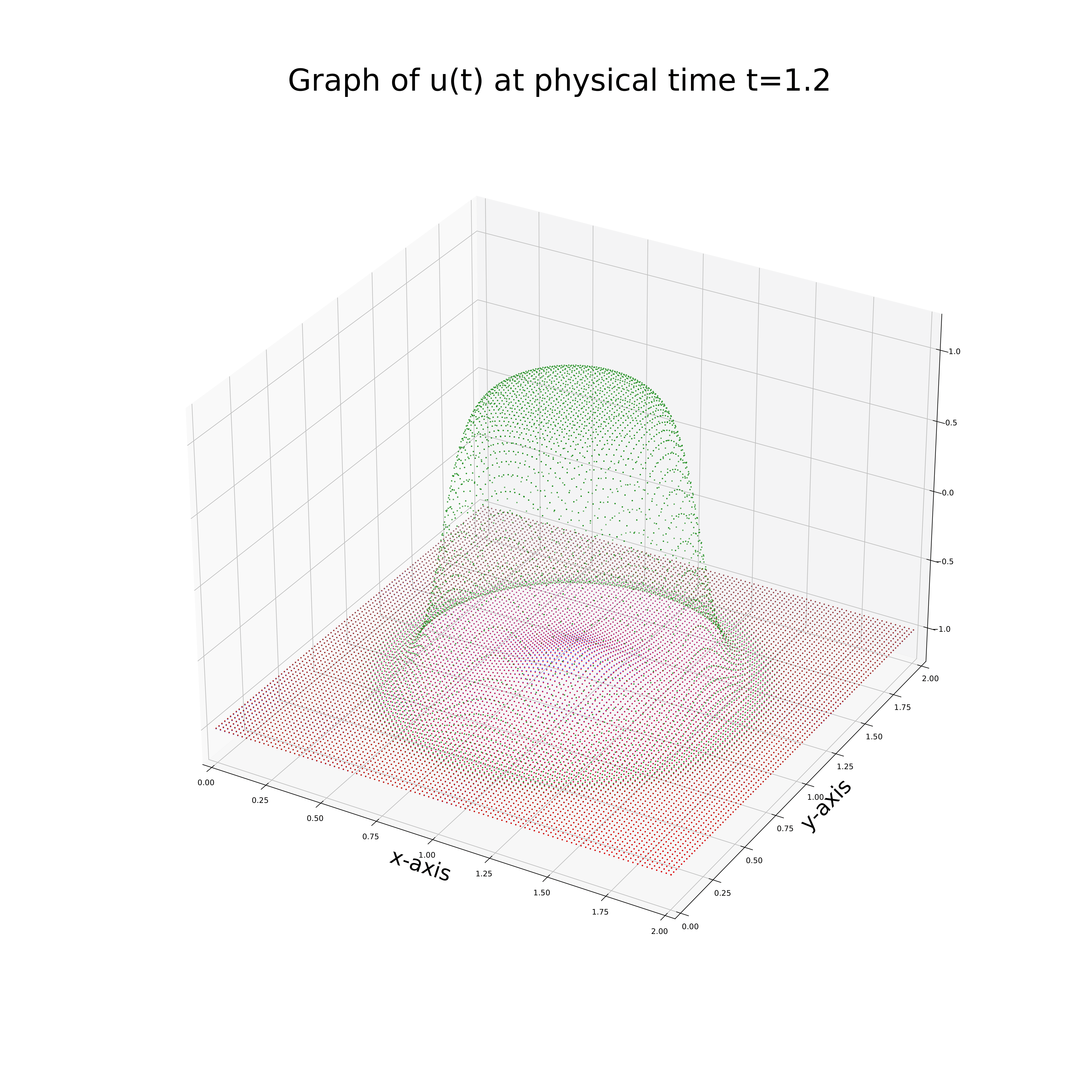}
    \end{subfigure}
    \begin{subfigure}[b]{0.18\textwidth}
        \includegraphics[trim={0 0 0 4.5cm},clip, width=\textwidth]{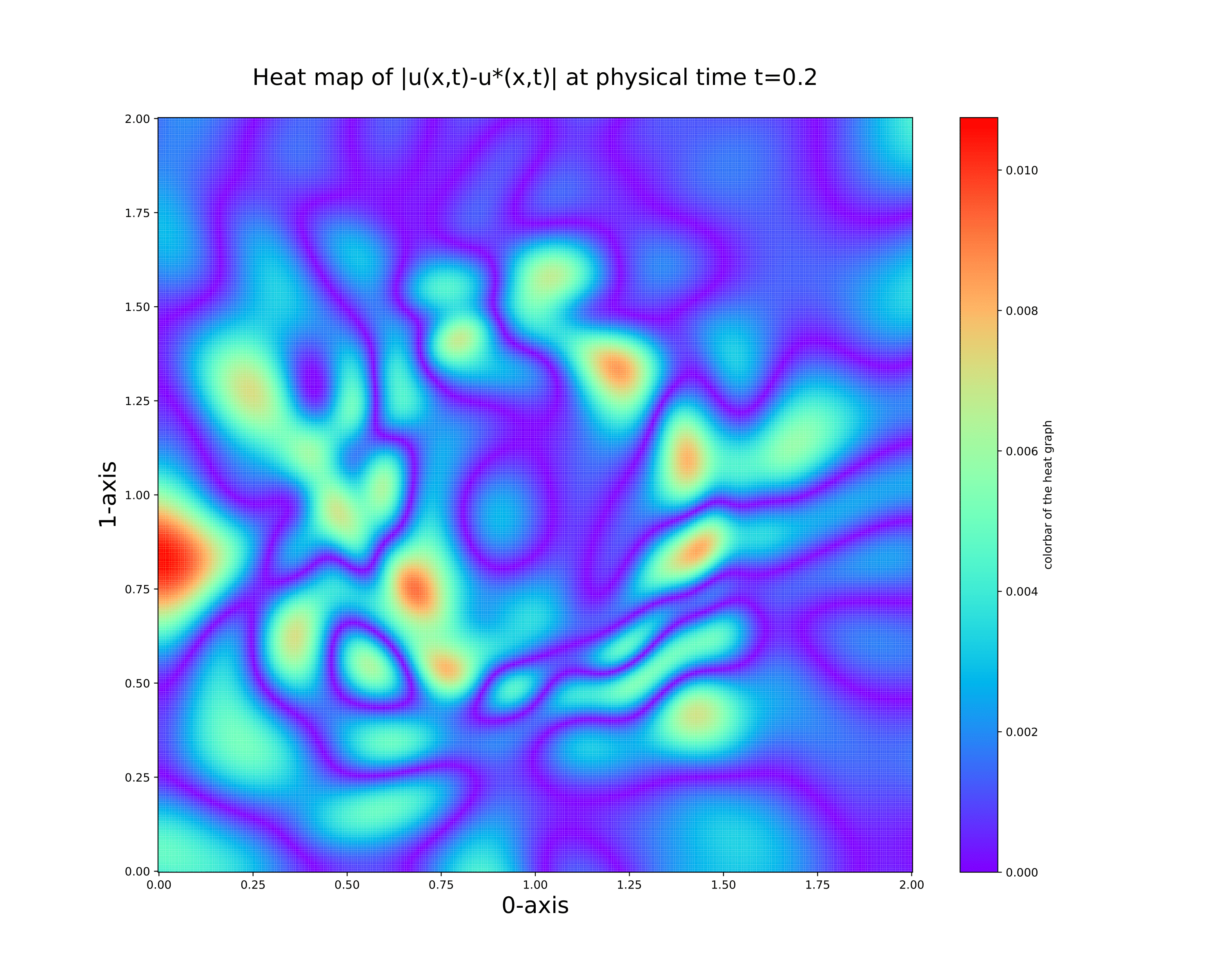}
    \end{subfigure}
        \begin{subfigure}[b]{0.18\textwidth}
        \includegraphics[trim={0 0 0 4.5cm},clip, width=\textwidth]{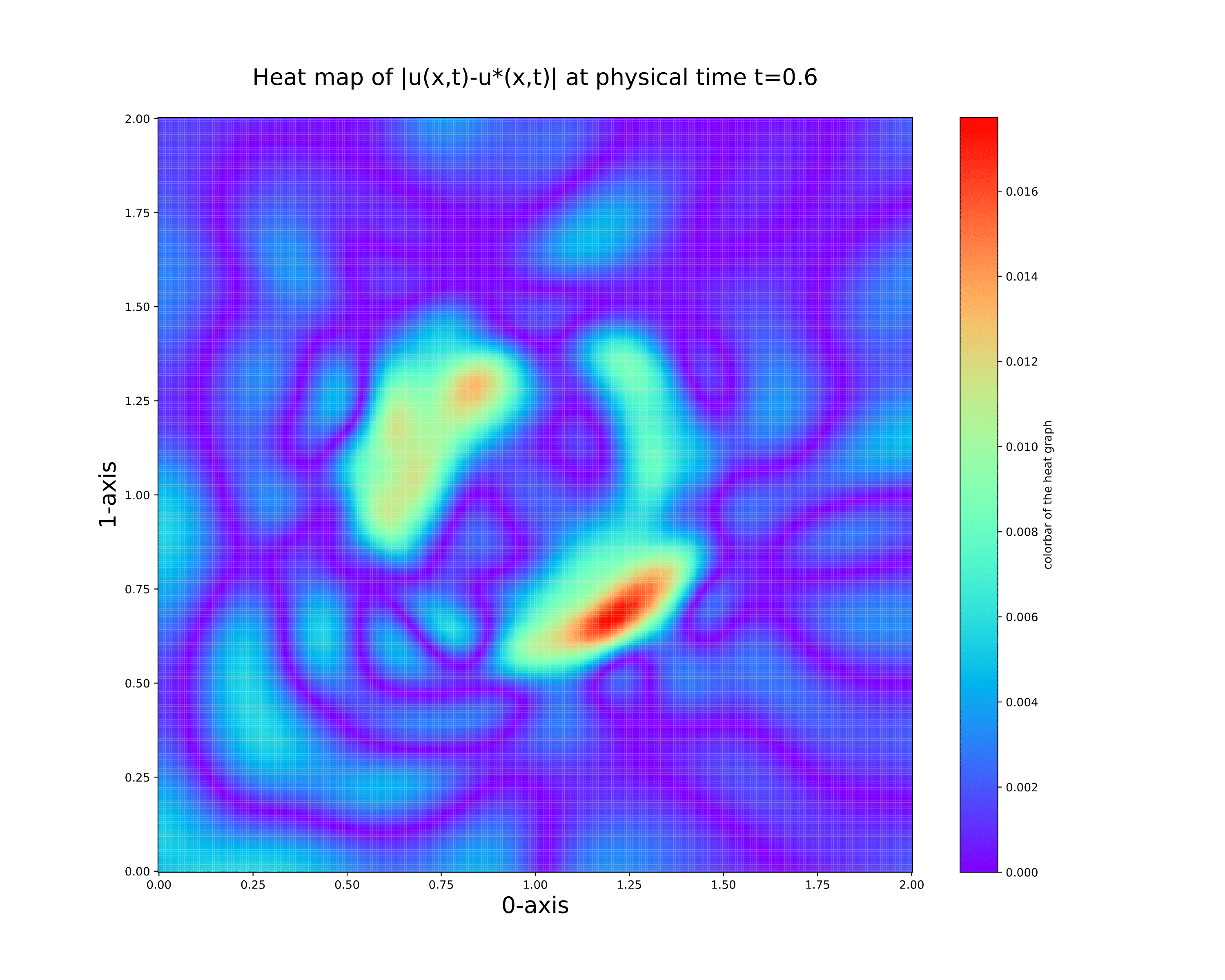}
    \end{subfigure}
        \begin{subfigure}[b]{0.18\textwidth}
        \includegraphics[trim={0 0 0 4.5cm},clip, width=\textwidth]{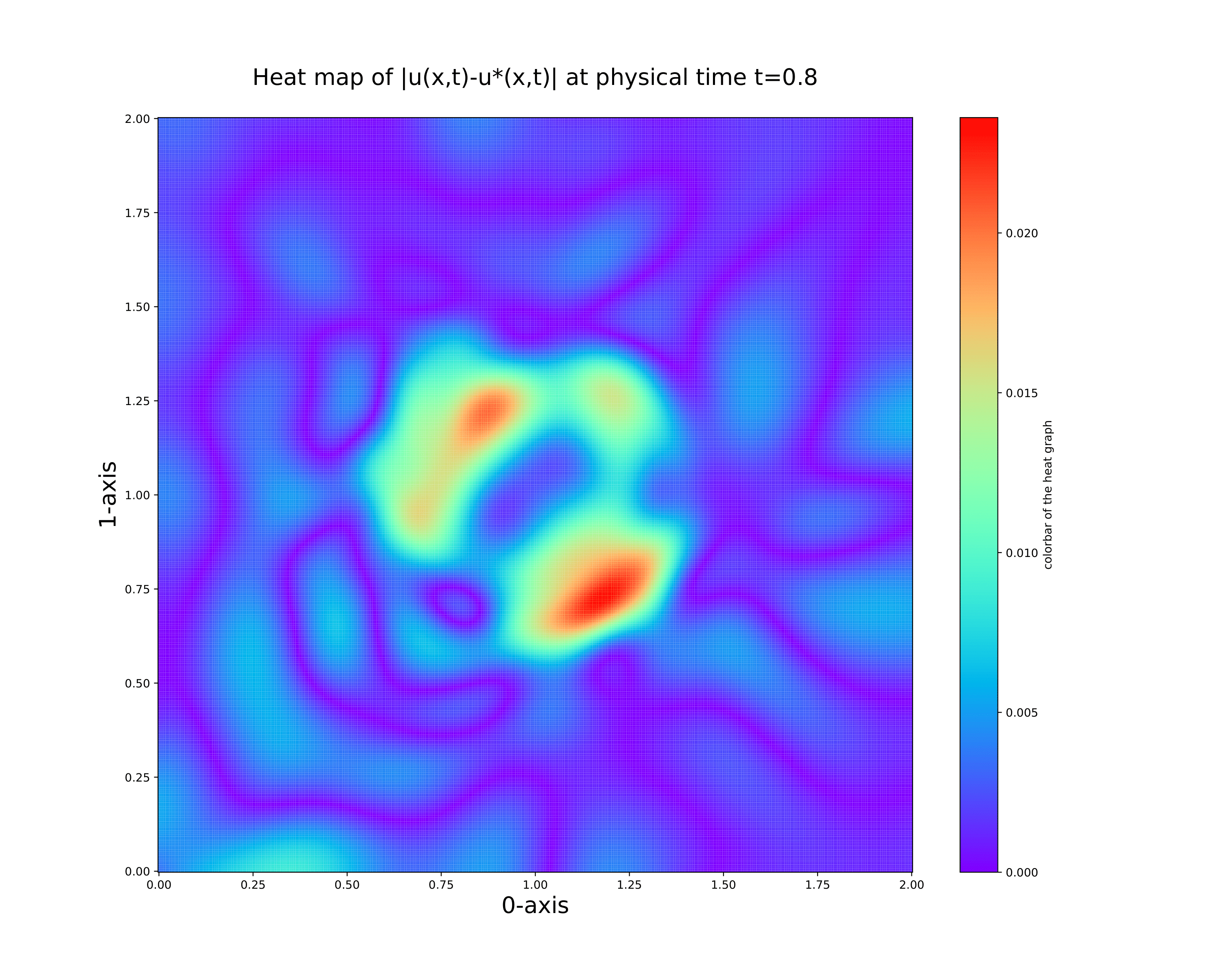}
    \end{subfigure}
        \begin{subfigure}[b]{0.18\textwidth}
        \includegraphics[trim={0 0 0 4.5cm},clip, width=\textwidth]{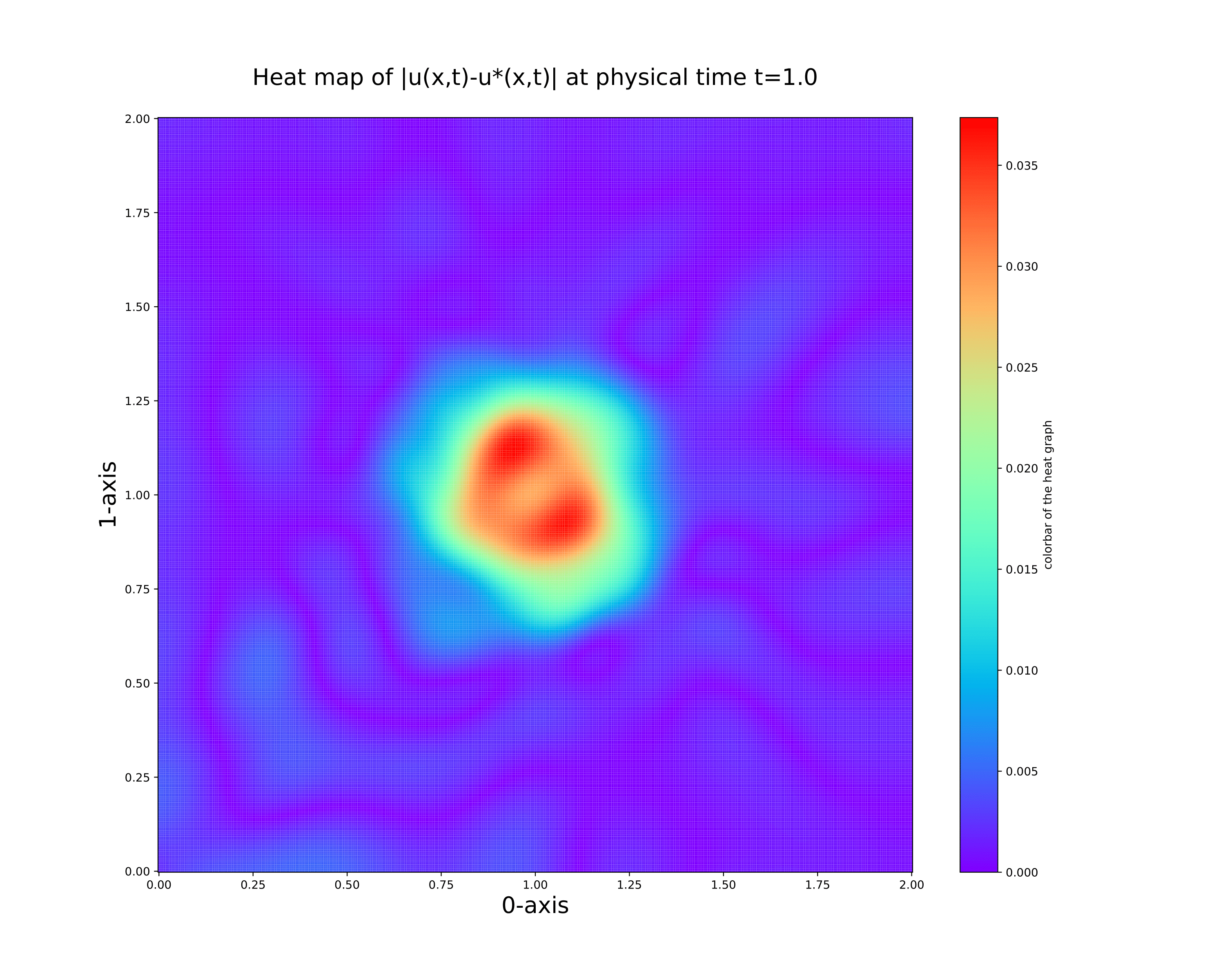}
    \end{subfigure}
        \begin{subfigure}[b]{0.18\textwidth}
        \includegraphics[trim={0 0 0 4.5cm},clip, width=\textwidth]{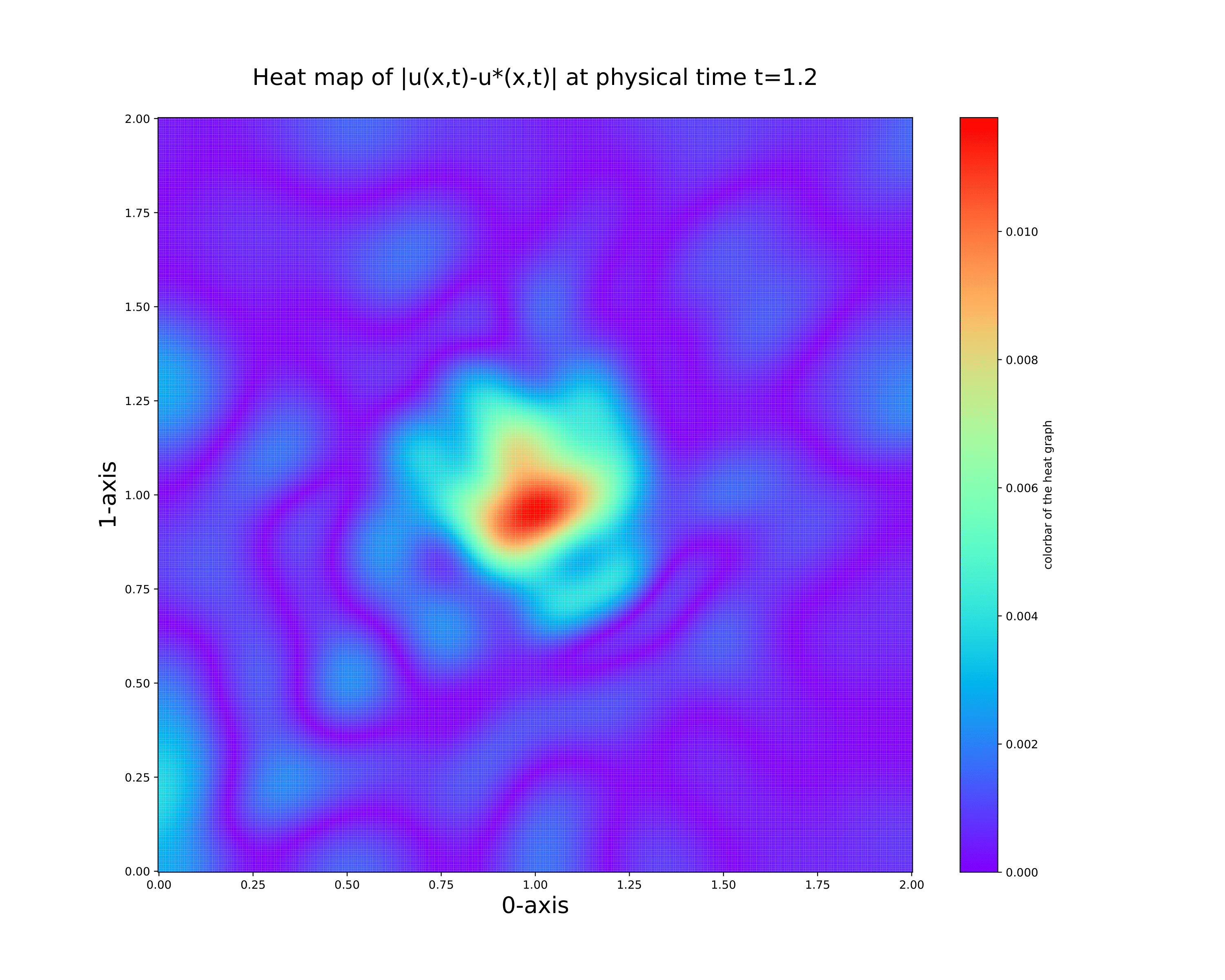}
    \end{subfigure}
    \caption{\textbf{Up row}: plots of $u_{\theta_k}$ (blue) together with the numerical solution $\{U_{ij}^k\}$ solved from implicit finite-difference scheme (red) on $\Omega$ at physical time $0.2, 0.6, 0.8, 1.0, 1.2$. The initial function $u_0$ is marked with green color; \textbf{Down row}: heatmaps of the error term $|u_{\theta_k}(\cdot)-U^k|$ at physical time $0.2, 0.6, 0.8, 1.0, 1.2$.  }
    \label{fig: AC3}
\end{figure}

{\color{black}
\begin{remark}
  Recall the semi-log plots of $\sqrt{\mathrm{MSE\; loss}}$ vs. training time for both 1D and 2D Allen Cahn equations presented in Figure \ref{subfig: AC semi-log MSE vs comptime} and \ref{subfig: log MSE vs iteration}, respectively. It is worth noting that in the 1D case, the accuracy of the numerical solution deteriorates as the physical time $t_k$ increases. This behavior can be attributed to the stationary profile of the equation: its highly localized and strongly curved geometry makes it increasingly difficult for the MLP to accurately approximate. In contrast, the stationary profile associated with the 2D equation collapses to a flat plane, making it fairly easy for the MLP to approximate. Thus, we observe increasing accuracy as $t_k$ increases.
\end{remark}}

\vspace{-0.3cm}
\section{Primal-Dual algorithm using Adam optimizer for Optimal Transport problem}\label{append: PD Adam OT}
In this section, we briefly describe the PD-Adam algorithm tested in Section \ref{sec: MA}. Recall the loss functional $\mathcal L(T, \varphi)$ defined in \eqref{MA loss functional L }, we parametrize both the map $T$ and the dual function $\varphi$ by neural networks $T_\theta, \varphi_\eta$. We aim at solving the following saddle point problem
\vspace{-0.2cm}
\begin{align}
  \max_\eta \min_\theta ~ \mathcal{L}(T_\theta,\varphi_\eta) := &  \int_{\mathbb{R}^d} \frac12  \|x - T(x)\|^2~\rho_0\,\dd x + \int_{\mathbb{R}^d} \varphi(T(x))~\rho_0\,\dd x - \int_{\mathbb{R}^d} \varphi(y)~\rho_1\,\dd y  \nonumber  \\ 
  \approx & \frac{1}{N} \sum_{i=1}^N \frac{1}{2}\|{\boldsymbol{X}_i} - T_\theta({\boldsymbol{X}_i}) \|^2 - \varphi_\eta(T_\theta( {\boldsymbol{X}_i})) + \varphi_\eta( {\boldsymbol{Y}_i}),
  \label{OT loss function PD-adam}
\end{align}
\vspace{-0.04cm}
where $N$ is the size of the datasets, $\{\boldsymbol{X}_i\}_{i=1}^N, \{\boldsymbol{Y}_i\}_{i=1}^N$ are samples drawn by $\rho_0$ and $\rho_1$. The PD-Adam algorithm is summarized in Algorithm \ref{alg: PD-adam}. 

\begin{algorithm}[!htb]
\caption{Computing optimal Monge map from $\rho_a$ to $\rho_b$}
\begin{algorithmic}
\State \textbf{Input}:
Marginal distributions $\rho_0$ and $\rho_1$, learning rate $lr_u, lr_{\varphi}$ of the Adam algorithm; Batch size $N$, total iteration number $N_{iter}$.
\State Initialize ${T}_{\theta}, {\varphi}_{\eta}$.
\For{ $iter=1$ to $N_{iter}$ }
\State Sample $\{\boldsymbol{X}_i\}_{i=1}^N$ $\sim \rho_a$. Sample $\{\boldsymbol{Y}_i\}_{i=1}^N$ $\sim \rho_b$.
\State Update ${\theta}$ to {decrease} \eqref{OT loss function PD-adam} by Adam algorithm with learning rate $lr_u$ for $K_1$ steps.
\State Update ${\eta}$ to {increase} \eqref{OT loss function PD-adam} by Adam algorithm with learning rate $lr_\varphi$ for $K_2$ steps.
\EndFor
\State \textbf{Output}: The transport map $T_\theta$.
\end{algorithmic}
\label{alg: PD-adam}
\end{algorithm}

In all tests, we always set $K_1=K_2=1$. We summarize all the other hyperparameters of the PD-Adam algorithm in Section \ref{sec: MA} in Table \ref{tab: PD-adam}. 
\begin{table}[htb!]
    \centering
    \begin{tabular}{|c|c|c|c|c|c|}
    \hline
         \multicolumn{2}{|c|}{} & $lr_u, lr_\varphi $ & $N_{iter}$ & $N$ &  NN architecture  \\  \hline \hline 
       \multicolumn{2}{|c|}{ \ref{sec: MA1D} (1D)} & $0.5\cdot 10^{-3}, 0.5\cdot 10^{-3}$ & $40000$ & $800$ & $\texttt{MLP}_{\textrm{PReLU}}(1, 50, 1, 3)$ \\  \hline
       \multicolumn{2}{|c|}{ \ref{sec: MA5D} (5D)} & $0.5\cdot 10^{-4}, 0.5\cdot 10^{-4} $ &  $ 200000 $  & $2000$ &  $\texttt{MLP}_{\textrm{PReLU}}(5, 80, 5, 4)$  \\  \hline
       \multirow{2}{*}{ \ref{sec: MA10D50D} } & {(10D)} & $0.5\cdot 10^{-4}, 0.5\cdot 10^{-4}$ & $100000$ & $2000$ & $\texttt{MLP}_{\textrm{PReLU}}(10, 120, 10, 6)$ \\    \cline{2-6}
        & {(50D)} & $10^{-5}, 10^{-5}$ & $300000$ & $2000$ & $\texttt{MLP}_{\textrm{PReLU}}(50, 120, 50, 6)$ \\    \hline
    \end{tabular}
    \caption{Some hyperparameters used in the PD-Adam algorithm tested in Section \ref{sec: MA}.}
    \label{tab: PD-adam}
\end{table}

\vspace{-0.2cm}

\section{Further numerical results regarding Section \ref{sec: MA10D50D}}\label{append: MA}
For the OT problem from $\rho_0$ to $\rho_b$, 
we provide the intermediate results obtained by the NPDG algorithms as well as the PD-Adam algorithms in Figure \ref{append fig: MA50D nonequal gaussians heat maps}. PD-Adam behaves unstably in this example, while NPDG method performs robustly for both preconditions.

\vspace{0.2cm}

\begin{figure}[ht!]
     \centering
     \begin{subfigure}{0.2\textwidth}
         \captionsetup{justification=centering}
         \centering
         \includegraphics[width=\textwidth, valign=t]{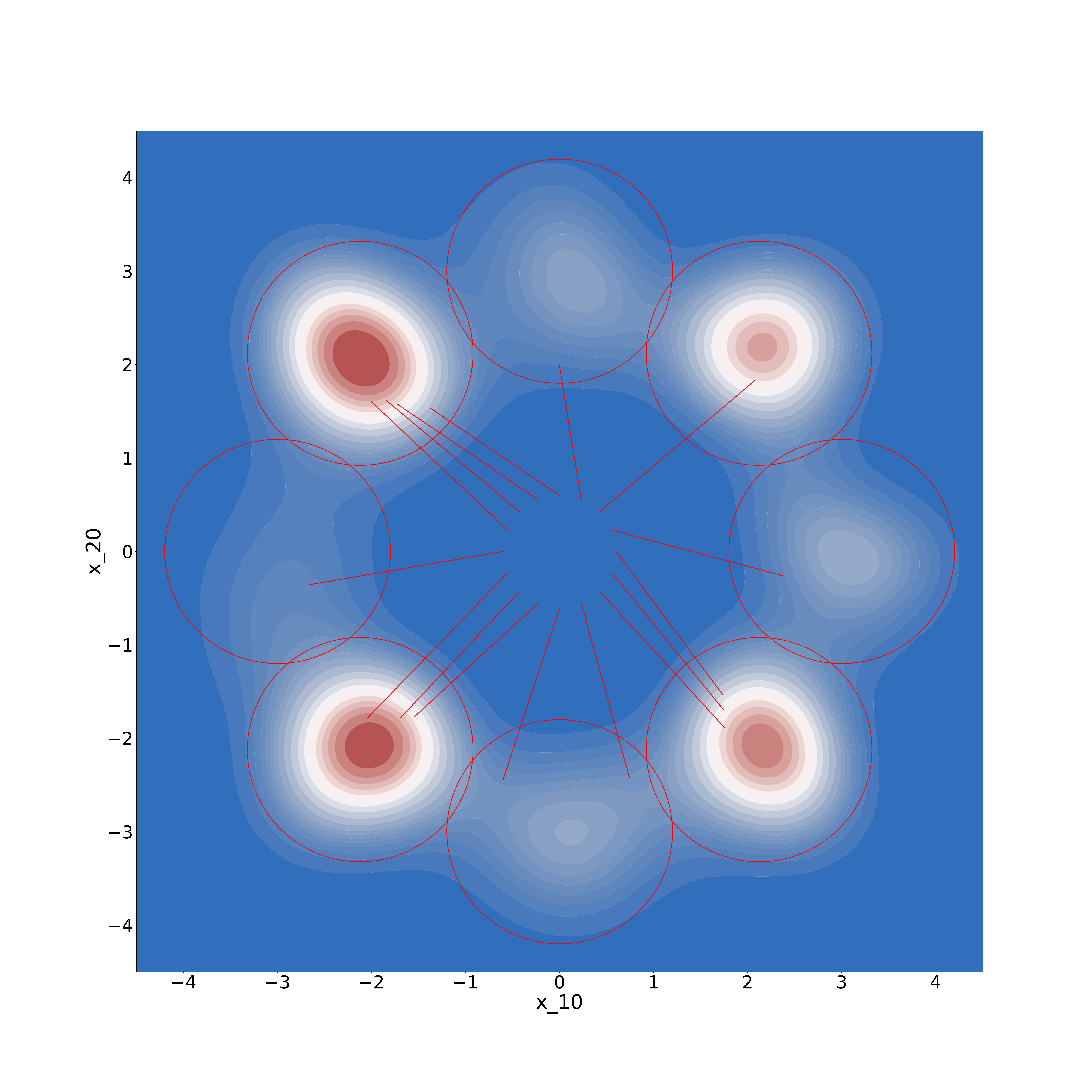}
     \end{subfigure}
     \hfill
     \begin{subfigure}{0.2\textwidth}
     \captionsetup{justification=centering}
         \centering
         \includegraphics[width=\textwidth, valign=t]{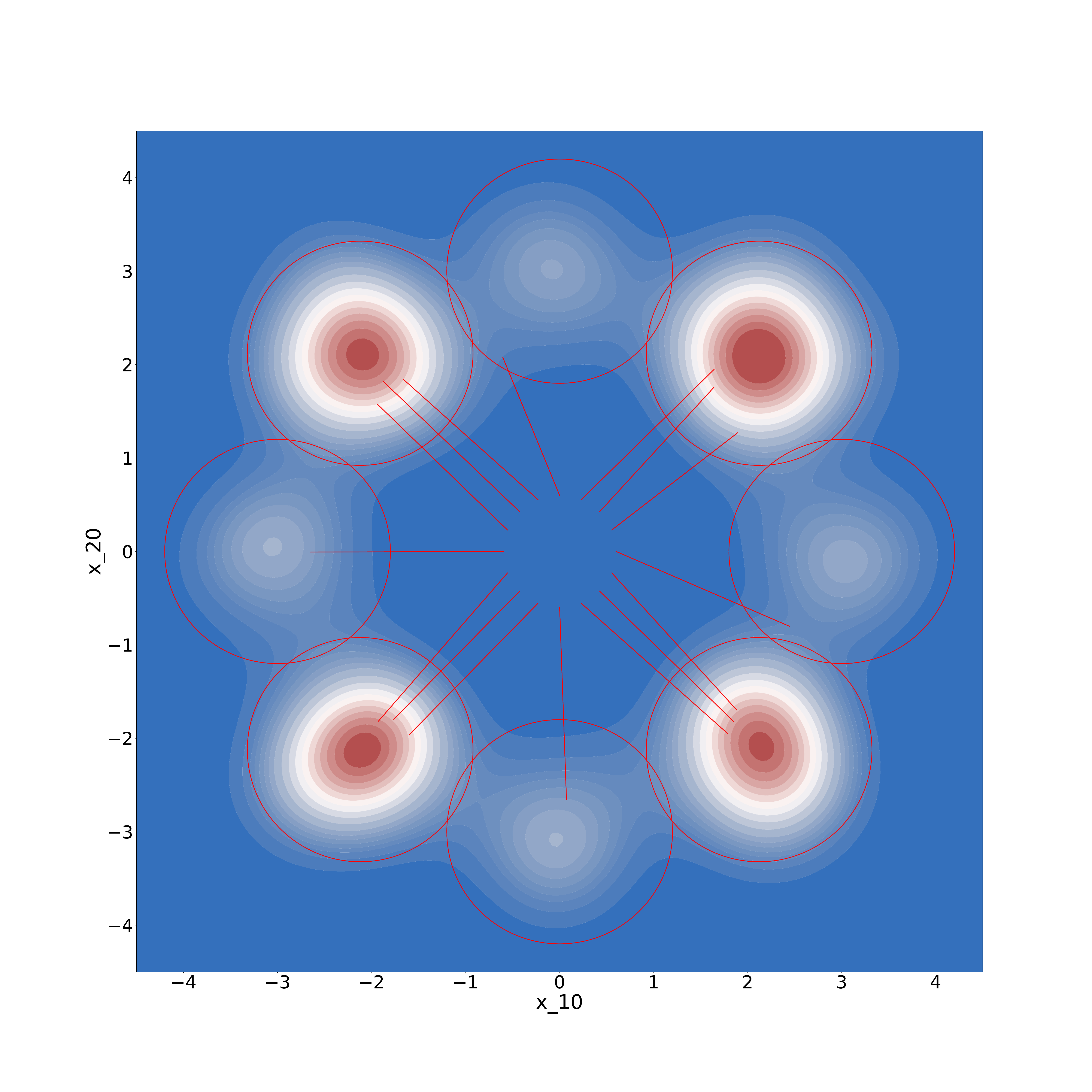}
     \end{subfigure}
     \hfill
     \begin{subfigure}{0.2\textwidth}
     \captionsetup{justification=centering}
         \centering
         \includegraphics[width=\textwidth, valign=t]{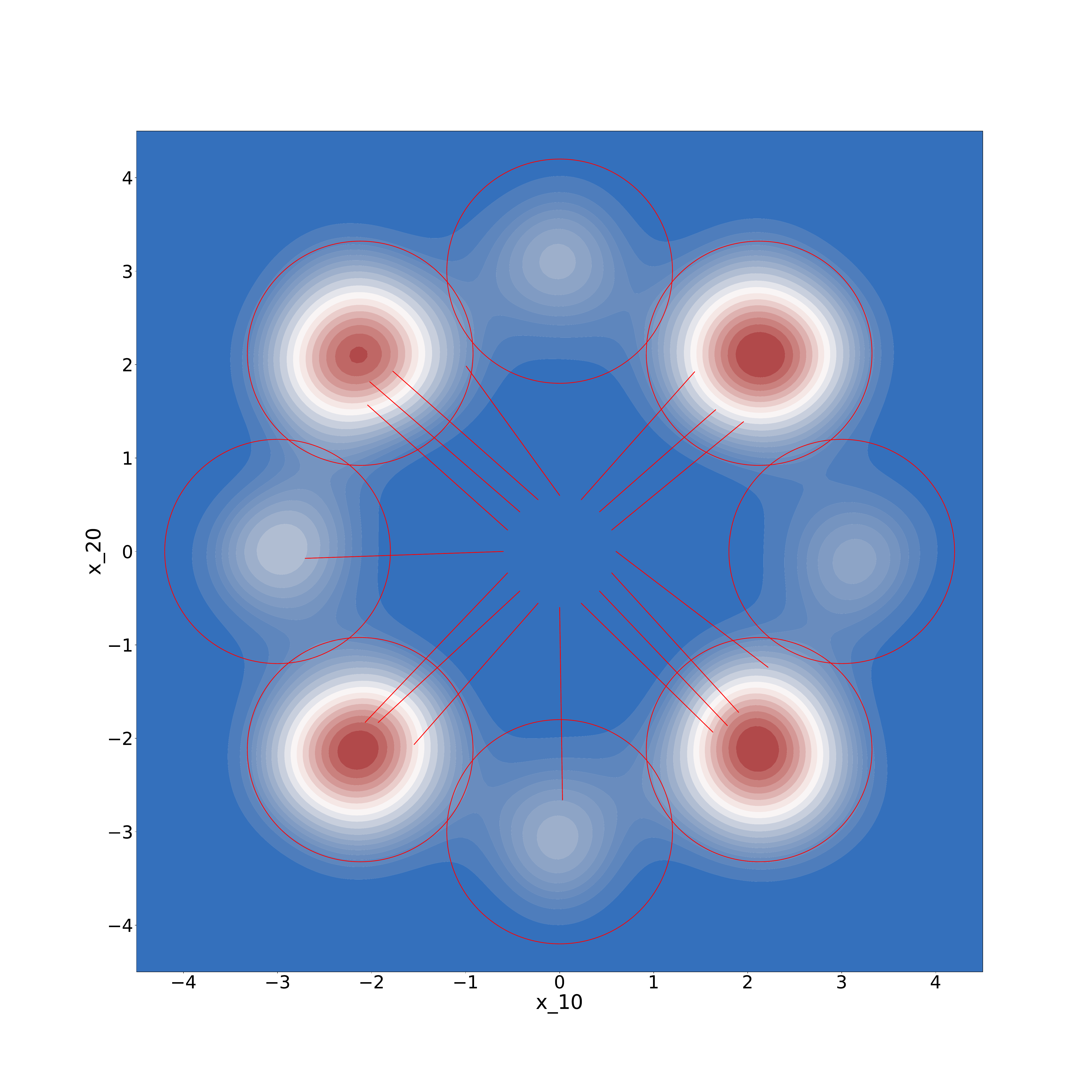}
     \end{subfigure}
     \hfill
     \begin{subfigure}{0.2\textwidth}
     \captionsetup{justification=centering}
         \centering
         \includegraphics[width=\textwidth, valign=t]{MongeAmpere/high_dim_gaussian_to_mixed_gaussian/50D_nonequal_gaussians/50D_NPDG_precond_id_nonequal_gaussians_20000.pdf}
     \end{subfigure}
     \centering
     \begin{subfigure}{0.2\textwidth}
     \captionsetup{justification=centering}
         \centering
         \includegraphics[width=\textwidth, valign=t]{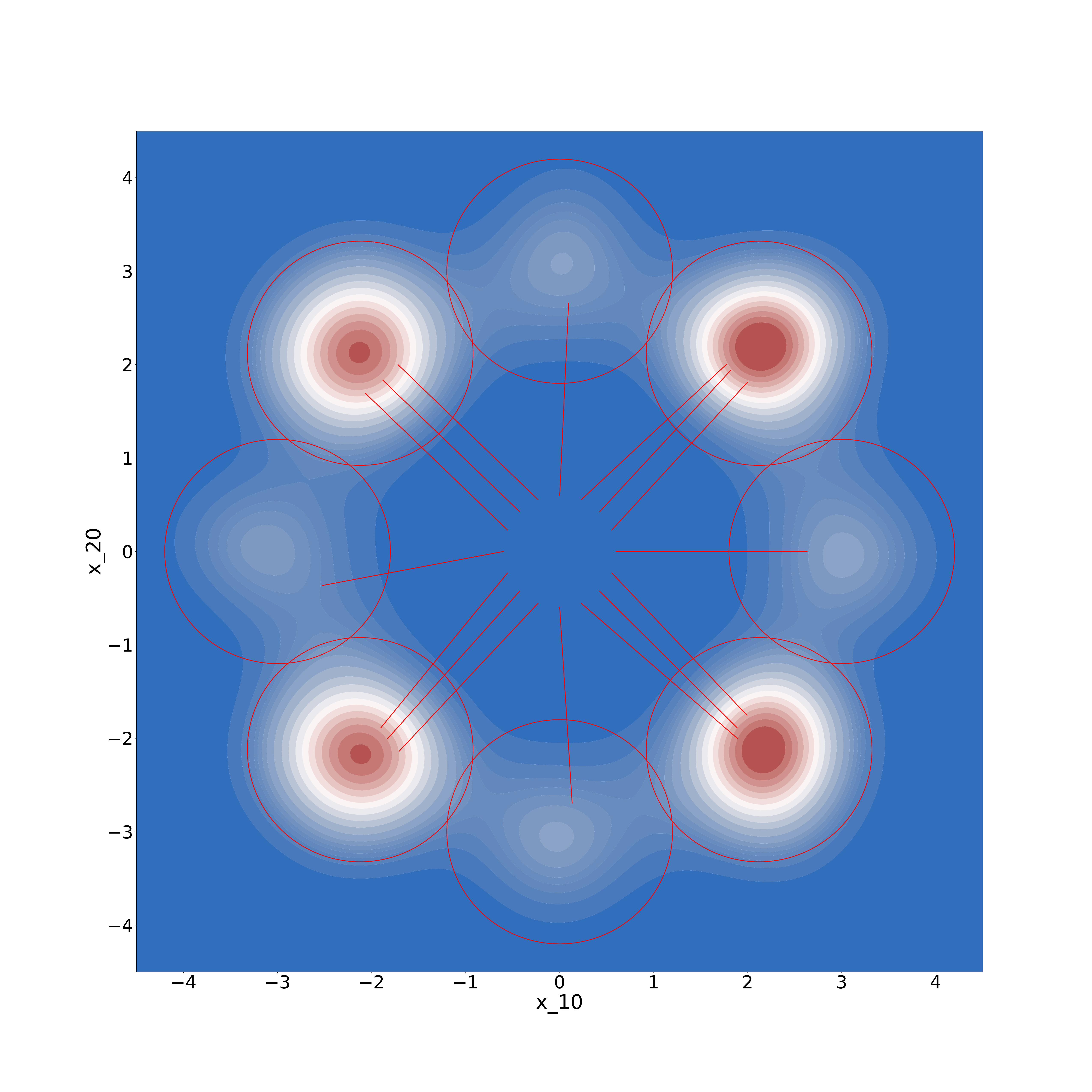}
         \caption{NPDG iter \\ $5000$}
         \label{subfig: MA50D nonequal gaussian NPDG grad 5000}
     \end{subfigure}
     \hfill
     \begin{subfigure}{0.2\textwidth}
     \captionsetup{justification=centering}
         \centering
         \includegraphics[width=\textwidth, valign=t]{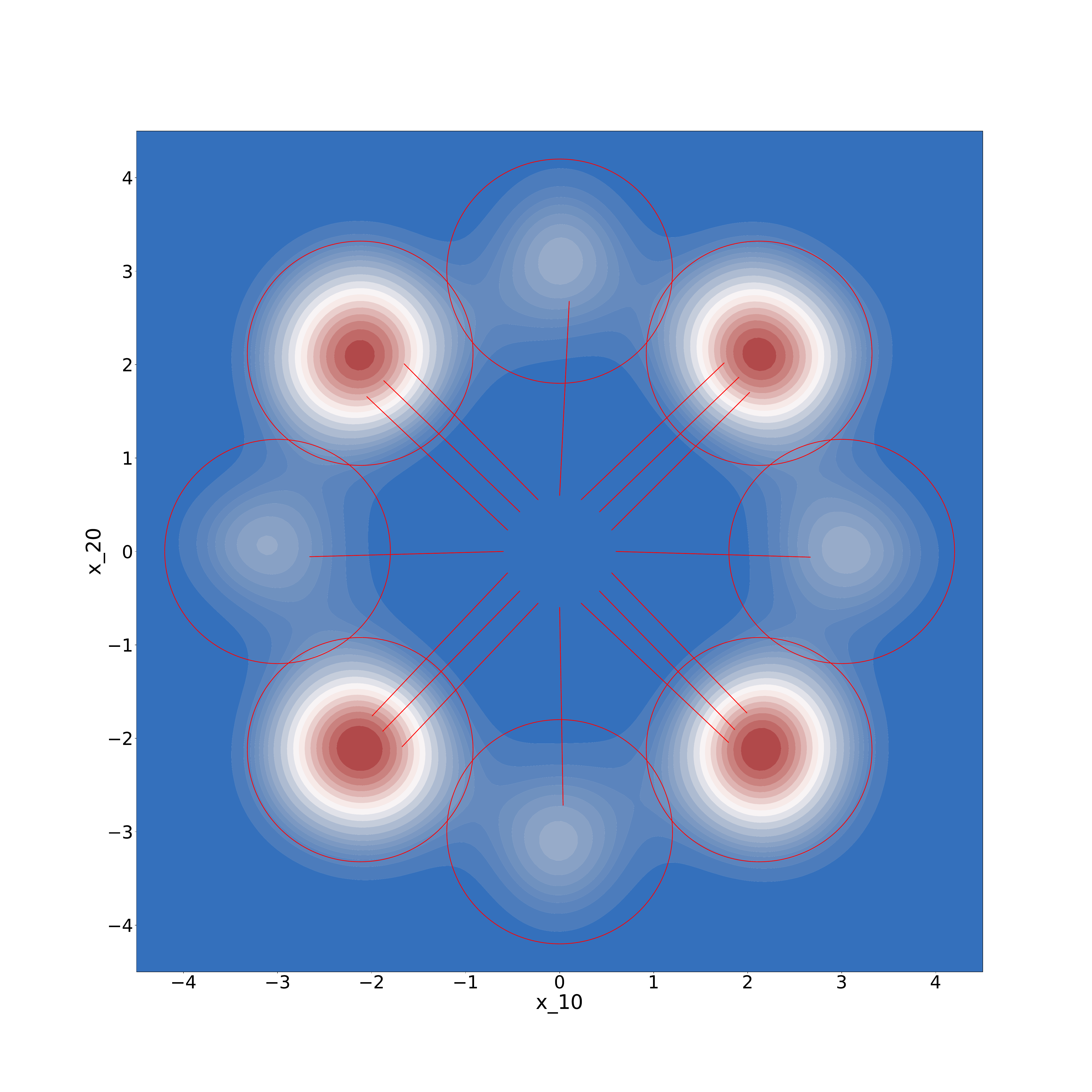}
         \caption{NPDG iter $10000$}
         \label{subfig: MA50D nonequal gaussian NPDG grad 10000}
     \end{subfigure}
     \hfill
     \begin{subfigure}{0.2\textwidth}
     \captionsetup{justification=centering}
         \centering
         \includegraphics[width=\textwidth, valign=t]{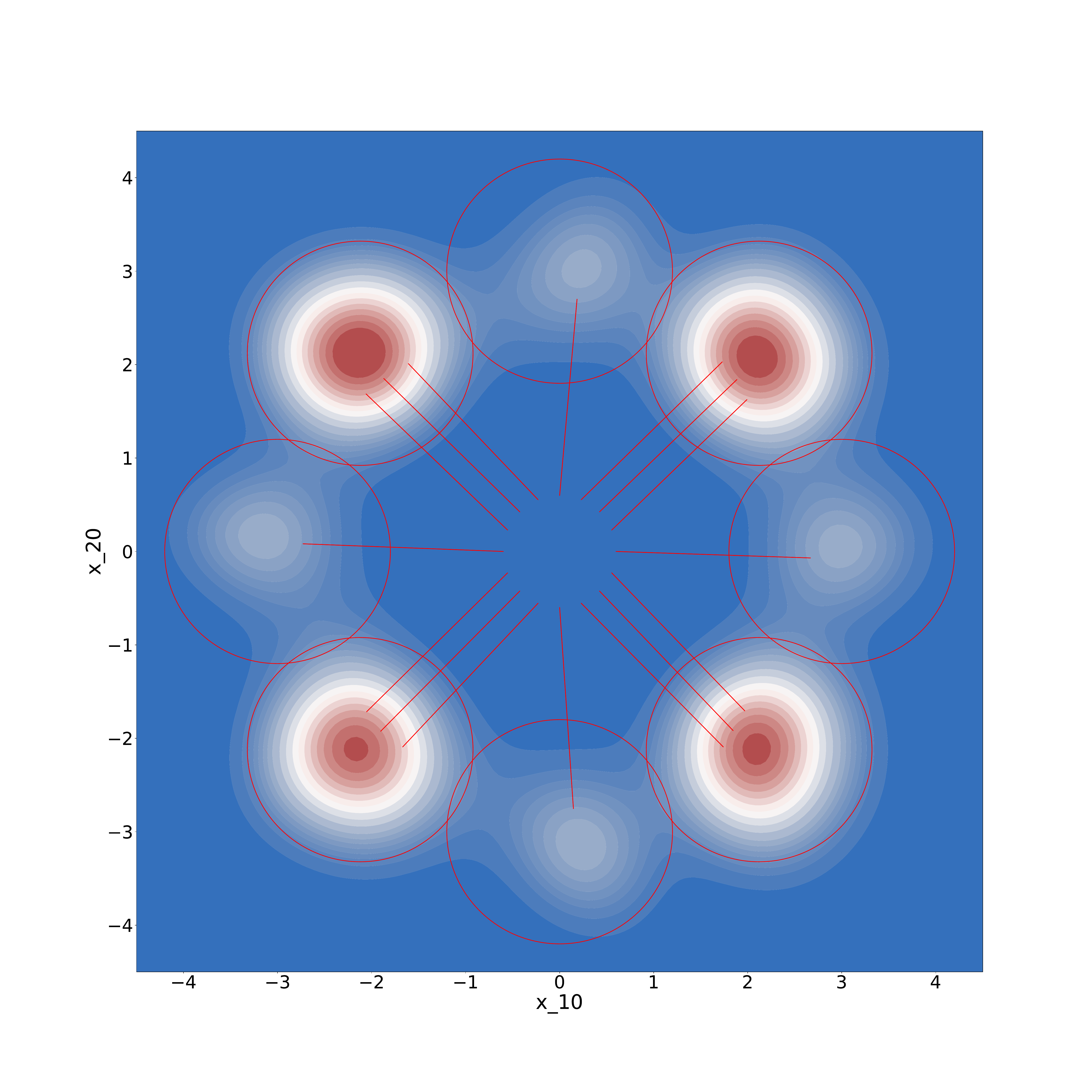}
         \caption{NPDG iter $15000$}
         \label{subfig: MA50D nonequal gaussian NPDG grad 15000}
     \end{subfigure}
     \hfill
     \begin{subfigure}{0.2\textwidth}
     \captionsetup{justification=centering}
         \centering
         \includegraphics[width=\textwidth, valign=t]{MongeAmpere/high_dim_gaussian_to_mixed_gaussian/50D_nonequal_gaussians/50D_NPDG_precond_grad_nonequal_gaussians_20000.pdf}
         \caption{NPDG iter $20000$}
         \label{subfig: MA50D nonequal gaussian NPDG grad 20000}
     \end{subfigure}
     \centering     
     \begin{subfigure}{0.2\textwidth}
     \captionsetup{justification=centering}
         \centering
         \includegraphics[width=\textwidth, valign=t]{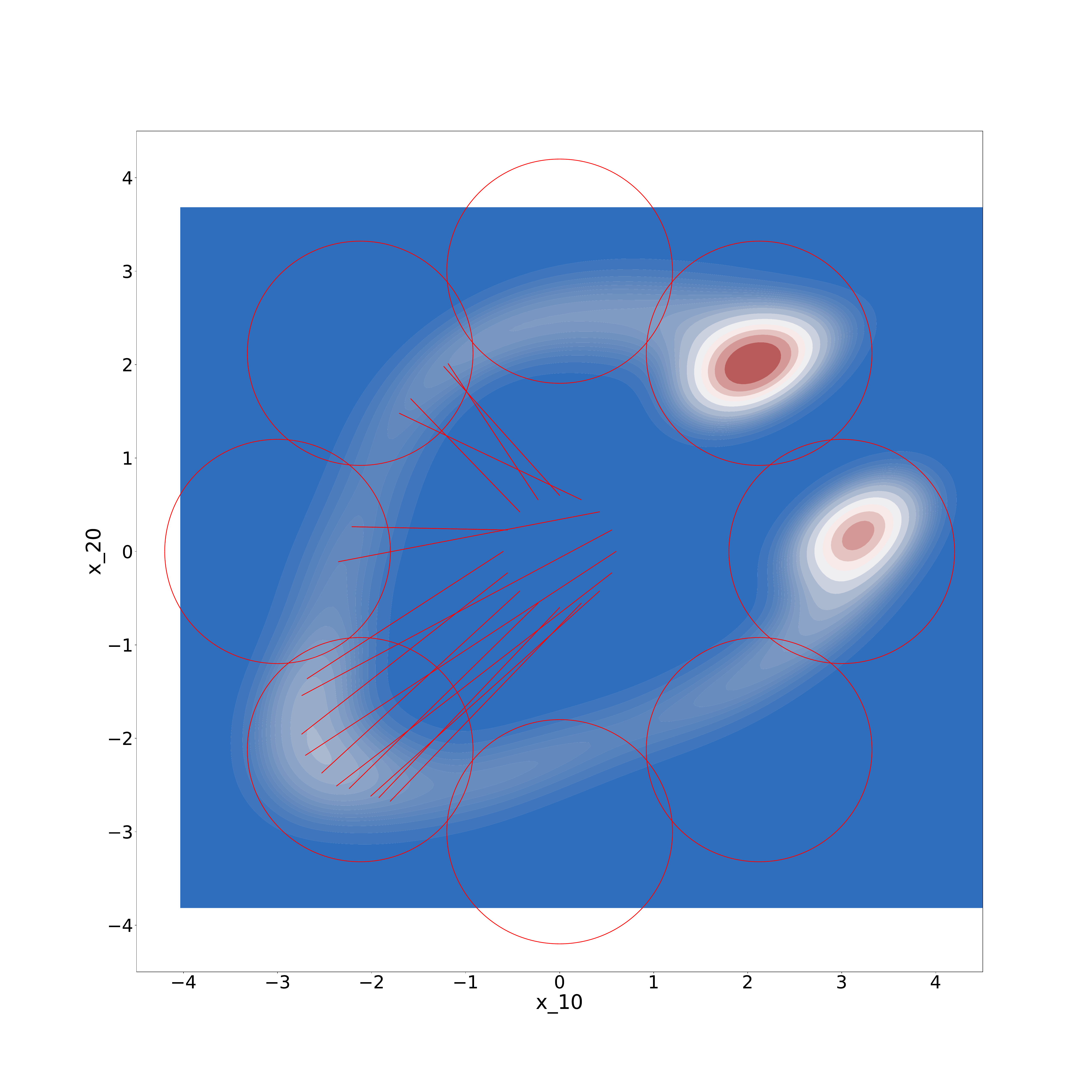}
         \caption{PD-Adam iter $150000$}
         \label{subfig: MA50D nonequal gaussian PD adam 150000}
     \end{subfigure}
     \hfill
     \begin{subfigure}{0.2\textwidth}
     \captionsetup{justification=centering}
         \centering
         \includegraphics[width=\textwidth, valign=t]{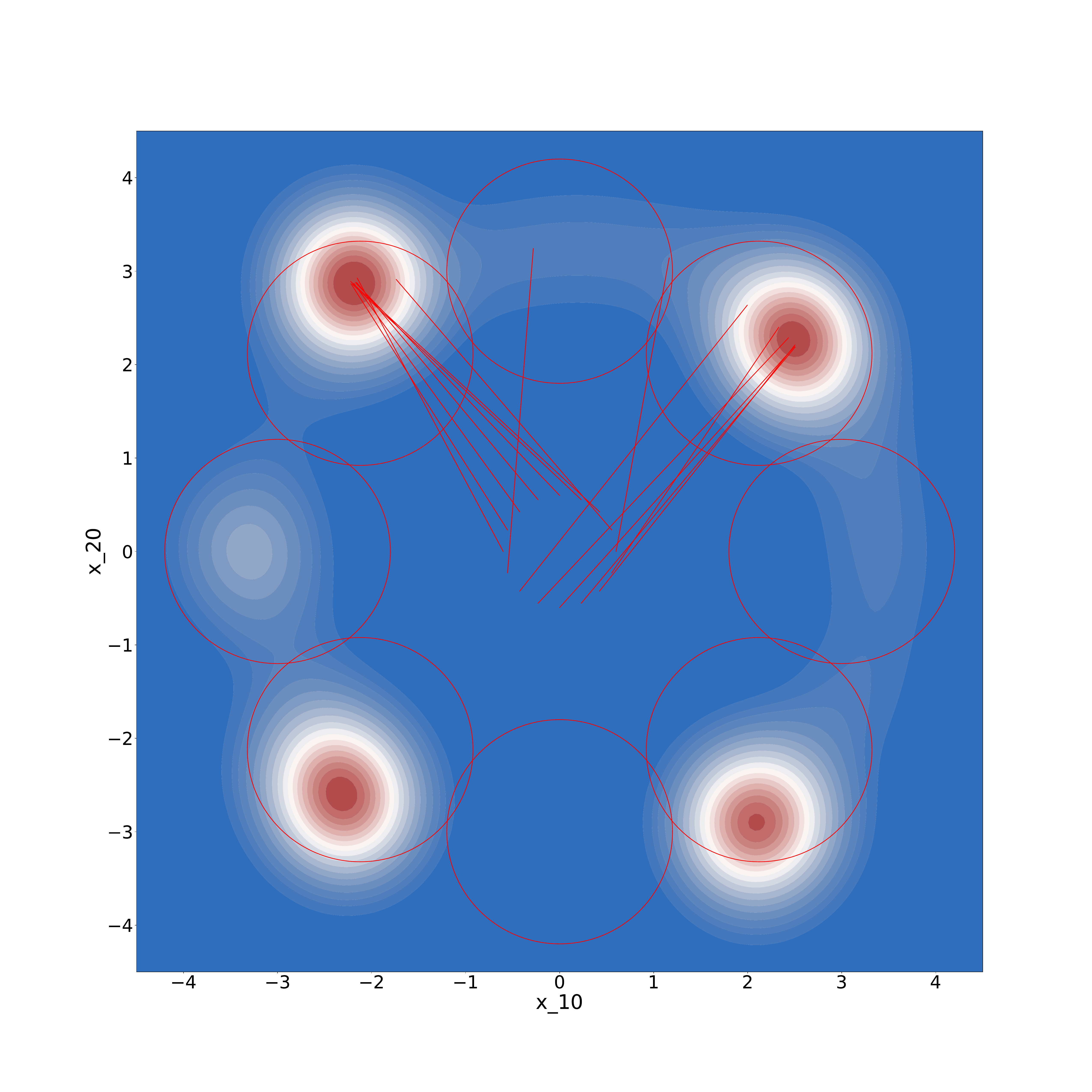}
         \caption{PD-Adam iter $200000$}\label{subfig: MA50D nonequal gaussian PD adam 200000}
     \end{subfigure}
     \hfill
     \begin{subfigure}{0.2\textwidth}
     \captionsetup{justification=centering}
         \centering
         \includegraphics[width=\textwidth, valign=t]{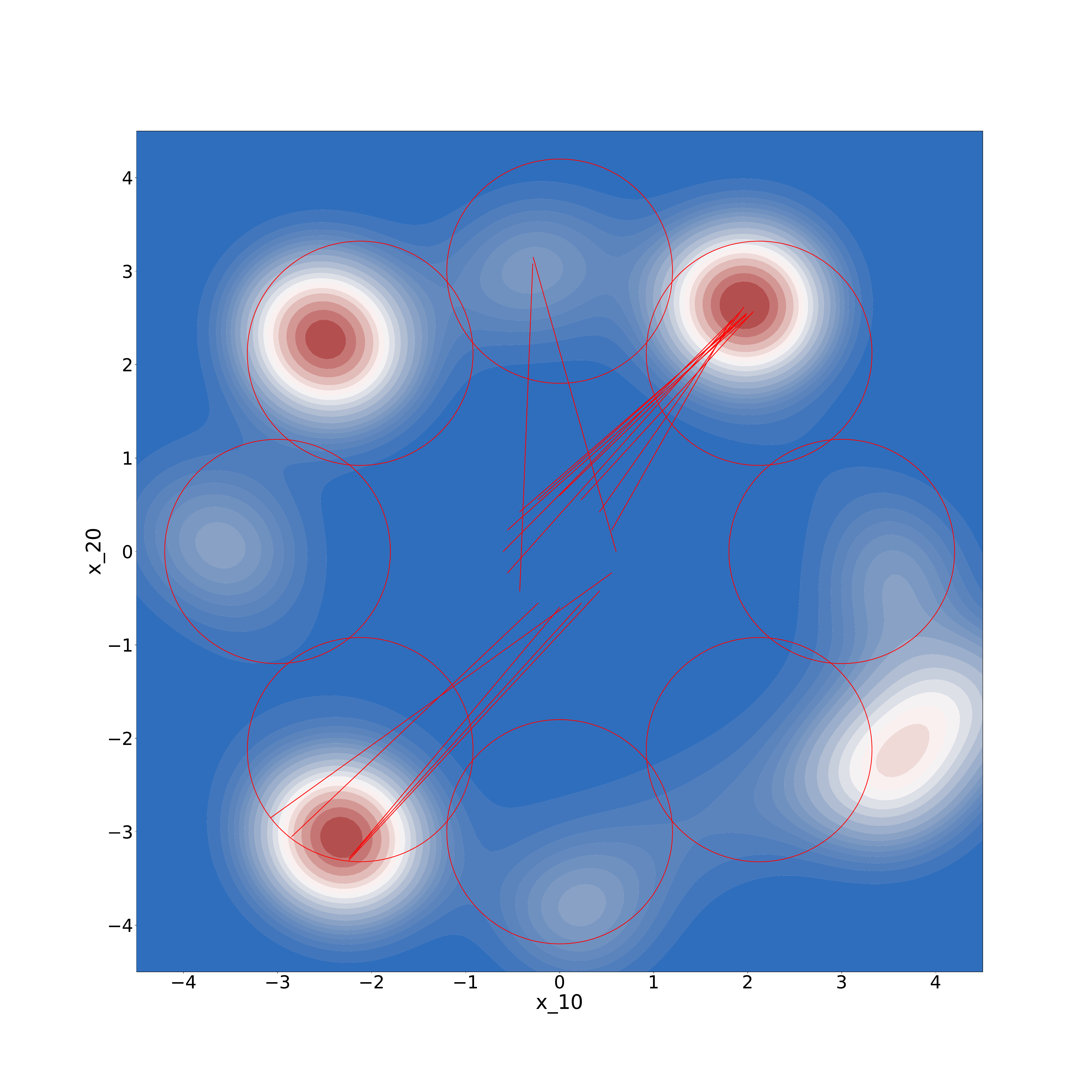}
         \caption{PD-Adam iter $250000$}\label{subfig: MA50D nonequal gaussian PD adam 250000}
     \end{subfigure}
     \hfill
     \begin{subfigure}{0.2\textwidth}
     \captionsetup{justification=centering}
         \centering
         \includegraphics[width=\textwidth, valign=t]{MongeAmpere/high_dim_gaussian_to_mixed_gaussian/50D_nonequal_gaussians/50D_PD_adam_nonequal_gaussians_300000.pdf}
         \caption{PD-Adam iter $300000$}\label{subfig: MA50D nonequal gaussian PD adam 300000}
     \end{subfigure}
     \caption{OT problem from $\rho_0$ to $\rho_b$: Plots of the pushforwarded densities $T_{\theta \sharp}\rho_0$ of the computed $T_\theta$ obtained by NPDG method (1st row (\eqref{pull Id as precond Monge problem} as preconditioning) \& 2nd row (\eqref{canonical precond for Monge problem} as preconditioning)) and PD-Adam method (3rd row). All figures are plotted on the $10-20$ coordinate plane.} \label{append fig: MA50D nonequal gaussians heat maps}
\end{figure}

\vskip 0.2in
\bibliography{references}

\end{document}